\documentclass[10pt]{amsart}
\usepackage{etex}
\usepackage{mathtools,tensor}
\usepackage{latexsym,amssymb,amsthm,amsmath,enumerate,epsfig,setspace,bbding,color,graphicx,caption}
\usepackage[noadjust]{cite}
\usepackage{stmaryrd}
\usepackage[all]{xy}
\usepackage[retainorgcmds]{IEEEtrantools}
\usepackage{hyperref}
\usepackage{subfiles}
\usepackage[margin=10pt,font=small,labelfont=bf]{caption}
\usepackage{tikz-cd}
\usepackage{tkz-graph}
\usepackage{float}
\usetikzlibrary{arrows}

\tikzset{
    labl/.style={anchor=south, rotate=90, inner sep=.4mm}
}

\newcommand{\Addresses}{{
  \bigskip
  \footnotesize

  \textsc{Max Planck Institute for Mathematics, Bonn, Germany}\par\nopagebreak
  \textit{E-mail address}: \texttt{maarten.mol.math@gmail.com}

}}
\newtheorem{thm}{Theorem}[section]
\newtheorem{Thm}{Theorem}
\newtheorem{cor}[thm]{Corollary}

\newtheorem{prop}[thm]{Proposition}
\newtheorem{lemma}[thm]{Lemma}
\newtheorem*{Q}{Question}

\theoremstyle{definition}
\newtheorem{defi}[thm]{Definition}

\newtheorem{Ex}{Example}
\newtheorem{ex}[thm]{Example}
\newtheorem{rem}[thm]{Remark}
\newtheorem{Rem}{Remark}

\newtheorem*{Defi}{Definition}

\pgfmathsetmacro{\shift}{0.65ex}
\renewcommand{\L}{\mathcal{L}}

\newcommand{\G}{\mathcal{G}}

\newcommand{\U}{\text{$\mathcal{U}$}}
\newcommand{\g}{\mathfrak{g}}
\newcommand{\h}{\mathfrak{h}}
\renewcommand{\t}{\mathfrak{t}}
\newcommand{\F}{\mathcal{F}}

\renewcommand{\H}{\mathcal{H}}

\newcommand{\T}{\mathbb{T}}
\newcommand{\A}{\mathcal{A}}
\newcommand{\No}{\mathcal{N}}

\newcommand{\R}{\mathbb{R}}
\newcommand{\C}{\mathbb{C}}

\newcommand{\Z}{\mathbb{Z}}

\renewcommand{\dim}{\text{dim}}
\renewcommand{\ker}{\textrm{Ker}}
\newcommand{\im}{\textrm{Im}}

\newcommand{\rk}{\textrm{rank}}

\renewcommand{\O}{\mathcal{O}}

\renewcommand\S{\mathcal{S}}
\renewcommand\phi{\varphi}
\newcommand{\sT}{\mathcal{T}}

\newcommand{\GL}{\textrm{GL}}

\renewcommand{\d}{\textrm{d}}

\newcommand{\E}{\mathcal{E}}

\renewcommand{\c}{\mathfrak{c}}
\newcommand{\B}{\mathcal{B}}

\begin{document}
\title[\resizebox{5.5in}{!}{On the classification of multiplicity-free Hamiltonian actions by regular proper symplectic groupoids}]{On the classification of multiplicity-free Hamiltonian actions by regular proper symplectic groupoids}
\author[\resizebox{0.65in}{!}{Maarten Mol}]{Maarten Mol}
\date{}
\maketitle
\begin{abstract} In this paper we study a natural generalization of symplectic toric manifolds in the context of regular Poisson manifolds of compact types. To be more precise, we consider a class of multiplicity-free Hamiltonian actions by regular proper symplectic groupoids that we call faithful. Given such a groupoid, we classify its faithful multiplicity-free Hamiltonian actions in terms of what we call Delzant subspaces of its orbit space --  
certain `suborbifolds with corners' satisfying the Delzant condition relative to the integral affine orbifold structure of the orbit space. This encompasses both the classification of symplectic toric manifolds (due to Delzant) in terms of Delzant polytopes 
and the classification of proper Lagrangian fibrations over an integral affine base manifold (due to Duistermaat) in terms of a sheaf cohomology group. Each Delzant subspace comes with an orbifold version of this cohomology, the degree one part of which classifies faithful multiplicity-free Hamiltonian actions with momentum map image equal to the Delzant subspace, provided there exists such an action. The obstruction to existence is encoded by a degree two class in this cohomology: the Lagrangian Dixmier-Douady class. In addition to the above, we introduce another invariant, which leads to a variation of our classification result involving only classical sheaf cohomology and the group cohomology of certain modules for the isotropy groups of the groupoid. 

\end{abstract}
\tableofcontents

\section*{Introduction}
This paper concerns the classification of {multiplicity-free} Hamiltonian actions of regular proper symplectic groupoids. Recall that the notion of Hamiltonian action for symplectic groupoids (introduced in \cite{WeMi}) unifies various momentum map theories appearing in Poisson geometry, with the common feature that in each case the momentum map is a Poisson map:
\begin{equation*} J:(S,\omega)\to (M,\pi)
\end{equation*} from a symplectic manifold into a specified Poisson manifold $(M,\pi)$, equipped with a Hamiltonian action of a natural symplectic groupoid integrating $(M,\pi)$ (see e.g. \cite{We7,Zu} and the references therein). Here, we focus on such actions for which the Poisson structure is regular and the symplectic groupoid is proper, using the framework developed in \cite{CrFeTo} for such symplectic groupoids and their associated Poisson structures. Furthermore, the actions that we consider are multiplicity-free, in the sense of \cite{GS3}. The main point of this paper is to give a classification of a large class of such actions, that we call faithful.\\ 

To explain this classification we first focus on actions of symplectic torus bundles, which form a class of symplectic groupoids analogous to the class of Lie groups formed by tori. Like tori, these admit a particularly concrete description in terms of lattices. Indeed, a symplectic torus bundle $(\sT,\Omega)$ induces an integral affine structure on its base $M$, encoded by a Lagrangian lattice bundle $\Lambda$ in the cotangent bundle $T^*M$, and $(\sT,\Omega)$ is entirely encoded by $\Lambda$. More precisely, it is canonically isomorphic to the symplectic torus bundle $(\sT_\Lambda,\Omega_\Lambda)$ over $M$, where $\sT_\Lambda:=T^*M/\Lambda$ denotes the bundle of tori associated to $\Lambda$ (a groupoid over $M$ with multiplication given by fiberwise addition) and $\Omega_\Lambda$ is the symplectic form induced by the canonical symplectic form $\d \lambda_\textrm{can}$ on the cotangent bundle $T^*M$ (see e.g. \cite[Proposition 3.1.6]{CrFeTo}). The starting point for our classification in the case of $(\sT,\Omega)$-spaces (Theorem \ref{classtorsymptorbunthm} below) was to find the connection between two classical results: the classification of symplectic toric manifolds in terms of Delzant polytopes (due to Delzant \cite{De}) and the classification of proper Lagrangian fibrations via their Lagrangian Chern class (implicit in Duistermaat's work \cite{Du}), recalled in Examples \ref{symptormanex} and \ref{proplagrfibrex} below. {Inspired by these, we consider Hamiltonian $(\sT,\Omega)$-spaces $J:(S,\omega)\to M$ with the three properties below. }
\begin{itemize}\item[i)] The $\sT$-action is free on a dense subset of $S$.
\item[ii)] $\dim(S)=2\dim(M)$ and the fibers of $J$ are connected. 
\item[iii)] The momentum map $J:S\to M$ is proper as a map into its image. 
\end{itemize} {We call these \textbf{faithful toric} $(\sT,\Omega)$-spaces.} An equivalent set of conditions consists of i) together with the two conditions below (cf. Appendix A).
\begin{itemize} \item[ii')] The $\sT$-orbits coincide with the fibers of $J$.
\item[iii')] The map $S/\sT\to M$ induced by $J$ is a topological embedding. 
\end{itemize} Note that the last condition here is purely topological, as the second condition means that the induced map $S/\sT\to M$ is injective. {Like symplectic toric manifolds, faithful toric $(\sT,\Omega)$-spaces are closely related to integrable systems. For instance, in local coordinates for $M$, the momentum map $J$ of such a $(\sT,\Omega)$-space is an integrable system with only elliptic singularities.}  

\begin{Thm}[Classification of faithful toric spaces]\label{classtorsymptorbunthm} Let $(\sT,\Omega)$ be a symplectic torus bundle over $M$ with induced integral affine structure $\Lambda$ on $M$.
\begin{itemize}\item[a)] For any {faithful} toric $(\sT,\Omega)$-space $J:(S,\omega)\to M$, the image of the momentum map $J$ is a Delzant subspace of the integral affine manifold $(M,\Lambda)$ (Definition \ref{delzsubmandefi}; cf. Example \ref{symptormanex}).
\item[b)] For any Delzant subspace $\Delta$ in $(M,\Lambda)$, there is a canonical bijection: 
\begin{equation*} \left\{\begin{aligned} &\text{ }\text{ }\textrm{Isomorphism classes of} \\ &\textrm{ {faithful} toric $(\sT,\Omega)$-spaces } \\ & \textrm{ with momentum image $\Delta$} \end{aligned}\right\}\longleftrightarrow H^1(\Delta,\L),
\end{equation*} {defined by associating} to such a $(\sT,\Omega)$-space its `Lagrangian Chern class' (cf. Example \ref{proplagrfibrex}). Here the right-hand side denotes the degree one cohomology of the sheaf $\L$ on $\Delta$ consisting of Lagrangian sections {of 
$\sT$ over $\Delta$} (Definition \ref{shfinvlagrsecdefi:torbunversion}).
\end{itemize} So, there is a canonical bijection:
\begin{equation}\label{finalbijclassthmtorbunactintro}
\left\{\begin{aligned} &\text{ }\text{ }\textrm{Isomorphism classes of\text{ }} \\ &\textrm{{faithful} toric $(\sT,\Omega)$-spaces } \end{aligned}\right\}\longleftrightarrow \left\{\begin{aligned} &\text{Pairs $(\Delta,\kappa)$ consisting of a Delzant subspace $\Delta$}\\ & \quad\text{ of $(M,\Lambda)$ and an element $\kappa$ of $H^1(\Delta,\L)$ }\end{aligned}\right\}.
\end{equation} 
\end{Thm} 
This theorem is an amalgam of the classifications in the work of Delzant and Duistermaat mentioned before. The Delzant subspaces appearing in part $a$ are codimension zero submanifolds with corners that are `fully compatible with the integral affine structure'. {The Delzant polytopes in Delzant's classification} are examples of these (see Example \ref{symptormanex} below). On the other hand, the idea behind the bijection in part $b$ is the same as that in the classification in Duistermaat's work (briefly recalled in Example \ref{proplagrfibrex} below). {As in those classifications, the integral affine structure plays a key role in Theorem \ref{classtorsymptorbunthm}}. 
\begin{Ex}\label{symptormanex} Delzant's classification of toric symplectic manifolds can be recovered from Theorem \ref{classtorsymptorbunthm}, as follows. Let $T$ be a torus with Lie algebra $\t$ and character lattice $\Lambda_T^*$ in $\t^*$. Recall that a Delzant polytope in the integral affine vector space $(\t^*,\Lambda_T^*)$ is a convex polytope $\Delta$ in $\t^*$ with the property that at each vertex $x$ the polyhedral cone $C_x(\Delta)$ spanned by the edges meeting at $x$ is generated by a $\Z$-basis of $\Lambda_T^*$. Delzant showed that there is a canonical bijection:
\begin{equation*}
\left\{\begin{aligned} &\quad\textrm{ Isomorphism classes of compact, \text{ }} \\ &\textrm{ connected toric symplectic $T$-spaces } \end{aligned}\right\}\longleftrightarrow \left\{\begin{aligned} \textrm{Delzant polytopes in $(\t^*,\Lambda_T^*)$}\end{aligned}\right\},
\end{equation*} defined by assigning to such a $T$-space the image of its momentum map. This can be recovered from the theorem above, applied to the {cotangent} groupoid (see e.g. \cite[Section 3]{WeMi}):
\begin{equation*} (\sT,\Omega)=(T{\times} \t^*,\Omega_\textrm{can}),
\end{equation*}
by means of the following two observations:
\begin{itemize}\item $\{\textrm{Delzant polytopes in $(\t^*,\Lambda_T^*)$}\}=\{\textrm{Compact, connected Delzant subspaces of $(\t^*,\Lambda_T^*)$}\}$,
\item for any Delzant polytope $\Delta$ the cohomology $H^1(\Delta,\L)$ vanishes, due to convexity of $\Delta$.
\end{itemize}
\end{Ex}
\begin{Ex}\label{proplagrfibrex} For any integral affine manifold $(M,\Lambda)$, $\Delta:=M$ itself is a Delzant subspace. In this case Theorem \ref{classtorsymptorbunthm} recovers the aforementioned classification of proper Lagrangian fibrations. To explain this, recall that a proper Lagrangian fibration is a symplectic manifold $(S,\omega)$ together with a proper surjective submersion $J:(S,\omega)\to M$ with Lagrangian fibers. We will always assume such fibrations to have connected fibers, without further notice. Any such fibration over $M$ induces an integral affine structure $\Lambda$ on $M$. 
Duistermaat implicitly showed that, given any fixed integral affine manifold $(M,\Lambda)$, there is a canonical bijection: \begin{equation}\label{duisbijlagrchernclassintro} \left\{\begin{aligned} &\quad\quad\quad\quad\textrm{ Isomorphism classes of} \\ &\quad\textrm{ proper Lagrangian fibrations over $M$ } \\ &\textrm{ with induced integral affine structure $\Lambda$ } \end{aligned}\right\}\longleftrightarrow H^1(M,\L)
\end{equation} that associates to $[J:(S,\omega)\to M]$ its so-called Lagrangian Chern class (called Lagrangian class in \cite{Zu1,Sj1}). The key insight leading to (\ref{duisbijlagrchernclassintro}) is that fibrations as in (\ref{duisbijlagrchernclassintro}) can be viewed as certain principal $\sT_\Lambda$-bundles, called symplectic $(\sT_\Lambda,\Omega_\Lambda)$-torsors in \cite{Sj1,CrFeTo}. Indeed, any such fibration is a Poisson map into $(M,0)$ that admits a (necessarily unique) Hamiltonian $(\sT_\Lambda,\Omega_\Lambda)$-action along it, which turns it into a symplectic $(\sT_\Lambda,\Omega_\Lambda)$-torsor (cf. \cite[pg. 23, pg. 88]{CrFeTo}). From that point of view the construction of the bijection (\ref{duisbijlagrchernclassintro}) is analogous to that in the usual classification of principal $G$-bundles in terms of \v{C}ech cohomology. {Given a} symplectic torus bundle $(\sT,\Omega)$ over $M$, the notion of a symplectic $(\sT,\Omega)$-torsor coincides with that of a {faithful} toric $(\sT,\Omega)$-space with momentum image $M$. So, in the setting of Theorem \ref{classtorsymptorbunthm} there is a canonical bijection: 
\begin{equation*} \left\{\begin{aligned} &\quad\quad\quad\quad\textrm{Isomorphism classes of} \\ &\quad\textrm{proper Lagrangian fibrations over $M$ } \\ &\textrm{ with induced integral affine structure $\Lambda$ } \end{aligned}\right\}\longleftrightarrow \left\{\begin{aligned} &\text{ }\text{ }\text{ }\textrm{Isomorphism classes of} \\ &\textrm{ {faithful} toric $(\sT,\Omega)$-spaces } \\ & \text{}\textrm{ with momentum image $M$} \end{aligned}\right\}.
\end{equation*}
In this way Theorem \ref{classtorsymptorbunthm} recovers the classification of proper Lagrangian fibrations above. 
\end{Ex}

Next, we explain our classification {in full generality}. Throughout this paper, we use the following as a working definition {for the class of actions that we consider}.
\begin{Defi}\label{toractdefi} Let $(\G,\Omega)\rightrightarrows M$ be a regular and proper symplectic groupoid and let $\sT$ be the torus bundle over $M$ with fibers the identity components of the isotropy groups of $\G$ (also see Subsection \ref{iaorbsec}). By a \textbf{faithful multiplicity-free} $(\G,\Omega)$-space we mean a Hamiltonian $(\G,\Omega)$-space $J:(S,\omega)\to M$ with the following three properties. 
\begin{itemize}
\item[i)] The induced $\sT$-action is free on a dense subset of $S$.
\item[ii)] The $\sT$-orbits coincide with the fibers of $J$.
\item[iii)] The transverse momentum map $\underline{J}:\underline{S}\to \underline{M}$ is a topological embedding. 
\end{itemize} Here and throughout:
\begin{itemize}\item $\underline{S}:=S/\G$ denotes the orbit space of the $\G$-action,
\item $\underline{M}:=M/\G$ denotes the orbit space of $\G$, 
\item $\underline{J}$ denotes the map induced by $J$, which we refer to as the \textbf{transverse momentum map}.
\end{itemize} 
\end{Defi} 
Of course, when $(\G,\Omega)$ is a symplectic torus bundle this boils down to the notion of faithful toric space considered before. {As in that case,} the last condition above is purely topological, since $\underline{J}$ is injective by the second condition. Moreover, there is an equivalent set of conditions (given in Proposition \ref{toracttorrepprop-cd}) of the same type as the first set of conditions for faithful toric spaces. From these conditions it is clearer that faithful multiplicity-free $(\G,\Omega)$-spaces are closely related to non-commutative integrable systems (e.g. in the sense of \cite{LaMiVa}).\\ 

For the rest of this introduction, let $(\G,\Omega)\rightrightarrows M$ and $\sT$ be as in the above definition. The fact that symplectic torus bundles induce an integral affine structure on their base generalizes as follows. The symplectic groupoid $(\G,\Omega)$ induces an integral affine structure, not on $M$, but on its orbit space $\underline{M}$, which has the structure of an orbifold encoded by the orbifold groupoid $\B:=\G/\sT\rightrightarrows M$. The induced integral affine structure on this orbifold is encoded by a $\B$-invariant lattice bundle $\Lambda$ in the co-normal bundle $\No^*\F$ to the {orbits} of $\G$. This is shown in \cite{CrFeTo} and recalled in Subsection \ref{iaorbsec}. Part $a$ of Theorem \ref{classtorsymptorbunthm} generalizes as follows.
\begin{Thm}\label{momimtoricthm} Let $(\G,\Omega)$ be a regular and proper symplectic groupoid. For any faithful multiplicity-free $(\G,\Omega)$-space $J:(S,\omega)\to M$, the image of the transverse momentum map: 
\begin{equation*} \underline{J}(\underline{S})\subset \underline{M}
\end{equation*} is a Delzant subspace of the orbit space $\underline{M}$ (in the sense of Definition \ref{defdelzorb}, with respect to the integral affine orbifold structure mentioned above).
\end{Thm}
Here, intuitively, one can think of Delzant subspaces as codimension zero `suborbifolds with corners' that are `fully compatible with the integral affine structure'.
We split the generalization of part $b$ of Theorem \ref{classtorsymptorbunthm} into two theorems. The first explains to which extent the momentum map image determines the isomorphism class of a given faithful multiplicity-free $(\G,\Omega)$-space.
\begin{Thm}[\label{firststrthm}{First structure theorem}] Let $(\G,\Omega)\rightrightarrows M$ be a regular and proper symplectic groupoid, let $\underline{\Delta}$ be a Delzant subspace of the orbit space $\underline{M}$ (in the same sense as in Theorem \ref{momimtoricthm}) and let $\Delta$ be the corresponding invariant subspace of $M$. If:
\begin{equation}\label{isoclwithprescrmomimnonempty} \left\{\begin{aligned} &\quad\quad\text{ }\text{ }\text{ }\textrm{Isomorphism classes of} \\ &\textrm{ faithful multiplicity-free $(\G,\Omega)$-spaces } \\ & \quad\text{ }\text{ }\textrm{with momentum map image $\Delta$} \end{aligned}\right\}\neq \emptyset,
\end{equation} then it is a torsor with abelian structure group:
\begin{equation}\label{strgpfirststrthm} H^1(\B\vert_{\Delta},\L),
\end{equation} the degree one {cohomology} of the $\B\vert_\Delta$-sheaf $\L$ {consisting of isotropic} sections of $\sT$ over $\Delta$ (defined in {Subsection \ref{subsec:equivshfcohom}}). This action is natural with respect to Morita equivalences (in the sense of Proposition \ref{naturalityofactprop}).
\end{Thm}

The second addresses the condition (\ref{isoclwithprescrmomimnonempty}).

\begin{Thm}[\label{globalsplitthm}{Splitting theorem}] In the setting of Theorem \ref{firststrthm}, {the following are equivalent.
\begin{itemize}\item[a)] The condition (\ref{isoclwithprescrmomimnonempty}) holds.
\item[b)] As pre-symplectic groupoid with corners over $\Delta$ (see Remark \ref{etalecasedelzsubmanrem} and Definition \ref{sympgpwithcorndef}), $(\G,\Omega)\vert_\Delta$ and $(\B\Join \sT,\textrm{pr}_{\sT}^*\Omega_{\sT})\vert_\Delta$ are Morita equivalent.
\item[c)] The Lagrangian Dixmier-Douady class (see Definition \ref{def:lagrdixmierdouadyclass}): 
\begin{equation}\label{lagrdixmdouadcl} \emph{\textrm{c}}_2(\G,\Omega,\underline{\Delta})\in H^2(\B\vert_\Delta,\L)
\end{equation}
vanishes.
\end{itemize}}
\end{Thm}
Here $\B\Join \sT$ denotes the semi-direct product groupoid over $M$ associated to the $\B$-action along $\sT\to M$ via conjugation in $\G$. Moreover, $\Omega_\sT$ is the restriction {to $\sT$} of the symplectic form $\Omega$ on $\G$. 
\begin{Rem}\label{canbijremintro} Theorem \ref{firststrthm} and the proof of Theorem \ref{globalsplitthm} lead to a sharper conclusion in the setting of Theorem \ref{classtorsymptorbunthm}. Indeed, Theorem \ref{constrtorictorbunspthm} shows that for any Delzant subspace $\Delta$ of $(M,\Lambda)$ there is a canonically associated faithful toric $(\sT,\Omega)$-space $J_\Delta:(S_\Delta,\omega_\Delta)\to M$ with momentum image $\Delta$. So, there is a canonical section of the map:
\begin{equation*} \left\{\begin{aligned} &\quad\textrm{Isomorphism classes of} \\ &\textrm{ faithful toric $(\sT,\Omega)$-spaces } \end{aligned}\right\} \longrightarrow \{\textrm{Delzant subspaces of $(M,\Lambda)$}\}
\end{equation*}{defined by sending} such a $(\sT,\Omega)$-space to its momentum map image. Combined with Theorem \ref{firststrthm} this leads to the bijection (\ref{finalbijclassthmtorbunactintro}), where a pair $(\Delta,\kappa)$ corresponds to the isomorphism class obtained by letting $\kappa$ act on the isomorphism class of $J_\Delta$. More generally, Theorem \ref{firststrthm} leads to such a bijection when $(\G,\Omega)$ is a semi-direct product symplectic groupoid $(\B\Join \sT,\textrm{pr}_{\sT}^*\Omega_\sT)$, with $\B$ an etale orbifold groupoid acting on a symplectic torus bundle $(\sT,\Omega_\sT)$ as in Subsection \ref{constrtorspsemidirprodsec} (by the same reasoning, using Proposition \ref{exttoractsemidirectprodprop} in addition to Theorem \ref{constrtorictorbunspthm}).

\end{Rem}

\begin{Ex}\label{almabcompLiegpexintro0} Let $G$ be an infinitesimally abelian compact Lie group, meaning that its Lie algebra $\g$ is abelian. For Hamiltonian $G$-spaces, the outcome of Theorems \ref{firststrthm} and \ref{globalsplitthm} can be rephrased in terms of the group $G$. To be more precise, recall that Hamiltonian $G$-spaces correspond to Hamiltonian $(\G,\Omega)$-spaces for the {cotangent} groupoid: 
\begin{equation*} (\G,\Omega):=(G\ltimes \g^*,\Omega_\textrm{can})
\end{equation*} over $\g^*$ (see e.g. \cite{WeMi}). A compact and connected Hamiltonian $G$-space $J:(S,\omega)\to \g^*$ corresponds to a faithful multiplicity-free $(\G,\Omega)$-space if and only if $J:(S,\omega)\to \g^*$ is a {toric symplectic manifold} with respect to the induced action of the identity component $T$ of $G$ (a torus). So, these are toric symplectic manifolds with additional symmetry, which is reflected by the fact that the Delzant polytope corresponding to such a $T$-space (its momentum map image) is invariant with respect to the induced action of the finite group $\varGamma:=G/T$ on $\g^*$. Let us call such Hamiltonian $G$-spaces \textbf{maximally toric}. Our theorems lead to the following conclusions, which (somewhat surprisingly) we could not find in the literature. Let $\Delta$ be a $\varGamma$-invariant Delzant polytope in $(\g^*,\Lambda_T^*)$ and let $J:(S,\omega)\to \g^*$ be a toric symplectic $T$-space with momentum map image $\Delta$. The latter is unique up to $T$-equivariant symplectomorphism and any isomorphism class of {maximally toric Hamiltonian $G$-spaces} with momentum map image $\Delta$ can be represented by a {Hamiltonian} $G$-action on $(S,\omega)$ with momentum map $J$ (cf. Example \ref{symptormanex}).  
\begin{itemize}\item The condition (\ref{isoclwithprescrmomimnonempty}) can be rephrased as requiring that the symplectic $T$-action on $(S,\omega)$ extends to such an action of $G$ that is compatible with the $\varGamma$-action on $\Delta$. It follows from Theorem \ref{globalsplitthm} that this holds if and only if the short exact sequence of groups:
\begin{equation}\label{extinfabliegpintro} 1\to T\to G\to \varGamma\to 1
\end{equation} admits a right splitting (i.e. the extension is trivial up to isomorphism). {Of course, there is a natural group cohomology class in $H^2(\varGamma,T)$ that encodes the obstruction to the existence of such a splitting, where $T$ is viewed as $\varGamma$-module via the action by conjugation in $G$. This class is recovered by the Lagrangian Dixmier-Douady class via a canonical isomorphism between $H^2(\varGamma,T)$ and $H^2(\B\vert_\Delta,\L)$ (see Proposition \ref{prop:strgps:caseofgpactions}$a$).}
\item Theorem \ref{firststrthm} explains in which different ways the $T$-action on $(S,\omega)$ extends to a $G$-action as above. Namely: if the $T$-action extends, then {up to isomorphism of Hamiltonian $G$-spaces} there is exactly one such $G$-action for each group cohomology class in $H^1(\varGamma,T)$. This is because the structure group (\ref{strgpfirststrthm}) is isomorphic to $H^1(\varGamma,T)$ (see Proposition \ref{prop:strgps:caseofgpactions}$a$). Explicitly, {given such a $G$-action, the $H^1(\varGamma,T)$-action of Theorem \ref{firststrthm} twists the $G$-action by a group $1$-cocycle} ${c}:\varGamma\to T$: 
\begin{equation*} g\cdot_{{c}} p=g{c}([g])\cdot p,\quad g\in G,\quad p\in S.
\end{equation*} So, the $T$-action extends in different ways, each corresponding to an automorphism of the extension (\ref{extinfabliegpintro}), modulo automorphisms given by conjugation by elements of $T$.
\end{itemize} 
\end{Ex}

\begin{Ex}\label{thirdexampleintro} Example \ref{proplagrfibrex} generalizes as follows: for source-connected $(\G,\Omega)\rightrightarrows M$, faithful multiplicity-free $(\G,\Omega)$-spaces with momentum image $\Delta=M$ correspond to certain proper isotropic fibrations over $M$. To elaborate, recall that proper isotropic fibrations are defined by replacing `Lagrangian' by `isotropic' in Example \ref{proplagrfibrex}, with the additional requirement that the symplectic orthogonal $\ker(\d J)^\omega$ is an involutive distribution (these are symplectically complete isotropic fibrations in the sense of \cite{DaDe}). Any such fibration $J:(S,\omega)\to M$ induces a regular Poisson structure $\pi$ on $M$, uniquely determined by the property that $J:(S,\omega)\to (M,\pi)$ is a Poisson map (as a consequence of Libermann's theorem \cite{Li}). Moreover, as shown in \cite{DaDe}, such a fibration induces a transverse integral affine structure to the foliation $\F_\pi$ on $M$ by symplectic leaves of $\pi$, encoded by a lattice bundle $\Lambda$ in the co-normal bundle $\No^*\F_\pi$. Suppose now that $\F_\pi$ is of proper type, meaning that its holonomy groupoid $\text{Hol}(M,\F_\pi)$ is proper. It is shown in \cite{CrFeTo} that there is a natural source-connected proper symplectic groupoid: 
\begin{equation}\label{holgpoidintro} (\textrm{Hol}_J(M),\Omega_J)\rightrightarrows M
\end{equation} associated to $J:(S,\omega)\to M$, with the following properties.
\begin{itemize} \item It integrates $(M,\pi)$. 
\item The (transverse) integral affine structure induced by $(\textrm{Hol}_J(M),\Omega)$ is that induced by $J$.
\item The associated orbifold structure on $\underline{M}$ is that encoded by $\text{Hol}(M,\F_\pi)$, because the associated orbifold groupoid is the integration $\textrm{Hol}(M,\F_\pi)$ of $\F_\pi$. 
\end{itemize}
The map $J:(S,\omega)\to M$ admits a canonical Hamiltonian $(\textrm{Hol}_J(M),\Omega)$-action, which is makes it a faithful multiplicity-free $(\textrm{Hol}_J(M),\Omega)$-space. More generally, for any orbifold groupoid $\B$ integrating $\F_\pi$, there is a natural integration $(\G,\Omega)$ of $(M,\pi)$ associated to $J$, for which the associated orbifold groupoid $\G/\sT$ is the integration $\B$ of $\F_\pi$. This is called the $\B$-integration of $(M,\pi)$ relative to $J$ (\cite[pg. 82]{CrFeTo}). It also induces the same (transverse) integral affine structure as $J$ and acts canonically along $J$, making it a faithful multiplicity-free $(\G,\Omega)$-space. This explains how such proper isotropic fibrations can naturally be viewed as faithful multiplicity-free spaces. For the converse, let $(\G,\Omega)$ be any source-connected, regular and proper symplectic groupoid, let $\B=\G/\sT$ and let $\pi$ be the induced Poisson structure on its base $M$. Then for any faithful multiplicity-free $(\G,\Omega)$-space $J:(S,\omega)\to M$ with $J(S)=M$, the momentum map is a proper isotropic fibration that induces this same Poisson structure on $M$. In fact, $(\G,\Omega)$ is the $\B$-integration of $(M,\pi)$ relative to $J$. An isomorphism of two such proper isotropic fibrations:
\begin{center}
\begin{tikzcd} (S_1,\omega_1)\arrow[rr,"\sim"]\arrow[rd,"J_1"'] &  & (S_2,\omega_2)\arrow[ld,"J_2"] \\
& M & 
\end{tikzcd}
\end{center}
is automatically $\G$-equivariant, or in other words, it is an isomorphism of Hamiltonian $(\G,\Omega)$-spaces. Hence, there is a canonical bijection:
\begin{equation*} \left\{\begin{aligned} &\quad\quad\quad\textrm{ Isomorphism classes of } \\ &\textrm{proper isotropic fibrations over $(M,\pi)$} \\ & \quad\quad\quad\textrm{with $\B$-integration $(\G,\Omega)$} \end{aligned}\right\}\longleftrightarrow \left\{\begin{aligned} &\quad\quad\quad\text{ }\textrm{ Isomorphism classes of} \\ &\textrm{ faithful multiplicity-free $(\G,\Omega)$-spaces } \\ & \quad\quad\quad\textrm{with momentum image $M$} \end{aligned}\right\}.
\end{equation*} 
This explains the relationship between faithful multiplicity-free $(\G,\Omega)$-spaces with momentum image $\Delta=M$ and proper isotropic fibrations over $M$. Finally, let us point out: by the splitting theorem, $(\G,\Omega)$ is the $\B$-integration relative to some proper isotropic fibration if and only if $(\G,\Omega)$ is Morita equivalent to $(\B\Join \sT,\textrm{pr}_{\sT}^*\Omega_\sT)$. This recovers a result of \cite{CrFeTo}. 
\end{Ex}

Our final two main theorems provide an alternative way of listing the elements of (\ref{isoclwithprescrmomimnonempty}), by means of an additional invariant of faithful multiplicity-free $(\G,\Omega)$-spaces that we call the \textbf{ext-invariant} (Definition \ref{extinvdefi}). This invariant is a global section of what we call the \textbf{ext-sheaf} (Definition \ref{extinvsheafdef}):
\begin{equation}\label{flatsecshintro} \mathcal{I}^1=\mathcal{I}^1_{(\G,\Omega,\underline{\Delta})}.
\end{equation} This is a sheaf of sets on $\underline{\Delta}$ naturally associated to $(\G,\Omega)$. Its stalk at an orbit $\O_x$ of $\G$ through a point $x\in \Delta$ is the set: 
\begin{equation*} I^1(\G_x,\sT_x),
\end{equation*} consisting of right-splittings of the short exact sequence of isotropy groups:
\begin{equation}\label{sesisogpsatxintro} 1\to \sT_x\to \G_x\to \B_x\to 1,
\end{equation} modulo the $\sT_x$-action by conjugation (cf. Definition \ref{classgpcohomdefi} and Proposition \ref{splitcharprop}). 
\begin{Thm}[\label{torsortorspthm}{Second structure theorem}] Suppose that we are in the setting of Theorem \ref{firststrthm} and let $e$ be a global section of the ext-sheaf (\ref{flatsecshintro}). If: 
\begin{equation}\label{isoclsetsecondstrthmnonempty} \left\{\begin{aligned} &\quad\quad\text{ }\text{ }\textrm{Isomorphism classes of} \\ &\textrm{ faithful multiplicity-free $(\G,\Omega)$-spaces } \\ & \quad\textrm{ with momentum map image $\Delta$} \\& \quad\quad\textrm{ and with ext-invariant $e$ } \end{aligned}\right\}\neq \emptyset,
\end{equation} then it is a torsor with abelian structure group: 
\begin{equation}\label{strgpsecstrthm} H^1(\underline{\Delta},\underline{\L}),
\end{equation} the degree one cohomology of the sheaf $\underline{\L}$ on $\underline{\Delta}$ associated to the $\B\vert_\Delta$-sheaf $\L$ by considering $\B$-invariant sections (see Subsection \ref{sec:shfofautosandinvisotrsections}). This action is natural with respect to Morita equivalences (in the sense explained in Subsection \ref{extinvconstsec}).
\end{Thm}
The structure groups in the first and second structure theorem are related by a natural injective group homomorphism:
\begin{equation}\label{inclstrgpsintro} H^1(\underline{\Delta},\underline{\L}) \hookrightarrow H^1(\B\vert_\Delta,\L), 
\end{equation} which is compatible with the actions in these theorems. This leads to:
\begin{Thm}[\label{thirdstrthm}{Third structure theorem}] In the setting of Theorem \ref{firststrthm}, if (\ref{isoclwithprescrmomimnonempty}) holds, then the image of the map:
\begin{equation}\label{extinvariantmapthrdstrthm} \left\{\begin{aligned} &\quad\quad\quad\text{ }\text{}\textrm{Isomorphism classes of} \\ &\textrm{ faithful multiplicity-free $(\G,\Omega)$-spaces } \\ & \quad\quad\textrm{with momentum map image $\Delta$} \end{aligned}\right\}\longrightarrow \mathcal{I}^1(\underline{\Delta})
\end{equation} that associates to an isomorphism class of such a $(\G,\Omega)$-space its ext-invariant, is a torsor with abelian structure group the quotient:
\begin{equation}\label{quotgpthirdstrthmeqintro} \frac{H^1(\B\vert_\Delta,\L)}{H^1(\underline{\Delta},\underline{\L})}.
\end{equation} This action is natural with respect to Morita equivalences {as well}.
\end{Thm}

\begin{Rem}\label{expldescrpactthirdstrcthmrem} The injection (\ref{inclstrgpsintro}) fits in an exact sequence:
\begin{equation}\label{eqn:fivetermexseqlowdeg} 0\to H^1(\underline{\Delta},\underline{\L})\to H^1(\B\vert_\Delta,\L)\to \H^1(\underline{\Delta})\to H^2(\underline{\Delta},\underline{\L})\to H^2(\B\vert_\Delta,\L),\quad\H^1:=R^1(q_\Delta)_*(\L),
\end{equation} which is the exact sequence in low degrees of a Leray-type spectral sequence for the canonical map of topological groupoids $q_\Delta$ from $\B\vert_\Delta$ to the unit groupoid of $\underline{\Delta}$. Here $\H^1$ denotes the first right-derived functor of the push-forward along $q_\Delta$ applied to the $\B\vert_\Delta$-sheaf $\L$. In Subsection \ref{explcdescrpactthirdstrthmsec} we show that there is a natural free and transitive action of the sheaf of abelian groups $\H^1$ on $\mathcal{I}^1$. Moreover, we show that the action of (\ref{quotgpthirdstrthmeqintro}) on the image of (\ref{extinvariantmapthrdstrthm}) in Theorem \ref{thirdstrthm} (which is inherited from the action of (\ref{strgpfirststrthm}) in Theorem \ref{firststrthm}) factors through this $\H^1$-action via the middle map in (\ref{eqn:fivetermexseqlowdeg}).   
\end{Rem}

Together, the second and third structure theorem provide a way of listing the isomorphism classes of faithful multiplicity-free $(\G,\Omega)$-spaces different from that in the first structure theorem.

\begin{Ex}\label{almabcompLiegpexintro} To illustrate this difference we return to Example \ref{almabcompLiegpexintro0}. From the second and third structure theorem and the splitting theorem we obtain a canonical bijection:
\begin{equation*}\left\{\begin{aligned} &\text{ Isomorphism classes of compact, connected, }\\ &\quad\text{ {maximally} toric Hamiltonian $G$-spaces }\end{aligned}\right\}\longleftrightarrow \left\{\begin{aligned} &\text{$\varGamma$-invariant Delzant}\\ & \text{polytopes in $(\g^*,\Lambda_T^*)$}\end{aligned}\right\}\times I^1(G,T),
\end{equation*} that associates to a class $[J:(S,\omega)\to \g^*]$ the pair consisting of the image $\Delta$ of $J$ and the germ of its ext-invariant at any $\varGamma$-fixed point in $\Delta$. This is because for any $\varGamma$-invariant Delzant polytope $\Delta$ in $(\g^*,\Lambda_T^*)$ the following hold.
\begin{itemize}\item The structure group (\ref{strgpsecstrthm}) is trivial, due to convexity of $\Delta$ (see Proposition \ref{prop:strgps:caseofgpactions}$b$).
\item As mentioned before: the structure group (\ref{strgpfirststrthm}) is naturally isomorphic to the degree one group cohomology $H^1(\varGamma,T)$. 
\item The set of global sections of (\ref{flatsecshintro}) is naturally in bijection with the set $I^1(G,T)$ consisting of right splittings of (\ref{extinfabliegpintro}) modulo the $T$-action by conjugation. Moreover, under this and the previous identification the action of (\ref{strgpfirststrthm}) on the set of global sections of (\ref{flatsecshintro}) is identified with the free and transitive $H^1(\varGamma,T)$-action on $I^1(G,T)$ obtained by viewing $H^1(\varGamma,T)$ as the group of automorphisms of the extension (\ref{extinfabliegpintro}) modulo the subgroup of automorphisms given by conjugation by elements of $T$ (see Proposition \ref{prop:strgps:caseofgpactions}$c$). 
\end{itemize} 
\end{Ex}
To further illustrate the use of the second and third structure theorem, in Subsection \ref{flatsecdefsec} we give a more concrete example (Example \ref{example:concretecompsecondstrthm}) of a symplectic groupoid and a Delzant subspace to which we apply these theorems, in order to obtain an explicit list of all isomorphism classes of faithful multiplicity-free actions of that symplectic groupoid with that Delzant subspace as their momentum map image. \\

{ In fact, the above theorems apply to a more general class of Hamiltonian group actions than that in Example \ref{almabcompLiegpexintro}, to give the following classification.  
\begin{Ex}\label{gencompLiegpexintro} 
Let $G$ be any compact Lie group. A compact and connected Hamiltonian $G$-space is called multiplicity-free if all its symplectic reduced spaces are points. Let us call a Hamiltonian $G$-space \textbf{regular} if its momentum map image is contained in $\g^*_\textrm{reg}$ \textemdash the union of coadjoint orbits of maximal dimension. Fix a maximal torus $T$ and an open Weyl chamber $\c$ in $\t^*$. Further, let $N(\c)$ be the normalizer of $\c$ in $G$. 
It follows from the second and third structure theorem and the splitting theorem (see Example \ref{ex:gencompLiegpexintro:details}) that compact, connected, regular, multiplicity-free Hamiltonian $G$-spaces are classified up to isomorphism by triples of:
\begin{itemize}
\item a closed normal subgroup $K$ of $N(\c)$ that is contained in $T$,
\item an $N(\c)$-invariant polytope $\Delta$ in $\c$ that is Delzant as polytope in an integral affine subspace of $(\t^*,\Lambda_T^*)$ parallel to $((\t/\mathfrak{k})^*,\Lambda_{T/K}^*)$,
\item an element of $I^1(N(\c)/K,T/K)$.
\end{itemize}
\end{Ex}
Whilst for connected $G$ Examples \ref{almabcompLiegpexintro} and \ref{gencompLiegpexintro} are a simple case of the existing classification of all multiplicity-free Hamiltonian actions of such Lie groups \cite{De,De1,Wo,Kn} (since in that case $N(\c)=T$), for disconnected $G$ it seems that there are no such classifications in the literature.} \\ 

Finally, let us mention that the main theorems above readily generalize to quasi-Hamiltonian actions of regular, proper, twisted pre-symplectic groupoids (in the sense of \cite{Xu1,BuCrWeZh}), essentially because every regular twisted pre-symplectic groupoid is Morita equivalent to a regular (non-twisted) symplectic groupoid. In particular, for any compact Lie group $G$, this applies to a class of quasi-Hamiltonian $G$-spaces (in the sense of \cite{AlMaMe}) for which the momentum map image is contained in the regular part of $G$. This provides some new insight into the classification problem for multiplicity-free quasi-Hamiltonian $G$-spaces, for in the current results on this \cite{Kn1} the Lie group $G$ is assumed to be simply-connected. \\

\textbf{\underline{Brief outline:}} In Part 1 we introduce Delzant subspaces and prove that the image of the momentum map of a faithful multiplicity-free Hamiltonian action is such a subspace (Theorem \ref{momimtoricthm}). The classification for the case of symplectic torus bundles (Theorem \ref{classtorsymptorbunthm}) is proved in Part 2. In Part $3$ we prove the first structure theorem and the splitting theorem (Theorem \ref{firststrthm} and Theorem \ref{globalsplitthm}) and in Part 4 we introduce the ext-invariant and the ext-sheaf, and prove the second and third structure theorems (Theorem \ref{torsortorspthm} and Theorem \ref{thirdstrthm}). Apart from the section on the third structure theorem, Part 4 can be read independently of Part 3. More detailed outlines can be found at the beginning of each of these parts. \\

\textbf{\underline{Acknowledgements:}} First, I wish to thank my PhD supervisor Marius Crainic for suggesting this project to me, for useful discussions on the cohomology of orbifold sheaves and for commenting on an earlier version of this text. The latter has surely helped me to improve the exposition. Furthermore, I would like to thank Rui Loja Fernandes for sharing a private note on proper isotropic fibrations with me. Th{at} short note contains a statement similar to the outcome of the first structure theorem above when applied to the particular case of Example \ref{thirdexampleintro}, which was a source of inspiration for the statement of the first structure theorem. This work was supported by NWO (Vici Grant no. 639.033.312) and the Max Planck Institute for Mathematics.\\

\textbf{\underline{Conventions:}} Throughout, we require smooth manifolds (with or without corners) to be both Hausdorff and second countable and we require the same for both the base and the space of arrows of a Lie groupoid (with or without corners). Furthermore, given a groupoid $\G\rightrightarrows X$ we use the notation $\underline{X}:=X/\G$ for its orbit space and we denote subsets of $\underline{X}$ as $\underline{A}$, where $A$ denotes the corresponding invariant subset of $X$. 

\newpage
\section{Delzant subspaces and the momentum map image}
This part concerns the properties of the image of the momentum map and its relation to those of the Hamiltonian action. \\

In Section \ref{classtorrepsec} we formulate and prove linear versions of the classification results in Example \ref{almabcompLiegpexintro0} and Example \ref{almabcompLiegpexintro}, which illustrate our classification results in a particularly simple case. Moreover, these form the foundation for the construction of the ext-invariant and the proof of Theorem \ref{torsortorspthm}, given in Part 4. In Section \ref{momimsec} we introduce Delzant subspaces, present some of their basic features and recall the necessary background on integral affine orbifolds. This can be read independently from Section \ref{classtorrepsec}, except for the fact that some of the terminology on linear integral affine geometry used in Section \ref{momimsec} is introduced in Subsection \ref{torrepoftorsec}. Lastly, in Section \ref{sec:momimisdelzsubsp} we prove Theorem \ref{momimtoricthm} and Proposition \ref{conedescrpsympnormrepprop}, which connect the objects studied in Section \ref{classtorrepsec} to those in Section \ref{momimsec} and explain how to read off some invariants of the Hamiltonian action from the momentum map image. 


\subsection{Classification of maximally toric representations}\label{classtorrepsec} 
\subsubsection{Statement of the classification theorem}
Let us first introduce some terminology. 
\begin{defi} By a \textbf{symplectic representation} $(V,\omega)$ of a Lie group $H$ we mean finite-dimensional real symplectic vector space $(V,\omega)$ together with a morphism $H\to \textrm{Sp}(V,\omega)$ into the Lie group of linear symplectic automorphisms of $(V,\omega)$.
\end{defi} 
For any such representation $(V,\omega)$, the $H$-action is Hamiltonian with equivariant momentum map:
\begin{equation}\label{quadsympmommap} J_V:(V,\omega)\to \h^*, \quad \langle J_V(v),\xi \rangle=\frac{1}{2}\omega(\xi\cdot v,v).
\end{equation}  
\begin{defi}\label{torrepdef} We call a Lie group $H$ \textbf{infinitesimally abelian} if its Lie algebra is abelian. Given an infinitesimally abelian compact Lie group $H$ with identity component $T$, we call a symplectic representation $H\to \textrm{Sp}(V,\omega)$ \textbf{maximally toric} if the induced torus representation $T\to \textrm{Sp}(V,\omega)$ is faithful and $\dim(V)=2\dim(T)$. 
\end{defi} 
{The representations in this definition are linear versions of the Hamiltonian group actions considered in Example \ref{almabcompLiegpexintro0} and Example \ref{almabcompLiegpexintro}. The theorem below is a linear version of the classification result in Example \ref{almabcompLiegpexintro}. It will also be important later for understanding the local properties of general faithful multiplicity-free actions. }
\begin{thm}\label{isoclasstorrepthm} Let $H$ be an infinitesimally abelian compact Lie group with identity component $T$ and let $\varGamma:=H/T$ be the group of connected components of $H$. There is a canonical bijection:
\begin{equation*} \left\{\begin{aligned}&\quad\quad\text{Isomorphism classes of\text{ }}\\
&\text{ maximally toric $H$-representations }\end{aligned}\right\}\longleftrightarrow \left\{\begin{aligned}&\text{ \quad\quad Smooth \& $\varGamma$-invariant }\\ &\text{ pointed polyhedral cones in $(\t^*,\Lambda_T^*)$ }\end{aligned}\right\}\times I^1(H,T)
\end{equation*} defined by sending the class of such an $H$-representation $(V,\omega)$ to the pair consisting of the image $\Delta_V$ of its momentum map (\ref{quadsympmommap}) and its ext-class $\textrm{e}(V,\omega)$ (see Definition \ref{extclasstoricrep}). 
\end{thm}
In this section we explain and prove this statement, which is a modest extension of a well-known result for representations of tori (Proposition \ref{classtorrep} recalled below). First, we elaborate on the set $I^1(H,T)$. Notice that $T$ is a torus and the quotient ${\varGamma}=H/T$ is a finite group. Furthermore, we have a canonical short exact sequence of Lie groups:
\begin{equation}\label{sesamabliegp} 1\to T\to H\to \varGamma\to 1,
\end{equation} 
via which we view $H$ as an extension of $\varGamma$ by the abelian group $T$ and we view $T$ both as an $H$-module and as a $\varGamma$-module, equipped with the actions by conjugation in $H$. We denote by $H^1(H,T)$ the degree one group cohomology of the $H$-module $T$, represented {throughout this section} as the abelian group of maps ${c}:H\to T$ satisfying the cocycle condition:
\begin{equation*} {c}(h_1h_2)=({c}(h_1)\cdot h_2){c}(h_2), \quad h_1,h_2\in H,
\end{equation*} modulo the subgroup of those ${c}:H\to T$ for which there is a $t\in T$ such that ${c}$ is given by:
\begin{equation*} {c}(h)=(t\cdot h)t^{-1},\quad h\in H. 
\end{equation*} 
\begin{defi}\label{classgpcohomdefi} Let $H$ be an infinitesimally abelian compact Lie group with identity component $T$. We let $I^1(H,T)$ denote the subset of $H^1(H,T)$ consisting of cohomology classes whose representatives restrict to the identity map on $T$. 
\end{defi}
\begin{rem}\label{degonegpcohomtorsorrem} If $I^1(H,T)$ is non-empty, then it is naturally a torsor with abelian structure group $H^1(\varGamma,T)$ (the degree one group cohomology), which makes it amenable to computation in explicit examples. Here, the action is defined by assigning to $[\kappa]\in H^1(\varGamma,T)$ and $[{c}]\in I^1(H,T)$ the class $[\kappa]\cdot [{c}]$ represented by the $1$-cocycle $H\to T$ that maps $h\in H$ to $\kappa([h]){c}(h)\in T$.   
\end{rem} 
The set $I^1(H,T)$ is non-empty if and only if $H$ is split (by Proposition \ref{splitcharprop}), in the sense below.
\begin{defi}\label{almabliegpsplitdefi} We will say that an infinitesimally abelian compact Lie group $H$ is \textbf{split} if the short exact sequence of groups (\ref{sesamabliegp}) admits a right-splitting. 
\end{defi} 
In Subsection \ref{splitabgpextsec} we will prove the proposition below, which is a linear version of the forward implication in the Theorem \ref{globalsplitthm}. In view of this proposition, the statement of Theorem \ref{isoclasstorrepthm} is trivially true when $H$ is not split, since in this case both sets in the bijection are empty. 
\begin{prop}\label{cirepsplit} Let $H$ be an infinitesimally abelian compact Lie group. If $H$ admits a maximally toric representation, then it is split. \end{prop}
The proof of this will reveal the group cohomology class (\ref{gpcohclasstorrep}). In the coming subsections we will explain the remaining parts of the statement and give a proof of Theorem \ref{isoclasstorrepthm}.

\subsubsection{Isomorphism versus equivalence of symplectic representations} Throughout, we will use the following notions of equivalence between symplectic representations.
\begin{defi}\label{isoequivsymprepdef} An \textbf{isomorphism} of symplectic $H$-representations is an $H$-equivariant symplectic linear isomorphism. Given two Lie groups $H_1$ and $H_2$, by an \textbf{equivalence} between a symplectic $H_1$-representation $(V_1,\omega_1)$ and a symplectic $H_2$-representation $(V_2,\omega_2)$ we mean a pair of maps:
\begin{equation}\label{symprepeq} (\phi,\psi):(H_1,(V_1,\omega_1))\to (H_2,(V_2,\omega_2))
\end{equation} consisting of an isomorphism of Lie groups $\phi:H_1\to H_2$ and a symplectic linear isomorphism $\psi:(V_1,\omega_1)\to (V_2,\omega_2)$ that are compatible with the actions, in the sense that:
\begin{equation*} \phi(h)\cdot\psi(v)=\psi(h\cdot v),\quad h\in H_1,\text{ } v\in V_1.
\end{equation*}
\end{defi}
Notice that the momentum map (\ref{quadsympmommap}) is an invariant of a symplectic representation, in the sense that an equivalence (\ref{symprepeq}) induces an identification:
\begin{center} \begin{tikzcd} (V_1,\omega_1)\arrow[r,"\psi"] \arrow[d,"J_{V_1}"] & (V_2,\omega_2) \arrow[d,"J_{V_2}"] \\
 \h_1^* \arrow[r,"\phi_*"] & \h_2^*
\end{tikzcd} 
\end{center} 

\subsubsection{{Maximally} toric representations of tori}\label{torrepoftorsec} If an infinitesimally abelian compact Lie group is connected, then it is simply a torus. In this case, Theorem \ref{isoclasstorrepthm} boils down to:
\begin{prop}[\cite{De}]\label{classtorrep} Let $T$ be an $n$-dimensional torus. Then:
\begin{enumerate}
\item [a)] A symplectic representation of $T$ is {maximally} toric if and only if its weight-tuple forms a basis of the character lattice $\Lambda_T^*$ in $\t^*$.
\item[b)] The map that associates to each unordered basis $\{\alpha_1,...,\alpha_n\}$ of $\Lambda_T^*$ the polyhedral cone generated by $(\alpha_1,...,\alpha_n)$ is a bijection from the set of such $n$-tuples to the set of smooth pointed polyhedral cones in $(\t^*,\Lambda_T^*)$. 
\end{enumerate} 
Therefore we have a bijection:
\begin{equation*}
\left\{\begin{aligned}&\quad\quad\text{ Isomorphism classes of }\\
 &\text{ {maximally} toric $T$-representations }\end{aligned}\right\}\longleftrightarrow \left\{\begin{aligned}&\quad\quad\text{ Smooth pointed}\\ &\text{ polyhedral cones in }(\t^*,\Lambda_T^*)\text{ }\end{aligned}\right\}
\end{equation*} that associates to such a $T$-representation $(V,\omega)$ the image $\Delta_V$ of the momentum map $(\ref{quadsympmommap})$. 
\end{prop}
Let us briefly recall the meaning of the various notions appearing in this statement. The character lattice $\Lambda_T^*$ of $T$ in $\t^*$ is the dual of the full rank lattice $\Lambda_T:=\ker(\exp_T)$ in $\t$. This gives $\t^*$ the structure of an \textbf{integral affine vector space}, by which we mean a pair $(V,\Lambda)$ consisting of a finite-dimensional real vector space $V$ and a full rank lattice $\Lambda$ in $V$. Any such lattice $\Lambda$ admits a \textbf{basis} (that is, a basis of $V$ consisting of vectors that $\Z$-span $\Lambda$). \\

Recall that symplectic $T$-representations are classified by their weight-tuples, as follows. Associated to $\alpha\in \Lambda_T^*$ is the irreducible symplectic representation $T\to \textrm{Sp}(\C,\omega_\textrm{st})$ defined via the character: \begin{equation}\label{chareq} \chi_\alpha:T\to \mathbb{S}^1, \quad \chi_\alpha(\exp_T(\xi))=e^{2\pi i \langle \alpha,\xi\rangle},
\end{equation} and the $\mathbb{S}^1$-action on $\C$ by complex multiplication. Here $(\C,\omega_\textrm{st})$ is viewed as real vector space equipped with the standard linear symplectic form: 
\begin{equation*} \omega_{\textrm{st}}(w,z)=\frac{1}{2\pi i}(w\overline{z}-\overline{w}z)\in \R,\quad w,z\in \C. 
\end{equation*} We denote this symplectic $T$-representation associated to $\alpha\in \Lambda_T^*$ as:
\begin{equation*} (\C_\alpha,\omega_\textrm{st}).
\end{equation*}
The \textbf{weight-tuple} of a symplectic $T$-representation $(V,\omega)$ is the unique unordered tuple $(\alpha_1,...,\alpha_n)$ such that: 
\begin{equation}\label{weightclassisosympTrep} (V,\omega)\cong (\C_{\alpha_1},\omega_\textrm{st})\oplus...\oplus (\C_{\alpha_n},\omega_\textrm{st})
\end{equation} as symplectic $T$-representation (for details on this, see for instance the appendices of \cite{LeTo,KaLe}). This explains the terminology in part $a$ of the proposition. Turning to part $b$, by a \textbf{polyhedral cone} in a finite-dimensional real vector space $V$, we mean a subset of the form:
\begin{equation*} \textrm{Cone}(v_1,...,v_k):=\left\{ \sum_{i=1}^k s_i v_i\mid s_i\geq 0\right\}, 
\end{equation*} with $v_1,...,v_k\in V$. The tuple $(v_1,...,v_k)$ is said to \textbf{generate} the polyhedral cone. We call a polyhedral cone \textbf{pointed} if it does not contain a line through the origin in $V$. Finally, we call a polyhedral cone $C$ in an integral affine vector space $(V,\Lambda)$ \textbf{smooth} if there is a basis $\{v_1,...,v_n\}$ of the lattice $\Lambda$ in $V$ and a $k\leq n$ such that: 
\begin{equation*} C=\textrm{Cone}(v_1,...,v_k,v_{k+1},-v_{k+1},...,v_n,-v_n).
\end{equation*} This explains the terminology in part $b$. For the remainder, notice that under an identification of symplectic $T$-representations (\ref{weightclassisosympTrep}) the momentum map $J_V$ is identified with:
\begin{equation}\label{quadmom} J_V(z_1,...,z_n)=\sum_{i=1}^n|z_i|^2\alpha_i,
\end{equation} where $z_i$ denotes the standard complex linear coordinate on $\C_{\alpha_i}$. Therefore, the image $\Delta_V$ of $J_V$ is the polyhedral cone in $\t^*$ generated by the weight-tuple $(\alpha_1,...,\alpha_n)$, which is smooth and pointed if the symplectic $T$-representation is maximally toric. 

\begin{rem}\label{torrepcharrem} Let $H$ be an infinitesimally abelian compact Lie group with identity component $T$ and let $(V,\omega)$ be a symplectic $H$-representation. For future reference, let us point out that the representation $H\to \textrm{Sp}(V,\omega)$ is maximally toric if and only if the map of Lie groups induced by the weights of the induced representation $T\to \textrm{Sp}(V,\omega)$:
\begin{equation}\label{weightisotori} (\chi_{\alpha_1},...,\chi_{\alpha_n}):T\to \T^n
\end{equation} is an isomorphism. 
\end{rem}

\subsubsection{Splittings of abelian group extensions}\label{splitabgpextsec}
The main point of this subsection will be to prove Proposition \ref{cirepsplit}. It will be useful to have some alternative descriptions of right splittings of abelian group extensions. Let $A$ and $K$ be groups and suppose that $A$ is abelian. Recall the following.
\begin{itemize}\item An extension of $K$ by $A$ is a short exact sequence of groups:
\begin{equation*} 1\to A\to G\to K\to 1.
\end{equation*} When the maps are clear, we simply refer to $G$ as an extension of $K$ by $A$. Given any such extension, conjugation yields an action of $G$ on $A$ by group automorphisms, which descends to an action of $K$ on $A$. 
\item A morphism of two extensions of $K$ by $A$:
\begin{align*} &1\to A\to G_1\to K\to 1\\
&1\to A\to G_2\to K\to 1
\end{align*} is a morphism of groups $\phi:G_1\to G_2$ that makes the diagram:
\begin{center} \begin{tikzcd} 1\arrow[r]& A \arrow[r] \arrow[d,equal] &G_1 \arrow[d,"\phi"] \arrow[r] &K \arrow[d,equal] \arrow[r] & 1\\
1 \arrow[r]& A\arrow[r] & G_2\arrow[r] &K \arrow[r] & 1
\end{tikzcd} 
\end{center} 
commute. By a version of the five-lemma, any such morphism $\phi$ must be an isomorphism. 
\item A right splitting $\sigma:K\to G$ of an extension as above is a group homomorphism that is a section of the map $G\to K$.
\item Given an action of $K$ on $A$ by group automorphisms, the semi-direct product $K\ltimes A$ is canonically an extension of $K$ by $A$. 
\item Given an extension as above, a $1$-cocycle (with respect to the right action of $G$ on $A$ by conjugation) is a map ${c}:G\to A$ satisfying:
\begin{equation*} {c}(g_1g_2)=({c}(g_1)\cdot g_2){c}(g_2), \quad g_1,g_2\in G.
\end{equation*} 
\end{itemize}
We can now give the desired characterizations of right splittings.
\begin{prop}\label{splitcharprop} Consider an extension 
\begin{equation*} 1\to A\to G\to K\to 1
\end{equation*} of a group $K$ by an abelian group $A$. There are natural bijections between:
\begin{itemize}\item[i)] The set of right splittings $K\to G$.
\item[ii)] The set of isomorphisms of extensions: \begin{equation*} K\ltimes A\to G,\end{equation*} where $K$ acts on $A$ via conjugation in $G$.
\item[iii)] The set of $1$-cocycles $G\to A$ that restrict to the identity map on $A$.
\item[iv)] The set of subgroups $K_G$ in $G$ such that the composite $K_G\hookrightarrow G\to K$ is an isomorphism. 
\end{itemize} 
\end{prop}
\begin{proof} We will define a square of maps:
\begin{center} 
\begin{tikzcd} \textrm{(i)}\arrow[r] & \textrm{(ii)} \arrow[d]\\
 \textrm{(iv)}\arrow[u]& \textrm{(iii)}\arrow[l]
\end{tikzcd}
\end{center} Firstly, given a right splitting $\sigma:K\to G$, the corresponding isomorphism of extensions is:
\begin{equation}\label{splitiso} K\ltimes A\to G,\quad (k,a)\mapsto\sigma(k)a.
\end{equation} 
Secondly, given an isomorphism of extensions $\phi:K\ltimes A\to G$, the corresponding $1$-cocycle is the composition: \begin{equation*}G\xrightarrow{\phi^{-1}}K\ltimes A\xrightarrow{\textrm{pr}_A} A.
\end{equation*} Thirdly, given a $1$-cocycle ${c}:G\to A$ that restricts to the identity map on $A$, the subset: 
\begin{equation}\label{kernel1cocycleeq} K_{{c}}:=\{g\in G\mid {c}(g)=1\}
\end{equation} is a subgroup of $G$ and, because ${c}$ restricts to the identity map on $A$, the composite $K_{{c}}\hookrightarrow G\to K$ is an isomorphism. Finally, given a subgroup $K_G$ of $G$ with the property that the composite $K_G\hookrightarrow G\to K$ is an isomorphism, the inverse to this map yields the corresponding right splitting: 
\begin{equation*} K\to K_G\hookrightarrow G,
\end{equation*} which completes the square. We leave it to the reader to verify that each of the four maps in this triangle is inverse to the composition by the other three, so as to complete the proof. 
\end{proof}

With this and the lemma below, we are ready to prove Proposition \ref{cirepsplit} . 
\begin{lemma} A maximally toric representation of a torus has a unique decomposition into irreducible symplectic subrepresentations. 
\end{lemma}  
\begin{proof} The existence of such a decomposition holds for all symplectic representations of tori. Now suppose $T\to \text{Sp}(V,\omega)$ is a maximally toric representation of an $n$-dimensional torus. Let $V=V_{\alpha_1}\oplus ... \oplus V_{\alpha_n}$ be a decomposition into irreducible symplectic subrepresentations indexed by the corresponding weights $(\alpha_1,...,\alpha_n)$. By Proposition \ref{classtorrep}$a$ the weights are linearly independent. So, since after identifying $V_{\alpha_i}$ with $\C_{\alpha_i}$ as symplectic $T$-representations the momentum map is given by (\ref{quadmom}), each subspace $V_{\alpha_i}$ coincides with $(J_V)^{-1}(\R_+\cdot\alpha_i)$. Hence, the decomposition is indeed unique. 
\end{proof}
\begin{proof}[Proof of Proposition \ref{cirepsplit}] Let $H\to \text{Sp}(V,\omega)$ be a maximally toric representation. After a choice of isomorphism of symplectic $T$-representations, we can assume that: \begin{equation}\label{isoirrep} (V,\omega)=(\C_{\alpha_1},\omega_\textrm{st})\oplus ...\oplus (\C_{\alpha_n},\omega_\textrm{st}), \end{equation} as symplectic $T$-representations. Let $h\in H$. The pair of maps: 
\begin{equation*} (C_h,m_h):(T,(V,\omega))\to (T,(V,\omega)),
\end{equation*} consisting of conjugation and multiplication by $h$, is a self-equivalence of the induced symplectic representation of $T$ on $(V,\omega)$. Therefore, $[h]\in \varGamma$ (which acts on $\t^*$ via the coadjoint $H$-action) permutes the weights $(\alpha_1,...,\alpha_n)$ and, by uniqueness of the decomposition into irreducibles, it follows that $h$ acts as a symplectic $\R$-linear map $\C_{\alpha_i}\to \C_{[h]\cdot \alpha_i}$ for each $i$. Since $J_V$ is given by (\ref{quadmom}) and is $H$-equivariant, any $h$ maps the unit circle $(J_V)^{-1}(\alpha_i)$ to the unit circle $(J_V)^{-1}([h]\cdot \alpha_i)$, and it follows that $h:\C_{\alpha_i}\to \C_{[h]\cdot \alpha_i}$ acts by a rotation determined by an element $h_i\in \mathbb{S}^1$. So, we can associate to every $h\in H$ an element $(h_1,...,h_n)\in \T^n$. Via the isomorphism (\ref{weightisotori}) we obtain a map ${c}:H\to T$. This is a $1$-cocycle for the extension (\ref{sesamabliegp}) that restricts to the identity map on $T$. So, in view of Proposition \ref{splitcharprop}, $H$ is split. 
\end{proof}
\subsubsection{The classification of maximally toric representations}
We will now address the associated group cohomology class appearing in Theorem \ref{isoclasstorrepthm}. Let $(V,\omega)$ be a maximally toric $H$-representation. As in the proof of Proposition \ref{cirepsplit}, the choice of an isomorphism of symplectic $T$-representations:
\begin{equation}\label{choiceisosympTreps} \psi:(V,\omega)\xrightarrow{\sim}(\C_{\alpha_1},\omega_\textrm{st})\oplus ...\oplus (\C_{\alpha_n},\omega_\textrm{st})
\end{equation} induces a $1$-cocycle ${c}_\psi:H\to T$, determined by the fact that for each $h\in H$, $v\in V$ and each $\alpha_i$:
\begin{equation}\label{defprop1coc} \chi_{\alpha_i}({c}_\psi(h))\cdot \psi(v)_{\alpha_i}=\psi(h\cdot v)_{[h]\cdot \alpha_i}.
\end{equation} This $1$-cocycle restricts to the identity map on $T$ and for different choices of $\psi$ the resulting cocycles are cohomologous. Hence, the class: \begin{equation}\label{gpcohclasstorrep} 
e(V,\omega):=[{c}_\psi]\in I^1(H,T)
\end{equation} is independent of the choice of $\psi$. 
\begin{defi}\label{extclasstoricrep} We call (\ref{gpcohclasstorrep}) the \textbf{ext-class} of the maximally toric $H$-representation $(V,\omega)$. 
\end{defi}
In fact, this class is the same for any two maximally toric representations that are isomorphic as symplectic $H$-representations. 
This explains the remaining part of the statement of Theorem \ref{isoclasstorrepthm}. We now turn to its proof. 
\begin{proof}[Proof of Theorem \ref{isoclasstorrepthm}] For injectivity, let $(V_1,\omega_1)$ and $(V_2,\omega_2)$ be maximally toric $H$-representation such that $\Delta_{V_1}=\Delta_{V_2}$ and $e(V_1,\omega_1)=e(V_2,\omega_2)$. Then since $\Delta_{V_1}=\Delta_{V_2}$, it follows from Proposition \ref{classtorrep} that the two symplectic representations have the same weight-tuple, say $(\alpha_1,...,\alpha_n)$. Hence, there are isomorphisms of symplectic $T$-representations:
\begin{equation*} (V_1,\omega_1)\xrightarrow{\psi_1} (\C_{\alpha_1},\omega_\textrm{st})\oplus ...\oplus (\C_{\alpha_n},\omega_\textrm{st}) \xleftarrow{\psi_2} (V_2,\omega_2).
\end{equation*} Since $e(V_1,\omega_1)=e(V_2,\omega_2)$, there is a $t\in T$ such that for each $h\in H$:
\begin{equation}\label{cohomcocyceq} {c}_{\psi_1}(h)=(t\cdot h)t^{-1}{c}_{\psi_2}(h).
\end{equation} Consider: 
\begin{equation*} \psi_t:(V_1,\omega_1)\to (\C_{\alpha_1},\omega_\textrm{st})\oplus ...\oplus (\C_{\alpha_n},\omega_\textrm{st}) , \quad \psi_t(v)=t^{-1}\cdot \psi_1(v),
\end{equation*} which is again an isomorphism of symplectic $T$-representations.  As one readily verifies, it follows from (\ref{cohomcocyceq}) that $\psi_2^{-1}\circ \psi_t:(V_1,\omega_1)\to (V_2,\omega_2)$ is $H$-equivariant, and hence it is an isomorphism of symplectic $H$-representations. This proves injectivity. \\

For surjectivity, let $\Delta$ be a $\varGamma$-invariant and smooth pointed polyhedral cone in $(\t^*,\Lambda_T)$, and let ${c}:H\to T$ be a $1$-cocycle that restricts to the identity on $T$. Let $(\alpha_1,...,\alpha_n)$ be a tuple that generates $\Delta$ and forms a basis of $\Lambda_T^*$. Consider the maximally toric $T$-representation: \begin{equation*}(V,\omega):=(\C_{\alpha_1},\omega_\textrm{st})\oplus ...\oplus (\C_{\alpha_n},\omega_\textrm{st}). 
\end{equation*} Since $\varGamma$ leaves both $\Delta$ and $\Lambda_T^*$ invariant, by Proposition \ref{classtorrep}$b$ it must permute the ordered tuple $(\alpha_1,...,\alpha_n)$. This induces an action of $\varGamma$ on $V$ by permuting the components indexed by this tuple. Using this, the $T$-representation extends to a representation $r:\varGamma\ltimes T \to \textrm{Sp}(\C^n,\omega_\textrm{st})$ by setting $(\gamma,t)\cdot z=\gamma\cdot (t\cdot z)$. By construction, this representation is maximally toric and has momentum image $\Delta$. Next, consider the map:
\begin{equation*} \phi_{{c}}:H\to \varGamma\ltimes T, \quad \phi_{{c}}(h)=([h],{c}(h)). 
\end{equation*} This is the isomorphism of extensions of $\varGamma$ by $T$ corresponding to ${c}$ via the bijection in Proposition \ref{splitcharprop}. Composing $\phi_{{c}}$ with $r$, we obtain a maximally toric $H$-representation for which $\Delta$ is the momentum image. Moreover, the cohomology class associated to this is represented by the $1$-cocycle ${c}_\psi$ where we may pick $\psi$ to be the identity map on $\C^n$. Using (\ref{defprop1coc}), this $1$-cocycle is readily seen to coincide with the given $1$-cocycle ${c}$. So we have constructed a maximally toric $H$-representation with momentum image $\Delta$ and associated group cohomology class $[{c}]$. 
\end{proof}
\begin{rem}\label{firststrthmlinversrem} The analogue of Example \ref{almabcompLiegpexintro0} holds as well in this linear setting: by Theorem \ref{isoclasstorrepthm} and Remark \ref{degonegpcohomtorsorrem} the set of isomorphism of class of maximally toric $H$-representations with momentum image a prescribed $\Delta$ is a torsor with structure group $H^1(\varGamma,T)$, provided $H$ is split, and (using the lemma below) it is readily verified that this action can be described more explicitly, as follows. Consider the isomorphism between the abelian group of $1$-cocycles ${c}:\varGamma\to T$ and the group of automorphisms of the extension (\ref{sesamabliegp}) (with group structure given by composition of maps), that associates to a $1$-cocycle ${c}$ the automorphism: 
\begin{equation*} \phi_{{c}}:H\to H, \quad h\mapsto h{c}([h]). 
\end{equation*} This descends to an isomorphism between $H^1(\varGamma,T)$ and the group of automorphisms of (\ref{sesamabliegp}) modulo the subgroup of automorphisms given by conjugation by elements of $T$. Now, the $H^1(\varGamma,T)$-action on the set of isomorphism classes of maximally toric $H$-representations is given by:
\begin{equation*} [{c}]\cdot[(V,\omega)]:=[(V_{{c}},\omega)],
\end{equation*} where $(V_{{c}},\omega)$ is equal to $(V,\omega)$ as symplectic vector space, but equipped with the linear symplectic action given by $h\cdot_{{c}} v=\phi_{{c}}(h)\cdot v$, for $h\in H$ and $v\in V$, where the right-hand dot denotes the original action of $H$ on $V$.
\end{rem}

In the above remark we referred to:
\begin{lemma}\label{smpeqtorrepindisoextinvprop} Suppose that we are given an equivalence of maximally toric representations of infinitesimally abelian compact Lie groups:
\begin{equation*} (\phi,\psi):(H_1,(V_1,\omega_1))\to (H_2,(V_2,\omega_2)).\end{equation*} 
Then the induced isomorphism: 
\begin{equation*} \phi_*:I^1(H_1,T_1)\xrightarrow{\sim} I^1(H_2,T_2)
\end{equation*} sends $e(V_1,\omega_1)$ to $e(V_2,\omega_2)$.
\end{lemma}

This lemma is straightforward to verify. It will also be useful for later reference. 
\subsection{Delzant subspaces of integral affine orbifolds}\label{momimsec}
\subsubsection{Background on orbifolds}\label{orbdefsec} Following \cite{CrFeTo}, we use the terminology below for orbifolds.
\begin{itemize} 
\item An \textbf{orbifold groupoid} is a proper foliation groupoid (that is, a proper Lie groupoid with discrete isotropy groups).
\item By an \textbf{orbifold atlas} on a topological space $B$ we mean an orbifold groupoid $\B\rightrightarrows M$, together with a homeomorphism $p$ between the {orbit} space $\underline{M}$ and $B$.
\item By an \textbf{orbifold} we mean a pair consisting of a topological space $B$ together with an orbifold atlas $(\B,p)$ on $B$.
\item We call two orbifold atlases on $B$ \textbf{equivalent} if there is a Morita equivalence between the given orbifold groupoids that intertwines the respective homeomorphisms between their {orbit} spaces and $B$.   
\end{itemize} This approach to orbifolds (using groupoids instead of atlases of charts) is in the spirit of \cite{Mo1}.
\subsubsection{Background on integral affine structures on orbifolds}\label{iaorbdefsec} 
An integral affine atlas on a manifold $M$ is one for which the coordinate changes are (restrictions of) integral affine transformations of $\R^n$. A maximal such atlas is called an \textbf{integral affine structure} on $M$. Such a structure can be encoded globally, as follows. Given a vector bundle $E\to M$, a smooth lattice in $E$ is a subbundle $\Lambda$ with the property that, for each $x_0\in M$, there is a local frame $e$ of $E$ defined on an open neighbourhood $U$ of $x_0$, such that for all $x\in U$:
\begin{equation*} \Lambda_x=\Z (e_1)_x\oplus...\oplus \Z (e_n)_x.
\end{equation*} In particular, $\Lambda_x$ is a full rank lattice in $E_x$ for each $x\in M$. The data of an integral affine structure on $M$ is equivalent to that of a smooth lattice $\Lambda$ in the cotangent bundle $T^*M$ satisfying the integrability condition that $\Lambda$ is locally spanned by closed $1$-forms, or equivalently, that $\Lambda$ is Lagrangian as submanifold of $(T^*M,\Omega_\textrm{can})$. {The corresponding lattice is locally spanned by the coframe associated to any integral affine chart. See \cite{CrFeTo} for more details.
\begin{ex} The $n$-torus $\T^n$ has an integral affine structure with lattice $\Z\d\theta_1\oplus ...\oplus \Z\d\theta_n$. 
\end{ex}
\begin{rem} Besides the $2$-torus, the only other compact surface that admits an integral affine structure is the Klein bottle \cite{Benz,Miln} and the various such structures on these have been classified \cite{Mish1,Sep}. There are also classification results for integral affine structures on three-dimensional compact manifolds \cite{Koz}. 
\end{rem} }

{The }global description is well-suited for a generalization of integral affine structures to orbifolds. First of all, the notion of vector bundle generalizes: a vector bundle over an orbifold $(B,\B,p)$ is a representation of $\B$, meaning that it is a vector bundle $E\to M$ (in the sense of manifolds) equipped with a fiberwise linear action of $\B$. For example, the tangent bundle of an orbifold $(B,\B,p)$ is the canonical representation of $\B$ on the normal bundle $\No\F$ to the foliation $\F$ on $M$ by connected components of the {orbits} of $\B$. Explicitly, this representation is given by:
\begin{equation}\label{normrep} g\cdot [v]=[\d t(\widehat{v})], \quad g\in \B, \quad v\in \No_{s(g)}\F,
\end{equation} where $\widehat{v}\in T_g\B$ is any choice of tangent vector such that $\d s(\widehat{v})=v$. If $\B$ is source-connected, then this coincides with the linear holonomy representation. The cotangent bundle of the orbifold is the dual representation of $\B$ on the co-normal bundle:
\begin{equation*} \No^*\F=T\F^0\subset T^*M.
\end{equation*} Now, the definition of an integral affine structure generalizes to orbifolds, as follows.
\begin{defi} An integral affine structure on an orbifold $(B,\B,p)$ is a smooth lattice $\Lambda$ in $\No^*\F$ with the property that $\Lambda$ is Lagrangian as submanifold of $(T^*M,\Omega_\textrm{can})$ and $\Lambda$ is invariant with respect to the co-normal representation of $\B$. We call $(B,\B,p,\Lambda)$ an \textbf{integral affine orbifold}.
\end{defi}
\begin{ex}\label{example:IAorb:discrgpaction}
 {For a proper action of a countable discrete group $\Gamma$ on a manifold $M$, the data of an integral affine structure on the orbit space $M/\Gamma$ (viewed as orbifold with orbifold groupoid the action groupoid $\Gamma\ltimes M$) is that of a $\Gamma$-invariant integral affine structure on $M$.   }
\end{ex} 
\begin{ex} Given a foliated manifold $(M,\F)$, the data of a smooth Lagrangian lattice in $\No^*\F$ is the same as that of a transverse integral affine structure on $(M,\F)$. If an orbifold groupoid $\B\rightrightarrows M$ is source-connected, then every smooth Lagrangian lattice in $\No^*\F$ is automatically $\B$-invariant, so that in this case the data of an integral affine structure on $(B,\B,p)$ is simply that of a transverse integral affine structure on the associated foliation $\F$ on $M$.
\end{ex}
\subsubsection{Background on the integral affine orbifold associated to a regular proper symplectic groupoid}\label{iaorbsec} Let $(\G,\Omega)\rightrightarrows M$ be a regular and proper symplectic groupoid. Let $\F$ denote the foliation of $M$ by connected components of the $\G$-orbits, and let $\underline{M}:=M/\G$ denote the {orbit} space of $\G$. It follows from \cite[Proposition 2.5]{Mo} that there is a canonical short exact sequence of Lie groupoids over $M$:
\begin{equation}\label{sestorbun} 1\to \sT\to \G\to \B\to 1
\end{equation} where $\sT$ is the bundle of Lie groups with fiber $\sT_x$ the identity component of the isotropy group $\G_x$ of $\G$ at $x\in M$, and $\B=\G/\sT$ is an orbifold groupoid\index{Groupoid!orbifold} over $M$. The fibers of $\sT$ are in fact tori and the orbifold $(\underline{M},\B,\textrm{Id}_{\underline{M}})$ comes with a natural integral affine structure. To see this, recall first that, as for any symplectic groupoid, the conormal space $\No^*_x\F$ at $x\in M$ can be canonically identified with the isotropy Lie algebra $\g_x$ of $\G$ at $x$ via the isomorphism of Lie algebroids:
\begin{equation}\label{imsymp} \rho_\Omega:T_\pi^*M\to \text{Lie}(\G), \quad \iota_{\rho_\Omega(\alpha)}\Omega_{1_x}=(\d t_{1_x})^*\alpha, \quad \alpha\in T^*_xM,\text{ } x\in M,
\end{equation} where $T^*_\pi M$ denotes the cotangent bundle equipped with the Lie algebroid structure associated to the Poisson structure $\pi$ on $M$ induced by $(\G,\Omega)$. Since the Poisson structure $\pi$ is regular, its isotropy Lie algebras are abelian. Hence, so are the isotropy Lie algebras of $\G$. Since $\G$ is proper, its isotropy groups are compact. Therefore, $\sT$ is a bundle of tori and the kernel of each exponential map $\g_x\to \G_x^0=\sT_x$ determines a full rank lattice $\Lambda_x$ in $\No^*_x\F$. All together, this yields a map of Lie groupoids:
\begin{equation}\label{exptorbun} \No^*\F\to \sT
\end{equation} with kernel the desired smooth lattice $\Lambda$ in $\No^*\F$. 
\begin{rem}\label{presymptorbunrem} The map (\ref{exptorbun}) factors through an isomorphism of Lie groupoids:
\begin{equation}\label{exptorbun2} \sT_\Lambda:=\No^*\F/\Lambda\xrightarrow{\sim} \sT,
\end{equation} 
and the co-normal representation of $\B$ on $\No^*\F$ descends to an action of $\B$ on $\sT_\Lambda$, which under the above isomorphism is identified with the action of $\B$ on $\sT$ by conjugation. The symplectic form $\Omega$ on $\G$ restricts to a pre-symplectic form $\Omega_\sT$ on $\sT$, which makes:\begin{equation}\label{presymtorbuneq} (\sT,\Omega_\sT)\to M
\end{equation} into a pre-symplectic torus bundle. On the other hand, the canonical symplectic form on $T^*M$ restricts to a pre-symplectic form on $\No^*\F$, which in turn descends to a pre-symplectic form $\Omega_\Lambda$ on $\sT_\Lambda$. The map (\ref{exptorbun2}) identifies $\Omega_\sT$ with $\Omega_\Lambda$. So, (\ref{presymtorbuneq}) is fully encoded by the integral affine orbifold associated to $(\G,\Omega)$. 
\end{rem}

\subsubsection{Delzant subspaces}\label{delzsuborbdefsec} We define Delzant subspaces of integral affine manifolds as follows. 
\begin{defi}\label{delzsubmandefi} A \textbf{Delzant subspace $\Delta$ of an integral affine manifold} $(M,\Lambda)$ is a subset of $M$ with the property that for every $x\in \Delta$ and every (or equivalently some) integral affine chart $(U,\chi)$ around $x$ into $\R^n$, there is a smooth polyhedral cone $C^\chi_x(\Delta)$ in $(\R^n,\Z^n)$ (in the sense of Subsection \ref{torrepoftorsec}) such that the germ of $\chi(U\cap \Delta)$ at $\chi(x)$ in $\R^n$ is that of $\chi(x)+C^\chi_x(\Delta)$ at $\chi(x)$.
\end{defi}
{ 
As mentioned in the introduction: Delzant subspaces are embedded submanifolds with corners of codimension zero (as in Definition \ref{embsubmanwithcorndefi}), Delzant polytopes are examples of Delzant subspaces, and so are integral affine manifolds (being Delzant subspaces of themselves). Below we give a few other examples. 
\begin{ex} Let $(M,\Lambda):=(\R\times \mathbb{S}^1,\Z\d x\oplus\Z\d\theta)$ \textemdash a cylinder equipped with its standard integral affine structure. The cylinder with boundary $\Delta:=[0,1]\times \mathbb{S}^1$ is a Delzant subspace. 
\end{ex}
\begin{ex}\label{example:exoticdelzsubsp}
Let $(M,\Lambda):=(\R\times\mathbb{S}^1,\Z(\d x+(\theta-\frac{1}{2})\d\theta)\oplus \Z\d\theta)$ \textemdash a cylinder equipped with a non-standard integral affine structure. The subspace:
\begin{equation*} \Delta:=\{(x,e^{2\pi i\theta})\in \R\times\mathbb{S}^1\mid x+\theta(\theta-1)/2\geq 0,\text{ }\theta\in [0,1[\}
\end{equation*} is a Delzant subspace with a single vertex and a single stratum of codimension one (see Figure \ref{figure1}).
\end{ex}}
\begin{figure}[H]
  \centering
  \includegraphics[scale=0.2]{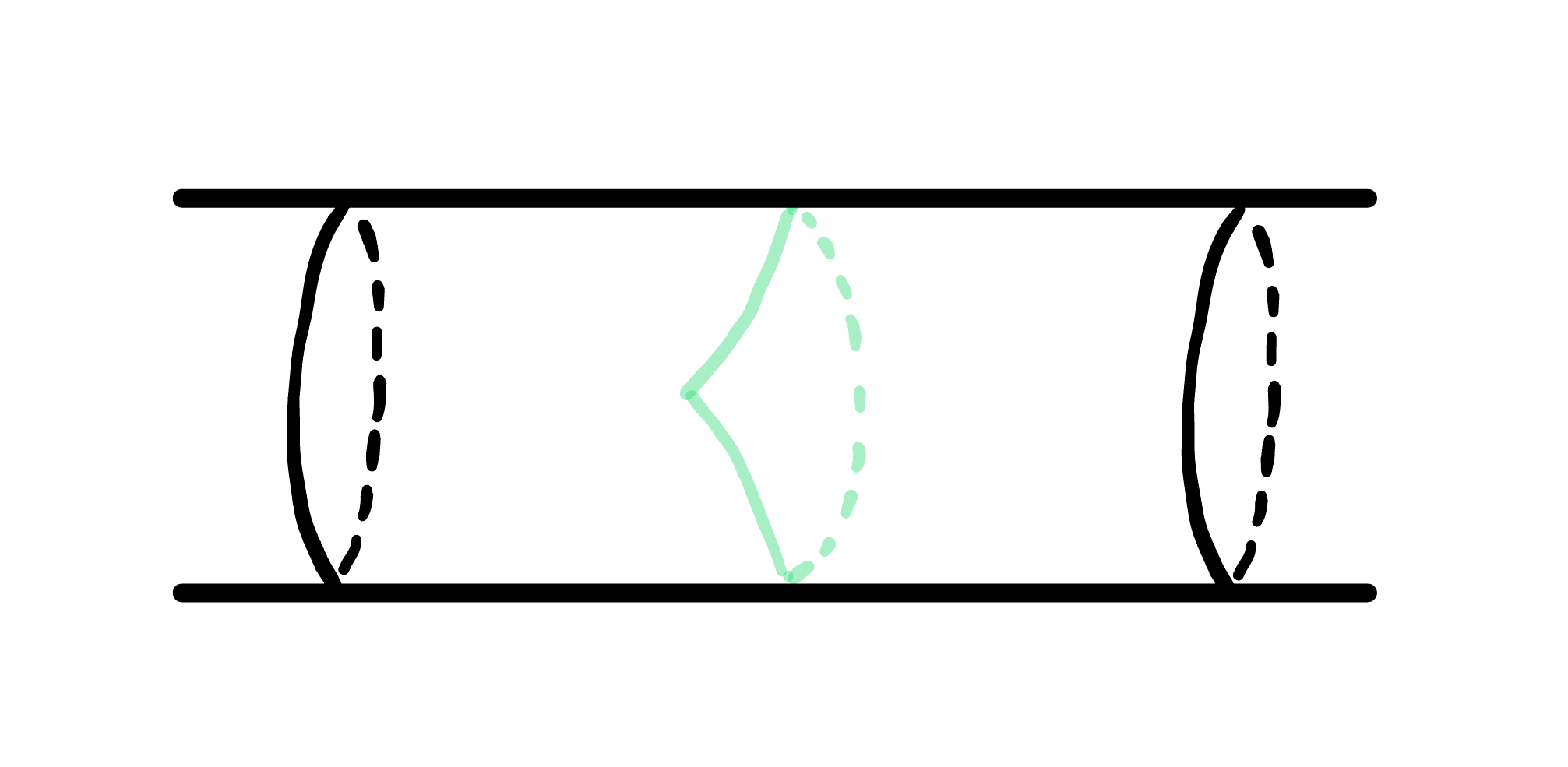}
    \caption{\footnotesize{A sketch of the boundary $\partial\Delta$ (coloured in green) of the Delzant subspace in Example \ref{example:exoticdelzsubsp}.}}
    \label{figure1}
  \end{figure}
To extend this definition to orbifolds (in the sense of Subsection \ref{orbdefsec}) it is convenient to have a coordinate-free description of Delzant subspaces of integral affine manifolds. To this end, notice that around each point in an integral affine manifold $(M,\Lambda)$ there is one natural choice of integral affine `chart'. More precisely, around each $x\in M$ there is unique map germ:
\begin{equation*} \log_x\in \textrm{Germ}_x(M;T_xM)
\end{equation*} induced by an integral affine isomorphism $\iota$ from $(U,\Lambda)$ onto an open in $(T_xM,\Lambda_x^*)$, that maps $x$ to the origin in $T_xM$ and the derivative of which at $x$ is the identity map on $T_xM$. Given a subset $\Delta$ of $M$ and an $x\in \Delta$, there is an associated set germ $\log_x(\Delta)$ at the origin in $T_xM$, defined as the germ of $\iota(U\cap \Delta)$ at the origin, which is independent of the choice of $\iota$ as above. Now, $\Delta$ is a Delzant subspace of $(M,\Lambda)$ if and only if for every $x\in \Delta$ the set germ $\log_x(\Delta)$ is the germ of a smooth polyhedral cone in the integral affine vector space $(T_xM,\Lambda_x^*)$ (in the sense of Subsection \ref{torrepoftorsec}). To define Delzant subspaces of integral affine orbifolds we will now generalize this characterization, starting with: 
\begin{prop}\label{logmapgermprop} Let $(M,\F,\Lambda)$ be a foliated manifold with a transverse integral affine structure and let $x\in M$. There is a unique map germ:
\begin{equation*} \log_x\in \textrm{Germ}_x(M;\No_x\F)
\end{equation*} induced by a submersion $\nu$ defined on an open $U$ around $x$ in $M$, with the following properties.
\begin{itemize} \item[i)] The tangent distribution to the fibers of $\nu$ coincides with that of the foliation $\F$ over $U$.
\item[ii)] It maps $x$ to the origin in $\No_x\F$ and its differential at $x$ is the projection $T_xM\to \No_x\F$.
\item[iii)] It is compatible with the integral affine structure, in the sense that for each $y\in U$:
\begin{equation*} \d \nu_y:(\No_y\F,\Lambda_y^*)\to (\No_x\F,\Lambda_x^*)
\end{equation*} is an isomorphism of integral affine vector spaces. 
\end{itemize} 
\end{prop}
\begin{proof} First, we prove existence. Since $\Lambda$ is a smooth lattice in $\No^*\F$, we can choose a local frame $\alpha$ of $\No^*\F$, defined on an open $U$ around $x$, such that:
\begin{equation*} \Lambda\vert_U=\Z\text{ }\alpha_1\oplus ...\oplus \Z\text{ }\alpha_n.
\end{equation*} Since $\Lambda$ is Lagrangian in $T^*M$, all of its local sections are closed $1$-forms. So, by the Poincar\'e Lemma, we can (after shrinking $U$) arrange that $\alpha_i=\d f_i$ for some $f_i\in C^\infty(U)$ such that $f_i(x)=0$. Consider: \begin{equation*} f=(f_1,...,f_n):U\to \R^n.
\end{equation*} Then $(\underline{\d f})_x: \No_x\F\to \R^n$ is a linear isomorphism, so that we can define 
\begin{equation*} \nu=(\underline{\d f})_x^{-1}\circ f: U\to \No_x\F.
\end{equation*} As is readily verified, this has the desired properties. To prove uniqueness, let $\nu_1:U_1\to \No_x\F$ and $\nu_2:U_2\to \No_x\F$ be two submersions as above. Since both $\nu_1$ and $\nu_2$ satisfy property i), we can find an open neighbourhood $U$ of $x$ in $U_1\cap U_2$, together with a connected transversal $\Sigma$ to $\F$ through $x$, with the property that:
\begin{itemize} \item $\Sigma$ is contained in $U$ and every leaf of the foliation on $U$ induced by $\F$ intersects $\Sigma$,
\item both $\nu_1\vert_{\Sigma}$ and $\nu_2\vert_{\Sigma}$ are open embeddings into $\No_x\F$.  
\end{itemize} The transversal $\Sigma$ inherits an honest integral affine structure $\Lambda_\Sigma$ from the transverse integral affine structure $\Lambda$, and by property iii) both $\nu_1\vert_{\Sigma}$ and $\nu_2\vert_{\Sigma}$ are isomorphisms of integral affine manifolds onto their image in $(\No_x\F,\Lambda_x^*)$ with respect to $\Lambda_\Sigma$. Therefore, $\nu_1\vert_{\Sigma}\circ (\nu_2\vert_\Sigma)^{-1}$ is a morphism of integral affine manifolds between connected opens in the integral affine vector space $(\No_x\F,\Lambda_x^*)$. By the lemma below, this means that it must be the restriction of an integral affine transformation of $(\No_x\F,\Lambda_x^*)$. So, it is determined by its value and its derivative at the origin in $\No_x\F$. By property ii), $\nu_1\vert_{\Sigma}\circ (\nu_2\vert_\Sigma)^{-1}$ fixes the origin and its derivative at the origin is the identity map. Hence, $\nu_1\vert_{\Sigma}\circ (\nu_2\vert_\Sigma)^{-1}$ must be the restriction of the identity map on $\No_x\F$ to $\nu_2(\Sigma)$, so that $\nu_1\vert_{\Sigma}=\nu_2\vert_{\Sigma}$. It follows from this and property i) that in fact $\nu_1\vert_U=\nu_2\vert_U$, because every leaf of the foliation on $U$ induced by $\F$ intersects $\Sigma$. So, $\nu_1$ and $\nu_2$ have the same germ at $x$, as was to be shown.
\end{proof}
\begin{lemma}\label{iamorphintaffvectsplemma} Let $(V_1,\Lambda_1)$ and $(V_2,\Lambda_2)$ be integral affine vector spaces. Then every morphism of integral affine manifolds from a connected open in $(V_1,\Lambda_1)$ into $(V_2,\Lambda_2)$ is of the form $v\mapsto Av+b$ for some linear map $A:V_1\to V_2$ that maps $\Lambda_1$ into $\Lambda_2$ and some $b\in V_2$. 
\end{lemma}
\begin{proof} After a choosing bases for $\Lambda_1$ and $\Lambda_2$ we can assume that $(V_1,\Lambda_1)=(\R^{n_1},\Z^{n_1})$ and $(V_2,\Lambda_2)=(\R^{n_2},\Z^{n_2})$. Let $f$ be a morphism of integral affine manifolds from a connected open $U_1$ in $(\R^{n_1},\Z^{n_1})$ into $(\R^{n_2},\Z^{n_2})$, meaning that its partial derivatives take values in $\Z$. Since $U_1$ is connected and $\Z$ is a discrete subspace of $\R$, the Jacobian of $f$ must be constant. Fix a $u_1\in U_1$. Since any two points in $U_1$ can be connected to $u_1$ by a smooth path, it follows by integrating along such paths that $f(v)=A(v)+b$ for all $v\in U_1$, where $A$ is the constant value of the Jacobian of $f$ and $b=f(u_1)-A(u_1)$. 
\end{proof}
Now, let $(M,\F,\Lambda)$ be a foliated manifold with a transverse integral affine structure, let $\Delta$ be a subset of $M$ and let $x\in \Delta$. Then, as for integral affine manifolds, there is an associated set germ $\log_x(\Delta)$ at the origin in $\No_x\F$. To define this, let us call a submersion $\nu:U\to \No_x\F$ representing $\log_x$ \textbf{$\Delta$-adapted} if: 
\begin{equation*} \nu^{-1}(\nu(U\cap \Delta))=U\cap \Delta.
\end{equation*} The set-germ of $\nu(U\cap \Delta)$ at the origin in $\No_x\F$ is independent of the choice of $\Delta$-adapted submersion $\nu:U\to \No_x\F$ representing $\log_x$. Moreover, if $\Delta$ is $\F$-invariant, we can always find a small enough open $U$ around $x$ in $M$ that admits a $\Delta$-adapted submersion $\nu:U\to \No_x\F$ representing $\log_x$. Therefore, it makes sense to define:
\begin{defi} Let $(M,\F,\Lambda)$ be a foliated manifold with a transverse integral affine structure, let $\Delta$ be an $\F$-invariant subset of $M$ and let $x\in \Delta$. We define $\log_x(\Delta)$ to be the set germ of $\nu(U\cap \Delta)$ at the origin in $\No_x\F$, for any $\Delta$-adapted submersion $\nu:U\to \No_x\F$ representing $\log_x$. 
\end{defi}
We are now ready to define Delzant subspaces of integral affine orbifolds. 
\begin{defi}\label{defdelzorb} Let $(B,\B,p,\Lambda)$ be an integral affine orbifold. A \textbf{Delzant subspace} $\underline{\Delta}$ is a subset of $B$ with the property that for every $x\in \Delta$ (the corresponding invariant subset of $M$), the set germ $\log_x(\Delta)$ is the germ of a smooth polyhedral cone in the integral affine vector space $(\No_x\F,\Lambda_x^*)$ (in the sense of Subsection \ref{torrepoftorsec}). For each $x\in \Delta$, we denote this polyhedral cone in $\No_x\F$ (which is necessarily unique) by $C_x(\underline{\Delta})$ and call it the \textbf{cone of $\underline{\Delta}$ at $x$}.
\end{defi}
\begin{rem}\label{etalecasedelzsubmanrem} Let $(B,\B,p,\Lambda)$ be an integral affine orbifold and let $\underline{\Delta}$ be a Delzant subspace. Then $\Delta$ is an embedded submanifold with corners of $M$ of codimension zero (as in Definition \ref{embsubmanwithcorndefi}), with tangent cone $C_x(\Delta)$ the pre-image of $C_x(\underline{\Delta})$ under the projection $T_xM\to \No_x\F$. The restriction $\B\vert_\Delta$ is a Lie groupoid with corners (as in Definition \ref{liegpwithcorndef}; cf. Example \ref{prototamesubmex}). 
If $\underline{M}$ is the {orbit} space of a regular and proper symplectic groupoid $(\G,\Omega)$ (equipped with the associated integral affine orbifold structure), then $(\G,\Omega)\vert_\Delta$ is a symplectic groupoid with corners (as in Definition \ref{sympgpwithcorndef}). 
\end{rem}
\begin{ex}{ Let $\B\rightrightarrows (M,\Lambda)$ be an etale integral affine orbifold groupoid. For any $\B$-invariant Delzant subspace $\Delta$ of the integral affine manifold $(M,\Lambda)$ (in the sense of Definition \ref{delzsubmandefi}), the quotient $\underline{\Delta}$ is a Delzant subspace of $\underline{M}$. In fact, any Delzant subspace of $\underline{M}$ is of this form. 
}
\end{ex} 
{
\begin{ex}\label{ex:compLiegp:Delzantsubspleafsp:1} Let $G$ be a compact Lie group. Consider the open and dense subset $\g^*_\textrm{reg}$ of $\g^*$ consisting of coadjoint $G$-orbits of maximal dimension and consider the regular and proper symplectic groupoid $(\G,\Omega):=(G\ltimes \g^*_\textrm{reg},-\d\lambda_\textrm{can})$ over $M:=\g^*_\textrm{reg}$ (the restriction of the cotangent symplectic groupoid of $G$). Fix a maximal torus $T$ in $G$ and let $\c$ be the interior of a Weyl chamber in $\t^*$. As is common, we identify $\t^*$ with the $T$-fixed point set in $\g^*$. Let $N(\c)$ be the normalizer of $\c$ in $G$. Delzant subspaces of $\underline{M}=\g^*_\textrm{reg}/G$ are naturally in bijection with Delzant subspaces of the integral affine vector space $(\t^*,\Lambda_T)$ that are contained in $\c$ and that are invariant with respect to the natural action of the finite group $N(\c)/T$ on $\c$. If $G$ is connected, then $N(\c)=T$, so that the invariance condition becomes vacuous. A conceptual explanation for this correspondence will be given in Example \ref{ex:compLiegp:Delzantsubspleafsp:2}, using Corollary \ref{delzsubspmorinvcor} below. 
\end{ex}
}
We think of a Delzant subspace as what should be an integral affine suborbifold with corners, without making this precise. In particular, these objects should be well-behaved with respect to equivalences of the ambient integral affine orbifold that respect the integral affine structure. The latter we will make precise, for besides its conceptual value it will be of use throughout.  
\begin{defi}\label{iaorbgpoiddef} By an \textbf{integral affine orbifold groupoid} $\B \rightrightarrows (M,\Lambda)$ we mean an orbifold groupoid $\B\rightrightarrows M$ with an integral affine structure $\Lambda$ on the orbifold ($\underline{M},\B,\textrm{Id}_{\underline{M}}$) (its {orbit} space). We call a Delzant subspace $\underline{\Delta}$ of this integral affine orbifold simply a Delzant subspace of $\underline{M}$. 
\end{defi}
\begin{defi}
By an \textbf{integral affine Morita equivalence} between integral affine orbifold groupoids $\B_1\rightrightarrows (M_1,\Lambda_1)$ and $\B_2\rightrightarrows (M_2,\Lambda_2)$ we mean a Morita equivalence: 
\begin{center}
\begin{tikzpicture} \node (G1) at (0,0) {$\B_1$};
\node (M1) at (0,-1.3) {$M_1$};
\node (S) at (1.4,0) {$P$};
\node (M2) at (2.7,-1.3) {$M_2$};
\node (G2) at (2.7,0) {$\B_2$};
 
\draw[->,transform canvas={xshift=-\shift}](G1) to node[midway,left] {}(M1);
\draw[->,transform canvas={xshift=\shift}](G1) to node[midway,right] {}(M1);
\draw[->,transform canvas={xshift=-\shift}](G2) to node[midway,left] {}(M2);
\draw[->,transform canvas={xshift=\shift}](G2) to node[midway,right] {}(M2);
\draw[->](S) to node[pos=0.25, below] {$\text{ }\text{ }\alpha_1$} (M1);
\draw[->] (0.8,-0.15) arc (315:30:0.25cm);
\draw[<-] (1.9,0.15) arc (145:-145:0.25cm);
\draw[->](S) to node[pos=0.25, below] {$\alpha_2$\text{ }} (M2);
\end{tikzpicture}
\end{center} with the additional property that $\alpha_1^*(\Lambda_1)=\alpha_2^*(\Lambda_2)$ as subbundles of $\alpha_1^*(\No^*\F_1)=\alpha_2^*(\No^*\F_2)$. 
\end{defi}
\begin{rem}\label{altdefiamoreqrem} A Morita equivalence as above is integral affine if and only if for each $p\in P$, writing $x_1=\alpha_1(p)$ and $x_2=\alpha_2(p)$, the induced linear isomorphism:
\begin{equation}\label{moreqindisoisotreps} \psi_p:\No_{x_1}\F_1\xrightarrow{\sim} \No_{x_2}\F_2,\quad [v]\mapsto [\d\alpha_2(\widehat{v})],
\end{equation} where $\widehat{v}\in T_pP$ is any tangent vector with the property that $\d\alpha_1(\widehat{v})=v$, is an isomorphism of integral affine vector spaces:
\begin{equation*} \psi_p:(\No_{x_1}\F_1,(\Lambda_1)^*_{x_1})\xrightarrow{\sim} (\No_{x_2}\F_2,(\Lambda_2)^*_{x_2}).
\end{equation*} In particular, for each such $p\in P$ there is an induced isomorphism of tori:
\begin{equation*} (\psi_p)_*:(\sT_{\Lambda_1})_{x_1}\xrightarrow{\sim}(\sT_{\Lambda_2})_{x_2}.
\end{equation*} 
\end{rem}

\begin{rem}\label{iamoreqtransportrem} Given a Morita equivalence between orbifold groupoids $\B_1\rightrightarrows M_1$ and $\B_2\rightrightarrows M_2$, and an integral affine structure $\Lambda_1$ on the orbifold $(\underline{M}_1,\B_1,\textrm{Id}_{\underline{M}_1})$, there is a unique integral affine structure $\Lambda_2$ on the orbifold $(\underline{M}_2,\B_2,\textrm{Id}_{\underline{M}_2})$ with respect to which the given Morita equivalence becomes integral affine.
\end{rem}

\begin{ex}\label{transversalmoreqex} Let $\B\rightrightarrows (M,\Lambda)$ be an integral affine orbifold groupoid and $\Sigma$ a transversal for $\B$ {(by which we mean a transversal to the foliation $\F$ by connected components of $\B$-orbits, of complementary dimension)}. There is a canonical Morita equivalence:
\begin{center}
\begin{tikzpicture} \node (G1) at (-0.6,0) {$\B\vert_{\widehat{\Sigma}}$};
\node (M1) at (-0.6,-1.3) {$\widehat{\Sigma}$};
\node (S) at (1.4,0) {$s_\B^{-1}(\Sigma)$};
\node (M2) at (3,-1.3) {$\Sigma$};
\node (G2) at (3,0) {$\B\vert_\Sigma$};
 
\draw[->,transform canvas={xshift=-\shift}](G1) to node[midway,left] {}(M1);
\draw[->,transform canvas={xshift=\shift}](G1) to node[midway,right] {}(M1);
\draw[->,transform canvas={xshift=-\shift}](G2) to node[midway,left] {}(M2);
\draw[->,transform canvas={xshift=\shift}](G2) to node[midway,right] {}(M2);
\draw[->](S) to node[pos=0.25, below] {$\text{ }\text{ }t_\B$} (M1);
\draw[->] (0.6,-0.15) arc (315:30:0.25cm);
\draw[<-] (2.1,0.15) arc (145:-145:0.25cm);
\draw[->](S) to node[pos=0.3, below] {$s_\B$\text{ }} (M2);
\end{tikzpicture} 
\end{center} where $\widehat{\Sigma}:=t(s^{-1}(\Sigma))$ denotes the saturation of $\Sigma$ with respect to $\B$ (which is open in $M$). The manifold $\Sigma$ inherits an honest integral affine structure $\Lambda_\Sigma$ from $\Lambda$, $\B\vert_{\Sigma}\rightrightarrows (\Sigma,\Lambda_{\Sigma})$ is an etale integral affine orbifold groupoid and the above Morita equivalence becomes integral affine. 
 
\end{ex}

\begin{ex}\label{sympmoreqiamoreq} Let $(\G_1,\Omega_1)\rightrightarrows M_1$ and $(\G_2,\Omega_2)\rightrightarrows M_2$ be regular and proper symplectic groupoids and let $\B_1\rightrightarrows (M_1,\Lambda_1)$ and $\B_2\rightrightarrows (M_2,\Lambda_2)$ be the associated integral affine orbifold groupoids (as in the previous subsection). A symplectic Morita equivalence:
\begin{center}
\begin{tikzpicture} \node (G1) at (-0.5,0) {$(\G_1,\Omega_1)$};
\node (M1) at (-0.5,-1.3) {$M_1$};
\node (S) at (1.4,0) {$(P,\omega_P)$};
\node (M2) at (3.2,-1.3) {$M_2$};
\node (G2) at (3.2,0) {$(\G_2,\Omega_2)$};
 
\draw[->,transform canvas={xshift=-\shift}](G1) to node[midway,left] {}(M1);
\draw[->,transform canvas={xshift=\shift}](G1) to node[midway,right] {}(M1);
\draw[->,transform canvas={xshift=-\shift}](G2) to node[midway,left] {}(M2);
\draw[->,transform canvas={xshift=\shift}](G2) to node[midway,right] {}(M2);
\draw[->](S) to node[pos=0.25, below] {$\text{ }\text{ }\alpha_1$} (M1);
\draw[->] (0.65,-0.15) arc (315:30:0.25cm);
\draw[<-] (2.05,0.15) arc (145:-145:0.25cm);
\draw[->](S) to node[pos=0.25, below] {$\alpha_2$\text{ }} (M2);
\end{tikzpicture}
\end{center} 
induces an integral affine Morita equivalence:
\begin{center}
\begin{tikzpicture} \node (G1) at (-0.4,0) {$\B_1$};
\node (M1) at (-0.4,-1.3) {$(M_1,\Lambda_1)$};
\node (S) at (1.4,0) {$\underline{P}$};
\node (M2) at (3.1,-1.3) {$(M_2,\Lambda_2)$};
\node (G2) at (3.1,0) {$\B_2$};
 
\draw[->,transform canvas={xshift=-\shift}](G1) to node[midway,left] {}(M1);
\draw[->,transform canvas={xshift=\shift}](G1) to node[midway,right] {}(M1);
\draw[->,transform canvas={xshift=-\shift}](G2) to node[midway,left] {}(M2);
\draw[->,transform canvas={xshift=\shift}](G2) to node[midway,right] {}(M2);
\draw[->](S) to node[pos=0.25, below] {$\text{ }\text{ }\underline{\alpha}_1$} (M1);
\draw[->] (0.75,-0.15) arc (315:30:0.25cm);
\draw[<-] (1.95,0.15) arc (145:-145:0.25cm);
\draw[->](S) to node[pos=0.25, below] {$\underline{\alpha}_2$\text{ }} (M2);
\end{tikzpicture}
\end{center} 
where $\underline{P}=P/\sT_1=P/\sT_2$. To see this, let $p\in P$ and denote $x_1=\alpha_1(p)$ and $x_2=\alpha_2(p)$. The given Morita equivalence induces an isomorphism of Lie groups:
\begin{equation}\label{moreqindisoisotgps} \phi_p:(\G_1)_{x_1}\to (\G_2)_{x_2},\end{equation} uniquely determined by the property that for each $g\in (\G_1)_{x_1}$:
\begin{equation*} g\cdot p=p\cdot \phi_p(g). 
\end{equation*} Since $\phi_p$ is an isomorphism of Lie groups, it identifies the identity component of $(\G_1)_{x_1}$ with the identity component of $(\G_2)_{x_2}$. In other words, it identifies $(\sT_1)_{x_1}$ with $(\sT_2)_{x_2}$ and hence it follows that the $\sT_1$-orbit through $p$ coincides with the $\sT_2$-orbit through $p$. This shows that $P/\sT_1=P/\sT_2$. From the lemma below, it is clear that the induced Morita equivalence is integral affine.
\end{ex}
\begin{lemma}\label{sympmoreqnormreplem} Suppose that we are given a symplectic Morita equivalence:
\begin{center}
\begin{tikzpicture} \node (G1) at (-0.5,0) {$(\G_1,\Omega_1)$};
\node (M1) at (-0.5,-1.3) {$M_1$};
\node (S) at (1.4,0) {$(P,\omega_P)$};
\node (M2) at (3.2,-1.3) {$M_2$};
\node (G2) at (3.2,0) {$(\G_2,\Omega_2)$};
 
\draw[->,transform canvas={xshift=-\shift}](G1) to node[midway,left] {}(M1);
\draw[->,transform canvas={xshift=\shift}](G1) to node[midway,right] {}(M1);
\draw[->,transform canvas={xshift=-\shift}](G2) to node[midway,left] {}(M2);
\draw[->,transform canvas={xshift=\shift}](G2) to node[midway,right] {}(M2);
\draw[->](S) to node[pos=0.25, below] {$\text{ }\text{ }\alpha_1$} (M1);
\draw[->] (0.65,-0.15) arc (315:30:0.25cm);
\draw[<-] (2.05,0.15) arc (145:-145:0.25cm);
\draw[->](S) to node[pos=0.25, below] {$\alpha_2$\text{ }} (M2);
\end{tikzpicture}
\end{center} For each $p\in P$, writing $x_1=\alpha_1(p)$ and $x_2=\alpha_2(p)$, we have a commutative square:
\begin{center}
\begin{tikzcd} \No_{x_2}{\O_2} \arrow[r,"\psi_p^{-1}"] & \No_{x_1}{\O_1} \\
(\g_2)^*_{x_2} \arrow[u,"\rho_{\Omega_2}^*"]\arrow[r, "({\phi_p})^*"] & (\g_1)^*_{x_1} \arrow[u, "\rho_{\Omega_1}^*"'] 
\end{tikzcd}
\end{center} with horizontal arrows defined as in (\ref{moreqindisoisotreps}), respectively (\ref{moreqindisoisotgps}), and vertical arrows defined as in (\ref{imsymp}).

\end{lemma}
\begin{proof}
Notice (by dualizing and unravelling the definition of the upper horizontal map) that we ought to prove the commutativity of the diagram:
\begin{center}
\begin{tikzcd} & \No_{p}^*{\O}& \\
 \No_{x_1}^*\O_1\arrow[ru,"(\d \alpha_1)^*"]& & \No_{x_2}^*\O_2\arrow[lu,"(\d \alpha_2)^*"'] \\
(\g_1)_{x_1}\arrow[u,"(\rho_{\Omega_1})^{-1}"]\arrow[rr,"(\phi_p)_*"'] & & (\g_2)_{x_2}\arrow[u,"(\rho_{\Omega_2})^{-1}"']   
\end{tikzcd}
\end{center} First notice that for all $\xi \in (\g_1)_{x_1}$:
\begin{equation}\label{iamoreq1} \exp(\xi)\cdot p=p\cdot \exp((\phi_p)_*(\xi)).
\end{equation} 
Now consider the respective Lie algebroid actions:
\begin{equation*} a_L:\alpha_1^*(T^*M_1)\to TP \quad \& \quad a_R:\alpha_2^*(T^*M_2)\to TP
\end{equation*} induced by the $(\G_1,\Omega_1)$-action via $\rho_{\Omega_1}$ and by the $(\G_2,\Omega_2)$-action via $\rho_{\Omega_2}$. By definition of $a_L$ and $a_R$, for every $\xi \in (\g_1)_{x_1}$ and $\eta\in (\g_2)_{x_2}$ we have:
\begin{equation*} a_L\left((\rho_{\Omega_1})^{-1}(\xi)\right)=\left.\frac{\d}{\d t}\right|_{t=0} \exp(t\xi)\cdot p \quad \& \quad a_R\left((\rho_{\Omega_2})^{-1}(\eta)\right)=\left.\frac{\d}{\d t}\right|_{t=0} p\cdot \exp(-t\eta),
\end{equation*} which combined with (\ref{iamoreq1}) gives:
\begin{equation}\label{iamoreq2} a_L\left((\rho_{\Omega_1})^{-1}(\xi)\right)=-a_R\left((\rho_{\Omega_2})^{-1}((\phi_p)^*(\xi))\right).
\end{equation} Since the (left) $(\G_1,\Omega_1)$-action and the (right) $(\G_2,\Omega_2)$-action are Hamiltonian, $a_L$ and $a_R$ satisfy the momentum map condition. This means that for all $\beta_1\in \Omega^1(M_1)$ and $\beta_2\in \Omega^1(M_2)$:
\begin{equation*} \iota_{a_L(\beta_1)}\omega_P=\alpha_1^*(\beta_1) \quad\&\quad  \iota_{a_R(\beta_2)}\omega_P=-\alpha_2^*(\beta_2).
\end{equation*} Combined with (\ref{iamoreq2}) this implies the desired commutativity, which concludes the proof.
\end{proof}
Below we give a precise meaning to the statement that Delzant subspaces are well-behaved with respect to integral affine Morita equivalences of the ambient integral affine orbifold.
\begin{prop}\label{iamoreqdelzsubconeprop} Suppose that we are given an integral affine Morita equivalence:
\begin{center}
\begin{tikzpicture} \node (G1) at (0,0) {$\B_1$};
\node (M1) at (0,-1.3) {$M_1$};
\node (S) at (1.4,0) {$P$};
\node (M2) at (2.7,-1.3) {$M_2$};
\node (G2) at (2.7,0) {$\B_2$};
 
\draw[->,transform canvas={xshift=-\shift}](G1) to node[midway,left] {}(M1);
\draw[->,transform canvas={xshift=\shift}](G1) to node[midway,right] {}(M1);
\draw[->,transform canvas={xshift=-\shift}](G2) to node[midway,left] {}(M2);
\draw[->,transform canvas={xshift=\shift}](G2) to node[midway,right] {}(M2);
\draw[->](S) to node[pos=0.25, below] {$\text{ }\text{ }\alpha_1$} (M1);
\draw[->] (0.8,-0.15) arc (315:30:0.25cm);
\draw[<-] (1.9,0.15) arc (145:-145:0.25cm);
\draw[->](S) to node[pos=0.25, below] {$\alpha_2$\text{ }} (M2);
\end{tikzpicture}
\end{center} that relates a given subset $\underline{\Delta}_1$ of $\underline{M}_1$ to a subset $\underline{\Delta}_2$ of $\underline{M}_2$. Then for each $p\in P$ such that $x_1:=\alpha_1(p)\in \Delta_1$, and $x_2:=\alpha_2(p)\in \Delta_2$, it holds that:
\begin{equation*} \psi_p(\log_{x_1}(\Delta_1))=\log_{x_2}(\Delta_2),
\end{equation*} where $\psi_p$ is defined as in (\ref{moreqindisoisotreps}).
\end{prop}  
\begin{cor}\label{delzsubspmorinvcor} In the setting of Proposition \ref{iamoreqdelzsubconeprop}, $\underline{\Delta}_1$ is a Delzant subspace of $\underline{M}_1$ if and only if $\underline{\Delta}_2$ is a Delzant subspace of $\underline{M}_2$. In this case, for each $p\in P$ such that $x_1:=\alpha_1(p)\in \Delta_1$, and $x_2:=\alpha_2(p)\in \Delta_2$, their cones at $x_1$ and $x_2$ are related as:
\begin{equation*} \psi_p(C_{x_1}(\underline{\Delta}_1))=C_{x_2}(\underline{\Delta}_2),
\end{equation*}  where $\psi_p$ is defined as in (\ref{moreqindisoisotreps}).
\end{cor} 
\begin{proof}[Proof of Proposition \ref{iamoreqdelzsubconeprop}] First notice that the respective foliations $\F_1$ and $\F_2$ pull back along $\alpha_1$ and $\alpha_2$ to the same foliation $\F$ on $P$. Moreover, since the Morita equivalence is integral affine, the smooth lattices $\Lambda_1$ and $\Lambda_2$ pull back to the same smooth lattice $\Lambda$ in $\No^*\F$. One readily verifies that this lattice $\Lambda$ is Lagrangian in $(T^*P,\Omega_\textrm{can})$. So, it defines a transverse integral affine structure on the foliated manifold $(P,\F)$. Being related by the Morita equivalence, the respective invariant subsets $\Delta_1$ in $M_1$ and $\Delta_2$ in $M_2$, corresponding to $\underline{\Delta}_1$ and $\underline{\Delta}_2$, have the same pre-image $\Delta$ in $P$ under $\alpha_1$ and $\alpha_2$. Let $p\in \Delta$. Notice that, to prove the proposition, it is enough to show that for both $i\in\{1,2\}$ the linear isomorphism $(\underline{\d\alpha}_i)_p:\No_p\F\to \No_{x_i}\F_i$ identifies the set germ $\log_p(\Delta)$ with the set-germ $\log_{x_i}(\Delta_i)$. To see that this is indeed the case, observe that if $\nu_i:U_i\to \No_{x_i}\F_i$ is a $\Delta_i$-adapted submersion as in Proposition \ref{logmapgermprop}, with respect to $(M_i,\F_i,\Lambda_i)$, then the composition:
\begin{equation*} \alpha_i^{-1}(U_i)\xrightarrow{\nu_i\circ \alpha_i} \No_{x_i}\F_i\xrightarrow{(\underline{\d\alpha}_i)^{-1}_p} \No_p\F
\end{equation*} is a $\Delta$-adapted submersion as in Proposition \ref{logmapgermprop}, with respect to $(P,\F,\Lambda)$. 
\end{proof}
\begin{cor}\label{coneatxisinvrem} Let $(B,\B,p,\Lambda)$ be an integral affine orbifold and let $\underline{\Delta}$ be a Delzant subspace. For each $x\in \Delta$ the cone $C_x(\underline{\Delta})$ of $\underline{\Delta}$ at $x$ is $\B_x$-invariant. \end{cor}
\begin{proof}[Proof of Corollary \ref{coneatxisinvrem}] Apply Corollary \ref{delzsubspmorinvcor} to the identity equivalence. 
\end{proof}

\begin{ex}\label{ex:compLiegp:Delzantsubspleafsp:2} {The description of the Delzant subspaces in Example \ref{ex:compLiegp:Delzantsubspleafsp:1} can be obtained as follows. Recall that $\c$ is a complete transversal for the foliation of $\g^*_\textrm{reg}$ by coadjoint $G^0$-orbits, where $G^0$ denotes the identity component of $G$. The transverse integral affine structure $\Lambda_\G$ induced by $(\G,\Omega):=(G\ltimes \g^*_\textrm{reg},-\d\lambda_\textrm{can})$ restricts to the same integral affine structure on $\c$ as that inherited from the integral affine vector space $(\t^*,\Lambda_T)$ (of which $\c$ is an open subset). Moreover, the restriction of the orbifold groupoid $\B=\G/\sT$ to $\c$ is canonically isomorphic to the action groupoid $N(\c)/T\ltimes \c$ (cf. \cite[Proposition 3.15.1]{DuKo}). So, via Example \ref{iamoreqtransportrem} we obtain an integral affine Morita equivalence:}
\begin{center}
\begin{tikzpicture} \node (G1) at (-0.8,0) {$\B$};
\node (M1) at (-0.8,-1.3) {$(\g^*_\textrm{reg},\Lambda_\G)$};
\node (S) at (1.4,0) {$s_\B^{-1}(\c)$};
\node (M2) at (3.5,-1.3) {$(\c,\Lambda_T)$};
\node (G2) at (3.5,0) {$N(\c)/T\ltimes \c$};
 
\draw[->,transform canvas={xshift=-\shift}](G1) to node[midway,left] {}(M1);
\draw[->,transform canvas={xshift=\shift}](G1) to node[midway,right] {}(M1);
\draw[->,transform canvas={xshift=-\shift}](G2) to node[midway,left] {}(M2);
\draw[->,transform canvas={xshift=\shift}](G2) to node[midway,right] {}(M2);
\draw[->](S) to node[pos=0.25, below] {} (M1);
\draw[->] (0.65,-0.15) arc (315:30:0.25cm);
\draw[<-] (2.05,0.15) arc (145:-145:0.25cm);
\draw[->](S) to node[pos=0.25, below] {} (M2);
\end{tikzpicture}
\end{center} Applying Corollary \ref{delzsubspmorinvcor} to this leads to the bijection in Example \ref{ex:compLiegp:Delzantsubspleafsp:1}.
\end{ex}

\subsection{On the image of the momentum map}\label{sec:momimisdelzsubsp}
\subsubsection{The symplectic normal representations and the momentum image}\label{momimsubsectoract} The objects in the previous two sections are related via the proposition below, which we prove in the next subsection along with Theorem \ref{momimtoricthm}. 
\begin{prop}\label{conedescrpsympnormrepprop} Let $(\G,\Omega)$ be a regular and proper symplectic groupoid and let $J:(S,\omega)\to M$ be a faithful multiplicity-free Hamiltonian $(\G,\Omega)$-space. Further, let $p\in S$ and $x:=J(p)$. 
\begin{itemize}\item[a)] The symplectic normal representation $(\S\No_p,\omega_p)$ at $p$ is maximally toric (Definition \ref{torrepdef}).
\item[b)] The inclusion of the isotropy group $\G_p$ of the action into the isotropy group $\G_x$ of $\G$ induces an isomorphism between their groups of connected components: 
\begin{equation}\label{fundgpisogps} \varGamma_{\G_p}\xrightarrow{\sim} \varGamma_{\G_x}.
\end{equation} 
\item[c)] The cone at $x$ of $\underline{\Delta}:=\underline{J}(\underline{S})$ (which is a Delzant subpace by Theorem \ref{momimtoricthm}) is given by:
\begin{equation}\label{conedescrpsympnormrepeq} C_x(\underline{\Delta})=(\rho_{\Omega})_x^*(\pi_{\g_p^*}^{-1}(\Delta_{J_{\S\No_p}})),
\end{equation} where:
\begin{itemize}\item $(\rho_\Omega)_x:\No_x^*\F \to \g_x$ is defined as in (\ref{imsymp}), 
\item $\pi_{\g_p^*}:\g_x^*\to \g_p^*$ denotes the canonical projection (dual to the inclusion $\g_p\hookrightarrow \g_x$),
\item $J_{\S\No_p}:(\S\No_p,\omega_p)\to \g_p^*$ denotes the quadratic momentum map of the symplectic normal representation at $p$, as defined in (\ref{quadsympmommap}), and $\Delta_{J_{\S\No_p}}$ denotes its image. 
\end{itemize}
\end{itemize}
\end{prop} Recall here that, given a Hamiltonian $(\G,\Omega)$-space $J:(S,\omega)\to M$, for each $p\in S$ there is a naturally associated symplectic representation:
\begin{equation*} (\S\No_p,\omega_p)\in \textrm{SympRep}(\G_p),
\end{equation*} of the isotropy group $\G_p$ of the action at $p$. We call this the \textbf{symplectic normal representation} at $p$. Explicitly, this consists of the symplectic normal space to the $\G$-orbit $\O\subset (S,\omega)$ through $p$:
\begin{equation*} \S\No_p:=\frac{T_p\O^\omega}{T_p\O\cap T_p\O^\omega},
\end{equation*} equipped with the linear symplectic form induced by $\omega$ and with the $\G_p$-action:
\begin{equation*} g\cdot [v]=[\d m_{(g,p)}(0,v)], \quad g\in \G_p,\quad v\in T_p\O^\omega,
\end{equation*} where $m:\G\ltimes S\to S$ denotes the action map. Here $T_p\O^\omega$ denotes the $\omega$-orthogonal to the tangent space $T_p\O$ of the orbit at $p$. This generalizes the so-called symplectic normal (or slice) representations for Hamiltonian Lie group actions. For further details on this definition in the generality of symplectic groupoid actions we refer to \cite{Mol1}.
\begin{rem} {As for symplectic toric manifolds, there is a natural way to read off the isomorphism type of the isotropy groups (and so, also the canonical stratification of the orbit space) of a faithful multiplicity-free Hamiltonian action from its momentum image. Indeed, in view of Proposition \ref{cirepsplit} and Proposition \ref{conedescrpsympnormrepprop}, the isotropy group $\G_p$ at $p\in S$ is isomorphic to the semi-direct product of $\Gamma_{\G_x}$ with the $\Gamma_{\G_x}$-invariant subtorus of $(\sT_\Lambda)_x$ with Lie algebra given by the annihilator in $T^*_xM$ of the face of the polyhedral cone $C_x(\underline{\Delta})$ through the origin. }
\end{rem}
\subsubsection{Proofs of Theorem \ref{momimtoricthm} and Proposition \ref{conedescrpsympnormrepprop}} For torus actions, Theorem \ref{momimtoricthm} and Proposition \ref{conedescrpsympnormrepprop} boil down to the known facts below about symplectic toric manifolds (which follow from an argument using the Marle-Guillemin-Sternberg normal form; see \cite{De} and \cite[Lemma B.3]{KaLe}). 
\begin{prop}\label{prop:properties-non-compact-symp-tor-man} 
Let $T$ be a torus and let $J:(S,\omega)\to \t^*$ be a possibly non-compact and non-connected symplectic toric manifold for which $J$ descends to a topological embedding $S/T\to \t^*$. Let $p\in S$ and $x=J(p)$.
\begin{itemize} 
\item[a)] The symplectic normal representation $(\S\No_p,\omega_p)$ at $p$ is maximally toric.
\item[b)] The isotropy group $T_p$ of the $T$-action is connected.
\item[c)] The germ at $x$ of the image of $J$ is equal to that of the subset $x+C_x$, where $C_x$ is the pre-image of the polyhedral cone $\Delta_{J_{\S\No_p}}\subset \t_p^*$ under the canonical projection $\t^*\to \t_p^*$. 
\end{itemize}
\end{prop}
Our strategy of proof will be to reduce to this case. For this we will use: 
\begin{lemma}\label{proetgpoidlinlem} Let $\B\rightrightarrows M$ be a proper etale Lie groupoid. Every $x\in M$ admits a connected open neighbourhood $U$ such that for each $\gamma\in \B_x$ there exists a (necessarily unique) smooth local section $\sigma_\gamma:U\to \B\vert_U$ of the source-map that sends $x$ to $\gamma$, and such that the map:
\begin{equation*} \B_x\times U\to \B\vert_U,\quad (\gamma,y)\mapsto \sigma_\gamma(y)
\end{equation*} is surjective. This defines an isomorphism of Lie groupoids between the action groupoid $\B_x\ltimes U$ of the $\B_x$-action on $U$ given by $\gamma\cdot y:=t(\sigma_\gamma(y))$ and $\B\vert_U$. In particular, $\B\vert_U$ is source-proper. 
\end{lemma}
\begin{proof} See \cite[Proposition 5.30]{MoMr}.
\end{proof}
\begin{lemma}\label{lemma:indtorbunacttoric}
Let $(\G,\Omega)$ be regular and proper symplectic groupoid and let $J:(S,\omega)\to M$ be a faithful multiplicity-free $(\G,\Omega)$-space. If the associated orbifold groupoid $\B=\G/\sT$ is etale, then $J:(S,\omega)\to M$ is a faithful {toric} $(\sT,\Omega_\sT)$-space with respect to the induced $\sT$-action. 
\end{lemma}
\begin{proof} Note that if $\B$ is etale, then $(\sT,\Omega_\sT)$ is symplectic. To see that the induced Hamiltonian $(\sT,\Omega_\sT)$-action is faithful toric, the only thing to show is that the map $\underline{J}_{\sT}:S/{\sT}\to M$ is a topological embedding. This map clearly being a continuous injection, it remains to show that $\underline{J}_{\sT}:S/{\sT}\to \Delta$ is closed, or equivalently, that $J:S\to \Delta$ is closed. For this, it suffices to show that every $x\in \Delta$ admits an open neighbourhood $U$ in $\Delta$ such that $J:J^{-1}(U)\to U$ is closed. Let $x\in \Delta$, take $U_M$ to be an open in $M$ around $x$ as Lemma \ref{proetgpoidlinlem} and let $U=U_M\cap \Delta$. Then $\B\vert_{U_M}$ is source-proper, hence so is $\G\vert_{U_M}$ and so is the action groupoid of the restriction of the $\G$-action to $U_M$. Therefore, the quotient map $q_S:S\to \underline{S}$ restricts to a proper map from $J^{-1}(U)$ onto its image in $\underline{S}$. Combined with the fact that $\underline{J}:\underline{S}\to \underline{M}$ is a topological embedding, it follows that $J:J^{-1}(U)\to U$ is proper. Because any continuous proper map into a first-countable and Hausdorff space is closed (see e.g. \cite{Pa}), we conclude that $J^{-1}(U)\to U$ is indeed closed. 
\end{proof}
We are now ready to complete the proof. 
\begin{proof}[End of the proof of Theorem \ref{momimtoricthm} and Proposition \ref{conedescrpsympnormrepprop}] First suppose that the orbifold groupoid $\B=\G/\sT$ is etale, so that the integral affine structure $\Lambda$ induced by $(\G,\Omega)$ is an honest integral affine structure on $M$ and 
(by Lemma \ref{lemma:indtorbunacttoric}) $J:(S,\omega)\to M$ is a {faithful} toric $(\sT,\Omega_\sT)$-space with respect to the induced action. Let $p\in S$, $x:=J(p)$ and fix an integral affine isomorphism $\iota$ from an open $U$ in $(M,\Lambda)$ into an open $V$ in $(T_xM,\Lambda^*_x)$ that maps $x$ to the origin and the derivative of which at $x$ is the identity map on $T_xM$. This induces an isomorphism of symplectic torus bundles:
\begin{equation}\label{eqn:trivializationsymptorbun} (\sT,\Omega_{\sT})\vert_U\xrightarrow{(\ref{exptorbun2})}(\sT_\Lambda,\Omega_\Lambda)\vert_U \xrightarrow{\iota_*} (T\ltimes \t^*,-\d\lambda_\textrm{can})\vert_V, \quad T:=(\sT_\Lambda)_x.
\end{equation} Consider the Hamiltonian $(T\ltimes \t^*,-\d\lambda_\textrm{can})\vert_V$-action along $\iota\circ J:(J^{-1}(U),\omega)\to V$ corresponding to the restriction to $U$ of the Hamiltonian $(\sT,\Omega)$-action along $J$. The data of such a Hamiltonian action is the same as that of a Hamiltonian $T$-action with momentum map: 
\begin{equation}\label{eqn:hamTsploc:momimpf} \iota\circ J:(J^{-1}(U),\omega)\to T_xM=\t^*.
\end{equation} Since $J$ is a faithful multiplicity-free Hamiltonian $(\G,\Omega)$-space, (\ref{eqn:hamTsploc:momimpf}) is a Hamiltonian $T$-space as in Proposition \ref{prop:properties-non-compact-symp-tor-man}. Further notice that $T_p=\G_p\cap \sT_x$ and that the symplectic normal representation at $p$ of the Hamiltonian $T$-space (\ref{eqn:hamTsploc:momimpf}) coincides with the symplectic $T_p$-representation induced by the symplectic normal representation $\G_p\to \textrm{Sp}(\S\No_p,\omega_p)$ of the $(\G,\Omega)$-space $J$. So, it follows from parts $a$ and $b$ of Proposition \ref{prop:properties-non-compact-symp-tor-man} that the symplectic normal representation at $p$ of the $(\G,\Omega)$-space $J$ is maximally toric as well and that $\G_p\cap \sT_x$ is connected. The latter means that $\G_p\cap \sT_x$ is the identity component of $\G_p$, or in other words, (\ref{fundgpisogps}) is injective. For surjectivity, note that given $g\in \G_x$, both $p$ and $g^{-1}\cdot p$ belong to the same fiber of $J$ over $x$. So, since the $J$-fibers coincide with the $\sT$-orbits, there is a $t\in \sT_x$ such that $g^{-1}\cdot p=t\cdot p$. Then $gt\in \G_p$, and $[gt]\in \varGamma_{\G_p}$ is send to $[g]\in \varGamma_{\G_x}$. This proves parts $a$ and $b$ of Proposition \ref{conedescrpsympnormrepprop}. From part $c$ of Proposition \ref{prop:properties-non-compact-symp-tor-man} it follows that the germ of $\iota(U\cap \Delta)$ at the origin in $T_xM$ is that of the polyhedral cone $\pi_{\g_p}^{-1}(\Delta_{J_{\S\No_p}})$, where $\Delta:=J(S)$. This polyhedral cone in $(T_xM,\Lambda_x^*)$ smooth, since the polyhedral cone $\Delta_{J_{\S\No_p}}$ is smooth in $(\g_p^*,\Lambda_{T_p})$ and $\g_p^0\cap \Lambda_x$ is a full-rank lattice in $\g_p^0$ (because $\g_p$ is the Lie algebra of a subtorus of $\sT_x$). So, Theorem \ref{momimtoricthm} and Proposition \ref{conedescrpsympnormrepprop} hold when $\B=\G/\sT$ is etale. \\      

When $\B$ is not etale, we reduce to the etale case as follows. As before, let $p\in S$ and $x:=J(p)$. Choose a transversal $\Sigma$ for $\B$ through $x$. Then $\Sigma$ is a Poisson transversal in $(M,\pi)$, where $\pi$ is the Poisson structure on $M$ induced by $(\G,\Omega)$. Therefore, any Poisson map from a symplectic manifold into $(M,\pi)$ is transversal to $\Sigma$ and the pre-image of $\Sigma$ under such a map is a symplectic submanifold of its domain. So, $J^{-1}(\Sigma)$ is a symplectic submanifold of $(S,\omega)$ and $(\G,\Omega)\vert_\Sigma$ is a symplectic subgroupoid of $(\G,\Omega)$. The $(\G,\Omega)$-action along $J$ restricts to a Hamiltonian $(\G,\Omega)\vert_{\Sigma}$-action along: 
\begin{equation*} J_\Sigma:(J^{-1}(\Sigma),\omega\vert_{J^{-1}(\Sigma)})\to \Sigma.
\end{equation*} This is again a faithful multiplicity-free Hamiltonian action and its symplectic normal representation at $p$ is canonically isomorphic to that of the $(\G,\Omega)$-action along $J$. The symplectic groupoids $(\G,\Omega)$ and $(\G,\Omega)\vert_\Sigma$ are Morita equivalent via a bibundle like that in Example \ref{transversalmoreqex} and the induced the integral affine Morita equivalence (Example \ref{sympmoreqiamoreq}) is that in Example \ref{transversalmoreqex}. In view of Corollary \ref{delzsubspmorinvcor}, since the orbifold groupoid $\B\vert_\Sigma$ associated to $\G\vert_\Sigma$ is etale, this reduces the proof of Theorem \ref{momimtoricthm} and Proposition \ref{conedescrpsympnormrepprop} to the etale case. 
\end{proof}
\newpage
\section{Classification in the case of symplectic torus bundle actions}
In this part we complete the proof of Theorem \ref{classtorsymptorbunthm} (the classification in the case of symplectic torus bundle actions), for which it remains to prove part $b$. One reason for treating this case separately is that the classification simplifies, whilst the most of the main ideas of the classification are already present. The orbifold structure plays no role in this case (it is trivial) and the three structure theorems boil down to one and the same theorem, that we prove in Section \ref{sec:structurethm:torbunversion}. In Section \ref{consttorspoutofdelzsec} we complete the proof of Theorem \ref{classtorsymptorbunthm} by showing that for any Delzant subspace of an integral affine manifold, there is a canonical faithful toric space with that Delzant subspace as the image of its momentum map. This construction (and its naturality, in particular) will be key as well in the proofs of Theorem \ref{firststrthm} (the first structure theorem) and Theorem \ref{globalsplitthm} (the splitting theorem) in the next part. 

\subsection{The sheaf of Lagrangian sections and the structure theorem}\label{sec:structurethm:torbunversion}
\subsubsection{The sheaves of automorphisms and Lagrangian sections}\label{subsec:shfautomandlagrsec} In this subsection we explain and prove the theorem below, which will be key for the proof of Theorem \ref{classtorsymptorbunthm}$b$. 
\begin{thm}\label{isoautsymplagseccor:torbunversion} Let $(\sT,\Omega)$ be a symplectic torus bundle and let $J:(S,\omega)\to M$ be a faithful {toric} $(\sT,\Omega)$-space with associated Delzant subspace $\Delta:=J(S)$. There is an isomorphism of sheaves on $\Delta$ \textemdash induced by the map (\ref{torsectautiso:torbunversion}) below \textemdash between the sheaf $\textrm{Aut}_\sT(J,\omega)$ of automorphisms of the $(\sT,\Omega)$-space $J:(S,\omega)\to M$ and the sheaf $\L$ of Lagrangian sections of $\sT\vert_\Delta$ (as in Definition \ref{shfinvlagrsecdefi:torbunversion} below). 
\end{thm}
The sheaf of Lagrangian sections referred to here and in Theorem \ref{classtorsymptorbunthm} is defined as follows. {
\begin{defi}\label{shfinvlagrsecdefi:torbunversion} Let $(\sT,\Omega)\to M$ be a symplectic torus bundle with induced integral affine structure $\Lambda$ on $M$ and let $\Delta$ a Delzant subspace of $(M,\Lambda)$. We denote by:
\begin{itemize}
\item $\mathcal{C}^\infty_\Delta(\sT)$ the sheaf on $\Delta$ of smooth sections of the bundle $\sT\vert_\Delta \to \Delta$ (where smooth is meant as maps between manifolds with corners),
\item $\L:=\L_\Delta$ the subsheaf of $\mathcal{C}^\infty_\Delta(\sT)$ consisting of those sections ${\tau}$ that are Lagrangian, in the sense that ${\tau}^*\Omega=0$. 
\end{itemize} 
\end{defi}
\begin{rem}\label{smseccornextrem} 
Given a submersion $f:M\to N$ between manifolds without corners and an embedded submanifold with corners $\Delta$ in $N$ (see Definition \ref{embsubmanwithcorndefi}), a local section $\sigma:U\to f^{-1}(\Delta)$ of $f$, defined on an open $U$ in $\Delta$, is smooth as map of manifolds with corners if and only if for every $x\in U$ there is a smooth local \textit{section} $U_x\to M$ of $f$, defined on an open $U_x$ in $N$ around $x$, that coincides with $\sigma$ on $U\cap U_x$ (also see Remark \ref{charembofmanwithcornrem}). This can be taken as a working definition. 
\end{rem}}
\begin{rem} The sheaf $\L$ is a sheaf of \textit{abelian} groups (with group structure induced by that on the fibers of $\sT$). So, automorphisms of a faithful {toric} $(\sT,\Omega)$-space commute. 
\end{rem}
To explain the definition of the isomorphism in Theorem \ref{isoautsymplagseccor:torbunversion}, consider the sheaf of groups $\textrm{Aut}_\sT(J)$ on $\Delta$ that assigns to an open $U$ in $\Delta$ the group: 
\begin{equation}\label{gpofsTequivdiffeos:torbunversion} \textrm{Aut}_{\sT}(J)(U)
\end{equation} consisting of $\sT$-equivariant diffeomorphisms:
\begin{center}
\begin{tikzcd} J^{-1}(U) \arrow[rr,"\cong"] \arrow[dr, "J"'] & & J^{-1}(U) \arrow[dl,"J"] \\
 & U & 
\end{tikzcd}
\end{center}
The sheaf $\textrm{Aut}_{\sT}(J,\omega)$ in Theorem \ref{isoautsymplagseccor:torbunversion} is the subsheaf of $\textrm{Aut}_{\sT}(J)$ that assigns to such an open the subgroup of $\sT$-equivariant \textit{symplectomorphisms}. To relate these to the sheaves in Definition \ref{shfinvlagrsecdefi:torbunversion}, consider the map of sheaves (of groups) on $\Delta$:
\begin{equation}\label{torsectautiso:torbunversion} \mathcal{C}^\infty_\Delta(\sT)\to \textrm{Aut}_\sT(J), \quad {\tau} \mapsto \psi_{{\tau}},
\end{equation} where, for ${\tau}\in \mathcal{C}^\infty_\Delta(\sT)(U)$, we define:
\begin{equation*} \psi_{{\tau}}:J^{-1}(U)\to J^{-1}(U), \quad \psi_{{\tau}}(p)={\tau}(J(p))\cdot p. 
\end{equation*} Notice that (because the $\sT$-action is free on a dense subset):
\begin{prop} The map of sheaves (\ref{torsectautiso:torbunversion}) is injective.
\end{prop}
Furthermore, we have:
\begin{prop}\label{lagrseccharprop:torbunversion} A section ${\tau}\in \mathcal{C}^\infty_\Delta(\sT)(U)$ is Lagrangian if and only if $\psi_{{\tau}}$ is a symplectomorphism.
\end{prop}
\begin{proof} Using that the $(\sT,\Omega)$-action is Hamiltonian, we deduce:
\begin{align*} (\psi_{{\tau}})^*\omega&=({\tau}\circ J,\textrm{id}_{S})^*(m_S)^*\omega\\
&=({\tau}\circ J,\textrm{id}_{S})^*\left((\textrm{pr}_S)^*\omega+(\textrm{pr}_\sT)^*\Omega\right)\\
&=\omega+J^*({\tau}^*\Omega).
\end{align*} Therefore, $\psi_{{\tau}}$ is a symplectomorphism if and only if $J^*({\tau}^*\Omega)=0$. If the $\sT$-action is free at a point $p\in J^{-1}(U)$, then $J$ is a submersion at $p$, so that $J^*({\tau}^*\Omega)_p=0$ if and only if $({\tau}^*\Omega)_{J(p)}=0$. As the set of points where $\sT$ acts freely is dense in $S$, it follows from continuity that $J^*({\tau}^*\Omega)=0$ if and only if ${\tau}^*\Omega=0$. This proves the proposition.
\end{proof} 
In view of these propositions, the map (\ref{torsectautiso:torbunversion}) induces injective maps of sheaves:
\begin{align}\label{torsectautisoinv:torbunversion} \mathcal{C}^\infty_\Delta(\sT)&\to \textrm{Aut}_{\sT}(J), \\
\label{torsectautisoinvsymp:torbunversion} \mathcal{\L}_{\Delta}&\to \textrm{Aut}_{\sT}(J,\omega).
\end{align} Theorem \ref{isoautsymplagseccor:torbunversion} says that (\ref{torsectautisoinvsymp:torbunversion}) is an isomorphism. To prove this, by Proposition \ref{lagrseccharprop:torbunversion} it suffices to show: 
\begin{thm}\label{torsectautisoprop:torbunversion} The map (\ref{torsectautisoinv:torbunversion}) is an isomorphism of sheaves. 
\end{thm}
\begin{proof}[Proof of Theorem \ref{torsectautisoprop:torbunversion} (and hence of Theorem \ref{isoautsymplagseccor:torbunversion})] It remains to show that (\ref{torsectautisoinv:torbunversion}) is stalk-wise surjective. Let $x\in \Delta$ and let $\psi\in \textrm{Aut}_{\sT}(J)(U)$ for some open $U$ around $x$ in $\Delta$. We have to show that there is a smooth section ${\tau}$ of $\sT\vert_\Delta$, defined on an open neighbourhood of $x$ in $\Delta$, such that the germ of $\psi_{{\tau}}$ at $x$ coincides with that of $\psi$. The proof of this reduces to the case in which $M=\t^*$, $x$ is the origin and $(\sT,\Omega)$ is the symplectic torus bundle $(T\times \t^*,-\d \lambda_\textrm{can})$ associated to a torus $T$, by using an isomorphism (\ref{eqn:trivializationsymptorbun}) as in the proof of Theorem \ref{momimtoricthm} and Proposition \ref{conedescrpsympnormrepprop}. In that case, the data of a faithful complexity $(\sT,\Omega)$-space $J:(S,\omega)\to \t^*$ is the same as that of a (possibly non-compact and non-connected) symplectic toric $T$-manifold $(S,\omega)$ with momentum map $J$ that descends to a topological embedding $\underline{J}:S/T\to \t^*$. By appealing to the standard normal form theorem for such Hamiltonian $T$-spaces (see e.g. \cite[Lemma B.3]{KaLe}) and using that any $T$-invariant open in $S$ is the pre-image under $J$ of an open in $\t^*$ ($\underline{J}$ being a topological embedding), the proof further reduces to the subcase in which:
\begin{itemize}\item $T=\T^n$ (the standard torus), 
\item $S$ is an open in $\T^{n-k}\times \R^{n-k}\times \C^k$ of the form  $\T^{n-k}\times B^{n-k}\times B^k$, where $B^{n-k}\subset \R^{n-k}$ and $B^k\subset \C^k$ are open balls of the same radius and centered around the respective origins, and $\T^{n-k}\times \R^{n-k}\times \C^k$ is equipped with the standard symplectic form,
\item the $T$-action on $S$ is given by:
\begin{equation*} (\lambda_1,...,\lambda_n)\cdot (t_1,...,t_{n-k},x_1,...,x_{n-k},z_1,...,z_k)=(\lambda_{k+1}t_1,...,\lambda_{n}t_{n-k},x_1,...,x_{n-k},\lambda_1z_1,...,\lambda_kz_k),
\end{equation*} 
\item $J$ is the restriction of the map: 
\begin{align*} \T^{n-k}\times \R^{n-k}\times \C^k&\to \R^n,\\
 (t_1,...,t_{n-k}, x_1,....,x_{n-k},z_1,...,z_k)&\mapsto (|z_1|^2,...,|z_k|^2,x_1,...,x_{n-k}), 
\end{align*} 
\item $\psi$ is a $\T^n$-equivariant self-diffeomorphism of $\T^{n-k}\times B^{n-k}\times B^k$ such that $J\circ \psi=J$. 
\end{itemize} To conclude the proof we will show that there is map ${\tau}:\Delta\to \T^n$, defined on the image $\Delta$ of $\T^{n-k}\times B^{n-k}\times B^k$ under $J$, such that $\psi$ is given by:
\begin{equation}\label{firstpropsectoconstr} \psi(t,x,z)={\tau}(J(t,x,z))\cdot (t,x,z),\quad (t,x,z)\in \T^{n-k}\times B^{n-k}\times B^k, 
\end{equation} and such that ${\tau}$ extends to a smooth map from an open in $\R^n$ into $\T^n$. For this we use a variation of the argument used to prove \cite[Lemma 2.6]{De}. Let $\phi_j$ and $\psi_j$ denote the $j^\textrm{th}$ component of $\psi$ in $\T^{n-k}$ and $\C^k$, respectively. Using equivariance of $\psi$ and the fact that $\psi$ preserves $J$, we find:
\begin{equation}\label{delzlemeq1} \psi(t,x,z)=(t_1\cdot\phi_1(1,x,z),...,t_{n-k}\cdot\phi_{n-k}(1,x,z),x,\psi_1(1,x,z),...,\psi_k(1,x,z)). 
\end{equation}
Using equivariance of $\psi$ once more, it further follows that for all $(x,z)\in B^{n-k}\times B^k$ and $\lambda\in \T^k$:
\begin{align*} &\phi_j(1,x,\lambda\cdot z)=\phi_j(1,x,z),\\
&\psi_j(1,x,z)=\lambda_j\psi_j(1,x,z).
\end{align*} In particular, $(x,u)\mapsto \psi_j(1,x,u)$ restricts to a smooth function on $B^{n-k}\times (B^k\cap \textrm{Re}(\C^k))$, which is odd in the $u_j$ variable and even in the other $u$-variables, while $(x,u)\mapsto \phi_j(1,x,u)$ restricts to a smooth function on this domain as well, but is even in all $u$-variables. Therefore, a theorem due to Whitney \cite{Wh} (a particular case of Schwarz' theorem for finite groups \cite{Bi1,Schw1}) implies that there are continuous functions:
\begin{equation*} f_j,g_j:\Delta\to \C
\end{equation*} that satisfy:
\begin{align*} &\phi_j(1,x,u)=f_j(u_1^2,..., u_k^2,x)\\
& \psi_j(1,x,u)=u_jg_j(u_1^2,..., u_k^2,x),
\end{align*} for all $(x,u)\in B^{n-k}\times (B^k\cap \textrm{Re}(\C^k))$, and extend to smooth functions on some open neighbourhood of $\Delta$ in $\R^n$. The fact that $\psi$ preserves $J$ implies that:
\begin{equation*} u_j^2|g_j(u_1^2,...,u_k^2,x)|^2=|\psi_j(1,x,u)|^2=u_j^2
\end{equation*} for all $(x,u)\in B^{n-k}\times (B^k\cap \textrm{Re}(\C^k))$. Therefore $g_j$ takes values in $\mathbb{S}^1$ on the interior of $\Delta$ in $\R^n$, hence it must do so on all of $\Delta$ by a density argument. Note that $f_j$ does so as well, since $\phi_j$ does. Finally, observe that for $(x,z)\in B^{n-k}\times B^k$, writing $z_j=e^{i\theta_j}|z_j|$ one finds:
\begin{align}\label{delzlemeq2} \psi_j(1,x,z)&=e^{i\theta_j}\psi_j(1,x,|z_1|,...,|z_k|)\\ \nonumber
&=z_jg_j(|z_1|^2,...,|z_k|^2,x), \nonumber
\end{align} by equivariance of $\psi$. Similarly:
\begin{equation}\label{delzlemeq3} \phi_j(1,x,z)=f_j(|z_1|^2,...,|z_k|^2,x).
\end{equation} Now define:
\begin{equation*} f:\Delta\to \T^{n-k} \quad \& \quad g:\Delta\to \T^k
\end{equation*} to have $j^\textrm{th}$ component $f_j$ and $g_j$, respectively, and consider:
\begin{equation*} {\tau}:=(g,f):\Delta\to \T^{k}\times \T^{n-k}.
\end{equation*} Combining (\ref{delzlemeq1}), (\ref{delzlemeq2}) and (\ref{delzlemeq3}) we find that (\ref{firstpropsectoconstr}) holds. 
Moreover, by construction of $f$ and $g$, ${\tau}$ extends to a smooth map into $\C^n$ on an open neighbourhood of $\Delta$ in $\R^n$. Since none of its components vanish on a small enough such neighbourhood, by normalizing them we can find such an extension that maps smoothly into $\T^{n}$. So, ${\tau}$ has the desired properties. 
\end{proof}

\subsubsection{The structure theorem} For symplectic torus bundle actions the three structure theorems boil down to the statement below, which we prove in this subsection.
\begin{thm}\label{torsortorspthm:torbunversion}
Let $(\sT,\Omega)$ be a symplectic torus bundle over $M$ with induced integral affine structure $\Lambda$ and let $\Delta$ be a Delzant subspace of $(M,\Lambda)$. There is a natural free and transitive action of the group $H^1(\Delta,\L)$ (as in Theorem \ref{classtorsymptorbunthm}) on the set of isomorphism classes of faithful {toric} $(\sT,\Omega)$-spaces with momentum image $\Delta$.
\end{thm}
To prove Theorem \ref{torsortorspthm:torbunversion}, we will first show the following.
\begin{prop}\label{locclasstoricthm:torbunversion}
Let $(\sT,\Omega)$ be a symplectic torus bundle and suppose that we are given faithful {toric} $(\sT,\Omega)$-spaces $J_1:(S_1,\omega_1)\to M$ and $J_2:(S_2,\omega_2)\to M$. Let $x\in J_1(S_1)\cap J_2(S_2)$. Then there is an open neighbourhood $U$ of $x$ in $M$ and a $\sT$-equivariant symplectomorphism: 
\begin{center}
\begin{tikzcd} (J_1^{-1}(U),\omega_1) \arrow[rr,"\cong"] \arrow[dr, "J_1"'] & & (J_2^{-1}(U),\omega_2) \arrow[dl,"J_2"] \\
 & M & 
\end{tikzcd}
\end{center} if and only if the germs at $x$ in $M$ of $J_1(S_1)$ and $J_2(S_2)$ are equal.
\end{prop}
For this we use the theorem below, which is a generalization of a well-known theorem due to Marle \cite{Ma1} and Guillemin and Sternberg \cite{GS4} for Hamiltonian actions of compact Lie groups. 
\begin{thm}[\cite{Mol1}]\label{loceqthmhamact} Let $(\G,\Omega)\rightrightarrows M$ be a proper symplectic groupoid. Suppose we are given two Hamiltonian $(\G,\Omega)$-spaces $J_1:(S_1,\omega_1)\to M$ and $J_2:(S_2,\omega_2)\to M$. Let $p_1\in S_1$ and $p_2\in S_2$ be such that $J_1(p_1)=J_2(p_2)$, $\G_{p_1}=\G_{p_2}$ and $(\S\No_{p_1},\omega_{p_1})\cong (\S\No_{p_2},\omega_{p_2})$ are isomorphic symplectic representations. Then there are $\G$-invariant open neighbourhoods $U_i$ of $p_i$ in $S_i$ and an equivariant symplectomorphism that sends $p_1$ to $p_2$:
\begin{center}
\begin{tikzcd} (U_1,\omega_1,p_1) \arrow[rr,"\sim"] \arrow[dr, "J_1"'] & & (U_2,\omega_2,p_2) \arrow[dl,"J_2"] \\
 & M & 
\end{tikzcd}
\end{center} 
\end{thm} 
{For actions of symplectic torus bundles, Theorem \ref{loceqthmhamact} can be proved by reducing to the case of Hamiltonian torus actions via a trivialization of the symplectic torus bundle as in (\ref{eqn:trivializationsymptorbun}).  
\begin{proof}[Proof of Proposition \ref{locclasstoricthm:torbunversion}] The implication from left to right is straightforward. For the other implication, let $x\in J_1(S_1)\cap J_2(S_2)$ such that the germs at $x$ of $J_1(S_1)$ and $J_2(S_2)$ coincide. Since $\underline{J}:\underline{S}\to M$ is a topological embedding, every invariant $\sT$-invariant open in $S$ is of the form $J^{-1}(U)$ where $U$ is some open in $M$. So, in view of Theorem \ref{loceqthmhamact}, it remains to show that there are $p_1\in J_1^{-1}(x)$ and $p_2\in J_2^{-1}(x)$ such that $\sT_{p_1}=\sT_{p_2}$ and the symplectic representations $(\S\No_{p_1},\omega_{p_1})$ and $(\S\No_{p_2},\omega_{p_2})$ are isomorphic. Fix any $p_1\in J_1^{-1}(x)$ and $p_2\in J_2^{-1}(x)$. Since the germs at $x$ of $J_1(S_1)$ and $J_2(S_2)$ coincide, Proposition \ref{conedescrpsympnormrepprop}$c$ implies that:
\begin{equation}\label{eqpolcon} (\pi_{\t^*_{p_1}})^{-1}(\Delta_{\S\No_{p_1}})=(\pi_{\t^*_{p_2}})^{-1}(\Delta_{\S\No_{p_2}}).
\end{equation} Since the symplectic normal representations at $p_1$ and $p_2$ are maximally toric (Proposition \ref{conedescrpsympnormrepprop}$a$), the polyhedral cones $\Delta_{\S\No_{p_1}}$ and $\Delta_{\S\No_{p_2}}$ are pointed. So, (\ref{eqpolcon}) implies that $\ker(\pi_{\t^*_{p_1}})=\ker(\pi_{\t^*_{p_2}})$. Seeing as $\ker(\pi_{\t^*_{p_i}})$ is the annihilator of $\t_{p_i}=\textrm{Lie}(\sT_{p_i})$ in $\t_x^*$, this means that $\t_{p_1}=\t_{p_2}$, and hence $\sT_{p_1}=\sT_{p_2}$. In view of $(\ref{eqpolcon})$, we further conclude that $\Delta_{\S\No_{p_1}}=\Delta_{\S\No_{p_2}}$. So, it follows from Proposition \ref{classtorrep} that the symplectic normal representations at $p_1$ and $p_2$ are isomorphic. 
\end{proof} }

\begin{proof}[Proof of Theorem \ref{torsortorspthm:torbunversion}] The construction of the action uses a \v{C}ech model for sheaf cohomology, in much the same way as in the usual classification of principal torus bundles. Since the natural group homomorphism from \v{C}ech into sheaf cohomology:
\begin{equation}\label{eqn:cechtoderivedmorphism}
\check{H}^1(\Delta,\L)\to H^1(\Delta,\L)
\end{equation} is an isomorphism, given $\textrm{c}\in H^1(\Delta,\L)$ we can consider a \v{C}ech $1$-cocycle ${\tau}\in \check{C}^1_\U(\Delta,\L)$ with respect to some open cover $\U$ that represents $\textrm{c}$ via (\ref{eqn:cechtoderivedmorphism}). Write $\psi_{UV}:=\psi_{{\tau}(U,V)}$ as in (\ref{torsectautiso:torbunversion}). Furthermore, let $J:(S,\omega)\to M$ be a faithful toric $(\sT,\Omega)$-space with momentum image $\Delta$. Then we can define another $(\sT,\Omega)$-space $J_{{\tau}}:(S_{{\tau}},\omega_{{\tau}})\to M$, as follows. As topological space, define:
\begin{equation*} S_{{\tau}}:=\frac{\left( \bigsqcup_{U\in \U} J^{-1}(U) \right)}{\sim_{{\tau}}}
\end{equation*} where $(p,U)\sim_{{\tau}} (q,V)$ if and only if $J(p)=J(q)$ and $p=\psi_{UV}(q)$. This indeed defines an equivalence relation, because ${\tau}$ is a $1$-cocycle. Furthermore, because $\psi_{UV}\in \textrm{Aut}_\sT(J,\omega)(U\cap V)$ for all $U,V\in \U$, there is a unique smooth structure on $S_{{\tau}}$, a unique symplectic structure $\omega_{{\tau}}$ on $S_{{\tau}}$ and a unique $\sT$-action along the canonical map $J_{{\tau}}:S_{{\tau}}\to M$ such that, for each $U\in \U$, the canonical inclusion:
\begin{center}
\begin{tikzcd} (J^{-1}(U),\omega) \arrow[rr,hook] \arrow[dr, "J"'] & & (S_{{\tau}},\omega_{{\tau}}) \arrow[dl,"J_{{\tau}}"] \\
 & M & 
\end{tikzcd}
\end{center} is a smooth, symplectic and $\sT$-equivariant embedding. Clearly, $J_{{\tau}}:(S_{{\tau}},\omega_{{{\tau}}})\to M$ is a faithful toric $(\sT,\Omega)$-space with momentum image equal to $\Delta$. Now, define:
\begin{equation*} \textrm{c}\cdot [J:(S,\omega)\to M]:=[J_{{\tau}}:(S_{{\tau}},\omega_{{\tau}})\to M].
\end{equation*} As one readily verifies, this does not depend on the choice of representative ${\tau}$ or cover $\U$, 
and defines an action of $H^1(\Delta,\L)$. To see that this action is free, suppose that:
\begin{equation*} [J_{{\tau}}:(S_{{\tau}},\omega_{{\tau}})\to M]=[J:(S,\omega)\to M],
\end{equation*} meaning that there is an isomorphism of $(\sT,\Omega)$-spaces:
\begin{center}
\begin{tikzcd} (S_{{\tau}},\omega_{{\tau}}) \arrow[rr,"\psi"] \arrow[dr, "J_{{\tau}}"'] & & (S,\omega) \arrow[dl,"J"] \\
 & M & 
\end{tikzcd}
\end{center} For each $U\in \U$, consider: 
\begin{equation*} \psi_U:=\psi\vert_{J^{-1}(U)}\in \textrm{Aut}_{\sT}(J,\omega)(U).
\end{equation*} Since $\psi$ is well-defined, for each $q\in J^{-1}(U\cap V)$ it must hold that $\psi_V(q)=\psi_U(\psi_{UV}(q))$. Hence:
\begin{equation*} \psi_{UV}=\psi_U^{-1}\circ \psi_V.
\end{equation*} Letting $\sigma(U):U\to \sT$ be the Lagrangian section such that $\psi_{\sigma(U)}=\psi_U$ (which exists by Theorem \ref{isoautsymplagseccor:torbunversion}), we find that: 
\begin{equation*} {\tau}(U,V)=\sigma(V)-\sigma(U),
\end{equation*} hence ${\tau}$ is a \v{C}ech $1$-coboundary. This shows that the action is indeed free. To verify transitivity, let $J_1:(S_1,\omega_1)\to M$ and $J_2:(S_2,\omega_2)\to M$ be faithful toric $(\sT,\Omega)$-spaces with momentum image $\Delta$. By Proposition \ref{locclasstoricthm:torbunversion} we can find an open cover $\U$ of $M$, with for each $U\in \U$ an isomorphism of $(\sT,\Omega)$-spaces:
\begin{center}
\begin{tikzcd} (J_1^{-1}(U),\omega_1) \arrow[rr,"\psi_U"] \arrow[dr, "J_1"'] & & (J_2^{-1}(U),\omega_2) \arrow[dl,"J_2"] \\
 & M & 
\end{tikzcd}
\end{center} Due to Corollary \ref{isoautsymplagseccor:torbunversion}, there are unique Lagrangian sections ${\tau}(U,V):U\cap V\to \sT$ satisfying $\psi_{{\tau}(U,V)}=\psi_U\circ \psi_V^{-1}$ on $J_2^{-1}(U\cap V)$. As is readily verified, ${\tau}$ defines a \v{C}ech $1$-cocycle and:
\begin{equation*} [J_1:(S_1,\omega_1)\to M]=\textrm{c}\cdot [J_2:(S_2,\omega_2)\to M],
\end{equation*} where $\textrm{c}\in H^1(\Delta,\L)$ is the class represented by ${\tau}$ via (\ref{eqn:cechtoderivedmorphism}). Hence, the action is indeed transitive.
\end{proof}

\subsection{Constructing a (natural) toric space out of a Delzant subspace}\label{consttorspoutofdelzsec} 
\subsubsection{The result of the construction}\label{constrtorictorbunspthmintrosec} The aim of this section is to prove and explain the construction behind the theorem below. As  explained in Remark \ref{canbijremintro}, the assignment of a toric space in this theorem and the action in Theorem \ref{torsortorspthm:torbunversion} yield the bijection in Theorem \ref{classtorsymptorbunthm}$b$. 
\begin{thm}\label{constrtorictorbunspthm} Let $(M,\Lambda)$ be an integral affine manifold. For each Delzant subspace $\Delta$ of $(M,\Lambda)$, there is a {canonically} associated faithful toric $(\sT_\Lambda,\Omega_\Lambda)$-space:
\begin{equation*} J_\Delta:(S_\Delta,\omega_\Delta)\to M
\end{equation*} with momentum map image $\Delta$. {This }depends naturally and locally on $(M,\Lambda,\Delta)$ with respect to locally defined isomorphisms, in the sense explained below.
\end{thm} 
Here, by the statement that this depends \textbf{naturally and locally} on $(M,\Lambda,\Delta)$ with respect to locally defined isomorphisms we mean the following. If $(M_1,\Lambda_1)$ and $(M_2,\Lambda_2)$ are integral affine manifolds with respective Delzant subspaces $\Delta_1$ and $\Delta_2$, then for any two opens $U_1$ in $\Delta_1$ and $U_2$ in $\Delta_2$, and any diffeomorphism of manifolds with corners $\phi:U_1\to U_2$ such that $\phi^*\Lambda_2=\Lambda_1\vert_{U_1}$, there is an associated symplectomorphism:
\begin{equation}\label{assocsymplectotorbunspthm} \phi_*:(J_{\Delta_1}^{-1}(U_1),\omega_{\Delta_1})\to (J_{\Delta_2}^{-1}(U_2),\omega_{\Delta_2})
\end{equation} that fits into a commutative square:
\begin{center}
\begin{tikzcd} (J_{\Delta_1}^{-1}(U_1),\omega_{\Delta_1}) \arrow[r,"\phi_*"] \arrow[d,"J_{\Delta_1}"] & (J_{\Delta_2}^{-1}(U_2),\omega_{\Delta_2}) \arrow[d,"J_{\Delta_2}"] \\
  U_1 \arrow[r,"\phi"] & U_2 
\end{tikzcd} 
\end{center} and that is compatible with the actions, in the sense that for every $p\in S_{\Delta_1}$ and every $[\alpha]\in (\sT_{\Lambda_1})_{x}$ with $x=J_{\Delta_1}(p)$ it holds that:
\begin{equation*} \phi_*([\alpha]\cdot p)=[(\d \phi^{-1})^*\alpha]\cdot \phi_*(p).
\end{equation*} 
Furthermore, this association satisfies the following.
\begin{itemize} 
\item[i)] It is \textbf{natural}: given another integral affine manifold $(M_3,\Lambda_3)$ with a Delzant subspace $\Delta_3$ and a diffeomorphism of manifolds with corners $\psi:U_2\to U_3$ onto an open $U_3$ in $\Delta_3$ such that $\psi^*\Lambda_3=\Lambda_2\vert_{U_2}$, it holds that:
\begin{equation*} (\psi\circ\phi)_*=\psi_*\circ \phi_*.
\end{equation*} 
\item[ii)] It is \textbf{local}: if $V_1$ is another open in $\Delta_1$ such that $V_1\subset U_1$, then:
\begin{equation*} (\phi\vert_{V_1})_*=(\phi_*)\vert_{J_{\Delta_1}^{-1}(V_1)},
\end{equation*} where $\phi\vert_{V_1}:V_1\to \phi(V_1)$ denotes the restriction of $\phi:U_1\to U_2$. 
\end{itemize}
In the remainder of this section we give the construction behind Theorem \ref{constrtorictorbunspthm}. The essential idea behind this construction is the same as that behind the proof of \cite[Theorem 1.3.1]{KaLe}.
\subsubsection{The topology and the action}
Throughout this and the next subsection, let $(M,\Lambda)$ be a fixed integral affine manifold with a fixed Delzant subspace $\Delta$. As topological space, we define $S_\Delta$ as follows. Let $\F_\Delta$ be the set-theoretic bundle of groups over $\Delta$ with isotropy group at $x$ the torus:
\begin{equation*} (\F_\Delta)_x:=\frac{F_x(\Delta)^0}{\Lambda_x\cap F_x(\Delta)^0}, 
\end{equation*} where $F_x(\Delta)^0$ denotes the annihilator in $T_x^*M$ of the tangent space $F_x(\Delta)$ to the open face of $\Delta$ through $x\in \Delta$ (see Example \ref{stratmanwithcornex}). The groupoid $\F_\Delta$ includes canonically into $\sT_\Lambda\vert_\Delta$ as a set-theoretic wide normal subgroupoid and as such it acts along the bundle projection of $\sT_\Lambda\vert_\Delta$. We let $S_\Delta$ be the orbit space of this action:
\begin{equation*} S_\Delta:=\frac{\sT_\Lambda\vert_\Delta}{\F_\Delta},
\end{equation*} equipped with the quotient topology. The bundle projection descends to a continuous map:
\begin{equation}\label{mommaptoractconstrfromdelzsubman}  J_\Delta:S_\Delta\to M,
\end{equation} with image $\Delta$. Since it commutes with the $\F_\Delta$-action defined above, the canonical left $\sT_\Lambda\vert_\Delta$-action along the bundle projection of $\sT_\Lambda\vert_\Delta$ descends to a continuous left $\sT_\Lambda$-action along (\ref{mommaptoractconstrfromdelzsubman}). This defines the topological space and the action underlying the toric $(\sT_\Lambda,\Omega_\Lambda)$-space in Theorem \ref{constrtorictorbunspthm}. 
\begin{prop}\label{torproptoppartprop} The $\sT_\Lambda$-action along (\ref{mommaptoractconstrfromdelzsubman}) defined above has the following properties.
\begin{itemize}\item[a)] The action is free on the open and dense subset $J_\Delta^{-1}(\mathring{\Delta})$ of $S_\Delta$, with $\mathring{\Delta}$ as in Example \ref{stratmanwithcornex}.
\item[b)] The orbits coincide with the $J_\Delta$-fibers.
\item[c)] The transverse momentum map $\underline{J}_\Delta:\underline{S}_\Delta\to M$ is a homeomorphism onto $\Delta$.
\end{itemize}
\end{prop}
\begin{proof} Parts $a$ and $b$ follow from straightforward verifications. It follows from part $b$ that $\underline{J}_\Delta$ a continuous injection so that to prove part $c$ it remains to show that it is closed as map into its image $\Delta$. To this end, notice that the bundle projection $\sT_\Lambda\to M$ is proper, because it is a fiber bundle with compact fibers. This implies that $\underline{J}_\Delta$ is proper as map into $\Delta$. Since any continuous proper map into a first-countable and Hausdorff space is closed, this concludes the proof.
\end{proof}
\subsubsection{The smooth and symplectic structure}\label{smsympstrconstrtorspsec} Next, we define a structure of symplectic manifold on $S_\Delta$ that is compatible with the $\sT_\Lambda$-action along $J_\Delta:S_\Delta\to M$. First, we construct a local model for this. By a $\textbf{$\Delta$-admissible}$ triple we will mean a triple $(x_0,U,\chi)$ consisting of a point $x_0\in \Delta$ and an integral affine chart $(U,\chi)$ for $(M,\Lambda)$ around $x_0$ with the property that $\chi(x_0)=0$, $U$ is connected and: 
\begin{equation}\label{imadapiachart} \chi(U\cap \Delta)=\chi(U)\cap \R^n_k,
\end{equation} where $n=\dim(M)$ and $k=\textrm{depth}_\Delta(x_0)$. Every $x_0\in \Delta$ belongs to such a triple. {For each such triple, the symplectic torus bundle: 
\begin{equation}\label{trivsymptorbuneq} (\T^n\times \R^n,\sum_{j=1}^n \d \theta_j\wedge \d x_j)\vert_{\chi(U)}\xrightarrow{\textrm{pr}_{\R^n}} \chi(U)
\end{equation} comes with a faithful {toric} space with momentum image (\ref{imadapiachart}), that we denote as:
\begin{equation}\label{mommaplocmodadmtriple} J_{(x_0,U,\chi)}:(S_{(x_0,U,\chi)},\omega_{(x_0,U,\chi)})\to \chi(U).
\end{equation} The symplectic manifold $(S_{(x_0,U,\chi)},\omega_{(x_0,U,\chi)})$ is that obtained by symplectic reduction at the origin of the diagonal (right) Hamiltonian $\T^k$-action with momentum map:
\begin{equation}\label{eqn:mommapred:stdlocmod} \textrm{pr}_{\R^k}\circ\textrm{pr}_{\chi(U)}-J_{\C^k}\circ \textrm{pr}_{\C^k}:(\T^n\times \chi(U),\sum_{j=1}^n \d \theta_j\wedge \d x_j)\times (\C^k,\omega_\textrm{st})\to \R^k, 
\end{equation} where the $\T^k$-action on $\T^n\times \chi(U)$ is by translation on the first $k$ components of $\T^n$ and the $\T^k$-action on $\C^k$ is the anti-standard one, which is Hamiltonian with momentum map $-J_{\C^k}$, where:
\begin{equation*} J_{\C^k}:\C^k \to \R^k,\quad (z_1,...,z_k)\mapsto (|z_1|^2,...,|z_k|^2).
\end{equation*} 
The map (\ref{mommaplocmodadmtriple}) is that induced by the projection to $\chi(U)$ and the action of the torus bundle (\ref{trivsymptorbuneq}) is that induced by the action on itself by translation. Viewed as Hamiltonian $\T^n$-space, (\ref{mommaplocmodadmtriple}) is the standard local model for symplectic toric $\T^n$-manifolds.} The proposition below shows that (\ref{mommaplocmodadmtriple}) is a local model for the topological $\sT_\Lambda$-space $J_\Delta$. 

\begin{prop}\label{chartdefnattorspprop} The map:
\begin{align}\label{indhomdeltadmtripleeq} h_{(x_0,U,\chi)}:J_\Delta^{-1}(U)&\to S_{(x_0,U,\chi)}, \\ 
\left[\sum_{j=1}^n \theta_j \d\chi^j_x \mod \Lambda_x\right] &\mapsto \left[\left(e^{2\pi i\theta_1},...,e^{2\pi i \theta_n},\chi(x),\sqrt{\chi^1(x)},...,\sqrt{\chi^k(x)}\right)\right], \nonumber
\end{align} is a homeomorphism and it is compatible with the $\sT_\Lambda$-action along $J_\Delta$ in the sense that it intertwines (\ref{mommaplocmodadmtriple}) with: 
\begin{equation*} \chi\circ J_\Delta:J_\Delta^{-1}(U)\to \chi(U),
\end{equation*} and for each $p\in J_\Delta^{-1}(U)$ and $t\in (\sT_\Lambda)_x$ with $x=J(p)$ it satisfies:
\begin{equation*} h_{(x_0,\chi,U)}(t\cdot p)=\Phi_\chi(t)\cdot  h_{(x_0,\chi,U)}(p),
\end{equation*} where $\Phi_\chi$ is the isomorphism of symplectic torus bundles given by:
\begin{align}\label{trivsymptorbuniachart} \Phi_\chi: (\sT_\Lambda,\Omega_\Lambda)\vert_U&\xrightarrow{\sim} (\T^n\times \R^n,\sum_{j=1}^n \d\theta_j\wedge \d x_j)\vert_{\chi(U)},\\
 \sum_{j=1}^n \theta_j \d\chi^j_x \mod \Lambda_x &\mapsto (e^{2\pi i\theta_1},...,e^{2\pi i\theta_n},\chi(x)). \nonumber
\end{align} 
\end{prop}
\begin{proof} The bijectivity and the compatibility of (\ref{indhomdeltadmtripleeq}) with the $\sT_\Lambda$-action follow from a straightforward verification. Furthermore, (\ref{indhomdeltadmtripleeq}) is continuous and closed, because its composition with the canonical map $\sT_\Lambda\vert_U\to J_\Delta^{-1}(U)$ (which is a topological quotient map) is continuous and closed (since it factors as a continuous and closed map { into the fiber of (\ref{eqn:mommapred:stdlocmod}) over the origin}, composed with the quotient map from this fiber onto $S_{(x_0,U,\chi)}$, which is closed due to compactness of $\T^k$) . So, (\ref{indhomdeltadmtripleeq}) is indeed a homeomorphism.
\end{proof}
\begin{cor} The space $S_\Delta$ is Hausdorff and second countable.
\end{cor}
\begin{proof} By Proposition \ref{chartdefnattorspprop}, for each $\Delta$-admissible triple $(x_0,U,\chi)$ the open subspace $J_\Delta^{-1}(U)$ of $S_\Delta$ is Hausdorff and second countable. In view of this, the fact that $B$ is Hausdorff implies that $S_\Delta$ is Hausdorff and the fact that $B$ is second countable implies that $S_\Delta$ is second countable. 
\end{proof}
Next, we show that the local smooth and symplectic structures obtained via the homeomorphisms (\ref{indhomdeltadmtripleeq}) patch to a smooth and symplectic structure on all of $S_\Delta$. 
\begin{prop}\label{smoothandsympstrnattorspprop} The topological space $S_\Delta$ admits a smooth and symplectic structure, both uniquely determined by the property that for each $\Delta$-admissible triple $(x_0,U,\chi)$ the induced homeomorphism (\ref{indhomdeltadmtripleeq}) is a symplectomorphism onto the symplectic manifold $(S_{(x_0,U,\chi)},\omega_{(x_0,U,\chi)})$ defined above. 
\end{prop}
\begin{proof} We ought to show that any two given $\Delta$-admissible triples $(x_0,U,\chi)$ and $(y_0,V,\phi)$ induce the same smooth and symplectic structure on the intersection of $J_\Delta^{-1}(U)$ and $J_\Delta^{-1}(V)$ (via the homeomorphisms $h_{(x_0,U,\chi)}$ and $h_{(y_0,V,\phi)}$). Throughout, let $k=\textrm{depth}_\Delta(x_0)$ and $l=\textrm{depth}_\Delta(y_0)$. \\

First we address the smooth structure, starting with the case in which $(x_0,U)=(y_0,V)$. In this case, $\chi\circ \phi^{-1}$ is the restriction of an element of $A\in \textrm{GL}_n(\Z)$  that maps $\R^n_k$ onto $\R^n_k$ (because any linear map that maps an open neighbourhood of the origin in $\R^n_k$ to another such open must map $\R^n_k$ onto $\R^n_k$, as $\R^n_k$ is invariant under scaling by positive real numbers). So, $A$ is of the form:
\begin{equation*} 
A=\begin{pmatrix}
A_{k,k} & 0  \\
A_{n-k.k} & A_{n-k,n-k} 
\end{pmatrix}, \quad A_{k,k}\in \textrm{GL}_k(\Z), \quad A_{n-k,k}\in \textrm{M}_{n-k,k}(\Z), \quad A_{n-k,n-k}\in \textrm{GL}_{n-k}(\Z),
\end{equation*} where $A_{k,k}$ maps $[0,\infty[^k$ onto $[0,\infty[^k$. Any element of $\GL_k(\Z)$ that maps $[0,\infty[^k$ onto $[0,\infty[^k$ must permute the standard basis of $\R^k$ (this is a version of Proposition \ref{classtorrep}$b$ and is readily verified). Therefore, there is a permutation $\sigma$ of $\{1,...,k\}$ such that $\chi^{j}=\phi^{\sigma(j)}$ for all $j\in \{1,...,k\}$. The map $h_{(x_0,U,\chi)}\circ h_{(x_0,U,\phi)}^{-1}$ is given by:
\begin{equation*} S_{(x_0,U,\phi)}\to S_{(x_0,U,\chi)}, \quad [(t,x,z_1,...,z_k)]\mapsto [(A_*(t),A(x),...,A(x),z_{\sigma(1)},...,z_{\sigma(k)})],
\end{equation*} which is smooth. The inverse of this map is obtained by reversing the roles of $\chi$ and $\phi$. So, it is a diffeomorphism, which means that $(x_0,U,\chi)$ and $(x_0,U,\phi)$ indeed induce the same smooth structure on $J_\Delta^{-1}(U)$. \\

Next, we address the smooth structure in general, by reducing to the previous case. It is enough to show that for every $w_0\in U\cap V\cap \Delta$ there is an open neighbourhood $W$ of $w_0$ in $U\cap V$ such that the smooth structures on $J_\Delta^{-1}(U)$ and $J_\Delta^{-1}(V)$ induced by the respective triples $(x_0,U,\chi)$ and $(y_0,V,\phi)$ restrict to the same smooth structure on $J_\Delta^{-1}(W)$. To do so, let such $w_0$ be given and let $m=\textrm{depth}_\Delta(w_0)$. After possibly permuting the first $k$ components $\chi$ and the first $l$ components of $\phi$, we may assume without loss of generality that $\chi^j(w_0)=0$ if $j\leq m$, $\chi^j(w_0)>0$ if $m<j\leq k$, $\phi^j(w_0)=0$ if $j\leq m$ and $\phi^j(w_0)>0$ if $m<j\leq l$. Indeed, by the previous case, changing the given $\Delta$-admissible triples by such permutations leaves the induced smooth structures invariant. Now, choose a connected open neighbourhood $W$ of $w_0$ in $U\cap V$ such that for every $x\in W$:
\begin{align} \chi^{j}(x)>0 \quad &\textrm{if}\quad  m<j\leq k, \label{smoothandsympstrnattorsppropeq2}\\
 \phi^{j}(x)>0 \quad &\textrm{if}\quad m<j\leq l. \label{smoothandsympstrnattorsppropeq3}
\end{align} 
Consider the charts $\widetilde{\chi}:=\chi\vert_W-\chi(w_0)$ and $\widetilde{\phi}:=\phi\vert_W-\phi(w_0)$. By construction, the triples $(w_0,W,\widetilde{\chi})$ and $(w_0,W,\widetilde{\phi})$ are $\Delta$-admissible and by the previous case these induce the same smooth structure on the open $J_\Delta^{-1}(W)$.  The smooth structure on $J_\Delta^{-1}(U)$ induced by $h_{(x_0,U,\chi)}$ restricts to that on $J_\Delta^{-1}(W)$ induced by $h_{(w_0,W,\widetilde{\chi})}$, since the homeomorphism:
 \begin{equation*} h_{(x_0,U,\chi)}\circ h_{(w_0,W,\widetilde{\chi})}^{-1}:S_{(w_0,W,\widetilde{\chi})}\to J_{(x_0,U,\chi)}^{-1}(W) 
\end{equation*} is a diffeomorphism, for this map is given by:
\begin{equation*}  [(t,x,z_1,...z_m)]\mapsto \left[\left(t,x,z_1,...,z_m,\sqrt{\chi^{m+1}(x)},...,\sqrt{\chi^k(x)} \right)\right],
\end{equation*} and its inverse is given by:
\begin{equation*}
\left[\left(t_1,...,t_m,t_{m+1}\left(\frac{z_{m+1}}{\sqrt{\chi^{m+1}(x)}}\right),...,t_{k}\left(\frac{z_{k}}{\sqrt{\chi^{k}(x)}}\right),t_{k+1},...,t_n,x,z_1,...,z_m\right)\right]\mapsfrom [(t,x,z_1,...,z_k)],
\end{equation*} which are both smooth by (\ref{smoothandsympstrnattorsppropeq2}). Similarly, it follows from (\ref{smoothandsympstrnattorsppropeq3}) that the smooth structure on $J_\Delta^{-1}(V)$ induced by $h_{(y_0,V,\phi)}$ restricts to that on $J_\Delta^{-1}(W)$ induced by $h_{(w_0,W,\widetilde{\phi})}$. All together, this shows that the smooth structures on $J_\Delta^{-1}(U)$ and $J_\Delta^{-1}(V)$ induced by the respective triples $(x_0,U,\chi)$ and $(y_0,V,\phi)$ indeed restrict to the same smooth structure on $J_\Delta^{-1}(W)$. This proves the part of the proposition regarding the smooth structure. \\

For the part regarding the symplectic structure, it is enough to show (by continuity) that the symplectic forms induced by $h_{(x_0,U,\chi)}$ and $h_{(y_0,V,\phi)}$ coincide on the open  $J_\Delta^{-1}(U\cap V)\cap J_\Delta^{-1}(\mathring{\Delta})$, which is dense in $J_\Delta^{-1}(U\cap V)$. To see that these indeed coincide, notice that the quotient map $\sT_\Lambda\vert_\Delta\to S_\Delta$ restricts to a diffeomorphism: 
\begin{equation}\label{regpartdiffeoquotmap} \sT_\Lambda\vert_{\mathring{\Delta}}\to J_\Delta^{-1}(\mathring{\Delta})
\end{equation} and the symplectic forms on the open  $J_\Delta^{-1}(U\cap V)\cap J_\Delta^{-1}(\mathring{\Delta})$ induced by $h_{(x_0,U,\chi)}$ and $h_{(y_0,V,\phi)}$ both coincide with the push-forward of the symplectic form $\Omega_\Lambda$ along (\ref{regpartdiffeoquotmap}).
\end{proof}
\begin{rem} This proof shows that (\ref{regpartdiffeoquotmap}) is a symplectomorphism:
\begin{equation*} (\sT_\Lambda\vert_{\mathring{\Delta}},\Omega_\Lambda) \xrightarrow{\sim} (J_\Delta^{-1}(\mathring{\Delta}),\omega_\Delta).
\end{equation*} In particular, $(S_\Delta,\omega_\Delta)$ is simply $(\sT_\Lambda,\Omega_\Lambda)$ when $\Delta=M$.
\end{rem}
Next, we show that this structure of symplectic manifold on $S_\Delta$ is compatible with the $\sT_\Lambda$-action.
\begin{prop}\label{nattoractsmhamprop} The $\sT_\Lambda$-action along (\ref{mommaptoractconstrfromdelzsubman}) (defined in the previous subsection) is smooth and Hamiltonian with respect to the smooth and symplectic structure in Proposition \ref{smoothandsympstrnattorspprop}. 
\end{prop}
\begin{proof} This follows from the part of Proposition \ref{chartdefnattorspprop} on compatibility of (\ref{indhomdeltadmtripleeq}) with the $\sT_\Lambda$-action. 
\end{proof}
Henceforth, we consider $S_\Delta$ as smooth manifold with the smooth structure in Proposition \ref{smoothandsympstrnattorspprop} and we let $\omega_\Delta$ denote the symplectic structure in this proposition. Combining Proposition \ref{torproptoppartprop} and Proposition \ref{nattoractsmhamprop}, we conclude that (equipped with the $\sT_\Lambda$-action defined in the previous subsection): 
\begin{equation*} J_\Delta:(S_\Delta,\omega_\Delta)\to M
\end{equation*} is a faithful toric $(\sT_\Lambda,\Omega_\Lambda)$-space.
{
\begin{ex} Consider the standard integral affine cylinder $(M,\Lambda):=(\mathbb{S}^1\times \R,\Z\d\theta\oplus \Z\d x)$ with Delzant subspace the compact cylinder $\Delta:=\mathbb{S}^1\times [-1,1]$. In this case, $(S_\Delta,\omega_\Delta)$ is the symplectic manifold $(\T^2\times\mathbb{S}^2,\omega_{\T^2}\oplus \omega_{\mathbb{S}^2})$, with $\omega_{\T^2}$ and $\omega_{\mathbb{S}^2}$ the standard area forms (up to a factor $2\pi$). 
\end{ex}
}
\subsubsection{The natural and local dependence} To complete the proof of Theorem \ref{constrtorictorbunspthm} it remains to address the natural and local dependence of the construction given in the previous subsections. Let $(M_1,\Lambda_1)$ and $(M_2,\Lambda_2)$ be integral affine manifolds with respective Delzant subspaces $\Delta_1$ and $\Delta_2$, and let $\phi:U_1\to U_2$ be a diffeomorphism of manifolds with corners between respective opens $U_1$ in $\Delta_1$ and $U_2$ in $\Delta_2$, such that $\phi^*(\Lambda_2)=\Lambda_1\vert_{U_1}$. Then $\phi$ induces a map:
\begin{equation}\label{indmaptorbun} \sT_{\Lambda_1}\vert_{U_1}\to \sT_{\Lambda_2}\vert_{U_2},\quad [\alpha]\mapsto [(\d\phi^{-1})^*\alpha]. 
\end{equation} Since $\phi$ is a diffeomorphism of manifolds with corners between opens in $\Delta_1$ and $\Delta_2$, it holds that:
\begin{equation*} \d \phi_x(F_x(\Delta_1))=F_{\phi(x)}(\Delta_2),
\end{equation*} for all $x\in U_1$. Therefore, (\ref{indmaptorbun}) descends to a map:
\begin{equation*} J_{\Delta_1}^{-1}(U_1)\to J_{\Delta_2}^{-1}(U_2).
\end{equation*} 
\begin{defi} We define (\ref{assocsymplectotorbunspthm}) to be the map induced by (\ref{indmaptorbun}).
\end{defi}
\begin{prop} This map is indeed a symplectomorphism $(J_{\Delta_1}^{-1}(U_1),\omega_{\Delta_1})\xrightarrow{\sim} (J_{\Delta_2}^{-1}(U_2),\omega_{\Delta_2})$. 
\end{prop}
\begin{proof} Let $x_0\in U_1$ and let $k=\textrm{depth}_{\Delta_1}(x_0)$ and $n=\dim(M_1)$. Since $\phi$ is a diffeomorphism of manifolds with corners from an open in $\Delta_1$ onto an open in $\Delta_2$, it holds that $\textrm{depth}_{\Delta_2}(\phi(x_0))=k$ and $\dim(M_2)=n$. Now, fix integral affine charts $(V_1,\chi_1)$ around $x_0$ and $(V_2,\chi_2)$ around $\phi(x_0)$ with the following properties:
\begin{itemize}\item $(x_0,V_1,\chi_1)$ is $\Delta_1$-admissible and $(\phi(x_0),V_2,\chi_2)$ is $\Delta_2$-admissible, as in (\ref{imadapiachart}), 
\item $V_1\cap \Delta_1\subset U_1$, $V_2\cap \Delta_2\subset U_2$ and $\phi(V_1\cap \Delta_1)\subset V_2\cap \Delta_2$,
\item $\chi_1(V_1)$ is an open ball around the origin in $\R^n$,
\end{itemize} and consider the coordinate representation: 
\begin{equation}\label{coordrepiaintmap} \chi_2\circ \phi\circ \chi_1^{-1}:\chi_1(V_1)\cap\R^n_k\to \chi_2(V_2)\cap\R^n_k.
\end{equation} Using the same arguments as in the proof of Lemma \ref{iamorphintaffvectsplemma}, the fact that $\phi^*\Lambda_2=\Lambda_1\vert_{U_1}$, the fact that any point in $\chi(V_1)\cap \R^n_k$ can be connected to the origin by a smooth path in $\chi(V_1)\cap \R^n_k$ and the fact that (\ref{coordrepiaintmap}) maps the origin to itself, it follows that (\ref{coordrepiaintmap}) is the restriction of an element $A\in \GL_n(\Z)$. Since $A$ is a linear map that maps an open neighbourhood of the origin in $\R^n_k$ to another such open, it must map $\R^n_k$ onto $\R^n_k$. Therefore, $(x_0,V_1,A\circ \chi_1)$ is another $\Delta_1$-admissible triple. To conclude the proof, notice that:
\begin{equation*} h_{(\phi(x_0),V_2,\chi_2)}\circ (\phi)_* \circ h_{(x_0,V_1,A\circ \chi_1)}^{-1}:(S_{(x_0,V_1,A\circ \chi_1)},\omega_{(x_0,V_1,A\circ \chi_1)})\to (S_{(\phi(x_0),V_2,\chi_2)},\omega_{(\phi(x_0),V_2,\chi_2)})
\end{equation*} is given by the inclusion:
\begin{equation*} [(t,x,z)]\mapsto [(t,x,z)], 
\end{equation*} which is clearly smooth and symplectic. Since $x_0\in U_1$ was arbitrary, this shows that $\phi_*$ is smooth and symplectic. Since $\phi_*^{-1}=(\phi^{-1})_*$, the inverse of $\phi_*$ is smooth and symplectic as well. 
\end{proof}
Seeing as the remaining properties listed in Subsection \ref{constrtorictorbunspthmintrosec} are clearly satisfied, this concludes the construction behind (and hence the proof of) Theorem \ref{constrtorictorbunspthm}.
\newpage
\section{The first structure theorem and the splitting theorem}
{In this part we address Theorems \ref{firststrthm} and \ref{globalsplitthm}. In Section \ref{sec:backgrequivshvs} we introduce the equivariant sheaf appearing in these theorems and provide the necessary background on the cohomology of such sheaves. In particular, we recall a \v{C}ech-type model for these cohomology groups that will be used in the proofs of these theorems. The proof of Theorem \ref{firststrthm} is given in Section \ref{sec:pffirststrthm}. Along the way (see Proposition \ref{naturalityofactprop}), we make the meaning of the naturality in that theorem precise. The proof of Theorem \ref{globalsplitthm} is the content of Section \ref{pfsplitthmsec}. In that section we also give the definition of the Lagrangian Dixmier-Douady class.}
\subsection{On equivariant sheaves for topological groupoids}\label{sec:backgrequivshvs}
\subsubsection{{Equivariant sheaves and their cohomology}}\label{subsec:equivshfcohom}
Let $\G\rightrightarrows X$ be a topological groupoid. Recall that a \textbf{$\G$-sheaf of abelian groups} (simply called $\G$-sheaf throughout) is a sheaf $\A$ of abelian groups on $X$ equipped with a continuous (right) action of $\G$ along the etale map of $\A$ by fiberwise group homomorphisms. 
The category $\textsf{Sh}(\G)$ of $\G$-sheaves, with morphisms the $\G$-equivariant maps of the underlying sheaves, is abelian and has enough injectives. The additive functor $\Gamma^{\G}:\textsf{Sh}(\G)\to \textsf{Ab}$ that takes $\G$-invariant global sections is left-exact, and the cohomology groups of a $\G$-sheaf $\A$ are defined as the right-derived functors of $\Gamma^\G$, meaning that:
\begin{equation*} H^n(\G,\A):=H^n(\Gamma^\G(I^\bullet)),
\end{equation*} with $0\to \A\to I^0\to I^1\to \dots$ an injective resolution of $\A$ in $\textsf{Sh}(\G)$. 
\begin{ex}\label{ex:orbifoldsheafoflagsec} Our main example of interest is the sheaf $\L_\Delta$ appearing in theorems \ref{firststrthm} and \ref{globalsplitthm}. This $\B\vert_\Delta$-sheaf is defined for any integral affine orbifold groupoid $\B\rightrightarrows (M,\Lambda)$ and any Delzant subspace $\underline{\Delta}$ thereof, as follows. Let $\Delta$ denote the corresponding invariant subspace of $M$ and consider the pre-symplectic torus bundle $(\sT_\Lambda,\Omega_\Lambda)$ defined as in Remark \ref{presymptorbunrem}. For simplicity, suppose first that $\B$ is etale, so that the torus bundle $(\sT_\Lambda,\Omega_\Lambda)$ is symplectic. Then $\L_\Delta$ is the sheaf in Definition \ref{shfinvlagrsecdefi:torbunversion} and, in fact, both sheaves in Definition \ref{shfinvlagrsecdefi:torbunversion} are naturally $\B\vert_\Delta$-sheaves. For $\gamma:x\to y$ in $\B\vert_\Delta$ and $\tau:V\to \sT_\Lambda$ a section over an open $V$ in $\Delta$ around $y$, the action on the germ of $\tau$ at $y$ is given by: 
\begin{equation*} [\tau]_y\cdot \gamma:=[(t\circ\sigma)^*(\tau)]_x,
\end{equation*} with $\sigma:U\to \B$ any smooth section of $s_\B$ over an open $U$ in $\Delta$ around $x$ such that $\sigma(x)=\gamma$ and $t(\sigma(U))\subset V$. In general (when $\B$ is not necessarily etale), we let $\L_\Delta$ denote the sheaf of isotropic sections of $\sT_\Lambda\vert_\Delta$ (i.e. smooth local sections $\tau$ of $\sT_\Lambda\vert_\Delta$ such that $\tau^*\Omega_\Lambda=0$). This is a subsheaf of the sheaf of $\F$-basic smooth sections of $\sT_\Lambda\vert_\Delta$, with $\F$ the foliation by connected components of the {orbits} of $\B$. Both of these are $\B\vert_\Delta$-sheaves, with the action defined as in the etale case.
\end{ex}
When working with these cohomology groups we usually reduce to the case in which the groupoid is etale, by restricting to a complete transversal (one that intersects all {orbits}). This is possible due to Example \ref{transversalmoreqex} and:  
\begin{prop}\label{prop:morinvshfcohom}
A Morita equivalence of Lie groupoids with corners (as in Definition \ref{def:moreqLiegpoidswithcorners}):
\begin{center}
\begin{tikzpicture} \node (G1) at (0,0) {$\G_1$};
\node (M1) at (0,-1.3) {$X_1$};
\node (S) at (1.4,0) {$P$};
\node (M2) at (2.7,-1.3) {$X_2$};
\node (G2) at (2.7,0) {$\G_2$};
 
\draw[->,transform canvas={xshift=-\shift}](G1) to node[midway,left] {}(M1);
\draw[->,transform canvas={xshift=\shift}](G1) to node[midway,right] {}(M1);
\draw[->,transform canvas={xshift=-\shift}](G2) to node[midway,left] {}(M2);
\draw[->,transform canvas={xshift=\shift}](G2) to node[midway,right] {}(M2);
\draw[->](S) to node[pos=0.25, below] {$\text{ }\text{ }\alpha_1$} (M1);
\draw[->] (0.8,-0.15) arc (315:30:0.25cm);
\draw[<-] (1.9,0.15) arc (145:-145:0.25cm);
\draw[->](S) to node[pos=0.25, below] {$\alpha_2$\text{ }} (M2);
\end{tikzpicture}
\end{center} induces an equivalence of categories:
\begin{equation*}
P_*:\textsf{Sh}(\G_1)\xrightarrow{\sim}\textsf{Sh}(\G_2),
\end{equation*} and hence, for any $\G_1$-sheaf $\A$, it induces an isomorphism:
\begin{equation*} H^\bullet(\G_1,\A)\cong H^\bullet(\G_2,P_*\A).
\end{equation*} This is functorial with respect to composition of Morita equivalences.
\end{prop}
\begin{cor}\label{cor:morinvshfcohom:iaequiv}
An integral affine Morita equivalence:
\begin{center}
\begin{tikzpicture} \node (G1) at (0,0) {$\B_1$};
\node (M1) at (0,-1.3) {$(M_1,\Lambda_1)$};
\node (S) at (1.4,0) {$P$};
\node (M2) at (2.7,-1.3) {$(M_2,\Lambda_2)$};
\node (G2) at (2.7,0) {$\B_2$};
 
\draw[->,transform canvas={xshift=-\shift}](G1) to node[midway,left] {}(M1);
\draw[->,transform canvas={xshift=\shift}](G1) to node[midway,right] {}(M1);
\draw[->,transform canvas={xshift=-\shift}](G2) to node[midway,left] {}(M2);
\draw[->,transform canvas={xshift=\shift}](G2) to node[midway,right] {}(M2);
\draw[->](S) to node[pos=0.25, below] {$\text{ }\text{ }\alpha_1$} (M1);
\draw[->] (0.8,-0.15) arc (315:30:0.25cm);
\draw[<-] (1.9,0.15) arc (145:-145:0.25cm);
\draw[->](S) to node[pos=0.25, below] {$\alpha_2$\text{ }} (M2);
\end{tikzpicture}
\end{center} that relates a Delzant subspace $\underline{\Delta}_1\subset \underline{M}_1$ to $\underline{\Delta}_2\subset \underline{M}_2$ induces an isomorphism:
\begin{equation*} H^\bullet(\B_1\vert_{\Delta_1},\L_1)\cong H^\bullet(\B_2\vert_{\Delta_2},\L_2),
\end{equation*} between the cohomology groups of $\L_1:=\L_{\Delta_1}$ and $\L_2:=\L_{\Delta_2}$ (as in Example \ref{ex:orbifoldsheafoflagsec}). This is functorial with respect to composition of integral affine Morita equivalences. 
\end{cor}
\begin{proof}[Proof of Proposition \ref{prop:morinvshfcohom} and Corollary \ref{cor:morinvshfcohom:iaequiv}] Proposition \ref{prop:morinvshfcohom} can be proved as for Lie groupoids without corners (see \cite[Corollary 3.11]{MoMr1}). The functor $P_*$ sends a $\G_1$-sheaf $\A$ with etale space $E\to X_1$ to the $\G_2$-sheaf corresponding to the etale space: 
\begin{equation*} \alpha_2\circ \textrm{pr}_P:(P\times_{X_1}E)/\G_1\to X_2,
\end{equation*} equipped with the $\G_2$-action given by $[p,e]\cdot g:=[p\cdot g,e]$. The corollary follows from this, because for an integral affine Morita equivalence that relates $\underline{\Delta}_1$ to $\underline{\Delta}_2$ the $\B\vert_{\Delta_2}$-sheaves $P_*(\L_{\Delta_1})$ and $\L_{\Delta_2}$ are isomorphic, via the isomorphism given as follows. Let us call an open $U$ in $M_i$ simple if it is the domain of a foliation chart $\chi$ with $\chi(U)$ a product of a leafwise open cube and a transversal open cube centered around the origin, such that $\chi(U\cap \Delta)=\chi(U)\cap \R^n_k$ and the $\F_i$-basic coordinate $1$-forms span $\Lambda_i\vert_U$. Suppose that $U$ is a simple open in $M_2$ over which $\alpha_2$ admits a smooth section $\sigma:U\to P$ such that $\alpha_1\circ \sigma$ maps $U$ into a simple open $V$ in $M_1$ and $(\alpha_1\circ \sigma)(U)$ intersects all plaques of $\F_1$ in $V$. Since the Morita equivalence is integral affine and sections of $\L_1$ are $\F_1$-basic, pull-back along $\alpha_1\circ \sigma$ (whose differential at $x$ lifts the map $\psi^{-1}_{\sigma(x)}$ as in Remark \ref{altdefiamoreqrem}) induces an isomorphism:
\begin{equation*} \L_1(V\cap \Delta_1)\xrightarrow{\sim} \L_2(U\cap \Delta_2), \quad \tau \mapsto (\alpha_1\circ \sigma)^*(\tau).
\end{equation*}
On the other hand, we have an isomorphism:
\begin{equation*} \L_{1}(V\cap \Delta_1)\xrightarrow{\sim} P_*(\L_{1})(U\cap \Delta_2),\quad \tau\mapsto [\sigma(-),[\tau]_{(\alpha_1\circ \sigma)(-)}].
\end{equation*} Over opens of the form $U\cap \Delta_2$, which form a basis of $\Delta_2$, the isomorphism between $P_*(\L_{1})$ and $\L_{2}$ is given by the composition of these, which does not depend on the choice of $\sigma$ or $V$.  
 \end{proof}
\subsubsection{Bisections, embedding categories and their cohomology}\label{bisectionsembcatsec} 
Recall that a \textbf{continuous bisection} of a topological groupoid $\G\rightrightarrows X$ is a continuous section $\sigma:U\to \G$ of the source-map of $\G$, defined on an open $U$ in $X$, with the property that $t\circ\sigma:U\to X$ is an open embedding. For the image of this open embedding we use the notation:
\begin{equation*} \sigma\cdot U:=(t\circ \sigma)(U).
\end{equation*} Given two continuous bisections $\sigma_1:U_1\to \G$ and $\sigma_2:U_2\to \G$ such that $\sigma_2\cdot U_2\subset U_1$, we can consider their composition, which is again a continuous bisection of $\G$: \begin{equation*} \sigma_1\sigma_2:U_2\to \G, \quad (\sigma_1\sigma_2)(x):=\sigma_1((t\circ\sigma_2)(x))\sigma_2(x).
\end{equation*} In the smooth and symplectic setting we consider the following variations of this.
\begin{itemize} \item By a \textbf{smooth bisection} of a Lie groupoid with corners $\G\rightrightarrows X$ (see Definition \ref{liegpwithcorndef}) we mean a continuous bisection $\sigma:U\to \G$ with the property that $\sigma$ is smooth as map between manifolds with corners and $t\circ \sigma$ is a diffeomorphism onto its image. 
\item By a \textbf{Lagrangian bisection} of a symplectic groupoid with corners $(\G,\Omega)\rightrightarrows X$ (see Definition \ref{sympgpwithcorndef}) we mean a smooth bisection $\sigma:U\to \G$ with the property that $\sigma^*\Omega=0$.
\end{itemize} The composition of two smooth (resp. Lagrangian) bisections is again smooth (resp. Lagrangian). Moreover, continuous bisections of etale Lie groupoids with corners are automatically smooth. \\

We will use a \v{C}ech-type model from \cite{Mo2,Mo,CrMo1} for the cohomology of equivariant sheaves for etale groupoids. To recall this, let $\E\rightrightarrows X$ be a topological etale groupoid. Given a basis $\U$ for the topology of $X$, the \textbf{embedding category} over $\U$ is the category $\textsf{Emb}_\U(\E)$ with objects the opens in $\U$ and arrows $V\leftarrow U$ the continuous bisections $\sigma:U\to \E$ with $\sigma\cdot U\subset V$. We denote such an arrow as $V\xleftarrow{\sigma} U$. The composition of two arrows  $U_0\xleftarrow{\sigma_1}U_1\xleftarrow{\sigma_2}U_2$ is the arrow $U_0\xleftarrow{\sigma_1\sigma_2} U_2$ with underlying bisection the composition of $\sigma_1$ and $\sigma_2$. A \textbf{category approximating} $\E$ over $\U$ is a wide subcategory $\textsf{E}$ of $\textsf{Emb}_\U(\E)$ with the following properties. 
\begin{itemize}
\item[i)] For any $U,V\in \U$ with $U\subset V$, the arrow $V\hookleftarrow U$ (with bisection the unit map) lies in $\textsf{E}$.
\item[ii)] If $V\xleftarrow{\sigma} U$ belongs to \textsf{E}, then so does $V'\xleftarrow{\sigma} U'$, for any $U',V'\in \U$ with $\sigma\cdot U'\subset V'$.
\item[iii)] For any $V\in \U$ and $y\xleftarrow{g} x$ in $\E$ with $y\in V$, there is an arrow $V\xleftarrow{\sigma}U$ in \textsf{E} with $x\in U$ and $\sigma(x)=g$.
\end{itemize}
Give such a category $\textsf{E}$, an $\E$-sheaf $\A$ induces a pre-sheaf on $\textsf{E}$ (a contravariant functor from $\textsf{E}$ to the category of abelian groups) that assigns to a $U\in \U$ the abelian group $\A(U)$ and to an arrow $V\xleftarrow{\sigma}U$ the map $\A(V)\to \A(U)$ that sends $a\in \A(V)$ to the section $a\cdot \sigma\in \A(V)$ given by:
\begin{equation*} [a\cdot \sigma]_x:=[a]_{(t\circ \sigma)(x)}\cdot \sigma(x)\in \A_x,\quad x\in U. 
\end{equation*}
The cohomology that we will use is that of the small category $\textsf{E}$ with values in this pre-sheaf. We denote this as $\check{H}^\bullet(\textsf{E},\A)$, or as $\check{H}^\bullet_\U(\E,\A)$ when $\textsf{E}=\textsf{Emb}_\U(\E)$. Explicitly, $\check{H}^\bullet(\textsf{E},\A)$ is the cohomology of the cochain complex $(\check{C}(\textsf{E},\A),\check{\d})$, with $n$-cochains:
\begin{equation*} \check{C}^n(\textsf{E},\A)=\bigsqcap_{U_0\xleftarrow{\sigma_1} ... \xleftarrow{\sigma_n} U_n} \A(U_n)
\end{equation*} where the product runs through all $n$-strings of composable arrows in $\textsf{E}$. So, an $n$-cochain $c$ assigns to each such string $U_0\xleftarrow{\sigma_1} ... \xleftarrow{\sigma_n} U_n$ an element:
\begin{equation*} c(U_0\xleftarrow{\sigma_1} ... \xleftarrow{\sigma_n} U_n)\in \A(U_n). 
\end{equation*} The differential $\check{d}:\check{C}^n(\textsf{E},\A)\to \check{C}^{n+1}(\textsf{E},\A)$ is given by:
\begin{align*} \check{d}(c)(U_0\xleftarrow{\sigma_1} ...\xleftarrow{\sigma_{n+1}}U_{n+1})&=c(U_1\xleftarrow{\sigma_2} ...\xleftarrow{\sigma_{n+1}}U_{n+1})\\
&+\sum_{i=1}^n(-1)^ic(U_0\xleftarrow{\sigma_1} ... \xleftarrow{\sigma_{i-1}} U_{i-1} \xleftarrow{\sigma_{i}\sigma_{i+1}} U_{i+1}\xleftarrow{\sigma_{i+2}} ...\xleftarrow{\sigma_{n+1}}U_{n+1})\\
&+(-1)^{n+1}c(U_0\xleftarrow{\sigma_1}...\xleftarrow{\sigma_n} U_n)\cdot\sigma_{n+1}.
\end{align*} 
Like for usual \v{C}ech cohomology of sheaves on topological spaces, there is a natural morphism:
\begin{equation}\label{eqn:cechtoderivedmorph:equivshf} \phi_\textsf{E}:\check{H}^\bullet(\textsf{E},\A)\to H^\bullet(\E,\A),
\end{equation} induced, for instance, by the double complex obtained by applying $\check{C}^\bullet(\textsf{E},-)$ to an injective resolution of the $\G$-sheaf $\A$. This is an isomorphism if $\A$ is \textbf{$\U$-acyclic}, meaning that $H^n(U,\A)=0$ for each $n>0$ and each $U\in \U$. If $\U'$ is another basis and $\textsf{E}'$ a category approximating $\E$ over $\U'$, then is $\U\cup \U'$ is also a basis and the diagram:
\begin{center}
\begin{tikzcd} 
& H^\bullet(\E,\A) & \\
\check{H}^\bullet(\textsf{E},\A)\arrow[ru,"\phi_\textsf{E}"]\arrow[r,"(i_\textsf{E})^*"'] & \check{H}_{\U\cup \U'}^\bullet(\E,\A) \arrow[u,"\phi_{\textsf{E}''}"'] & \check{H}_{\U'}^\bullet(\textsf{E}',\A)\arrow[lu,"\phi_{\textsf{E}'}"']\arrow[l,"(i_{\textsf{E}'})^*"]
\end{tikzcd}
\end{center} commutes, where $\textsf{E}'':=\textsf{Emb}_{\U\cup\U'}(\E)$ and $i_\textsf{E}:\textsf{E}\to \textsf{E}''$ and $i_{\textsf{E}'}:\textsf{E}'\to \textsf{E}''$ are the inclusion functors. \\

\begin{rem}\label{vanlemrem} {Given an etale integral affine orbifold groupoid $\B\rightrightarrows (M,\Lambda)$ and a Delzant subspace $\underline{\Delta}$ with corresponding invariant subspace $\Delta\subset M$, there always is a basis $\U$ of $\Delta$ for which $\L:=\L_\Delta$ is $\U$-acyclic. This follows from the lemma below.}
\end{rem}
\begin{lemma}\label{vanlem} {Let $(V,\Lambda_V)$ be an integral affine vector space and let $\Delta$ be a convex Delzant subspace. Then $H^\bullet(\Delta,\L)$ vanishes in all degrees greater than zero.}
\end{lemma}
\begin{proof} By the Poincar\'{e} lemma for manifolds with corners we have a short exact sequence:
\begin{equation*} 0\to (\mathcal{C}^\infty)_\Lambda\to \mathcal{C}^\infty\xrightarrow{\d} \L\to 0,
\end{equation*} where $\L:=\L_\Delta$, $\mathcal{C}^\infty:=\mathcal{C}^\infty_\Delta$ and $(\mathcal{C}^\infty)_\Lambda$ is the subsheaf of $\mathcal{C}^\infty$ of those functions $f$ such that $\d f$ takes value in the lattice bundle $\Lambda\subset T^*V$ of $(V,\Lambda_V)$. Secondly, we have a short exact sequence:
\begin{equation*}0\to \R\to (\mathcal{C}^\infty)_\Lambda\xrightarrow{\d} \mathcal{C}^\infty(\Lambda)\to 0,
\end{equation*}
where $\R$ denotes the sheaf on $\Delta$ of locally constant functions with values in $\R$ and $\mathcal{C}^\infty(\Lambda)$ denotes the sheaf of smooth sections of $\Lambda\vert_\Delta$. Since $\mathcal{C}^\infty$ is a fine sheaf and $H^\bullet(\Delta,\R)$ vanishes in degree greater than zero ($\Delta$ being convex), we derive from the resulting long exact sequences that:
\begin{equation*} H^p(\Delta,\L)\cong H^{p+1}(\Delta,(\mathcal{C}^\infty)_\Lambda)\cong H^{p+1}(\Delta,\mathcal{C}^\infty(\Lambda))
\end{equation*} for all $p>0$. The lemma follows from this, because $\mathcal{C}_\Delta^\infty(\Lambda)$ is the sheaf of locally constant functions with values in $\Lambda_V^*$ (since $\Lambda=\Lambda^*_V\times V$), so that the cohomology $H^\bullet(\Delta,\mathcal{C}^\infty(\Lambda))$ vanishes in all degrees greater than zero (by convexity of $\Delta$). 
\end{proof}

\subsection{Proof of the first structure theorem}\label{sec:pffirststrthm} 
\subsubsection{Lifting to a Lagrangian cocycle}\label{liftlagrcocyclesec} 
The aim of this subsection is to prove the lemma below, which will be key in the proofs of both the first structure theorem and the splitting theorem.
\begin{lemma}\label{lifrlagrcocyleprop} Let $(\G,\Omega)$ be a regular and proper symplectic groupoid for which the associated orbifold groupoid $\B=\G/\sT$ is etale, let $\underline{\Delta}$ be a Delzant subspace of $\underline{M}$ and let $\U$ be a basis of $\Delta$ for which $\L_\Delta$ is $\U$-acyclic. If $\Delta$ is the momentum image of a faithful multiplicity-free $(\G,\Omega)$-space, then for each arrow $V\xleftarrow{\sigma} U$ in $\textsf{Emb}_\U(\B\vert_\Delta)$ there is a Lagrangian bisection:
\begin{equation}\label{lagrbislift} g(V\xleftarrow{\sigma} U):U\to (\G,\Omega)\vert_\Delta
\end{equation} lifting $\sigma:U\to \B\vert_\Delta$, and these lifts can be chosen so as to satisfy the cocycle condition: 
\begin{equation}\label{cocyccondlagrbis} g(U_0\xleftarrow{\sigma_1\sigma_2} U_2)=g(U_0\xleftarrow{\sigma_1} U_1)g(U_1\xleftarrow{\sigma_2} U_2)
\end{equation} for any two composable arrows $U_0\xleftarrow{\sigma_1} U_1\xleftarrow{\sigma_2} U_2$. 
\end{lemma}

To prove this, the following observation will be key.
\begin{prop}\label{liftinglagrbiswithsymplectoprop} Let $(\G,\Omega)$ be as in Lemma \ref{lifrlagrcocyleprop}, let $J:(S,\omega)\to M$ be a faithful multiplicity-free $(\G,\Omega)$-space and let $\Delta:=J(S)$. Further, suppose that we are given a continuous (or equivalently, smooth) bisection $\sigma:U\to \B\vert_\Delta$ defined on an open $U$ in $\Delta$. Then the following hold.
\begin{itemize}\item[a)] For any Lagrangian bisection $g_\sigma:U\to (\G,\Omega)$ lifting $\sigma$, the map: 
\begin{equation*} \psi_{g_\sigma}:(J^{-1}(U),\omega)\to (J^{-1}(\sigma\cdot U),\omega), \quad \psi_{g_\sigma}(p)=g_\sigma(J(p))\cdot p,
\end{equation*} is a symplectomorphism that fits into a commutative square:
\begin{equation}\label{commsqliftinglagrbiswithsymplectoprop}
\begin{tikzcd} (J^{-1}(U),\omega) \arrow[r,"\psi_{g_\sigma}"] \arrow[d,"J"] & (J^{-1}(\sigma\cdot U),\omega) \arrow[d,"J"] \\
 U \arrow[r,"t\circ \sigma"] & \sigma\cdot U 
\end{tikzcd}
\end{equation} and is compatible with the $\sT$-action in the sense that for each $p\in J^{-1}(U)$ and $t\in \sT_{J(p)}$:
\begin{equation*} \psi_{g_\sigma}(t\cdot p)=(\sigma(J(p))\cdot t)\cdot \psi_{g_\sigma}(p).
\end{equation*} 
\item[b)] Conversely, for any symplectomorphism $\psi:(J^{-1}(U),\omega)\to (J^{-1}(\sigma\cdot U),\omega)$ that fits into a commutative square and is compatible with the $\sT$-action as above, there is a unique Lagrangian bisection $g_\sigma:U\to (\G,\Omega)$ that lifts $\sigma$ and is such that $\psi=\psi_{g_\sigma}$.
\end{itemize}
\end{prop} 
\begin{proof}
Notice first that, given a smooth bisection $g_\sigma:U\to \G$ lifting $\sigma$, the associated map $\psi_{g_\sigma}$ is a diffeomorphism, fits into a commutative square as above and is compatible with the $\sT$-action. As in {Proposition \ref{lagrseccharprop:torbunversion}} one can prove that $\psi_{g_\sigma}$ is a symplectomorphism if and only if the smooth bisection $g_\sigma:U\to (\G,\Omega)$ is Lagrangian. This proves $a$ and shows that, to prove $b$, it is enough to prove that for any diffeomorphism $\psi:J^{-1}(U)\to J^{-1}(\sigma\cdot U)$ that fits into a commutative square like (\ref{commsqliftinglagrbiswithsymplectoprop}) and is compatible with the $\sT$-action as above, there is a unique smooth bisection $g_\sigma:U\to \G$ that lifts $\sigma$ and is such that $\psi=\psi_{g_\sigma}$. For uniqueness, suppose that $g_\sigma,h_\sigma:U\to \G$ are two such smooth bisections. Let $p\in J^{-1}(U)$ such that $\sT$ acts freely at $p$. Then $(g_\sigma^{-1}h_\sigma)(J(p))$ belongs to $\sT$ (since both $g_\sigma$ and $h_\sigma$ lift $\sigma$) and it fixes $p$ (since $\psi_{g_\sigma}=\psi_{h_\sigma}$). Hence, $g_\sigma(J(p))=h_\sigma(J(p))$ for all $p\in J^{-1}(U)$ at which $\sT$ acts freely. Since the set of such $p\in S$ is dense in $S$, it follows by continuity that $g_\sigma=h_\sigma$, as claimed. In view of this uniqueness, to prove existence it is enough to show that for every $x\in U$ there is an open neighbourhood $U_x$ of $x$ in $U$ and a smooth section $g_\sigma:U_x\to \G$ of the source map such that $\psi(p)=g_\sigma(J(p))\cdot p$ for all $p\in J^{-1}(U_x)$. To this end, let $x\in U$ and let $h_\sigma:U_x\to \G$ be any smooth bisection, defined on some open neighbourhood $U_x$ of $x$ in $U$. Then $\psi_{h_\sigma}^{-1}\circ (\psi\vert_{J^{-1}(U_x)})$ belongs to $\textrm{Aut}_\sT(J)(U_x)$, {as in (\ref{gpofsTequivdiffeos:torbunversion})}. Hence, it follows from {Theorem \ref{torsectautisoprop:torbunversion} }applied to the $(\sT,\Omega_\sT)$-space $J:(S,\omega)\to M$ {(which is faithful toric by Lemma \ref{lemma:indtorbunacttoric})}, that there is a section $\tau\in \mathcal{C}^\infty_\Delta(\sT)(U_x)$ with the property that: 
\begin{equation*} \psi_{h_{\sigma}}^{-1}(\psi(p))=\tau(J(p))\cdot p
\end{equation*} for every $p\in J^{-1}(U_x)$. The smooth bisection $g_\sigma=h_\sigma\tau:U_x\to \G$ has the desired property.  
\end{proof}
\begin{proof}[Proof of Lemma \ref{lifrlagrcocyleprop}] Since {$H^1(U,\L_\Delta)=0$} for each $U\in \U$, it follows from {Theorem \ref{torsortorspthm:torbunversion}} that there is an isomorphism of Hamiltonian $(\sT,\Omega_\sT)$-spaces:
\begin{center}
\begin{tikzcd} (J_\Delta^{-1}(U),\omega_\Delta)\arrow[rr,"\psi_U"]\arrow[rd,"J_\Delta"'] & & (J^{-1}(U),\omega)\arrow[ld,"J"]\\  
& M & 
\end{tikzcd}
\end{center} where $\sT$ acts along $J_\Delta$ via the $\sT_\Lambda$-action in Theorem \ref{constrtorictorbunspthm} and (\ref{exptorbun2}). Fix such an isomorphism for each $U\in \U$. Now, let $V\xleftarrow{\sigma} U$ be an arrow in $\textsf{Emb}_\U(\B\vert_\Delta)$. It follows from $\B$-invariance of the lattice $\Lambda$ in $T^*M$ that $(t\circ\sigma)^*\Lambda=\Lambda\vert_U$. So, since $t\circ \sigma$ is a diffeomorphism of manifolds with corners from $U$ onto an open in $\Delta$, by Theorem \ref{constrtorictorbunspthm} we have an induced symplectomorphism:
\begin{equation*} (J_\Delta^{-1}(U),\omega)\xrightarrow{(t\circ \sigma)_*} (J^{-1}_\Delta(\sigma\cdot U),\omega).
\end{equation*} Now notice that:
\begin{equation*} \psi_V\circ (t\circ \sigma)_* \circ\psi_U^{-1}:(J^{-1}(U),\omega)\to (J^{-1}(\sigma\cdot U),\omega)
\end{equation*} meets the assumptions of Proposition \ref{liftinglagrbiswithsymplectoprop}$b$ (as is readily verified using Lemma \ref{sympmoreqnormreplem}), so that there is a unique Lagrangian bisection:
\begin{equation*} g(V\xleftarrow{\sigma}U):U\to (\G,\Omega)\vert_\Delta
\end{equation*} that lifts $\sigma$ and satisfies:
\begin{equation*} g(V\xleftarrow{\sigma}U)(J(p))\cdot p=(\psi_V\circ (t\circ \sigma)_* \circ\psi_U^{-1})(p)
\end{equation*} for all $p\in J^{-1}(U)$. Using the natural and local dependence in Theorem \ref{constrtorictorbunspthm}, it is readily verified that these choices of lifts satisfy the cocycle condition (\ref{cocyccondlagrbis}). 
\end{proof}
\subsubsection{The rest of the proof} We start with a proof of the first structure theorem for the case in which $\B$ is etale, excluding naturality with respect to symplectic Morita equivalences. The proof in full generality and the naturality will be addressed afterwards. 
\begin{proof}[Proof of Theorem \ref{firststrthm}; definition of the action on (\ref{isoclwithprescrmomimnonempty}) in the etale case] Let $(\G,\Omega)\rightrightarrows M$ be a regular and proper symplectic groupoid for which the associated orbifold groupoid $\B=\G/\sT$ is etale and let $\underline{\Delta}\subset \underline{M}$ be a Delzant subspace. To define the action, {fix a basis $\U$ of $\Delta$ such that $\L:=\L_\Delta$ is $\U$-acyclic}. {Let $\textrm{c}\in H^1(\B\vert_\Delta,\L)$ and consider a $1$-cocycle $\tau$ representing the class $[\tau]\in \check{H}^1_\U(\B\vert_\Delta,\L)$ corresponding to $\textrm{c}$ via the isomorphism (\ref{eqn:cechtoderivedmorph:equivshf}).} Let $J:(S,\omega)\to M$ be a faithful multiplicity-free $(\G,\Omega)$-space with momentum image $\Delta$. First consider the topological space:
\begin{equation}\label{disjunionspeq} \bigsqcup_{U\in \U} J^{-1}(U).
\end{equation} Given two elements $(p_1,U_1)$ and $(p_2,U_2)$ of (\ref{disjunionspeq}), we write $(p_1,U_1)\hookleftarrow (p_2,U_2)$ if $U_2\subset U_1$ and:
\begin{equation*} \tau(U_1\hookleftarrow U_2)(J(p_2))\cdot p_2=p_1,
\end{equation*} where $U_2 \hookleftarrow U_1$ denotes the arrow in $\textsf{Emb}_\U(\B\vert_\Delta)$ with underlying bisection the unit map. This relation is reflexive and transitive (as follows from the fact that $\tau$ is a $1$-cocycle), but not necessarily symmetric. The equivalence relation generated by this relation is given by: $(p_1,U_1)\sim (p_2,U_2)$ if and only if there is a pair $(p_{12},U_{12})$, with $U_{12}\in \U$ and $p_{12}\in U_{12}\subset U_1\cap U_2$, such that $(p_1,U_1)\hookleftarrow (p_{12},U_{12})$ and $(p_{12},U_{12})\hookrightarrow (p_2,U_2)$. Indeed, this relation is clearly reflexive, symmetric and contains the first relation. To prove transitivity, suppose that $(p_1,U_1)\sim (p_2,U_2)$ and $(p_2,U_2)\sim (p_3,U_3)$. Then there are $(p_{12},U_{12})$ and $(p_{23},U_{23})$ with $U_{12},U_{23}\in \U$, $p_{12}\in U_{12}\subset U_1\cap U_2$ and $p_{23}\in U_{23}\subset U_2\cap U_3$, such that:
\begin{center}
\begin{tikzcd} (p_1,U_1) & & (p_2,U_2) & & (p_3,U_3)\\
& (p_{12},U_{12}) \arrow[ul, hook']\arrow[ur, hook] & & (p_{23},U_{23})\arrow[ul, hook']\arrow[ur, hook] &  
\end{tikzcd}
\end{center} Let $U_{123}\in \U$ be such that $J(p_2)\in U_{123}\subset U_{12}\cap U_{23}$ and consider the element
\begin{equation*} p_{123}:=\tau(U_2\hookleftarrow U_{123})(J(p_2))^{-1}\cdot p_2\in J^{-1}(U_{123}),
\end{equation*} which is well-defined since the source and target of $\tau(U_2\hookleftarrow U_{123})(J(p_2))$ coincide. It follows from the cocycle condition (\ref{cocyccondlagrbis}) that $(p_{12},U_{12})\hookleftarrow (p_{123},U_{123})$ and $(p_{123},U_{123})\hookrightarrow (p_{23},U_{23})$, and hence that $(p_1,U_1)\hookleftarrow (p_{123},U_{123})$ and $(p_{123},U_{123})\hookrightarrow (p_{3},U_{3})$. This shows that $(p_1,U_1)\sim (p_3,U_3)$, which proves transitivity. Now, consider the quotient space:
\begin{equation*} S_\tau:=\frac{\left(\bigsqcup_{U\in \U} J^{-1}(U)\right)}{\sim}. 
\end{equation*} From the above description of the equivalence relation it is clear that for each $U\in \U$ the map:
\begin{equation}\label{injopenintoquot} j_U:J^{-1}(U)\to S_\tau, \quad p\mapsto [p,U]
\end{equation} is an injection. Since for each inclusion $V\hookleftarrow U$ the map:
\begin{equation*} (J^{-1}(U),\omega)\to (J^{-1}(U),\omega), \quad p\mapsto \tau(V\hookleftarrow U)(J(p))\cdot p
\end{equation*} is a symplectomorphism (which follows from {Proposition \ref{lagrseccharprop:torbunversion}}), there is a unique structure of symplectic manifold on $S_\tau$ with the property that for each $U\in \U$ the injection (\ref{injopenintoquot}) is a symplectomorphism onto an open in $S_\tau$ (with respect to the symplectic form $\omega$ on $J^{-1}(U)$). Let $\omega_\tau$ be the corresponding symplectic form on $S_\tau$. Next, note that $J:(S,\omega)\to M$ induces a smooth map:
\begin{equation*} J_\tau:(S_\tau,\omega_\tau)\to M, \quad [p,U]\mapsto J(p),
\end{equation*} along which $(\G,\Omega)$ acts in a Hamiltonian fashion, as follows. Given $g\in \G$ and $p_\tau\in S_\tau$ such that $s(g)=J_\tau(p_\tau)$, let $(p,U)$ be a representative of $p_\tau$ with $U\in \U$ small enough such that there is an arrow $V\xleftarrow{\sigma} U$ in $\textsf{Emb}_\U(\B\vert_\Delta)$ with the property that $\sigma(s(g))=[g]$, and set:
\begin{equation*} g\cdot p_\tau=[g\cdot\tau(V\xleftarrow{\sigma} U)(s(g))\cdot p,V]. 
\end{equation*} It follows from the fact that $\tau$ is a $1$-cocycle that this is independent of the choice of representative $(p,U)$ and the choice of arrow $V\xleftarrow{\sigma} U$, and that this defines a $\G$-action. Furthermore, the fact that $\tau(V\xleftarrow{\sigma} U)$ is Lagrangian implies that {the} action is Hamiltonian. 
 It is readily verified that the $(\G,\Omega)$-space $J_\tau:(S_\tau,\omega_\tau)\to M$ is {faithful multiplicity-free}, has momentum image $\Delta$ and that its isomorphism class {does not depend on the choice of representative $\tau$ or basis $\U$. 
So, we can define: 
\begin{equation*} \textrm{c}\cdot [J:(S,\omega)\to M]:=[J_\tau:(S_\tau,\omega_\tau)\to M].
\end{equation*} }We leave it for the reader to check that this defines an action of $H^1(\B\vert_\Delta,\L)$.
\end{proof}
\begin{proof}[Proof of Theorem \ref{firststrthm}; freeness and transivity of the action in the etale case] {Let $\U$ be a basis as above.} To see that the action is free, let $\tau$ be a $1$-cocycle as above and suppose that:{
\begin{equation*} [J_\tau:(S_\tau,\omega_\tau)\to M]=[J:(S,\omega)\to M],
\end{equation*} meaning that there is an isomorphism of Hamiltonian} $(\G,\Omega)$-spaces:
\begin{center}
\begin{tikzcd} (S_\tau,\omega_\tau) \arrow[rr,"\psi"]\arrow[rd,"J_\tau"'] & & (S,\omega)\arrow[ld, "J"] \\
& M & 
\end{tikzcd}
\end{center} Then for each $U\in \U$ we have an automorphism of {Hamiltonian} $(\sT,\Omega_{\sT})$-spaces:
\begin{equation}\label{autprimitivefreeactfirststrthm} \psi\circ j_U \in \textrm{Aut}_{\sT}(J,\omega)(U), 
\end{equation} and it follows from $\G$-equivariance of $\psi$ that the $0$-cochain:
\begin{equation*} c\in \check{C}^0_\U(\B\vert_\Delta,\L),
\end{equation*} that assigns to an open $U\in \U$ the Lagrangian section corresponding to (\ref{autprimitivefreeactfirststrthm}) {via (\ref{torsectautisoinvsymp:torbunversion})} is a primitive for the $1$-cocycle $\tau$ {(here Theorem \ref{isoautsymplagseccor:torbunversion} applies by Lemma \ref{lemma:indtorbunacttoric})}. So, the class {represented by $\tau$} is trivial, which proves freeness. For transitivity, suppose $J_1:(S_1,\omega_1)\to M$ and $J_2:(S_2,\omega_2)\to M$ are {faithful multiplicity-free} $(\G,\Omega)$-spaces with momentum image $\Delta$. By {Theorem \ref{torsortorspthm:torbunversion} (which applies by Lemma \ref{lemma:indtorbunacttoric})} there is, for each $U\in \U$, an isomorphism of Hamiltonian $(\sT,\Omega_{\sT})$-spaces:
\begin{center}
\begin{tikzcd} (J_1^{-1}(U),\omega_1) \arrow[rr,"\psi_U"]\arrow[rd,"J_1"'] & & (J_2^{-1}(U),\omega_2)\arrow[ld, "J_2"] \\
& M & 
\end{tikzcd}
\end{center} 
Furthermore, by Lemma \ref{lifrlagrcocyleprop} we can find for each arrow $V\xleftarrow{\sigma} U$ in $\textsf{Emb}_\U(\B\vert_\Delta)$ a Lagrangian bisection (\ref{lagrbislift}) lifting $\sigma$, chosen so as to satisfy the cocycle condition (\ref{cocyccondlagrbis}). By Proposition \ref{liftinglagrbiswithsymplectoprop}, this choice provides us with symplectomorphisms: 
\begin{equation*} \psi_{i,V\xleftarrow{\sigma} U}:(J_i^{-1}(U),\omega_i)\to (J_i^{-1}(\sigma\cdot U),\omega_i), \quad \psi_{i,V\xleftarrow{\sigma} U}(p)=g(V\xleftarrow{\sigma} U)(J_i(p))\cdot p,
\end{equation*} for each $i\in \{1,2\}$ and each arrow $V\xleftarrow{\sigma} U$. Now, consider the $1$-cochain: 
\begin{equation*} \tau\in \check{C}^1_\U(\B\vert_\Delta,\L)
\end{equation*} that assigns to an arrow $V\xleftarrow{\sigma} U$ the Lagrangian section corresponding {via (\ref{torsectautisoinvsymp:torbunversion})} to the automorphism of the $(\sT,\Omega_{\sT})$-space $J_2:(J_2^{-1}(U),\omega_2)\to M$ given by the composition:
\begin{center}
\begin{tikzcd} (J_2^{-1}(U),\omega_2) \arrow[rr,"\psi_U^{-1}"] & & (J_1^{-1}(U),\omega_1) \arrow[d,"\psi_{1,V\xleftarrow{\sigma} U}"] \\
(J_2^{-1}(\sigma\cdot U),\omega_2) \arrow[u,"\psi_{2,V\xleftarrow{\sigma} U}^{-1}"]   & & (J_1^{-1}(\sigma\cdot U),\omega_1) \arrow[ll,"\psi_V"]
\end{tikzcd}
\end{center}
From the cocycle condition (\ref{cocyccondlagrbis}) it follows that $\tau$ is a $1$-cocycle. Now, consider the unique map $\psi:S_1\to S_{2,\tau}$ defined by the property that $\psi\vert_{J_1^{-1}(U)}=j_{2,U}\circ \psi_U$ for each $U\in \U$, with $j_{2,U}$ as in (\ref{injopenintoquot}). To see that this is well-defined, notice first that (since $\U$ is a basis) this boils down to showing that if $V,U\in \U$ such that $U\subset V$, then $j_{2,V}\circ \psi_V$ restricts to $j_{2,U}\circ \psi_U$ on $J_1^{-1}(U)$, and this in turn readily follows from the fact that $\psi_U$ is $\sT$-equivariant and $g(V\hookleftarrow U)$ takes values in $\sT$, being a lift of the unit bisection of $\B$. Clearly, $\psi$ is a symplectomorphism that intertwines $J_1$ and $J_{2,\tau}$. Furthermore, it is $\G$-equivariant. To see this, let $g\in \G$ with source $x\in M$ and target $y\in M$, and let $p\in S_1$ be such that $J_1(p)=x$. Let $V\xleftarrow{\sigma} U$ be an arrow in $\textsf{Emb}_\U(\B\vert_\Delta)$ such that $x\in U$ and $\sigma(x)=[g]$. Then $J_1(p)=x\in U$ and $J_1(g\cdot p)=y\in V$, so that:
\begin{equation*} g\cdot \psi(p)=g\cdot [\psi_U(p),U]=[g\cdot \tau(V\xleftarrow{\sigma} U)(x)\cdot \psi_U(p),V],
\end{equation*} whereas:
\begin{equation*} \psi(g\cdot p)=[\psi_V(g\cdot p),V].
\end{equation*} To see that these are equal, notice that:
\begin{equation*} g\cdot \tau(V\xleftarrow{\sigma} U)(x)\cdot \psi_U(p)=(g\cdot g(V\xleftarrow{\sigma} U)(x)^{-1})\cdot \psi_V(g(V\xleftarrow{\sigma} U)(x)\cdot p)=\psi_V(g\cdot p),
\end{equation*} where the first step follows by spelling out the definition of $\tau(V\xleftarrow{\sigma} U)$ and the second step follows from $\sT$-equivariance of $\psi_V$ and the observation that: 
\begin{equation*} g\cdot g(V\xleftarrow{\sigma} U)(x)^{-1}\in \sT,
\end{equation*} since both $g$ and $g(V\xleftarrow{\sigma} U)(x)$ project to $[g]\in \B$. So, $\psi$ is an isomorphism of Hamiltonian $(\G,\Omega)$-spaces, leading us to conclude that:
\begin{equation*} [J_1:(S_1,\omega_1)\to M]=\textrm{c}\cdot [J_2:(S_2,\omega_2)\to M],
\end{equation*} where $\textrm{c}\in H^1(\B\vert_\Delta,\L)$ is the class represented by $\tau$. This proves transitivity of the action. 
\end{proof} 
Next, we address the proof of the first structure theorem {in full generality, including the naturality. For this we will use the fact (due to \cite{Xu}) that a symplectic Morita equivalence $((P,\omega_P),\alpha_1,\alpha_2)$ between two symplectic groupoids $(\G_1,\Omega_1)$ and $(\G_2,\Omega_2)$ induces an equivalence:
\begin{equation}\label{eqn:moreq:indequivofcatshamsp} P_*:\textsf{Ham}(\G_1,\Omega_1)\xrightarrow{\sim} \textsf{Ham}(\G_2,\Omega_2)
\end{equation}
between the categories of Hamiltonian spaces of the respective symplectic groupoids, and hence it induces a bijection:
\begin{equation}\label{eqn:moreq:indbijisoclasseshamsp} P_*:\left\{\begin{aligned} &\quad\text{ }\textrm{ Isomorphism classes of } \\ & \text{ }\text{ }\textrm{Hamiltonian $(\G_1,\Omega_1)$-spaces } \end{aligned}\right\}\xrightarrow{\sim} \left\{\begin{aligned} &\quad\text{ }\textrm{ Isomorphism classes of} \\ & \text{ }\text{ }\textrm{Hamiltonian $(\G_2,\Omega_2)$-spaces }  \end{aligned}\right\}
\end{equation}
This associates to a Hamiltonian $(\G_1,\Omega_1)$-space $J:(S,\omega)\to M_1$ the Hamiltonian $(\G_2,\Omega_2)$-space:
\begin{equation*} P_*(J):(P\ast_{\G_1}S,\omega_{PS})\to M_2,
\end{equation*} with $P\ast_{\G_1}S$ the orbit space of the diagonal $\G_1$-action on $P\times_{M_1}S$, $\omega_{PS}$ the symplectic form inherited from $\omega_P$ and $\omega_S$, and  $P_*(J)$ the map induced by $\alpha_2\circ \textrm{pr}_P$. Our notation and conventions for this are as in \cite[Subsection 1.4.2]{Mol1}, where further details can also be found. We leave it for the reader to verify:
 \begin{prop}\label{prop:sympmoreqinv:faithfulcomplzerosp} Given a symplectic Morita equivalence $((P,\omega_P),\alpha_1,\alpha_2)$ between regular and proper symplectic groupoids $(\G_1,\Omega_1)$ and $(\G_2,\Omega_2)$ that relates a Delzant subspace $\underline{\Delta}_1\subset \underline{M}_1$ to $\underline{\Delta}_2\subset \underline{M}_2$,  (\ref{eqn:moreq:indbijisoclasseshamsp}) restricts to a bijection:
 \begin{equation}\label{indbijtorspsympmoreq} P_*:\left\{\begin{aligned} &\quad\quad\quad\textrm{ Isomorphism classes of } \\ & \textrm{faithful multiplicity-free $(\G_1,\Omega_1)$-spaces} \\ & \quad\quad\quad\textrm{with momentum image $\Delta_1$} \end{aligned}\right\}\xrightarrow{\sim} \left\{\begin{aligned} &\quad\quad\quad\textrm{ Isomorphism classes of} \\ & \textrm{faithful multiplicity-free $(\G_2,\Omega_2)$-spaces} \\ & \quad\quad\quad\textrm{with momentum image $\Delta_2$} \end{aligned}\right\}
\end{equation}
 \end{prop}
 }
\begin{proof}[Proof of Theorem \ref{firststrthm}; the torsor structure in full generality] Let $(\G,\Omega)\rightrightarrows M$ be a regular and proper symplectic groupoid. To define the action of {$H^1(\B\vert_\Delta,\L)$}, choose a complete transversal $\Sigma$. By the etale case, we have a free and transitive action of {$H^1(\B\vert_{\Sigma\cap\Delta},\L_{\Sigma\cap\Delta})$} on the set of isomorphism classes of {faithful multiplicity-free} $(\G,\Omega)\vert_\Sigma$-spaces with momentum image $\Sigma\cap \Delta$. {Via the isomorphism between $H^1(\B\vert_{\Sigma\cap\Delta},\L_{\Sigma\cap\Delta})$ and $H^1(\B\vert_\Delta,\L)$ obtained by applying Corollary \ref{cor:morinvshfcohom:iaequiv} to Example \ref{transversalmoreqex}, combined with }the bijection: 
\begin{equation*} \left\{\begin{aligned} &\quad\quad\quad\textrm{Isomorphism classes of } \\ & \textrm{{faithful multiplicity-free} $(\G,\Omega)\vert_\Sigma$-spaces} \\ & \quad\text{ }\textrm{ with momentum image $\Sigma\cap \Delta$} \end{aligned}\right\}\xrightarrow{\sim} \left\{\begin{aligned} &\quad\quad\quad\textrm{Isomorphism classes of} \\ & \textrm{{faithful multiplicity-free} $(\G,\Omega)$-spaces} \\ & \quad\quad\textrm{ with momentum image $\Delta$} \end{aligned}\right\}
\end{equation*} induced (as in {Proposition \ref{prop:sympmoreqinv:faithfulcomplzerosp}}) by the canonical symplectic Morita equivalence between $(\G,\Omega)$ and its restriction to $\Sigma$ {(like in Example \ref{transversalmoreqex})}, we obtain a free and transitive action of {$H^1(\B\vert_\Delta,\L)$} on the right-hand set. In Proposition \ref{naturalityofactprop} we show that this does not depend on the choice of $\Sigma$.  
\end{proof} 
To conclude the proof {of Theorem \ref{firststrthm}}, it remains to address the naturality, which is made precise in the proposition below. Suppose that we are given a symplectic Morita equivalence and Delzant subspaces as in Proposition \ref{prop:sympmoreqinv:faithfulcomplzerosp}. 
{In view of Corollary \ref{cor:morinvshfcohom:iaequiv}} the integral affine Morita equivalence $(\underline{P},\underline{\alpha}_1,\underline{\alpha}_2)$ associated to $((P,\omega_P),\alpha_1,\alpha_2)$ (see Example \ref{sympmoreqiamoreq}) induces an isomorphism:
\begin{equation}\label{indgpisocechlikecohomorshfsympmoreq} \underline{P}_*:{H^\bullet}(\B_1\vert_{\Delta_1},\L_1)\xrightarrow{\sim} {H^\bullet}(\B_2\vert_{\Delta_2},\L_2).
\end{equation}
\begin{prop}\label{naturalityofactprop} The action of ${H^1}(\B\vert_\Delta,\L)$ defined above is natural, in the following sense.
\begin{itemize}\item[a)] It does not depend on the choice of complete transversal $\Sigma$.
\item[b)] For any symplectic Morita equivalence as above, it holds that:
\begin{equation*} P_*({\emph{\textrm{c}}}\cdot[J:(S,\omega)\to M_1])=\underline{P}_*({\emph{\textrm{c}}})\cdot P_*([J:(S,\omega)\to M_1])
\end{equation*} for each {$\emph{\textrm{c}}\in H^1(\B_1\vert_{\Delta_1},\L_1)$} and each faithful multiplicity-free $(\G_1,\Omega_1)$-space $J:(S,\omega)\to M_1$ with momentum image $\Delta_1$.
\end{itemize}
\end{prop}
\begin{proof} First, suppose that $\B$ is etale. In this case, we defined the action for the choice of complete transversal $\Sigma:=M$, and it suffices to prove part $b$ for this action with $\B_1$ and $\B_2$ etale. {Let $\textrm{c}\in H^1(\B_1\vert_{\Delta_1},\L_1)$} and let $J:(S,\omega)\to M_1$ be a faithful multiplicity-free $(\G_1,\Omega_1)$-space with momentum image $\Delta_1$. {Choose a basis $\U_2$ of $\Delta_2$ for which $\L_2:=\L_{\Delta_2}$ is $\U_2$-acyclic, such that each $U\in \U_2$ admits a smooth local section $\underline{\zeta}_U:U_{M_2}\to \underline{P}$ of $\underline{\alpha}_2$, defined on an open $U_{M_2}$ in $M_2$ satisfying $U=U_{M_2}\cap \Delta_2$, with the property that:
\begin{equation*} f_U:=\alpha_1\circ \zeta_U=\underline{\alpha_1}\circ \underline{\zeta}_U:U_{M_2}\to M_1
\end{equation*} is a smooth open embedding. Here $\underline{\zeta}_U$ denotes the induced section $U_{M_2}\to \underline{P}$ of $\underline{\alpha}_2$. Fix such a collection $\{\zeta_U\mid U\in \U_2\}$. Since the Morita equivalence $(\underline{P},\underline{\alpha}_1,\underline{\alpha}_2)$ is integral affine, pull-back along $f_U$ induces an isomorphism of sheaves:
\begin{equation*} (f_U)^*:\L_1\vert_{f_U(U)}\xrightarrow{\sim} \L_2\vert_U
\end{equation*} covering the homeomorphism $f_U:U\to f_U(U)$. In particular, for each $U\in \U_2$ and each $n>0$:
\begin{equation*} H^n(f_U(U),\L_1)\cong H^n(U,\L_2)=0,
\end{equation*} and so we can extend $\{f_U(U)\mid U\in \U_2\}$ to a basis $\U_1$ of $\Delta_1$ for which $\L_1:=\L_{\Delta_1}$ is $\U_1$-acyclic. }
The composition of isomorphisms:
\begin{equation}\label{eqn:isoincechcohom:pfnaturalityfirststrthm}
 \check{H}^\bullet_{\U_1}(\B_1\vert_{\Delta_1},\L_1)\xrightarrow{(\ref{eqn:cechtoderivedmorph:equivshf})} H^\bullet(\B_1\vert_{\Delta_1},\L_1)\xrightarrow{\underline{P}_*} H^\bullet(\B_2\vert_{\Delta_2},\L_2)\xrightarrow{(\ref{eqn:cechtoderivedmorph:equivshf})} \check{H}^\bullet_{\U_2}(\B_2\vert_{\Delta_2},\L_2)
\end{equation} is induced by the map of complexes:
\begin{equation*} \check{C}^\bullet_{\U_1}(\B_1\vert_{\Delta_1},\L_1)\to \check{C}^\bullet_{\U_2}(\B_2\vert_{\Delta_2},\L_2),
\end{equation*} that associates to an $n$-cochain $\tau_1$ the $n$-cochain $\tau_2$ given by:
\begin{equation*} \tau_2(U_0\xleftarrow{\sigma_1} ...\xleftarrow{\sigma_n} U_n)=(f_{U_n})^*\tau_1(f_{U_0}(U_0)\xleftarrow{f(\sigma_1)} ...\xleftarrow{f(\sigma_n)} f_{U_n}(U_n)),
\end{equation*} where we denote by $f(\sigma_j)$ the unique continuous bisection of $\B_1\vert_{\Delta_1}$ on $f_{U_j}(U_j)$ satisfying:
\begin{equation*} f(\sigma_j)(f_{U_j}(x))\cdot \underline{\zeta}_{U_j}(x)=\underline{\zeta}_{U_{j-1}}((t\circ \sigma_j)(x))\cdot \sigma_j(x), 
\end{equation*} for each $x\in U_j$. 
{Consider a $1$-cochain $\tau_1$ representing the class $[\tau_1]\in \check{H}^1_{\U_1}(\B_1\vert_{\Delta_1},\L_1)$ corresponding to $\textrm{c}$ via (\ref{eqn:cechtoderivedmorph:equivshf}). From the above} it follows that the image of $[\tau_1]$ under (\ref{eqn:isoincechcohom:pfnaturalityfirststrthm}) is represented by the unique $1$-cocycle $\tau_2$ with the property that: 
\begin{equation}\label{relassociatedcocycleeq} \tau_2(V\xleftarrow{\sigma} U)(x)=\phi_{{\zeta_U(x)}}\left(\tau_1(f_V(V)\xleftarrow{f(\sigma)} f_U(U))(f_U(x))\right)
\end{equation} for each arrow $V\xleftarrow{\sigma} U$ in $\textsf{Emb}_{\U_2}(\B_2\vert_{\Delta_2})$ and each $x\in U$. Here $\phi_{\zeta_U(x)}:(\sT_1)_{f_U(x)}\xrightarrow{\sim} (\sT_2)_{x}$ denotes the group isomorphism given by (\ref{moreqindisoisotgps}). Consider the map:
\begin{align*} \psi:((P\ast_{\G_1}S)_{\tau_2},(\omega_{PS})_{\tau_2})&\to (P\ast_{\G_1}S_{\tau_1},\omega_{PS_{\tau_1}}),\\
 \left[[p_P,p_S],U\right]&\mapsto \left[{\zeta}_U(\alpha_2(p_P)),[[{\zeta}_U(\alpha_2(p_P)):p_P]\cdot p_S,f_U(U)]\right],
\end{align*} where $[{\zeta}_U(\alpha_2(p_P)):p_P]\in \G_1$ denotes the unique element satisfying: 
\begin{equation*} [{\zeta}_U(\alpha_2(p_P)):p_P]\cdot p_P={\zeta}_U(\alpha_2(p_P)).
\end{equation*} Using (\ref{relassociatedcocycleeq}) a somewhat tedious, but straightforward verification shows that $\psi$ is well-defined, $\G$-equivariant and bijective. Moreover, $\psi$ is readily verified to be smooth and symplectic (and hence immersive), so that for dimensional reasons it must be a symplectomorphism. So, it is an isomorphism of Hamiltonian $(\G_2,\Omega_2)$-spaces. Therefore: 
\begin{align*} P_*({\textrm{c}}\cdot [J:(S,\omega)\to M_1])&=[P_*(J_{\tau_1}):(P\ast_{\G_1}S_{\tau_1},\omega_{PS_{\tau_1}})\to M_2]\\
&=[P_*(J)_{\tau_2}:((P\ast_{\G_1}S)_{\tau_2},(\omega_{PS})_{\tau_2})\to M_2]\\
&=\underline{P}_*({\textrm{c}})\cdot P_*([J:(S,\omega)\to M_1]),
\end{align*} as claimed. This concludes the proof of the etale case.
In the general setting both parts $a$ and $b$ readily follow from the case treated above by using the fact that the bijection (\ref{indbijtorspsympmoreq}) only depends on the isomorphism class of the Hamiltonian bibundle $((P,\omega_P),\alpha_1,\alpha_2)$ and the fact that (\ref{indbijtorspsympmoreq}) and (\ref{indgpisocechlikecohomorshfsympmoreq}) are both functorial with respect to composition of symplectic Morita equivalences. 
\end{proof}

\subsection{Proof of the splitting theorem}\label{pfsplitthmsec}
\subsubsection{Overview of the proof} {In this section we prove Theorem \ref{globalsplitthm}. In Subsection \ref{constrtorspsemidirprodsec} we show that, if $\B$ is etale, then for any Delzant subspace the semi-direct product groupoid appearing in Theorem \ref{globalsplitthm} admits a faithful multiplicity-free space with momentum image the given Delzant subspace. In Subsection \ref{sympgpoidscornerssec} we prove that a symplectic Morita equivalence between the restrictions of two proper and regular symplectic groupoids to Delzant subspaces induces an equivalence between their categories of faithful multiplicity-free spaces with momentum image equal to the respective Delzant subspaces. Together with the result of Subsection \ref{constrtorspsemidirprodsec}, this leads to the implication from $b)$ to $a)$ in Theorem \ref{globalsplitthm}. After this, in Subsection \ref{sec:lagdixmierdouady} we turn to the construction of the Lagrangian Dixmier-Douady class and explain the implication from $a)$ to $c)$. The proof of the implication from $c)$ to $b)$ is the content of Subsection \ref{pfglobalsplitthmsec}. }
\subsubsection{Existence of a faithful multiplicity-free $(\B\Join \sT,\textrm{pr}_{\sT}^*\Omega)$-space}\label{constrtorspsemidirprodsec} The aim of this subsection is to prove:
\begin{prop}\label{exttoractsemidirectprodprop} Let $\B\rightrightarrows (M,\Lambda)$ be an etale integral affine orbifold groupoid and $\underline{\Delta}$ a Delzant subspace. Consider the corresponding invariant subspace $\Delta$ of $(M,\Lambda)$ (which is a Delzant {subspace}) and a faithful toric $(\sT_\Lambda,\Omega_\Lambda)$-space $J_\Delta:(S_\Delta,\omega_\Delta)\to M$ as in Theorem \ref{constrtorictorbunspthm}. This action extends to a faithful multiplicity-free action along $J_\Delta:(S_\Delta,\omega_\Delta)\to M$ of the semi-direct product symplectic groupoid $(\B\Join \sT_\Lambda,\textrm{pr}_{\sT_\Lambda}^*\Omega_\Lambda)$ (where we view $\sT_\Lambda$ as $\B$-space, as in Remark \ref{presymptorbunrem}). 
\end{prop} 
For this we will use the following lemma, the proof of which is straightforward. 
\begin{lemma}\label{semidirecprodactlem} Let $(\sT,\Omega_\sT)\rightrightarrows M$ be a symplectic torus bundle equipped with a symplectic action of an etale Lie groupoid $\B\rightrightarrows M$ by fiberwise group automorphisms, where by the action being symplectic we mean that $(m^\B_\sT)^*\Omega_\sT=\textrm{pr}_\sT^*\Omega_\sT$. Given a smooth map $J:(S,\omega)\to M$ from a symplectic manifold into $M$, there is a bijection between:
\begin{itemize}\item[i)] left Hamiltonian $(\B\Join \sT,\textrm{pr}_\sT^*\Omega_\sT)$-actions along $J:(S,\omega)\to M$,
\item[ii)] pairs consisting of a left Hamiltonian $(\sT,\Omega_{\sT})$-action along $J:(S,\omega)\to M$ and a left symplectic $\B$-action along $J:(S,\omega)\to M$ satisfying:
\begin{equation}\label{compcondsemidirecprodact} \gamma\cdot (t\cdot p)=(t\cdot \gamma^{-1}) \cdot (\gamma\cdot p),
\end{equation} for all $\gamma\in \B$, $t\in \sT$ and $p\in S$ such that $s_{\B}(\gamma)=\pi_\sT(t)=J(p)$,
\end{itemize} via which the two actions are related as:
\begin{equation*} (\gamma,t)\cdot p=\gamma\cdot (t\cdot p),
\end{equation*} for all $\gamma\in \B$, $t\in \sT$ and $p\in S$ such that $s_{\B}(\gamma)=\pi_\sT(t)=J(p)$. The same holds in the setting of symplectic groupoids with corners.
\end{lemma} 
\begin{proof}[Proof of Proposition \ref{exttoractsemidirectprodprop}] First, we define an action of $\B$ along $J_\Delta$, as follows. Given a $\gamma\in \B$ and $p\in S_\Delta$ such that $s_\B(\gamma)=J_\Delta(p)$, choose a smooth bisection $\sigma:U\to \B$ defined on an open $U$ in $M$ around $J_\Delta(p)$ such that $\sigma(J_\Delta(p))=\gamma$. Then $t\circ \sigma:U\cap \Delta\to (\sigma\cdot U)\cap \Delta$ is a diffeomorphism of manifolds with corners such that $(t\circ \sigma)^*\Lambda=\Lambda\vert_{U\cap \Delta}$. So, we can consider the associated symplectomorphism (as in Theorem \ref{constrtorictorbunspthmintrosec}):
\begin{equation}\label{indsymplectoexttoractsemidirectprodprop} (t\circ \sigma)_*:(J_\Delta^{-1}(U),\omega_\Delta)\to (J_\Delta^{-1}(\sigma\cdot U),\omega_\Delta),
\end{equation} and set:
\begin{equation*} \gamma\cdot p:=(t\circ \sigma)_*(p).
\end{equation*} It follows from the local dependence in Theorem \ref{constrtorictorbunspthmintrosec} that this does not depend on the choice of bisection, because the germ of the bisection chosen above is uniquely determined by the fact that it maps $J_\Delta(p)$ to $\gamma$ (since $\B$ is etale). Moreover, natural dependence and the compatibility of $(t\circ \sigma)_*$ with $J_\Delta$ imply that this indeed defines an action. From the facts that a bisection $\sigma$ as above is an open embedding and (\ref{indsymplectoexttoractsemidirectprodprop}) is a symplectomorphism it follows that this action is smooth and symplectic. The compatibilty of (\ref{indsymplectoexttoractsemidirectprodprop}) with the $\sT_\Lambda$-action implies that this $\B$-action and the $\sT_\Lambda$-action along $J_\Delta$ satisfy the compatibility condition (\ref{compcondsemidirecprodact}). So, in view of Lemma \ref{semidirecprodactlem} these actions define a Hamiltonian $(\B\Join \sT_{\Lambda},\Omega_\Lambda)$-action along $J_\Delta$. Because the $(\sT_\Lambda,\Omega_{\Lambda})$-action is {faithful multiplicity-free}, so is the $(\B\Join \sT_{\Lambda},\textrm{pr}_{\sT_\Lambda}^*\Omega_\Lambda)$-action. 
\end{proof}
\subsubsection{Morita equivalences between pre-symplectic groupoids with corners}\label{sympgpoidscornerssec}
Next, we will prove:
\begin{prop}\label{indequivofcatsmoreqsympgroupoidscorners} Let $(\G_1,\Omega_1)\rightrightarrows M_1$ and $(\G_2,\Omega_2)\rightrightarrows M_2$ be regular and proper symplectic groupoids and let $\underline{\Delta}_1$ and $\underline{\Delta}_2$ be Delzant subspaces of $\underline{M}_1$ and $\underline{M}_2$, respectively. A symplectic Morita equivalence $(P,\omega_P,\alpha_1,\alpha_2)$ between the restriction of $(\G_1,\Omega_1)$ to $\Delta_1$ and the restriction of $(\G_2,\Omega_2)$ to $\Delta_2$ induces an equivalence {between the category of faithful multiplicity-free $(\G_1,\Omega_1)$-spaces with momentum map image $\Delta_1$ and that of faithful multiplicity-free $(\G_2,\Omega_2)$-spaces with momentum map image $\Delta_2$.}
\end{prop}
\begin{proof} {When $\Delta_1=M_1$ and $\Delta_2=M_2$, the equivalence is given by the restriction of (\ref{eqn:moreq:indequivofcatshamsp}). The construction of this equivalence} extends to general Delzant subspaces, as follows. To define the functor from left to right, let $J:(S,\omega)\to M$ be a faithful multiplicity-free $(\G_1,\Omega_1)$-space. In view of Corollary \ref{fundproptamesubmcor}$c$, the fiber product $P\times_{\Delta_1}S$ is an embedded submanifold of $P\times S$ \textit{without} corners, since $S$ has no corners. So, since the diagonal $\G_1$-action along $\alpha_1\circ \textrm{pr}_P:P\times_{\Delta_1}S\to M_1$ is smooth, free and proper, the quotient:
\begin{equation*} P\ast_{\G_1}S:=\frac{(P\times_{\Delta_1}S)}{\G_1}
\end{equation*} is naturally a manifold without corners. Consider the left action of $\G_2$ along:
\begin{equation*} P_*(J):P\ast_{\G_1}S\to M_2, \quad [p_P,p_S]\mapsto \alpha_2(p_P),
\end{equation*} given by:
\begin{equation*} g\cdot[p_P,p_S]=[p_P\cdot g^{-1},p_S].
\end{equation*} The symplectic form $(-\omega_P)\oplus \omega_S$ descends to a symplectic form $\omega_{PS}$ on $P\ast_{\G_1}S$ and the $(\G_2,\Omega_2)$-action along $P_*(J)$ is Hamiltonian with respect to this. {It is straightforward to verify that this $(\G_2,\Omega_2)$-space is faithful multiplicity-free. Since this construction is clearly functorial, it yields a functor $P_*$. An entirely analogous construction from right to left gives an inverse functor. }
\end{proof}
With this at hand, we can prove one implication in the splitting theorem.
\begin{proof}[Proof of Theorem \ref{globalsplitthm}: {implication from $b)$ to $a)$}] The restriction of the pre-symplectic groupoid: 
\begin{equation*} (\B\Join \sT_\Lambda, \textrm{pr}_{\sT_\Lambda}^*\Omega_\Lambda)
\end{equation*} to a complete transversal $\Sigma$ for $\G$ is canonically isomorphic to the symplectic groupoid: 
\begin{equation*} (\B_{\vert_\Sigma}\Join \sT_{\Lambda_\Sigma},\textrm{pr}_{\sT_{\Lambda_\Sigma}}^*\Omega_{\Lambda_\Sigma})
\end{equation*} associated to the integral affine etale orbifold groupoid $\B\vert_\Sigma\rightrightarrows (\Sigma,\Lambda_\Sigma)$ (cf. Example \ref{transversalmoreqex}). Hence, by Proposition \ref{exttoractsemidirectprodprop} there exists a faithful multiplicity-free $(\B\ltimes \sT_\Lambda, \textrm{pr}_{\sT_\Lambda}^*\Omega_\Lambda)\vert_\Sigma$-space {with momentum map image $\Delta\cap\Sigma$}. If $(\G,\Omega)\vert_\Delta$ is pre-symplectic Morita equivalent to $(\B\Join \sT_\Lambda, \textrm{pr}_{\sT_\Lambda}^*\Omega_\Lambda)\vert_\Delta$, it follows from Proposition \ref{moreqwithcorniseqrelprop} and Remark \ref{presympmoreqissymprem} that $(\G,\Omega)\vert_\Delta$ is symplectic Morita equivalent to $(\B\Join \sT_\Lambda, \textrm{pr}_{\sT_\Lambda}^*\Omega_\Lambda)\vert_{\Delta\cap\Sigma}$. So, in view of Proposition \ref{indequivofcatsmoreqsympgroupoidscorners}, there exists a faithful multiplicity-free $(\G,\Omega)$-space {with momentum map image $\Delta$} as well.  
\end{proof}
\subsubsection{{The Lagrangian Dixmier-Douady class}}\label{sec:lagdixmierdouady} We will now extend the definition of the Lagrangian Dixmier-Douady class of a regular and proper symplectic groupoid in \cite{CrFeTo} to a version of this class for the restriction of the symplectic groupoid to a Delzant subspace. Let $(\G,\Omega)\rightrightarrows M$ be such a symplectic groupoid and $\underline{\Delta}\subset \underline{M}$ a Delzant subspace. First suppose that $\B:=\G/\sT$ is etale. Fix a basis $\U$ of $\Delta$ for which $\L:=\L_\Delta$ is $\U$-acyclic and consider the wide subcategory $\textsf{E}_\text{lift}$ of $\textsf{Emb}_\U(\B\vert_\Delta)$ consisting of those arrows $V\xleftarrow{\sigma} U$ that admit a Lagrangian bisection $U\to (\G,\Omega)\vert_\Delta$ lifting the underlying bisection $\sigma:U\to \B\vert_\Delta$. 

\begin{prop} This is a category approximating the etale groupoid $\B\vert_\Delta$ (see Subsection \ref{bisectionsembcatsec}). 
\end{prop}
\begin{proof} The first axiom holds because the unit section of a symplectic groupoid is Lagrangian, the second axiom is immediate and the third axiom follows from the fact that for any $\gamma\in \B$ the quotient map $(\G,\Omega)\to \B$ admits a local Lagrangian section defined around $\gamma$, since it is a proper Lagrangian fibration (in the sense of Example \ref{proplagrfibrex}).
\end{proof}
In view of this, $H^2(\B\vert_\Delta,\L)$ is isomorphic to $\check{H}^2(\textsf{E}_\text{lift},\L)$ via (\ref{eqn:cechtoderivedmorph:equivshf}). For each arrow $V\xleftarrow{\sigma} U$ in $\textsf{E}_\text{lift}$, choose a Lagrangian bisection: 
\begin{equation}\label{eqn:lagrlift:choice} g(V\xleftarrow{\sigma}U):U\to (\G,\Omega)\vert_\Delta
\end{equation} lifting $\sigma:U\to \B\vert_\Delta$. Consider the $2$-cochain $c\in \check{C}^2(\textsf{E}_\text{lift},\L)$ that assigns to a pair of composable arrows $U_0\xleftarrow{\sigma_1} U_1\xleftarrow{\sigma_2} U_2$ the unique Lagrangian section $c(U_0\xleftarrow{\sigma_1} U_1\xleftarrow{\sigma_2} U_2)\in \L(U_2)$ such that:
\begin{equation*}g(U_0\xleftarrow{\sigma_1} U_1)g(U_1\xleftarrow{\sigma_2} U_2)=g(U_0\xleftarrow{\sigma_1\sigma_2} U_2)c(U_0\xleftarrow{\sigma_1} U_1\xleftarrow{\sigma_2} U_2).
\end{equation*} As is readily verified, this defines a $2$-cocycle whose cohomology class $[c]\in\check{H}^2(\textsf{E}_\text{lift},\L)$ does not depend on the choice of lifts (\ref{eqn:lagrlift:choice}). Moreover, this cohomology class vanishes if and only if there is a choice of such lifts for which the cocycle condition (\ref{cocyccondlagrbis}) is satisfied. 
\begin{defi}\label{def:lagrdixmierdouadyclass}
Let $(\G,\Omega)\rightrightarrows M$ be a regular and proper symplectic groupoid and let $\underline{\Delta}\subset \underline{M}$ be a Delzant subspace. The \textbf{Lagrangian Dixmier-Douady class}:
\begin{equation}\label{eqn:defilagrdixmdouadycl} \textrm{c}_2(\G,\Omega,\underline{\Delta})\in H^2(\B\vert_\Delta,\L)
\end{equation} is defined as follows. When $\B:=\G/\sT$ is etale, it is the cohomology class corresponding via (\ref{eqn:cechtoderivedmorph:equivshf}) to the class in $\check{H}^2(\textsf{E}_\text{lift},\L)$ represented by the $2$-cocycle defined above. In general, consider a complete transversal $\Sigma$ and define (\ref{eqn:defilagrdixmdouadycl}) to be the class corresponding to $\textrm{c}_2((\G,\Omega)\vert_\Sigma,\underline{\Delta})$ via the isomorphism between $H^2(\B\vert_\Delta,\L)$ and $H^2(\B\vert_{\Delta\cap \Sigma},\L_{\Delta\cap \Sigma})$ induced (as in Corollary \ref{cor:morinvshfcohom:iaequiv}) by the canonical Morita equivalence between $(\G,\Omega)$ and its restriction to $\Sigma$. 
\end{defi} 
This does not depend on the choice of complete transversal.
\begin{proof}[Proof of Theorem \ref{globalsplitthm}: implication from $a)$ to $c)$] After restricting to a complete transversal $\Sigma$, the proof reduces to the case in which $\B:=\G/\sT$ is etale. The etale case follows from Lemma \ref{lifrlagrcocyleprop}. 
\end{proof}
\subsubsection{End of the proof: constructing a principal Hamiltonian bundle out of a Lagrangian cocycle}\label{pfglobalsplitthmsec} 
To prove the forward implication in the splitting theorem, we first show the following.
\begin{prop}\label{fromlagrcocycletofibhamprincbunprop} Let $(\G,\Omega)\rightrightarrows M$ be a regular and proper symplectic groupoid for which the associated orbifold groupoid $\B=\G/\sT$ is etale and let $\underline{\Delta}$ be a Delzant space of $\underline{M}$. Further, let $\U$ be a basis of the corresponding invariant subspace $\Delta$ of $M$ and suppose that for each arrow $V\xleftarrow{\sigma} U$ in $\textsf{Emb}_\U(\B\vert_\Delta)$ there is a Lagrangian bisection {(\ref{eqn:lagrlift:choice})} lifting $\sigma$, such that the collection of these Lagrangian bisections satisfies the cocycle condition (\ref{cocyccondlagrbis}). Then there is a principal Hamiltonian $(\G,\Omega)\vert_\Delta$-bundle `fibered over $\B\vert_\Delta$'. That is, there are: 
\begin{itemize}\item a $(\G\vert_\Delta,\B\vert_\Delta)$-bibundle $(P,\beta_1,\beta_2)$ (of Lie groupoids with corners), 
\item a symplectic form $\omega_P$ on $P$,
\item a smooth map $j:P\to \B\vert_\Delta$,
\end{itemize} that form a map of bi-bundles:
\begin{center}
\begin{tikzpicture} 
\node (H1) at (-1.6,0) {$(\G,\Omega)\vert_\Delta$};
\node (S1) at (-1.6,-1.3) {$\Delta$};
\node (Q) at (1.35,0) {$(P,\omega_P)$};
\node (S2) at (4.5,-1.3) {$\Delta$};
\node (H2) at (4.5,0) {$\B\vert_\Delta$};

\node (G1) at (-1,-3) {$\B\vert_\Delta$};
\node (M1) at (-1,-4.3) {$\Delta$};
\node (P) at (1.35,-3) {$\B\vert_\Delta$};
\node (M2) at (3.7,-4.3) {$\Delta$};
\node (G2) at (3.7,-3) {$\B\vert_\Delta$};

\draw[->, bend right=52](H1) to node[pos=0.45,below] {$\textrm{pr}\quad\quad\quad$} (G1);
\draw[->, bend right=22](S1) to node[pos=0.45,below] {$\textrm{Id}\quad\quad$} (M1);
\draw[->, bend left=52](H2) to node[pos=0.45,below] {$\quad\quad\quad\textrm{Id}$} (G2);
\draw[->, bend left=22](S2) to node[pos=0.45,below] {$\quad\quad \textrm{Id}$} (M2);
\draw[->](Q) to node[pos=0.45,below] {$j\quad\quad$} (P);
 
\draw[->,transform canvas={xshift=-\shift}](H1) to node[midway,left] {}(S1);
\draw[->,transform canvas={xshift=\shift}](H1) to node[midway,right] {}(S1);
\draw[->,transform canvas={xshift=-\shift}](H2) to node[midway,left] {}(S2);
\draw[->,transform canvas={xshift=\shift}](H2) to node[midway,right] {}(S2);
\draw[->](Q) to node[pos=0.25, below] {$\quad\text{ }\beta_1$} (S1);
\draw[->] (0.5,-0.15) arc (315:30:0.25cm);
\draw[<-] (2.2,0.15) arc (145:-145:0.25cm);
\draw[->](Q) to node[pos=0.25, below] {$\beta_2$\text{ }} (S2); 
 
\draw[->,transform canvas={xshift=-\shift}](G1) to node[midway,left] {}(M1);
\draw[->,transform canvas={xshift=\shift}](G1) to node[midway,right] {}(M1);
\draw[->,transform canvas={xshift=-\shift}](G2) to node[midway,left] {}(M2);
\draw[->,transform canvas={xshift=\shift}](G2) to node[midway,right] {}(M2);
\draw[->](P) to node[pos=0.25, below] {$\text{ }\text{ }\text{ }\text{ }t_{\B}$} (M1);
\draw[->] (0.37,-3.15) arc (315:30:0.25cm);
\draw[<-] (2.29,-2.85) arc (145:-145:0.25cm);
\draw[->](P) to node[pos=0.25, below] {$s_{\B}$\text{ }\text{ }} (M2);
\end{tikzpicture}
\end{center}

in which the lower Morita equivalence is the identity equivalence, the upper left bundle is principal and the action is Hamiltonian, whereas the upper right action is free and symplectic (meaning that $(m_P)^*\omega_P=(\textrm{pr}_P)^*\omega_P$, where $m_P,\textrm{pr}_P:P\rtimes \B\vert_\Delta\to P$ denote the action and projection maps). 
\end{prop}
\begin{proof} Consider the topological space:
\begin{equation}\label{disjunspeq} \bigsqcup_{U\in \U} s_\G^{-1}(U)
\end{equation} where $s_\G:\G\to M$ denotes the source map. Given two elements $(g_1,U_1)$ and $(g_2,U_2)$ of (\ref{disjunspeq}), we write $(g_1,U_1)\hookleftarrow (g_2,U_2)$ if $U_2\subset U_1$ and:
\begin{equation*} g(U_1\hookleftarrow U_2)(J(p_2))\cdot g_2=g_1,
\end{equation*} where $U_2 \hookleftarrow U_1$ denotes the arrow in $\textsf{Emb}_\U(\B\vert_\Delta)$ with underlying bisection the unit map. This relation is reflexive and transitive (by the cocycle condition (\ref{cocyccondlagrbis})), but not necessarily symmetric. The equivalence relation generated by this relation is given by: $(g_1,U_1)\sim (g_2,U_2)$ if and only if there is a pair $(g_{12},U_{12})$, with $U_{12}\in \U$ and $s_\G(g_{12})\in U_{12}\subset U_1\cap U_2$, such that $(g_1,U_1)\hookleftarrow (g_{12},U_{12})$ and $(g_{12},U_{12})\hookrightarrow (g_2,U_2)$. Indeed, this relation is clearly reflexive, symmetric and contains the first relation. To prove transitivity, suppose that $(g_1,U_1)\sim (g_2,U_2)$ and $(g_2,U_2)\sim (g_3,U_3)$. Then there are $(g_{12},U_{12})$ and $(g_{23},U_{23})$ with $U_{12},U_{23}\in \U$, $s_\G(g_{12})\in U_{12}\subset U_1\cap U_2$ and $s_\G(g_{23})\in U_{23}\subset U_2\cap U_3$, such that:
\begin{center}
\begin{tikzcd} (g_1,U_1) & & (g_2,U_2) & & (g_3,U_3)\\
& (g_{12},U_{12}) \arrow[ul, hook']\arrow[ur, hook] & & (g_{23},U_{23})\arrow[ul, hook']\arrow[ur, hook] &  
\end{tikzcd}
\end{center} Let $U_{123}\in \U$ be such that $s_\G(g_2)\in U_{123}\subset U_{12}\cap U_{23}$ and consider the element
\begin{equation*} g_{123}:=g(U_2\hookleftarrow U_{123})(s_\G(g_2))^{-1}\cdot g_2\in s_\G^{-1}(U_{123}),
\end{equation*} which is well-defined since the source and target of $g(U_2\hookleftarrow U_{123})(s_\G(g_2))$ coincide, being a lift of the unit map of $\B$. It follows from the cocycle condition (\ref{cocyccondlagrbis}) that $(g_{12},U_{12})\hookleftarrow (g_{123},U_{123})$ and $(g_{123},U_{123})\hookrightarrow (g_{23},U_{23})$, and hence that $(g_1,U_1)\hookleftarrow (g_{123},U_{123})$ and $(g_{123},U_{123})\hookrightarrow (g_{3},U_{3})$. This shows that $(g_1,U_1)\sim (g_3,U_3)$, which proves transitivity. Now, consider the quotient space:
\begin{equation*} P:=\frac{\left(\bigsqcup_{U\in \U} s_\G^{-1}(U)\right)}{\sim}.
\end{equation*} From the explicit description of the equivalence relation it is clear that for each $U\in \U$ the map:
\begin{equation}\label{caninjUeq} s_\G^{-1}(U)\to P, \quad g\mapsto [g,U] 
\end{equation} is an injection. Since for each inclusion $V\hookleftarrow U$ the map:
\begin{equation*} (s_\G^{-1}(U),\Omega)\to (s_\G^{-1}(U),\Omega), \quad h\mapsto g(V\hookleftarrow U)(s_\G(h))\cdot h
\end{equation*} is a symplectomorphism (which follows from the same type of arguments as for Proposition \ref{liftinglagrbiswithsymplectoprop}$a$), there is a unique structure of symplectic manifold with corners on the topological space $P$ with the property that for each $U\in \U$ the injection (\ref{caninjUeq}) is a symplectomorphism onto an open in $P$ (with respect to the symplectic form $\Omega$ on $s_\G^{-1}(U)$). Let $\omega_P$ be the corresponding symplectic form on $P$. Next, notice that the projection $\textrm{pr}:\G \to \B$ and the target and source map of $\G$ induce tame submersions:
\begin{equation*} j:P\to \B\vert_\Delta,\quad\quad \beta_1,\beta_2:P\to \Delta. 
\end{equation*} Furthermore, the canonical left Hamiltonian action of $(\G,\Omega)$ along its target maps induces a left Hamiltonian $(\G,\Omega)\vert_\Delta$-action along $\beta_1:(P,\omega_P)\to \Delta$, given by:
\begin{equation*} g\cdot [h,U]=[g\cdot h,U].
\end{equation*} On the other hand, $\B$ acts along $\beta_2$ from the right, as follows. Given $[h,V]\in P$ and $\gamma\in \B$ such that $\beta_2([h,V])=t_{\B}(\gamma)$, let $V\xleftarrow{\sigma} U$ be an arrow in $\textsf{Emb}_\U(\B\vert_\Delta)$ such that $s(\gamma)\in U$ and $\sigma(s_{\B}(\gamma))=\gamma$, and set:
\begin{equation*} [h,V]\cdot \gamma=[h\cdot g(V\xleftarrow{\sigma} U)(s_{\B}(\gamma)), U].
\end{equation*} It follows from the cocycle condition (\ref{cocyccondlagrbis}) that this depends neither on the choice of arrow $V\xleftarrow{\sigma} U$ nor on that of the representative of $[h,V]\in P$, and that this defines an action. Furthermore, the fact that $g(V\xleftarrow{\sigma} U)$ is Lagrangian implies that this $\B$-action is symplectic. These two actions and the map $j:P\to \B\vert_\Delta$ define a fibered principal Hamiltonian $(\G,\Omega)\vert_\Delta$-bundle over $\B\vert_\Delta$. 
\end{proof} 

\begin{proof}[Proof of Theorem \ref{globalsplitthm}: {implication from $c)$ to $b)$}] Restricting to a complete transversal for $\G$ reduces the proof to the case in which $\B:=\G/\sT$ is etale. In that case, {the vanishing of the Lagrangian Dixmier-Douady class} implies that there is a basis $\U$ of $\Delta$ as in Proposition \ref{fromlagrcocycletofibhamprincbunprop}, and so there is a fibered principal Hamiltonian $(\G,\Omega)\vert_\Delta$-bundle over $\B\vert_\Delta$. From this we can construct the desired Morita equivalence, as follows. Consider the left Hamiltonian $(\sT,\Omega_\sT)\vert_\Delta$-action along $\beta_2:(P,\omega_P)\to \Delta$ given by:
\begin{equation*} t\ast p=(j(p)\cdot t)\cdot p,
\end{equation*} where the second action on the right denotes the $\sT$-action induced by the $\G$-action on $P$, and the left symplectic $\B$-action along $\beta_2:(P,\omega_P)\to \Delta$ given by:
\begin{equation*} \gamma\cdot p=p\cdot \gamma^{-1}.
\end{equation*} These form a pair as in Lemma \ref{semidirecprodactlem}. So, they encode a left Hamiltonian $(\B\Join \sT,\textrm{pr}_{\sT}^*\Omega_{\sT})\vert_\Delta$-action along $\beta_2$. This commutes with the left $\G$-action since both of the above actions do so. Furthermore, the action is free and its orbits coincide with the $\beta_1$-fibers. Passing to a right action along $\beta_2$ via the groupoid inversion of $\B\Join \sT$ and the groupoid automorphism of $\B\Join \sT$ that maps $(\gamma,t)$ to $(\gamma,t^{-1})$ (which are both anti-symplectic with respect to $\textrm{pr}_{\sT}^*\Omega_{\sT}$) we obtain a right Hamiltonian $(\B\Join \sT,\textrm{pr}_{\sT}^*\Omega_{\sT})\vert_\Delta$-action along $\beta_2$, which completes the left principal Hamiltonian $(\G,\Omega)\vert_{\Delta}$-bundle $\beta_2:(P,\omega_P)\to \Delta$ to the desired symplectic Morita equivalence. \end{proof} 
\begin{rem}\label{sympgerberem} This proof shows that, in the language of \cite{CrFeTo}, the existence of a {faithful multiplicity-free} $(\G,\Omega)$-space with momentum image $\Delta$ is also equivalent to the condition that the symplectic gerbe represented by the symplectic central extension:
\begin{equation}\label{centextoverdelzsubps} 1\to (\sT,\Omega_\sT)\vert_\Delta\to (\G,\Omega)\vert_\Delta\to \B\vert_\Delta\to 1
\end{equation} is trivial. This means that (\ref{centextoverdelzsubps}) is Morita equivalent as pre-symplectic central extension to the trivial such extension:
\begin{equation*} 1\to (\sT,\Omega_{\sT})\vert_\Delta\to (\B\Join \sT,\textrm{pr}_\sT^*\Omega_\sT)\vert_\Delta\to \B\vert_\Delta\to 1.
\end{equation*} Here Morita equivalence of pre-symplectic extensions is meant in the sense of \cite{CrFeTo}, extended to the setting with corners using Definition \ref{sympgpwithcorndef}.
\end{rem}
\newpage
\section{The second and third structure theorems}
{In this part we address the second and third structure theorem. These theorems result naturally from a more classical approach to the classification of faithful multiplicity-free $(\G,\Omega)$-spaces than that taken in the previous part, which only involves the cohomology of sheaves on the topological space $\underline{\Delta}\subset \underline{M}$, rather than that of equivariant sheaves. The key ingredients in this approach are a generalization of Theorem \ref{isoautsymplagseccor:torbunversion} for the sheaf of $\G$-equivariant automorphisms (Theorem \ref{isoautsymplagseccor}) and the ext-invariant, both of which are introduced in Section \ref{locclassthmsec}. In that section we also prove the second structure theorem. The third structure theorem is proved in Section \ref{sec:thirdstrthm}.
}
\subsection{The ext-invariant and the second structure theorem}\label{locclassthmsec}
\subsubsection{The sheaves of automorphisms and invariant isotropic sections}\label{sec:shfofautosandinvisotrsections}
{The main point of this subsection is to prove the following generalization of Theorem \ref{isoautsymplagseccor:torbunversion}.}
\begin{thm}\label{isoautsymplagseccor} Let $(\G,\Omega)$ be a regular and proper symplectic groupoid with associated orbifold groupoid $\B=\G/\sT$ and let $J:(S,\omega)\to M$ be a faithful multiplicity-free $(\G,\Omega)$-space with associated Delzant subspace $\underline{\Delta}:=\underline{J}(\underline{S})$. There is an isomorphism of sheaves on $\underline{\Delta}$ \textemdash defined as in (\ref{torsectautiso:torbunversion}) \textemdash between the sheaf $\underline{\textrm{Aut}}_{\G}(J,\omega)$ of automorphisms of the $(\G,\Omega)$-space $J:(S,\omega)\to M$ and the sheaf $\underline{\L}_\Delta$ of $\B$-invariant isotropic sections of $\sT\vert_\Delta$ (as in Definition \ref{shfinvlagrsecdefi} below).
\end{thm}
{The sheaf $\underline{\L}_\Delta$ appearing in this theorem is defined as follows. 
\begin{defi}\label{shfinvlagrsecdefi} Let $\B\rightrightarrows (M,\Lambda)$ be an integral affine orbifold groupoid and let $\underline{\Delta}\subset \underline{M}$ be a Delzant subspace. We denote by:
\begin{itemize}
\item $\mathcal{C}^\infty_\Delta(\sT_\Lambda)$ the sheaf on $\Delta$ of smooth local section of $\sT_\Lambda\vert_\Delta\to \Delta$,
\item $\underline{\mathcal{C}}^\infty_\Delta(\sT_\Lambda)$ the sheaf on $\underline{\Delta}$ that assigns to an open $\underline{U}$ the subgroup of $\mathcal{C}^\infty_\Delta(\sT)(U)$ consisting of $\B$-invariant sections, with $U$ the invariant open in $\Delta$ corresponding to $\underline{U}$,
\item $\underline{\L}:=\underline{\L}_\Delta$ the subsheaf of $\underline{\mathcal{C}}^\infty_\Delta(\sT_\Lambda)$ consisting of those sections that are isotropic (in the same sense as in Example \ref{ex:orbifoldsheafoflagsec}). 
\end{itemize}
Here $\B$-invariance of a section $\tau:U\to \sT_\Lambda$ means that for all $\gamma:x\to y$ in $\B$:
\begin{equation*} \tau(y)\cdot \gamma=\tau(x). 
\end{equation*}
\end{defi}
When the integral affine orbifold is that associated to a regular and proper symplectic groupoid $(\G,\Omega)$, we will usually view the sheaves in Definition \ref{shfinvlagrsecdefi} as consisting of sections ${\tau}:U\to \sT$ via the isomorphism (\ref{exptorbun2}), without further notice. 
\begin{rem}\label{bsheaveslagrsecrem}
As for any topological groupoid, there is a canonical functor:
\begin{equation}\label{shadowfunctorbsh}
(q_\Delta)_*:\textsf{Sh}(\B\vert_\Delta)\to \textsf{Sh}(\underline{\Delta})
\end{equation} from the abelian category of $\B\vert_\Delta$-sheaves into that of sheaves on the topological space $\underline{\Delta}$. This is the push-forward along the canonical map of topological groupoids $q_\Delta:\B\vert_\Delta\to U(\underline{\Delta})$ into the unit groupoid over $\underline{\Delta}$. Explicitly, $(q_\Delta)_*$ associates to a $\B\vert_\Delta$-sheaf $\A$ the sheaf on $\underline{\Delta}$ that assigns to an open $\underline{U}$ the group of $\B$-invariant sections of $\A$ over the invariant open $U$ in $\Delta$ corresponding to $\underline{U}$. Here $\B$-invariance of a section $a\in \A(U)$ means that for all $\gamma:x\to y$ in $\B$:
\begin{equation*} [a]_y\cdot \gamma=[a]_x\in \A_x. 
\end{equation*} The sheaves in Definition \ref{shfinvlagrsecdefi} arise like this from the respective $\B\vert_\Delta$-sheaves in Example \ref{ex:orbifoldsheafoflagsec}. 
\end{rem}
To explain the isomorphism in Theorem \ref{isoautsymplagseccor}, note that there is an injective map of sheaves:
\begin{equation}\label{torsectautiso} \mathcal{C}^\infty_\Delta(\sT)\to \textrm{Aut}_\sT(J),\quad {\tau} \mapsto \psi_{{\tau}},
\end{equation} defined as in (\ref{torsectautiso:torbunversion}), using the induced $\sT$-action. Proposition \ref{lagrseccharprop:torbunversion} generalizes (by the same arguments):
\begin{prop}\label{lagrseccharprop} A section ${\tau}\in \mathcal{C}^\infty_\Delta(\sT)(U)$ is isotropic if and only if $\psi_{{\tau}}$ is a symplectomorphism.
\end{prop}}
It further holds that:
\begin{prop}\label{invsecprop} A section {${\tau}\in \mathcal{C}^\infty_\Delta(\sT)(U)$} is {$\B\vert_U$}-invariant if and only if $\psi_{{\tau}}$ is $\G\vert_U$-equivariant.
\end{prop}
\begin{proof} Notice that, for any arrow $g:x\to y$ in $\G\vert_U$ and any $p\in J^{-1}({x})$:
\begin{equation}\label{invseceq} ({\tau}({y})\cdot {[g]})\cdot p={g^{-1}}\cdot \psi_{{\tau}}({g}\cdot p).
\end{equation} It is clear from this that if ${\tau}$ is {$\B\vert_U$}-invariant, then $\psi_{{\tau}}$ is $\G\vert_U$-equivariant. On the other hand, the converse follows from (\ref{invseceq}) by a density argument. Indeed, suppose that $\psi_{{\tau}}$ is $\G\vert_U$-equivariant. By (\ref{invseceq}): if $\sT$ acts freely at $p$, then ${\tau(y)\cdot[g]}={\tau(x)}$. Now, let an arrow $g$ as above be given and pick $p\in J^{-1}({x})$. Let ${\sigma}$ be a smooth section of ${s:\G\vert_U\to U}$, defined on an open $V$ in $U$ around ${x}$, such that ${\sigma(x)}=g$. Since $\sT$ acts freely on a dense subset, there is a sequence of $p_n\in J^{-1}(V)$ at which $\sT$ acts freely, such that $p_n\to p$. Set $g_n={\sigma}(J(p_n)):x_n\to y_n$. By the discussion above we have ${\tau(y_n)\cdot [g_n]}={\tau}(x_n)$. So, taking $n\to \infty$, we find ${\tau(y)\cdot [g]}={\tau(x)}$.  
\end{proof}
In view of these propositions, the map (\ref{torsectautiso}) induces injective maps of sheaves on $\underline{\Delta}$:
\begin{align}\label{torsectautisoinv} \underline{\mathcal{C}}^\infty_\Delta(\sT)&\to \underline{\textrm{Aut}}_{\G}(J), \\
\label{torsectautisoinvsymp} \underline{\mathcal{\L}}_{\Delta}&\to \underline{\textrm{Aut}}_{\G}(J,\omega).
\end{align} {where $\underline{\textrm{Aut}}_{\G}(J)$ assigns to an open $\underline{U}$ the group of $\G$-equivariant self-diffeomorphisms of $J^{-1}(U)$ that preserve $J$ and $\underline{\textrm{Aut}}_{\G}(J,\omega)$ is the subsheaf consisting those diffeomorphisms that preserve $\omega$. }Theorem \ref{isoautsymplagseccor} states that (\ref{torsectautisoinvsymp}) is an isomorphism of sheaves. To prove this, by Proposition \ref{lagrseccharprop} it is enough to show: 
\begin{thm}\label{torsectautisoprop} The map (\ref{torsectautisoinv}) is an isomorphism of sheaves. 
\end{thm}
In the remainder of this subsection we give a proof of this, by reducing to Theorem \ref{torsectautisoprop:torbunversion} via:
\begin{prop}\label{morinvinvsecprop} An integral affine Morita equivalence between two integral affine orbifold groupoids $\B_1\rightrightarrows (M_1,\Lambda_1)$ and $\B_2\rightrightarrows (M_2,\Lambda_2)$ that relates a Delzant subspace $\underline{\Delta}_1\subset \underline{M}_1$ with $\underline{\Delta}_2\subset \underline{M}_2$ induces an isomorphism of sheaves:
\begin{equation*} \underline{\mathcal{C}}^\infty_{\Delta_1}(\sT_{\Lambda_1})\xrightarrow{\sim} \underline{\mathcal{C}}^\infty_{\Delta_2}(\sT_{\Lambda_2})
\end{equation*} covering the induced homeomorphism between $\underline{\Delta}_1$ and $\underline{\Delta}_2$. This restricts to an isomorphism:
\begin{equation*} \underline{\L}_{\Delta_1}\xrightarrow{\sim} \underline{\L}_{\Delta_2}.
\end{equation*}
\end{prop}
{The proof of this proposition is straightforward (cf. the proof of Corollary \ref{cor:morinvshfcohom:iaequiv}).}
\begin{ex}\label{transversalmoreqex1} In the setting of Example \ref{transversalmoreqex} there is a $\B\vert_\Sigma$-equivariant isomorphism: 
\begin{equation*} (\sT_\Lambda,\Omega_\Lambda)\vert_{\Sigma} \xrightarrow{\sim} (\sT_{\Lambda_\Sigma},\Omega_{\Lambda_\Sigma}), \quad [(x,\alpha)]\mapsto[(x,\alpha\vert_{T_x\Sigma})],
\end{equation*} For that example the isomorphisms of sheaves in Proposition \ref{morinvinvsecprop} are given by restriction of sections.  
\end{ex}
\begin{ex}\label{sympmoreqindiaeqisoofinvlagrsecex} { For an} integral affine Morita equivalence arising from a symplectic Morita equivalence as in Example \ref{sympmoreqiamoreq}, via (\ref{exptorbun2}) we obtain an isomorphism:
\begin{equation*} \underline{\mathcal{C}}_{\Delta_1}^\infty(\sT_1)\xrightarrow{\sim} \underline{\mathcal{C}}_{\Delta_2}^\infty(\sT_2)
\end{equation*} {from Proposition \ref{morinvinvsecprop},} that restricts to an isomorphism:
\begin{equation*}
 \underline{\L}_{\Delta_1}\xrightarrow{\sim} \underline{\L}_{\Delta_2}.
 \end{equation*}
This associates to $\tau_1\in \underline{\mathcal{C}}_{\Delta_1}^\infty(\sT_1)(\underline{U}_1)$ the section $\tau_2\in \underline{\mathcal{C}}_{\Delta_2}^\infty(\sT_2)(\underline{U}_2)$ given by:
\begin{equation*} \tau_2(x_2)=\phi_p(\tau_1(x_1)),
\end{equation*} where $p\in P$ is any choice of element such that $\alpha_2(p)=x_2$, $x_1:=\alpha_1(p)$ and $\phi_p$ is as in (\ref{moreqindisoisotgps}). 
\end{ex}

\begin{proof}[Proof of Theorem \ref{torsectautisoprop}] Let $x_0\in \Delta$ and let $\psi\in \underline{\textrm{Aut}}_{\G}(J)(\underline{U})$ for some $\G$-invariant open $U$ around $x_0$ in $\Delta$. We have to show that there is a $\B$-invariant smooth section ${\tau}$, defined in a $\G$-invariant open neighbourhood of $x_0$ in $\Delta$, such that the germ $\psi_{{\tau}}$ at $x_0$ coincides with that of $\psi$. To find such a section, choose a transversal $\Sigma$ for $\B$ through $x$ such that $\Delta\cap\Sigma\subset U$. As in the proof of Theorem \ref{momimtoricthm} and Proposition \ref{conedescrpsympnormrepprop}, the given $(\G,\Omega)$-space $J$ restricts to a faithful multiplicity-free $(\G,\Omega)\vert_{\Sigma}$-space: 
\begin{equation}\label{restrmommaptransveq} J_\Sigma:(J^{-1}(\Sigma),\omega\vert_{J^{-1}(\Sigma)})\to \Sigma.
\end{equation} After possibly shrinking $U$, we can assume that $U$ is the $\G$-saturation of $\Sigma\cap \Delta$. Combining Example \ref{transversalmoreqex1} and Example \ref{sympmoreqindiaeqisoofinvlagrsecex}, we find that any smooth $\B\vert_{\Sigma}$-invariant section ${{\tau}}_\Sigma:\Delta\cap \Sigma\to \sT$ extends uniquely to a smooth $\B$-invariant section ${\tau}:U\to \sT$, defined by the property that for any arrow $g:x\to y$ in $\G\vert_U$ with ${y}\in\Sigma$: 
\begin{equation*} {\tau(x)}={{\tau}_{\Sigma}(y)\cdot [g]}. 
\end{equation*} Further, notice that if $\psi\vert_{J^{-1}(\Sigma)}=\psi_{{{\tau}}_{\Sigma}}$, then $\psi=\psi_{{\tau}}$. So, to conclude the proof of the theorem it remains to find (after possibly shrinking $\Sigma$) a smooth $\B\vert_{\Sigma}$-invariant section ${{\tau}}_\Sigma:\Delta\cap \Sigma\to \sT$ with the property that $\psi\vert_{J^{-1}(\Sigma)}=\psi_{{{\tau}}_{\Sigma}}$. Since $(\sT,\Omega_\sT)\vert_\Sigma$ is a symplectic torus bundle and (by {Lemma \ref{lemma:indtorbunacttoric}}) its action along (\ref{restrmommaptransveq}) is faithful toric, it follows from {Theorem \ref{torsectautisoprop:torbunversion}} that, after possibly shrinking $\Sigma$, we can find a smooth section ${{\tau}}_{\Sigma}:\Delta\cap \Sigma\to \sT$ with the property that $\psi\vert_{J^{-1}(\Sigma)}=\psi_{{{\tau}}_\Sigma}$. In view of Proposition \ref{invsecprop}, this section must also be $\B\vert_{\Sigma}$-invariant.
\end{proof}

\subsubsection{The point-wise ext-invariant and proof of the second structure theorem}\label{extinvconstsec} The proof of Theorem \ref{torsortorspthm} will be analogous to that for the case of symplectic torus bundles (Theorem \ref{torsortorspthm:torbunversion}). The remaining ingredient for this is the following generalization of Proposition \ref{locclasstoricthm:torbunversion}. 
\begin{prop}\label{locclasstoricthm} Let $(\G,\Omega)$ be a regular and proper symplectic groupoid and suppose that we are given a faithful multiplicity-free $(\G,\Omega)$-actions along $J_1:(S_1,\omega_1)\to M$ and $J_2:(S_2,\omega_2)\to M$. Further, let $x\in J_1(S_1)\cap J_2(S_2)$. Then there is a $\G$-invariant open neighbourhood $U$ of $x\in M$ and a $\G$-equivariant symplectomorphism: 
\begin{center}
\begin{tikzcd} (J_1^{-1}(U),\omega_1) \arrow[rr,"\cong"] \arrow[dr, "J_1"'] & & (J_2^{-1}(U),\omega_2) \arrow[dl,"J_2"] \\
 & M & 
\end{tikzcd}
\end{center} if and only if both of the following hold.
\begin{itemize}\item[i)] The germs at $x$ of $J_1(S_1)$ and $J_2(S_2)$ are equal.
\item[ii)] The ext-classes $e(J_1)_x$ and $e(J_2)_x$ at $x$ are equal. 
\end{itemize}
\end{prop}  
{The ext-classes at $x$ in this proposition are certain elements of $I^1(\G_x,\sT_x)$ (as in Definition \ref{classgpcohomdefi}). To define these,} 
let $(\G,\Omega)$ be a regular and proper symplectic groupoid {and $J:(S,\omega)\to M$ a faithful multiplicity-free $(\G,\Omega)$-space}. For $x\in M$ and $p\in J^{-1}(x)$, consider the canonical extensions:
\begin{align} &1\to \sT_p\to \G_p\to \varGamma_{\G_p}\to 1 \label{isoextp}, \\
 &1\to \sT_x\to \G_x\to \varGamma_{\G_x}\to 1 \label{isoextx}.
\end{align} Any $1$-cocycle ${c}_p:\G_p\to \sT_p$ that restricts to the identity map on $\sT_p$ extends uniquely to a $1$-cocycle ${c}_x:\G_x\to \sT_x$ that restricts to the identity map on $\sT_x$. Indeed, let $\varGamma_{{c}}$ be the subgroup of $\G_p$ corresponding to ${c}_p$ via the bijection in Proposition \ref{splitcharprop}, applied to (\ref{isoextp}). Then $\varGamma_{{c}}\to \varGamma_{\G_p}$ is an isomorphism. So, since $\varGamma_{\G_p}\to \varGamma_{\G_x}$ is an isomorphism (Proposition \ref{conedescrpsympnormrepprop}$b$), $\varGamma_{{c}}$ is also a subgroup of $\G_x$ with the property that $\varGamma_{{c}}\to \varGamma_{\G_x}$ is an isomorphism. Therefore $\varGamma_{{c}}$ corresponds to a $1$-cocycle ${c}_x:\G_x\to \sT_x$ via the bijection in Proposition \ref{splitcharprop}, applied to (\ref{isoextx}). This is the desired cocycle extending ${c}_p$. The association of ${c}_p$ to ${c}_x$ descends to a map:
\begin{equation}\label{extinvcohommap} I^1(\G_p,\sT_p)\to I^1(\G_x,\sT_x),
\end{equation} as is readily verified. Since the symplectic normal representation $(\S\No_p,\omega_p)$ is {maximally} toric (Proposition \ref{conedescrpsympnormrepprop}$a$), we can consider its ext-class (as in Definition \ref{extclasstoricrep}):
\begin{equation}\label{extclasssympnormrep} e(\S\No_p,\omega_p)\in I^1(\G_p,\sT_p). 
\end{equation}

\begin{prop}\label{extinvwelldefprop} For any $p_1,p_2\in J^{-1}(x)$, the ext-classes $e(\S\No_{p_1},\omega_{p_1})$ and $e(\S\No_{p_2},\omega_{p_2})$ are mapped to the same class in $I^1(\G_x,\sT_x)$.
\end{prop}
\begin{defi}\label{extclassdefi} Let $(\G,\Omega)$ be a regular and proper symplectic groupoid and suppose that we are given a faithful multiplicity-free $(\G,\Omega)$-action along $J:(S,\omega)\to M$ with momentum image $\Delta:=J(S)$. The \textbf{ext-class} of the action \textbf{at a point} $x\in \Delta$:
\begin{equation*} e(J)_x\in I^1(\G_x,\sT_x),
\end{equation*} is the image under the map (\ref{extinvcohommap}) of the ext-class (\ref{extclasssympnormrep}) at any point $p\in S$ in the fiber of $J$ over $x$. 
\end{defi}
\begin{proof}[Proof of Proposition \ref{extinvwelldefprop}] Since the fibers of $J$ coincide with the $\sT$-orbits, it holds that $p_2=t\cdot p_1$ for some $t\in \sT_x$. Writing $p:=p_1$, we have $\G_{t\cdot p}=t\G_pt^{-1}$ and the pair:
\begin{equation*} (C_t,\psi):(\S\No_p,\omega_p)\to (\S\No_{t\cdot p},\omega_{t\cdot p}), \quad [v]\mapsto [\d(m_S)_{(t,p)}(0,v)],
\end{equation*} is an equivalence of symplectic representations, where $C_t:\G_p\to \G_{t\cdot p}$ denotes conjugation by $t$. Appealing to Lemma \ref{smpeqtorrepindisoextinvprop}, it follows that: \begin{equation*} e(\S\No_{t\cdot p},\omega_{t\cdot p})=(C_t)_*(e(\S\No_p,\omega_p)).
\end{equation*} As is readily verified, the square:
\begin{center} 
\begin{tikzcd} I^1(\G_p,\sT_p) \arrow[r,"(C_t)_*"]\arrow[d] & I^1(\G_{t\cdot p},\sT_{t\cdot p}) \arrow[d] \\ 
 I^1(\G_x,\sT_x)\arrow[r,"(C_t)_*"] &  I^1(\G_x,\sT_x)
\end{tikzcd} 
\end{center} commutes. By Lemma \ref{conjtoreltrivincohomlem} below the lower arrow is the identity map. So, the respective images of $e(\S\No_p,\omega_p)$ and $e(\S\No_{t\cdot p},\omega_{t\cdot p})$ under the vertical maps are equal. 
\end{proof} 
Here we used the lemma below, which is readily verified.
\begin{lemma}\label{conjtoreltrivincohomlem} Let $H$ be an infinitesimally abelian compact Lie group with identity component $T$. For any $h\in H$, conjugation by $h$ induces the identity map in $H^1(H,T)$.
\end{lemma} 
{Having defined the invariants appearing in Proposition \ref{locclasstoricthm}, we now turn to its proof.}
\begin{proof}[Proof of Proposition \ref{locclasstoricthm}] The implication from left to right is straightforward to verify. For the other implication, let $x\in J_1(S_1)\cap J_2(S_2)$ such that the germs at $x$ of $J_1(S_1)$ and $J_2(S_2)$ coincide and such that $e(J_1)_x=e(J_2)_x$. Since $\underline{J}:\underline{S}\to \underline{M}$ is a topological embedding, every invariant $\G$-invariant open in $S$ is of the form $J^{-1}(U)$ where $U$ is some $\G$-invariant open in $M$. So, in view of Theorem \ref{loceqthmhamact}, it remains to show that there are $p_1\in J_1^{-1}(x)$ and $p_2\in J_2^{-1}(x)$ such that $\G_{p_1}=\G_{p_2}$ and the symplectic representations $(\S\No_{p_1},\omega_{p_1})$ and $(\S\No_{p_2},\omega_{p_2})$ are isomorphic. Fix any $p\in J_1^{-1}(x)$ and $q\in J_2^{-1}(x)$. {The same argument as in the proof of Proposition \ref{locclasstoricthm:torbunversion} shows that $\sT_p=\sT_q$ and $\Delta_{\S\No_p}=\Delta_{\S\No_q}$.} Let $\varGamma_p$ and $\varGamma_q$ be the respective subgroups of $\G_p$ and $\G_q$ corresponding (as in Proposition \ref{splitcharprop}) to a choice of $1$-cocycles representing $e(\S\No_p,\omega_p)$ and $e(\S\No_q,\omega_q)$. By assumption, $e(\S\No_p,\omega_p)$ and $e(\S\No_q,\omega_q)$ are mapped to the same cohomology class in $I^1(\G_x,\sT_x)$. So, there is a $t\in \sT_x$ such that: \begin{equation}\label{eq2-locclasstoricthm} t\varGamma_pt^{-1}=\varGamma_q. 
\end{equation} 
Since any element of $\G_p$ can be written as product of an element of $\varGamma_p$ and one of $\sT_p$, while any element of $\G_q$ can be written as product of an element of $\varGamma_q$ and one of $\sT_q$, it follows that:
\begin{align*} \G_{t\cdot p}&=t\G_pt^{-1}\\
&=(t\varGamma_pt^{-1})\sT_p\\
&=\varGamma_q\sT_q\\
&=\G_q.
\end{align*} Moreover, as in the proof of Proposition \ref{extinvwelldefprop}, we find that $\G_{t\cdot p}=t\G_pt^{-1}$ and the subgroup $t\varGamma_pt^{-1}$ of $\G_{t\cdot p}$ corresponds to the $1$-cocycle representing the cohomology class $e(\S\No_{t\cdot p},\omega_{t\cdot p})$. In light of (\ref{eq2-locclasstoricthm}), this is the same as the $1$-cocycle representing the class $e(\S\No_q,\omega_q)$, hence:
\begin{equation*} e(\S\No_{t\cdot p},\omega_{t\cdot p})=e(\S\No_q,\omega_q).
\end{equation*} In combination with the fact that:
\begin{equation*} \Delta_{\S\No_{t\cdot p}}=\Delta_{\S\No_{p}}=\Delta_{\S\No_{q}}.
\end{equation*} we conclude from Theorem \ref{isoclasstorrepthm} that $(\S\No_{t\cdot p},\omega_{t\cdot p})$ and $(\S\No_q,\omega_q)$ are isomorphic as symplectic representations. Thus, $p_1:=t\cdot p\in J_1^{-1}(x)$ and $p_2:=q\in J_2^{-1}(x)$ are as required. 
\end{proof}
{ As mentioned in the introduction, the ext-invariant appearing in Theorem \ref{torsortorspthm} is a global section of a certain sheaf (\ref{flatsecshintro}) on $\underline{\Delta}$ (the ext-sheaf) with stalk at the orbit $\O_x\in \underline{\Delta}$ through $x\in \Delta$ given by $I^1(\G_x,\sT_x)$. To understand the proof of Theorem \ref{torsortorspthm}, it is not necessary to know the precise definition of the ext-sheaf or the ext-invariant. Rather, the only relevant fact for the proof is that the germ at $x$ of the ext-invariant of $J$ is the ext-class $e(J)_x$, so that two faithful multiplicity-free $(\G,\Omega)$-spaces with momentum image $\Delta$ have the same ext-invariant if and only if they have the same ext-class at every $x\in \Delta$. With this in mind, we give the proof of Theorem \ref{torsortorspthm} now and give the definition of the ext-sheaf and ext-invariant in the next subsection. 
\begin{proof}[Proof of Theorem \ref{torsortorspthm}] The $H^1(\underline{\Delta},\underline{\L})$-action is defined just as in the proof of Theorem \ref{torsortorspthm:torbunversion} and freeness and transitivity follow as in that proof, using Theorem \ref{isoautsymplagseccor} and Proposition \ref{locclasstoricthm}. By naturality of the action with respect to Morita equivalences we mean the following. Let $((P,\omega_P),\alpha_1,\alpha_2)$ be a symplectic Morita equivalence between regular and proper symplectic groupoids $(\G_1,\Omega_1)$ and $(\G_2,\Omega_2)$ with related Delzant subspaces $\underline{\Delta}_1\subset \underline{M}_1$ and $\underline{\Delta}_2\subset \underline{M}_2$. Consider the bijection $P_*$ in Proposition \ref{prop:sympmoreqinv:faithfulcomplzerosp}. By an argument similar to that for Proposition \ref{extinvwelldefprop}, it follows that if $x_2\in \Delta_2$ and $p\in \alpha_2^{-1}(x_2)$, then for any faithful multiplicity-free $(\G_1,\Omega_1)$-space $J$ with momentum map image $\Delta_1$, it holds that:
\begin{equation}\label{eqn:extinvofassociatedspace:moreq} e(P_*(J))_{x_2}=(\phi_p)_*(e(J)_{x_1}),
\end{equation} where $x_1:=\alpha_1(p)$ and $\phi_p:(\G_1)_{x_1}\xrightarrow{\sim}(\G_2)_{x_2}$ is the induced isomorphism of Lie groups (\ref{moreqindisoisotgps}). So, given a global section $e_1$ of the ext-sheaf $\mathcal{I}^1_{(\G_1,\Omega_1,\underline{\Delta}_1)}$, the bijection $P_*$ restricts to a bijection:
\begin{equation*}
P_*:\left\{\begin{aligned} &\quad\quad\quad\textrm{ Isomorphism classes of } \\ & \textrm{faithful multiplicity-free $(\G_1,\Omega_1)$-spaces} \\ & \quad\quad\quad\textrm{with momentum image $\Delta_1$} \\ & \quad\quad\quad\textrm{ and with ext-invariant $e_1$}\end{aligned}\right\}\xrightarrow{\sim} \left\{\begin{aligned} &\quad\quad\quad\textrm{ Isomorphism classes of} \\ & \textrm{faithful multiplicity-free $(\G_2,\Omega_2)$-spaces} \\ & \quad\quad\quad\textrm{with momentum image $\Delta_2$} \\ & \quad\quad\quad\textrm{ and with ext-invariant $e_2$}\end{aligned}\right\}
\end{equation*} for a unique global section $e_2$ of the ext-sheaf $\mathcal{I}^1_{(\G_2,\Omega_2,\underline{\Delta}_2)}$. By the aforementioned naturality, we mean that for any $\underline{c}\in H^1(\underline{\Delta}_1,\underline{\L}_1)$ and any faithful multiplicity-free $(\G_1,\Omega_1)$-space $J:(S,\omega)\to M_1$ with momentum image $\Delta_1$ and ext-invariant $e_1$, it holds that:
\begin{equation}\label{naturalityactionscndstrcthmeq} P_*(\underline{c}\cdot [J:(S,\omega)\to M_1])=\underline{P}_*(\underline{c})\cdot P_*([J:(S,\omega)\to M_1]),
\end{equation} where
\begin{equation*} \underline{P}_*:\check{H}^1(\underline{\Delta}_1,\underline{\L}_1)\xrightarrow{\sim} \check{H}^1(\underline{\Delta}_2,\underline{\L}_2)
\end{equation*} denotes the isomorphism induced by the associated isomorphism of sheaves as in Example \ref{sympmoreqindiaeqisoofinvlagrsecex}. The validity of (\ref{naturalityactionscndstrcthmeq}) is straightforward to verify and left to the reader.
\end{proof}}
\subsubsection{The ext-invariant as global section of the ext-sheaf}\label{flatsecdefsec} Let $(\G,\Omega)\rightrightarrows M$ be a regular and proper symplectic groupoid and $\underline{\Delta}\subset \underline{M}$ be a Delzant subspace $\underline{\Delta}\subset \underline{M}$ with corresponding invariant subspace $\Delta\subset M$. Consider the set-theoretic bundle:
\begin{equation}\label{discrstalkbunextinv} \bigsqcup_{x\in \Delta} I^1(\G_x,\sT_x) \to \Delta.
\end{equation} {Viewed as section of this bundle, the collection of ext-classes of a faithful multiplicity-free $(\G,\Omega)$-space with momentum map image $\Delta$ turns out to have the property that (in a sense to be made precise) its value at any given point $x\in \Delta$ determines its values in some invariant open neighbourhood of $x$. Roughly speaking, the ext-sheaf will be the sheaf of sections of (\ref{discrstalkbunextinv}) with this property. This sheaf and the fact that the ext-classes form a global section of it are both of conceptual value (see Theorem \ref{thm:extsheaftorsorfirstdirectimage}), as well as of use when applying the second and third structure theorem in concrete examples (see Example \ref{example:concretecompsecondstrthm} below)}. \\

{In order to define the ext-sheaf, note first that }the groupoid $\B:=\G/\sT$ acts along (\ref{discrstalkbunextinv}): a given $[g]:x\to y$ in $\B$ acts as the bijection
\begin{equation*} I^1(\G_x,\sT_x)\xrightarrow{\sim} I^1(\G_y,\sT_y)
\end{equation*} induced by the conjugation map $C_g:\G_x\xrightarrow{\sim}\G_y$ (cf. Lemma \ref{conjtoreltrivincohomlem}). 
\begin{defi}\label{setheoreticinvsecshfextinvdefi} We let $\mathcal{I}^1_\textrm{Set}=\mathcal{I}^1_{\textrm{Set},(\G,\Omega,\underline{\Delta})}$ denote the sheaf on $\underline{\Delta}$ consisting of $\B$-invariant set-theoretic local sections of (\ref{discrstalkbunextinv}). That is, $\mathcal{I}^1_\textrm{Set}(\underline{U})$ consists of set-theoretic sections $\sigma$ of (\ref{discrstalkbunextinv}) defined on $U\subset \Delta$, with the property that for every $[g]:x\to y$ in $\B\vert_U$:
\begin{equation*} \sigma(y)=[g]\cdot \sigma(x). 
\end{equation*} 
\end{defi} An argument similar to that for Proposition \ref{extinvwelldefprop} shows:
\begin{prop} Given a faithful multiplicity-free $(\G,\Omega)$-space with momentum map image $\Delta$, the section of (\ref{discrstalkbunextinv}) formed by its ext-classes is $\B$-invariant. 
\end{prop}
The ext-sheaf will be the subsheaf of $\mathcal{I}^1_\textrm{Set}$ consisting of what we call flat sections. To define the notion of flatness, first consider:  

\begin{lemma}\label{shfisodiscinvseclem} A Morita equivalence between regular and proper symplectic groupoids $(\G_1,\Omega_1)\rightrightarrows M_1$ and $(\G_2,\Omega_2)\rightrightarrows M_2$ that relates a Delzant subspace $\underline{\Delta}_1\subset \underline{M}_1$ to a Delzant subspace $\underline{\Delta}_2\subset \underline{M}_2$ induces an isomorphism of sheaves:
\begin{equation}\label{indisoshvsextinvsetsec} \mathcal{I}^1_\textrm{Set,1}:=\mathcal{I}^1_{\textrm{Set},(\G_1,\Omega_1,\underline{\Delta}_1)}\xrightarrow{\sim} \mathcal{I}^1_{\textrm{Set},(\G_2,\Omega_2,\underline{\Delta}_2)}=:\mathcal{I}^1_\textrm{Set,2}
\end{equation}
covering the induced homeomorphism between $\underline{\Delta}_1$ and $\underline{\Delta}_2$. This is functorial with respect to composition of Morita equivalences.
\end{lemma}
\begin{proof} Let a Morita equivalence:
\begin{center}
\begin{tikzpicture} \node (G1) at (-0.5,0) {$(\G_1,\Omega_1)$};
\node (M1) at (-0.5,-1.3) {$M_1$};
\node (S) at (1.4,0) {$(P,\omega_P)$};
\node (M2) at (3.2,-1.3) {$M_2$};
\node (G2) at (3.2,0) {$(\G_2,\Omega_2)$};
 
\draw[->,transform canvas={xshift=-\shift}](G1) to node[midway,left] {}(M1);
\draw[->,transform canvas={xshift=\shift}](G1) to node[midway,right] {}(M1);
\draw[->,transform canvas={xshift=-\shift}](G2) to node[midway,left] {}(M2);
\draw[->,transform canvas={xshift=\shift}](G2) to node[midway,right] {}(M2);
\draw[->](S) to node[pos=0.25, below] {$\text{ }\text{ }\alpha_1$} (M1);
\draw[->] (0.65,-0.15) arc (315:30:0.25cm);
\draw[<-] (2.05,0.15) arc (145:-145:0.25cm);
\draw[->](S) to node[pos=0.25, below] {$\alpha_2$\text{ }} (M2);
\end{tikzpicture}
\end{center}  be given. Given an open $\underline{U}_1$ in $\underline{\Delta}_1$, let $\underline{U}_2$ be its image under the induced homeomorphism between $\underline{\Delta}_1$ and $\underline{\Delta}_2$. This means that the respective invariant opens $U_1$ and $U_2$ in $M_1$ and $M_2$ satisfy: 
\begin{equation*} \alpha_1^{-1}(U_1)=\alpha_2^{-1}(U_2).
\end{equation*} Given $\sigma_1\in \mathcal{I}^1_\textrm{Set,1}(\underline{U}_1)$ and $x_2\in U_2$, choose a $p\in P$ such that $\alpha_2(p)=x_2$ and define:
\begin{equation*} \sigma_2(x_2):=(\phi_p)_*(\sigma_1(x_1))\in I^1((\G_2)_{x_2},\sT_{x_2}),
\end{equation*} where $x_1=\alpha_1(p)$ and $\phi_p:(\G_1)_{x_1}\xrightarrow{\sim} (\G_2)_{x_2}$ is the {induced} isomorphism of Lie groups  {(\ref{moreqindisoisotgps})}. It follows from $\B_1$-invariance of $\sigma_1$ that this does not depend on the choice of $p$. Indeed, any other element of $\widetilde{p}\in\alpha_2^{-1}(x_2)$ is of the form $g\cdot p$ for some $g\in \G_1$, and it holds that $\phi_{g\cdot p}=\phi_p\circ C_{g^{-1}}$, from which it follows that:
\begin{equation*} (\phi_{\widetilde{p}})_*(\sigma_1(\widetilde{x}_1))=(\phi_p)_*(C_{g^{-1}})_*(\sigma_1(\widetilde{x}_1))=(\phi_p)_*(\sigma_1(x_1)),
\end{equation*} where $\widetilde{x}_1=\alpha_1(\widetilde{p})$. A similar argument shows that $\sigma_2$ is $\B_2$-invariant. Hence, this defines a section $\sigma_2\in \mathcal{I}^1_\textrm{Set,2}(\underline{U}_2)$. In this way we obtain a map:
\begin{equation}\label{moremapdiscinvsec} \mathcal{I}^1_\textrm{Set,1}(\underline{U}_1)\to \mathcal{I}^1_\textrm{Set,2}(\underline{U}_2).
\end{equation} This is functorial with respect to composition of Morita equivalences, since the isomorphisms $\phi_p$ are. A completely analogous construction gives an inverse to (\ref{moremapdiscinvsec}). Moreover, (\ref{moremapdiscinvsec}) is compatible with restrictions to smaller opens. So, we have constructed the desired isomorphism of sheaves. 
\end{proof}

Due to the linearization theorem for proper symplectic groupoids \cite{Zu,CrMar1,CrFeTo1,CrFeTo2}, for every {orbit} $\O$ of $\G$ in $M$, there is an invariant open neighbourhood ${W}$ of $\O$ together with an infinitesimally abelian compact Lie group $G$, a $G$-invariant open neighbourhood ${W_{\g^*}}$ of the origin in $\g^*$ and a symplectic Morita equivalence:
\begin{center}
\begin{tikzpicture} \node (G1) at (-1.3,0) {$(\G,\Omega)\vert_W$};
\node (M1) at (-1.3,-1.3) {$W$};
\node (S) at (1.4,0) {$(P,\omega_P)$};
\node (M2) at (4.2,-1.3) {$W_{\g^*}$};
\node (G2) at (4.2,0) {$(G\ltimes \g^*,-\d \lambda_{\textrm{can}})\vert_{W_{\g^*}}$};
 
\draw[->,transform canvas={xshift=-\shift}](G1) to node[midway,left] {}(M1);
\draw[->,transform canvas={xshift=\shift}](G1) to node[midway,right] {}(M1);
\draw[->,transform canvas={xshift=-\shift}](G2) to node[midway,left] {}(M2);
\draw[->,transform canvas={xshift=\shift}](G2) to node[midway,right] {}(M2);
\draw[->](S) to node[pos=0.25, below] {$\text{ }\text{ }\alpha_1$} (M1);
\draw[->] (0.65,-0.15) arc (315:30:0.25cm);
\draw[<-] (2.05,0.15) arc (145:-145:0.25cm);
\draw[->](S) to node[pos=0.25, below] {$\alpha_2$\text{ }} (M2);
\end{tikzpicture}
\end{center} that relates $\O$ to the origin in $\g^*$. Let $\Delta_{\g^*}$ denote the invariant subspace of $\g^*$ related to $W\cap\Delta$ by this Morita equivalence. Then, via the induced isomorphism of sheaves of Lemma \ref{shfisodiscinvseclem}, to each local section of $\mathcal{I}^1_{\textrm{Set},(\G,\Omega,\underline{\Delta})}\vert_{\underline{W}\cap\underline{\Delta}}$ corresponds an invariant local section of the set-theoretic bundle:
\begin{equation}\label{set-theoric-bundle-ext-inv-cent} \bigsqcup_{\alpha\in \Delta_{\g^*}} I^1(G_\alpha,T)\to \Delta_{\g^*}. 
\end{equation} This is because the isotropy group of $G\ltimes \g^*$ at $\alpha\in \g^*$ is the isotropy group $G_\alpha$ of the coadjoint action and (since $G$ is infinitesimally abelian) the identity component of $G_\alpha$ is the identity component $T$ of $G$. Further notice that for each $\alpha\in \g^*$ there is a restriction map $I^1(G,T)\to I^1(G_\alpha,T)$. 
\begin{defi}\label{centeredsecdefi} We will call a local section $\sigma$ of (\ref{set-theoric-bundle-ext-inv-cent}) \textbf{centered} if the origin belongs to its domain and for each $\alpha$ in its domain it holds that:
\begin{equation*} \sigma(\alpha)=\sigma(0)\vert_{G_\alpha}\in I^1(G_\alpha,T).
\end{equation*}
\end{defi}
\begin{defi}\label{extinvsheafdef} Let $(\G,\Omega)\rightrightarrows M$ be a regular and proper symplectic groupoid together with a Delzant subspace $\underline{\Delta}\subset \underline{M}$. Given an open $\underline{V}$ in $\underline{\Delta}$ and an {orbit} $\O\in \underline{V}$, we call a section: 
\begin{equation}\label{sectionextinvsheafdefeq} \sigma\in \mathcal{I}^1_\textrm{Set}(\underline{V})
\end{equation} \textbf{flat at $\O$} if there is an open neighbourhood {$\underline{W}$} of $\O$ in $\underline{M}$ and a symplectic Morita equivalence as above, such that the invariant local section of (\ref{set-theoric-bundle-ext-inv-cent}) corresponding to $\sigma\vert_{\underline{W}\cap \underline{V}}$ (via the induced isomorphism of sheaves of Lemma \ref{shfisodiscinvseclem}) is centered. We call a section (\ref{sectionextinvsheafdefeq}) flat if it is so at all $\O\in \underline{V}$. The flat sections form a subsheaf of $\mathcal{I}^1_\textrm{Set}$ on $\underline{\Delta}$, that we denote as: 
\begin{equation}\label{locconstinvsecextsheaf} \mathcal{I}^1=\mathcal{I}^1_{(\G,\Omega,\underline{\Delta})}
\end{equation} and call the \textbf{ext-sheaf} of $(\G,\Omega)$ on $\underline{\Delta}$.
\end{defi}
\begin{rem}\label{rem:stalkofextsheaf}
Given $x\in \Delta$, evaluation at $x$ defines a bijection between the stalk of (\ref{locconstinvsecextsheaf}) at the {orbit} $\O_x$ and the set $I^1(\G_x,\sT_x)$. 
\end{rem}
Having defined the ext-sheaf, we now turn to the ext-invariant. The theorem below shows not only that the collection of ext-classes of a faithful multiplicity-free $(\G,\Omega)$-space is flat, but that  flatness is also a sufficient condition for there to be local solutions to the problem of existence of a faithful multiplicity-free $(\G,\Omega)$-space with momentum map image $\Delta$ and with collection of ext-classes a prescribed $\B$-invariant section of (\ref{discrstalkbunextinv}). More precisely:
\begin{thm}\label{exlocsolprop} Let $(\G,\Omega)\rightrightarrows M$ be a regular and proper symplectic groupoid with associated orbifold groupoid $\B=\G/\sT$. Further, let $\underline{\Delta}\subset \underline{M}$ be a Delzant subspace and let $\sigma$ be a $\B$-invariant section of the set-theoretic bundle (\ref{discrstalkbunextinv}). Then {$\sigma$ is flat at $\O\in\underline{\Delta}$} if and only if there is an invariant open $U$ in $M$ around {$\O$} and a faithful multiplicity-free $(\G,\Omega)$-space $J:(S,\omega)\to M$ {with $J(S)=U\cap\Delta$ and with $\sigma\vert_{\underline{U}\cap \underline{\Delta}}$ as the collection of its ext-classes}. In particular, the collection of ext-classes of a faithful multiplicity-free $(\G,\Omega)$-space is flat.  
\end{thm}
\begin{defi}\label{extinvdefi} Let $(\G,\Omega)\rightrightarrows M$ be a regular and proper symplectic groupoid. The \textbf{ext-invariant} $e(J)$ of a faithful multiplicity-free $(\G,\Omega)$-space $J:(S,\omega)\to M$ with momentum map image $\Delta$ is the global section of (\ref{locconstinvsecextsheaf}) given by the collection of its ext-classes. 
\end{defi}
\begin{rem}\label{extinvflatindchartrem} From Theorem \ref{exlocsolprop} one can derive that if a section (\ref{sectionextinvsheafdefeq}) is flat at $\O$, then for \textit{every} symplectic Morita equivalence as above the invariant local section of (\ref{set-theoric-bundle-ext-inv-cent}) corresponding to $\sigma\vert_{{\underline{W}}\cap\underline{V}}$ (via the induced isomorphism of sheaves of Lemma \ref{shfisodiscinvseclem}) is centered on some neighbourhood of the origin. 
\end{rem}
A proof of Theorem \ref{exlocsolprop} and Remark \ref{extinvflatindchartrem} will be given in the next subsection. Before that, we illustrate their use in a concrete example. 
\begin{ex}\label{example:concretecompsecondstrthm}
Consider the standard integral affine cylinder $(M,\Lambda):=(\R\times \mathbb{S}^1,\Z\d x\oplus \Z\d\theta)$ with Delzant subspace $\Delta:=[-1,1]\times \mathbb{S}^1$. Let $\Gamma:=\Z_2\times \Z_2$ and consider the $\Gamma$-action given by $\varepsilon\cdot (x,\mu)=(\varepsilon_1x,\mu^{\varepsilon_2})$, for $\varepsilon=(\varepsilon_1,\varepsilon_2)\in \Z_2\times \Z_2$ and $(x,\mu)\in \R\times \mathbb{S}^1$ (where we view $\Z_2$ as $\{\pm 1\}$). Let $(\G,\Omega)$ be the semi-direct product symplectic groupoid $(\B\Join \sT,\textrm{pr}_{\sT}^*\Omega_\sT)$ of the etale integral affine orbifold groupoid $\B:=\Gamma\ltimes M$ over $(M,\Lambda)$ with its canonical action on $\sT:=\sT_\Lambda$. There are eight isomorphism classes of faithful multiplicity-free $(\G,\Omega)$-spaces with momentum map image $\Delta$. These are represented by  $(\G,\Omega)$-spaces with the same underlying symplectic manifold $(S,\omega):=(\mathbb{S}^2\times\T^2,\omega_{\mathbb{S}^2}\oplus\omega_{\T^2})$ (where $\omega_{\mathbb{S}^2}$ and $\omega_{\T^2}$ denote the standard area forms) and momentum map $J:=h\times \textrm{pr}_{\mathbb{S}^1,2}$ (the product of the height function $h$ on $\mathbb{S}^2$ with the second coordinate projection from $\mathbb{T}^2$ to $\mathbb{S}^1$), but with different actions indexed by the eight group $1$-cocycles $\tau:\Gamma\to \L(\Delta)$ determined by the values listed in the table below. Here we canonically identify $\sT$ with $\T^2\times \R\times \mathbb{S}^1$.
\begin{center}
\begin{tabular}{ |c|c|c|c|c|c|c|c|c| } 
 \hline
 $\textrm{pr}_{\mathbb{T}^2}\left(\tau(1,-1)(x,\mu)\right)=$ & (1,1) & (-1,1) & (1,1) & (-1,1) & ($\mu$,$e^{2\pi i x}$) &  (-$\mu$,$e^{2\pi i x}$) &  ($\mu$,$e^{2\pi i x}$) &  (-$\mu$,$e^{2\pi i x}$) \\ 
 \hline
  $\textrm{pr}_{\mathbb{T}^2}\left(\tau(-1,1)(x,\mu)\right)=$ & (1,1) & (1,1) & (1,-1) & (1,-1) & ($\mu$,$e^{2\pi i x}$) &  ($\mu$,$e^{2\pi i x}$) &  ($\mu$,-$e^{2\pi i x}$) &  ($\mu$,-$e^{2\pi i x}$) \\  
 \hline
\end{tabular}
\end{center}
The corresponding actions are given by:
\begin{equation*}
\left(\varepsilon,\lambda_1,\lambda_2,z,\mu_2\right)\cdot_\tau\left(x+iy,z,\mu_1,\mu_2\right)=\left((\lambda_1\,\tau_1(\varepsilon)(z,\mu_2))^{\varepsilon_1}\,(x+i\varepsilon_1y),\varepsilon_1z,(\lambda_2\,\tau_2(\varepsilon)(z,\mu_2)\,\mu_1)^{\varepsilon_2},\mu_2^{\varepsilon_2}\right),
\end{equation*} for $\varepsilon=(\varepsilon_1,\varepsilon_2)\in \Z_2\times \Z_2=\Gamma$, $\lambda_1,\lambda_2,\mu_1,\mu_2\in \mathbb{S}^1\subset \C$, $(x+iy,z)\in \mathbb{S}^2\subset \C\times \R$ and $\theta_1,\theta_2\in \R$, where we canonically identify the space of arrows of $\B\Join \sT$ with $\Gamma\times \T^2\times \R\times \mathbb{S}^1$.\\

This claim can be verified using Theorem \ref{torsortorspthm}, Theorem \ref{exlocsolprop} and Remark \ref{extinvflatindchartrem}, as follows. Consider the section of $\mathcal{I}^1$, with germ at the orbit $\O_x\in \underline{\Delta}$ through $(x,\mu)\in \Delta$ represented by the $1$-cocycle: 
\begin{equation*} \textrm{pr}_{\sT_{(x,\mu)}}:\G_{(x,\mu)}=\Gamma_{(x,\mu)}\ltimes \sT_{(x,\mu)}\to \sT_{(x,\mu)}.
\end{equation*} Using this and Remark \ref{degonegpcohomtorsorrem} to identify the stalk of $\mathcal{I}^1$ at an orbit through $(x,\mu)$ with the degree one group cohomology of the $\Gamma_{(x,\mu)}$-module $\sT_{(x,\mu)}$, one can compute the stalks of $\mathcal{I}^1$ to be as in the following figure. \\

\begin{figure}[H]
  \centering
  \includegraphics[scale=0.2]{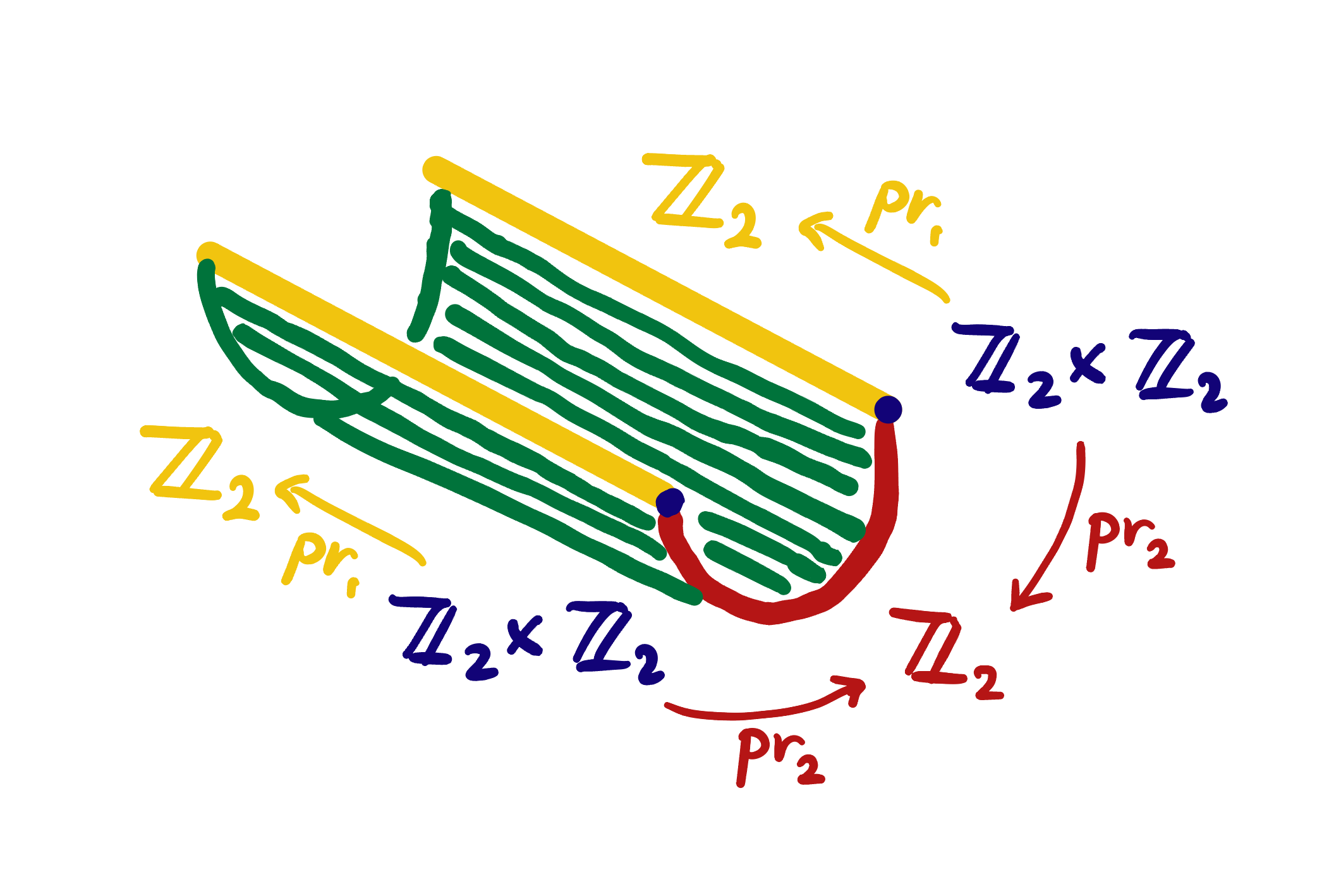}
    \caption{\footnotesize{A fundamental domain for the $\Gamma$-action on $\Delta$, partitioned by collecting points with the same isotropy group. For connected components of members with non-trivial isotropy groups, the corresponding stalks of $\mathcal{I}^1$ are depicted.}}
    \label{figure2}
  \end{figure}

Using Remark \ref{extinvflatindchartrem} one readily sees that any global section of $\mathcal{I}^1$ is determined by its germs at the two $\Gamma$-fixed points $(0,1)$ and $(0,-1)$ via the maps in this figure, and moreover, the second $\Z_2$-components of these two germs must coincide. Therefore, there are at most eight global sections of $\mathcal{I}^1$. Each of these possibilities is realized by the ext-invariants of the $(\G,\Omega)$-spaces listed above. So, to verify our claim, it remains to show that $H^1(\underline{\Delta},\underline{\L})=0$. This will be done in Example \ref{example:concretecompsecondstrthm2}, in the next section. 
\end{ex}

\subsubsection{Proof of Theorem \ref{exlocsolprop} and Remark \ref{extinvflatindchartrem}}\label{sec:pfextshfthmandrem}
In the proof we will use the following construction, which is a particular case of a well-known construction used in the Marle-Guillemin-Sternberg normal form for Hamiltonian Lie group actions \cite{GS4,Ma1}.  Consider the data of:
\begin{itemize} 
\item an infinitesimally abelian compact Lie group $G$,
\item a closed subgroup $H$ of $G$ for which the inclusion induces an isomorphism: 
\begin{equation*} {\varGamma}_H\xrightarrow{\sim}{\varGamma_G}
\end{equation*} between the groups of connected components of $H$ and $G$,
\item a maximally toric $H$-representation $(V,\omega_V)$.
\end{itemize} Associated to this is a Hamiltonian $G$-space: 
\begin{equation}\label{eqn:hamGsplocmod}
 J_\textrm{mod}:(S_\textrm{mod},\omega_\textrm{mod})\to \g^*,
\end{equation} where:
\begin{itemize}\item $(S_\textrm{mod},\omega_\textrm{mod})$ is the symplectic manifold obtained from symplectic reduction (at value $0$) of the (right) Hamiltonian $H$-space:
\begin{equation*} (G\times \g^*,-\d\lambda_\textrm{can})\times (V,\omega_V)\to \h^*, \quad (g,\alpha,v)\mapsto \alpha\vert_\h-J_V(v),
\end{equation*} with $J_V$ as in (\ref{quadsympmommap}) and with the $H$-action given by $(g,\alpha,v)\cdot h=(gh,h^{-1}\cdot\alpha,h^{-1}\cdot v)$,
\item the $G$-action on $S_\textrm{mod}$ is that inherited from the $G$-action on $G\times \g^*\times V$ by left translation on the $G$-component,
\item the momentum map $J_\textrm{mod}$ is the map induced by the projection map $\textrm{pr}_{\g^*}:G\times \g^*\times V\to \g^*$. 
\end{itemize} It follows from a straighforward verification that:
\begin{prop}\label{torlocmoddataprop}
Given the data above, consider the regular and proper symplectic groupoid $(\G,\Omega):=(G\ltimes\g^*,-\d\lambda_\textrm{can})$ (the cotangent groupoid, as in Example \ref{almabcompLiegpexintro0}). The Hamiltonian $(\G,\Omega)$-space corresponding to the Hamiltonian $G$-space (\ref{eqn:hamGsplocmod}) is a faithful multiplicity-free $(\G,\Omega)$-space. 
\end{prop}
Up to Morita equivalence, this construction provides a local model for faithful multiplicity-free actions of general regular proper symplectic groupoids. In view of this, it will be useful to first single out the following case of the {backward} implication in Theorem \ref{exlocsolprop}.   
\begin{prop}\label{torlocmodextinvconstprop} {Suppose that we are given $G$, $H$ and $(V,\omega_V)$ as above}, so that as for (\ref{extinvcohommap}) there is an induced map:
\begin{equation}\label{extinvcohommap2} I^1(H,T_H)\to I^1(G,T_G),
\end{equation} where $T_H$ and $T_G$ denote the respective identity components of $H$ and $G$. {Consider} the associated faithful multiplicity-free $(G\ltimes \g^*,-\d\lambda_\textrm{can})$-space ${J_\textrm{mod}}$ (as in Proposition \ref{torlocmoddataprop}). The following hold.   
\begin{itemize}\item[a)] The image of ${J_\textrm{mod}}$ is $\pi_{\h^*}^{-1}(\Delta_V)$, where $\Delta_V$ is the image of the momentum map (\ref{quadsympmommap}) and $\pi_{\h^*}:\g^*\to \h^*$ is the canonical projection.
\item[b)] The map (\ref{extinvcohommap2}) sends $e(V,\omega_V)\in I^1(H,T_H)$ to the ext-invariant $e({J_\textrm{mod}})_0\in I^1(G,T_G)$ of ${J_\textrm{mod}}$ at the origin in $\g^*$. 
\item[c)] The {collection of ext-classes} of ${J_\textrm{mod}}$ is centered (in the sense of Definition \ref{centeredsecdefi}).
\end{itemize}
\end{prop}
\begin{proof} Part $a$ readily follows from surjectivity of ${\varGamma}_H\to {\varGamma}_G$. For the remainder, let: \begin{equation*} q:(G\times \g^*)\times_{\h^*} V\to {S_\textrm{mod}}
\end{equation*} denote the quotient map. Part $b$ follows from the observation that the symplectic normal representation at $[1,0,0]\in {J_\textrm{mod}^{-1}(0)}$ is canonically isomorphic to $(V,\omega_V)$, via the $H$-equivariant linear symplectomorphism: 
\begin{equation*} (V,\omega_V)\xrightarrow{\sim}(\S\No_{[1,0,0]},({\omega_{\textrm{mod}}})_{[1,0,0]}), \quad v\mapsto [\d q_{(1,0,0)}(0,0,v)].
\end{equation*} We turn to part $c$. Fix an isomorphism of symplectic $T_H$-representations:
\begin{equation*} \psi:(V,\omega)\xrightarrow{\sim}(\C_{\alpha_1},\omega_\textrm{st})\oplus ...\oplus (\C_{\alpha_n},\omega_\textrm{st}),
\end{equation*} and let $\mathcal{W}=\{\alpha_1,...,\alpha_n\}$ denote the set of weights. First notice that (by part $b$) the class {$e({J_\textrm{mod}})_0$} is represented by the unique $1$-cocycle ${c}_0:G\to T_G$ that restricts to the identity map on $T_G$ and that restricts to ${c}_\psi$ on $H$. Next, let $\alpha\in {J_\textrm{mod}(S_\textrm{mod})}$. Then $\alpha\vert_\h\in \Delta_V$ by part $a$. From the description (\ref{quadmom}) it is clear that: 
\begin{equation*} \alpha\vert_\h=\sum_{\alpha_i\in \mathcal{W}} t_{\alpha_i}\alpha_i
\end{equation*} for unique $t_{\alpha_i}\in \R_{\geq 0}$. Let $\mathcal{W}_\alpha=\{\alpha_i\in \mathcal{W}\mid t_{\alpha_i}=0\}$ and let $p\in J_V^{-1}(\alpha\vert_{\h})$ be the element with component $\psi(p)_{\alpha_i}\in \C_{\alpha_i}$ equal to $\sqrt{t_{\alpha_i}}$ for each $\alpha_i\in \mathcal{W}$. Consider $[1,\alpha,p]\in J_\textrm{mod}^{-1}(\alpha)$. As is readily verified, the map:
\begin{equation*} (\S\No_p,\omega_p)\to (\S\No_{[1,\alpha,p]},({\omega_\textrm{mod}})_{[1,\alpha,p]}),\quad [v]\mapsto [\d q_{(1,\alpha,p)}(0,0,v)],
\end{equation*} is a symplectic linear isomorphism. Furthermore, it is equivariant with respect to the action of $H_{[1,\alpha,p]}=H_\alpha\cap H_p$. As the identity component of $H_\alpha\cap H_p$ is the identity component $T_{H_p}$ of $H_p$, it follows that: 
\begin{equation*} e\left(\S\No_{[1,\alpha,p]},({\omega_{\textrm{mod}}})_{[1,\alpha,p]}\right)=[{c}_{\psi_p}\vert_{H_\alpha\cap H_p}]\in I^1(H_\alpha\cap H_p,T_{H_p}).
\end{equation*} Therefore, $e({J_\textrm{mod}})_\alpha$ is represented by the unique $1$-cocycle ${c}_\alpha:G_\alpha\to T_G$ that restricts to the identity map on $T_G$ and to ${c}_{\psi_p}\vert_{H_\alpha\cap H_p}$ on $H_\alpha\cap H_p$. To complete the proof, we will now show that:
\begin{equation}\label{eq0torrepextinvconstrprop} {c}_0\vert_{G_\alpha}={c}_\alpha.
\end{equation} Using (\ref{quadmom}) one computes that:
\begin{equation*} \psi(\ker(\d J_V)_p)=\left(\bigoplus_{\alpha_i\in \mathcal{W}_\alpha} \C_{\alpha_i}\right)\oplus \left(\bigoplus_{\alpha_i\in \mathcal{W}-\mathcal{W_\alpha}} \sqrt{-1}\cdot \R_{\alpha_i}\right).
\end{equation*} On the other hand, it is clear that:
\begin{equation*} \psi(T_p\O)=\bigoplus_{\alpha_i\in \mathcal{W}-\mathcal{W_\alpha}} \sqrt{-1}\cdot \R_{\alpha_i}.
\end{equation*} It follows from this that $\psi$ induces a $(T_H)_p$-equivariant symplectic linear isomorphism:
\begin{equation*} \psi_p:(\S\No_p,\omega_p)\xrightarrow{\sim} \bigoplus_{\alpha_i\in \mathcal{W}_\alpha} \C_{\alpha_i},\quad \psi_p([v])_{\alpha_i}=\psi(v)_{\alpha_i}.
\end{equation*} Hence, the weights of the symplectic normal representation at $p$ are given by $\alpha_i\vert_{\t_p}\in \t_p^*$ for $\alpha_i\in \mathcal{W}_\alpha$ and from (\ref{defprop1coc}) it follows that the $1$-cocycle ${c}_{\psi_p}:H_p\to T_{H_p}$ representing $e(\S\No_p,\omega_p)\in I^1(H_p,T_{H_p})$ has the property that for each $h\in H_p$, $v\in \psi^{-1}(\bigoplus_{\alpha_i\in \mathcal{W}_\alpha} \C_{\alpha_i})$ and each $\alpha_i\in \mathcal{W}_\alpha$:
\begin{equation*} \chi_{\alpha_i}({c}_{\psi_p}(h))\cdot \psi(v)_{\alpha_i}=\psi(h\cdot v)_{[h]\cdot \alpha_i}. 
\end{equation*} The same property is satisfied by the $1$-cocycle ${c}_\psi:H\to T_H$ representing $e(V,\omega)\in I^1(H,T_H)$. Therefore, we conclude that: 
\begin{equation}\label{eq1torrepextinvconstrprop} \chi_{\alpha_i}\circ{c}_{\psi_p}=\chi_{\alpha_i}\circ {c}_{\psi}\vert_{H_p}, \quad \forall\text{ } \alpha_i\in \mathcal{W}_{\alpha}.
\end{equation} On the other hand, it holds that:
\begin{equation}\label{eq2torrepextinvconstrprop} \chi_{\alpha_i}\circ{c}_{\psi_p}=1=\chi_{\alpha_i}\circ {c}_{\psi}\vert_{H_p},\quad \forall\text{ } \alpha_i\in \mathcal{W}-\mathcal{W}_\alpha.
\end{equation} Indeed, the first equality in (\ref{eq2torrepextinvconstrprop}) follows from the observation that, because ${c}_{\psi_p}$ takes values in $(T_H)_p$, for all $h\in H_p$ and $\alpha_i\in \mathcal{W}$ we have:
\begin{equation*} \chi_{\alpha_i}({c}_{\psi_p}(h))\cdot \psi(p)_{\alpha_i}=({c}_{\psi_p}(h)\cdot \psi(p))_{\alpha_i}=\psi(p)_{\alpha_i}.
\end{equation*} For the latter equality, notice first that, since each $h\in H_p$ fixes $\alpha$ (as $J_V$ is $H$-equivariant), it holds that $t_{[h]\cdot \alpha_i}=t_{\alpha_i}$, and so $\psi(p)_{[h]\cdot \alpha_i}=\psi(p)_{\alpha_i}$ for each $\alpha_i\in \mathcal{W}$. Using this and (\ref{defprop1coc}), we find that for such $h$ and $\alpha_i$:
\begin{equation*} \chi_{\alpha_i}(c_\psi(h))\cdot \psi(p)_{\alpha_i}=\psi(h\cdot p)_{[h]\cdot \alpha_i}=\psi(p)_{[h]\cdot \alpha_i}=\psi(p)_{\alpha_i},
\end{equation*} from which we conclude that the second equality in (\ref{eq2torrepextinvconstrprop}) holds. Together, (\ref{eq1torrepextinvconstrprop}) and (\ref{eq2torrepextinvconstrprop}) imply that: 
\begin{equation}\label{eq3torrepextinvconstrprop} {c}_{\psi_p}={c}_{\psi}\vert_{H_p}. \end{equation} We conclude from this that (\ref{eq0torrepextinvconstrprop}) indeed holds, since ${c}_0\vert_{G_\alpha}:G_\alpha\to T_G$ is a $1$-cocycle that restricts to the identity on $T_G$, and that restricts to ${c}_{\psi_p}\vert_{H_\alpha\cap H_p}$ on $H_\alpha\cap H_p$ by (\ref{eq3torrepextinvconstrprop}). 
\end{proof}
We are now ready to complete the proof.
\begin{proof}[End of the proof of Theorem \ref{exlocsolprop} and Remark \ref{extinvflatindchartrem}] {First, in the setting of Theorem \ref{exlocsolprop}, suppose that there is an invariant open $U$ in $M$ around $\O\in \underline{\Delta}$ and a faithful multiplicity-free $(\G,\Omega)$-space $J:(S,\omega)\to M$ such that $J(S)=U\cap\Delta$ and the collection of its ext-invariants is $\sigma\vert_{\underline{U}\cap \underline{\Delta}}$. We will not only show that $\sigma$ is flat at $\O$, but that for \textit{every} symplectic Morita equivalence as after Lemma \ref{shfisodiscinvseclem}, with $W\subset U$, the invariant local section of (\ref{set-theoric-bundle-ext-inv-cent}) corresponding to $\sigma\vert_{\underline{W}\cap\underline{\Delta}}$ via (\ref{indisoshvsextinvsetsec}) is centered on some neighbourhood of the origin. Suppose that we are given any such symplectic Morita equivalence:}
\begin{center}
\begin{tikzpicture} \node (G1) at (-1.3,0) {$(\G,\Omega)\vert_W$};
\node (M1) at (-1.3,-1.3) {$W$};
\node (S) at (1.4,0) {$(P,\omega_P)$};
\node (M2) at (4.2,-1.3) {$W_{\g^*}$};
\node (G2) at (4.2,0) {$(G\ltimes \g^*,-\d \lambda_{\textrm{can}})\vert_{W_{\g^*}}$};
 
\draw[->,transform canvas={xshift=-\shift}](G1) to node[midway,left] {}(M1);
\draw[->,transform canvas={xshift=\shift}](G1) to node[midway,right] {}(M1);
\draw[->,transform canvas={xshift=-\shift}](G2) to node[midway,left] {}(M2);
\draw[->,transform canvas={xshift=\shift}](G2) to node[midway,right] {}(M2);
\draw[->](S) to node[pos=0.25, below] {$\text{ }\text{ }\alpha_1$} (M1);
\draw[->] (0.65,-0.15) arc (315:30:0.25cm);
\draw[<-] (2.05,0.15) arc (145:-145:0.25cm);
\draw[->](S) to node[pos=0.25, below] {$\alpha_2$\text{ }} (M2);
\end{tikzpicture}
\end{center} {After choosing a $p\in \alpha_1^{-1}(\O)$ and composing the Morita equivalence with the automorphism of $(G\ltimes \g^*,-\d \lambda_{\textrm{can}})$ associated to the isomorphism of isotropy groups $\phi_{p}:G\to G$ induced, as in (\ref{moreqindisoisotgps}), by the Morita equivalence, we can assume without loss of generality that there is a $p\in \alpha_1^{-1}(\O)$ such that $G$ is the isotropy group of $\G$ at $x:=\alpha_1(p)$ and such that $\phi_p$ is the identity map of $G$. Then the set-germ at the origin of the invariant subset of $W_{\g^*}$ corresponding to $W\cap \Delta$ is that of the $G$-invariant smooth polyhedral cone $C_{\g^*}:=((\rho_\Omega)_x^*)^{-1}(C_x(\Delta))$. This follows by combining Lemma \ref{sympmoreqnormreplem}, Proposition \ref{iamoreqdelzsubconeprop} and the observation that $(\rho_{-\d\lambda_\textrm{can}})_0^*:\g^*\to T_0\g^*$ (which is the canonical isomorphism of the vector space $\g^*$ with its tangent space at $0$) represents the map germ $\log_0$. Note here that the polyhedral cone $C_{\g^*}$ is indeed $G$-invariant, as follows from Corollary \ref{coneatxisinvrem} and $G$-equivariance of $(\rho_\Omega)_x^*:\g^*\to \No_x{\O}$ with respect to the coadjoint action of $G$ on $\g^*$ (which follows from Lemma \ref{sympmoreqnormreplem}, applied to the identity equivalence). Let $H$ and $(V,\omega_V)$ be the isotropy group and the symplectic normal representation of the $(\G,\Omega)$-space $J$ at some point in $J^{-1}(x)$, and consider the faithful multiplicity-free $(G\ltimes \g^*,-\d \lambda_{\textrm{can}})$-space $J_\textrm{mod}$ associated to this data, as in Proposition \ref{torlocmoddataprop} (which applies because of Proposition \ref{conedescrpsympnormrepprop}). In view of the above, (\ref{eqn:extinvofassociatedspace:moreq}), Proposition \ref{conedescrpsympnormrepprop}$c$ and Proposition \ref{torlocmodextinvconstprop}, the $(G\ltimes \g^*,-\d \lambda_{\textrm{can}})$-spaces $J_\textrm{mod}$ and $P_*(J)$ satisfy the criteria in Proposition \ref{locclasstoricthm} at the origin in $\g^*$. Hence, there is an invariant open neighbourhood of the origin in $\g^*$ such that the restrictions of the $(G\ltimes \g^*,-\d \lambda_{\textrm{can}})$-spaces $J_\textrm{mod}$ and $P_*(J)$ to this open neighbourhood are isomorphic. Since the ext-invariant of $J_\textrm{mod}$ is centered (Proposition \ref{torlocmodextinvconstprop}), the same holds for the restriction of $P_*(J)$ to this open neighbourhood. Because of (\ref{eqn:extinvofassociatedspace:moreq}), the collection of ext-invariants of $P_*(J)$ is the invariant local section of (\ref{set-theoric-bundle-ext-inv-cent}) corresponding to $\sigma\vert_{\underline{W}\cap\underline{\Delta}}$ via (\ref{indisoshvsextinvsetsec}), which is therefore centered on a neighbourhood of the origin, as we wanted to show.} \\

{Having shown this, to complete the proof of both Theorem \ref{exlocsolprop} and Remark \ref{extinvflatindchartrem}, it remains to prove the forward implication in Theorem \ref{exlocsolprop}. }To this end, suppose that $\sigma$ is flat at $\O$. By flatness and the same reasoning as above, there is a symplectic Morita equivalence $((P,\omega_P),\alpha_1,\alpha_2)$ with a $p\in \alpha_1^{-1}(\O)$ as above, with the additional property that the invariant local section corresponding to $\sigma\vert_{\underline{W}\cap\underline{\Delta}}$ via (\ref{indisoshvsextinvsetsec}) is centered. Since the set-germ at the origin of the invariant subset of $W_{\g^*}$ corresponding to $W\cap \Delta$ is that of the $G$-invariant smooth polyhedral cone $C_{\g^*}$ (defined as above), after possibly shrinking $W$ and $W_{\g^*}$ we can further arrange that $W\cap \Delta$ is related to the invariant subset $W_{\g^*}\cap C_{\g^*}$. To construct a faithful multiplicity-free $(\G,\Omega)$-space with momentum image $W\cap \Delta$ and ext-invariant $\sigma\vert_{\underline{W}\cap \underline{\Delta}}$, we will construct:
\begin{itemize}\item a closed subgroup $H$ of $G$ with the property that ${\varGamma}_H\to {\varGamma}_G$ is bijective,
\item a {maximally} toric $H$-representation $(V,\omega_V)$ such that $\pi_{\h^*}^{-1}(\Delta_V)=C_{\g^*}$ and such that $e(V,\omega_V)\in I^1(H,T_H)$ is mapped to $\sigma(x)\in I^1(G,T_G)$ by (\ref{extinvcohommap2}) (for $x:=\alpha_1(p)$ as above).
\end{itemize} {The restriction to $W_{\g^*}$ of the faithful multiplicity-free space $(G\ltimes \g^*,-\d \lambda_{\textrm{can}})$-space $J_\textrm{mod}$ associated to this data (as in Proposition \ref{torlocmoddataprop}) would have momentum map image $W_{\g^*}\cap C_{\g^*}$ and collection of ext-invariants equal to the invariant local section corresponding to $\sigma\vert_{\underline{W}\cap\underline{\Delta}}$ via (\ref{indisoshvsextinvsetsec}) (because both are centered and they have the same value at the origin). Therefore, the $(\G,\Omega)\vert_W$-space associated to this by the symplectic Morita equivalence $((P,\omega_P),\alpha_1,\alpha_2)$ would have momentum image $W\cap \Delta$ and ext-invariant $\sigma\vert_{\underline{W}\cap \underline{\Delta}}$. }\\

So, to complete the proof it remains to construct $H$ and $(V,\omega_V)$ satisfying the requirements listed above. For this, let ${c}:G\to T_G$ be a $1$-cocycle such that $[{c}]=\sigma(x)\in I^1(G,T_G)$. Further, let us (suggestively) denote by $\h^0$ the largest linear subspace of $\g^*$ that is contained in $C_{\g^*}$ and let $\h\subset \g$ be the annihilator of $\h^0$. Since the polyhedral cone $C_{\g^*}$ is smooth in $(\g^*,\Lambda_{T_G}^*)$, the lattice $\h\cap \Lambda_{T_G}$ has full rank in $\h$ and so $T_H:=\exp_{G}(\h)$ is a subtorus of $T_G$ with Lie algebra $\h$. Since $C_{\g^*}$ is invariant under the coadjoint action of $G$, $T_H$ is invariant under conjugation by elements of $G$. From this it readily follows that ${\varGamma}_G\times T_H$ is a subgroup of ${\varGamma}_G\ltimes T_G$. Now, let $H$ be the subgroup of $G$ corresponding to ${\varGamma}_G\times T_H$ under the isomorphism of Lie groups:
\begin{equation*} G\xrightarrow{\sim} {\varGamma}_G\ltimes T_G,\quad g\mapsto ([g],{c}(g)).
\end{equation*}
Then $H$ is a closed Lie subgroup of $G$ with Lie algebra $\h$ and ${\varGamma}_H\to {\varGamma}_G$ is bijective. Further notice that the $1$-cocycle ${c}$ restricts to a $1$-cocycle ${c}\vert_H:H\to T_H$, that restricts to the identity map on $T_H$. Next, we construct the desired representation of $H$. Let us (suggestively) denote by $\Delta_V$ the image of $C_{\g^*}$ under the projection $\pi_{\h^*}:\g^*\to \h^*$. By construction, $\Delta_V$ is a smooth and pointed polyhedral cone in the integral affine vector space $(\h^*,\Lambda_{T_H}^*)$. Furthermore, $\Delta_V$ is ${\varGamma}_H$-invariant, because $\pi_{\h^*}$ is $H$-equivariant and $C_{\g^*}$ is ${\varGamma}_G$-invariant. Hence, by Theorem \ref{isoclasstorrepthm} there is a {maximally} toric $H$-representation $(V,\omega_V)$ with momentum image $\Delta_V$ and $e(V,\omega_V)=[{c}\vert_H]\in I^1(H,T_H)$. As is clear from their construction, $H$ and $(V,\omega_V)$ satisfy the requirements.  
\end{proof}

\subsection{The third structure theorem}\label{sec:thirdstrthm}
\subsubsection{Proof of the third structure theorem} We now turn to the proof of Theorem \ref{thirdstrthm}, starting with:
\begin{prop}\label{compatfirstsecstrthmprop} The injection (\ref{inclstrgpsintro}) is compatible with the actions in the first and second structure theorems.
\end{prop}
\begin{proof} {After restricting to a complete transversal, the proof reduces to the case in which $\B$ is etale. Suppose we are in that case. Let $\underline{\textrm{c}}\in H^1(\underline{\Delta},\underline{\L})$ and let $J:(S,\omega)\to M$ be a faithful multiplicity-free $(\G,\Omega)$-space with momentum image $\Delta$. Via the natural isomorphism from \v{C}ech to sheaf cohomology we can represent $\underline{\textrm{c}}$ by a \v{C}ech $1$-cocycle $\underline{\tau}\in \check{C}^1_{q(\U)}(\underline{\Delta},\underline{\L})$ with respect to an open cover of the form $q(\U):=\{q(U)\mid U\in \U\}$, where $q_\Delta:\Delta\to \underline{\Delta}$ is the orbit projection and $\U$ is some basis of $\Delta$ for which $\L:=\L_\Delta$ is $\U$-acyclic. The composition:
\begin{equation*} \check{H}^1_{q(\U)}(\underline{\Delta},\underline{\L})\to \check{H}^1(\underline{\Delta},\underline{\L})\xrightarrow{\sim} H^1(\underline{\Delta},\underline{\L})\to H^1(\B\vert_\Delta,\L)\to \check{H}^1_\U(\B\vert_\Delta,\L)
\end{equation*} is induced by the map of complexes given in degree $n$ by:
\begin{equation*} (q_\Delta)^*:\check{C}^n_{q(\U)}(\underline{\Delta},\underline{\L})\to \check{C}^n_{\U}(\B\vert_\Delta,\L), \quad (q_\Delta)^*(c)(U_0\xleftarrow{\sigma_1}...\xleftarrow{\sigma_n} U_n):=c(q(U_1),...,q(U_n))\vert_{U_n}\in \L(U_n).
\end{equation*}
So, denoting $\tau:=(q_\Delta)^*(\underline{\tau})\in \check{C}^1_\U(\B\vert_\Delta,\L)$, to prove the proposition we ought to give an isomorphism of faithful multiplicity-free $(\G,\Omega)$-spaces:}
\begin{center}
\begin{tikzcd} (S_{\tau},\omega_{\tau}) \arrow[rr,"\psi"]\arrow[rd,"J_\tau"'] & & (S_{\underline{\tau}},\omega_{\underline{\tau}})\arrow[ld, "J_{\underline{\tau}}"] \\
& M & 
\end{tikzcd}
\end{center}{ where the left-hand $(\G,\Omega)$-space is that defined in the proof of Theorem \ref{firststrthm}, whereas the right-hand one is that defined as in the proof of Theorem  \ref{torsortorspthm:torbunversion}, but with $\L$ replaced by $\underline{\L}$. }Consider the map $\psi:S_{\tau}\to S_{\underline{\tau}}$ given by $\psi([p,U])=[p,\widehat{U}]$, where $\widehat{U}$ denotes the $\G$-saturation of $U$. A straightforward verification shows that this is well-defined, bijective and intertwines the maps $J_\tau$ and $J_{\underline{\tau}}$. Further notice that, for each $U\in \U$, the pre-composition of $\psi$ with the inclusion (\ref{injopenintoquot}) coincides with the composition:
\begin{equation*} (J^{-1}(U),\omega)\hookrightarrow (J^{-1}(\widehat{U}),\omega)\hookrightarrow (S_{\underline{\tau}},\omega_{\underline{\tau}}),
\end{equation*} which is a symplectomorphism onto an open in $(S_{\underline{\tau}},\omega_{\underline{\tau}})$. Therefore, $\psi$ is a symplectomorphism. Finally, the fact that ${\underline{\tau}}(q(U),q(V))$ is $\B$-invariant for all $U,V\in \U$ implies that $\psi$ is $\G$-equivariant. So, $\psi$ is an isomorphism of faithful multiplicity-free $(\G,\Omega)$-spaces.
\end{proof}
\begin{proof}[Proof of Theorem \ref{thirdstrthm}] {The second structure theorem and Proposition \ref{compatfirstsecstrthmprop} imply }that the action in the first structure theorem descends to an action of the quotient group (\ref{quotgpthirdstrthmeqintro}) on the image of (\ref{extinvariantmapthrdstrthm}), uniquely determined by the fact that (\ref{extinvariantmapthrdstrthm}) becomes equivariant with respect to the $\check{H}^1(\B\vert_\Delta,\L)$-action on its image corresponding to the quotient group action. It follows by combining the first and second structure theorem with Proposition \ref{compatfirstsecstrthmprop} that this action is free and transitive. 
\end{proof}
\subsubsection{{On the action in the third structure theorem}}\label{explcdescrpactthirdstrthmsec} 
In this subsection we give a more direct description of the action in the third structure theorem and show that the ext-sheaf comes with a free and transitive action of a more familiar sheaf on $\underline{\Delta}$: the first right-derived functor $\H^1:=R^1(q_\Delta)_*(\L)$ of the push-forward $(q_\Delta)_*$ (as in Remark \ref{bsheaveslagrsecrem}) applied to $\L:=\L_\Delta$.\\ 

Let $(\G,\Omega)\rightrightarrows M$ be a regular and proper symplectic groupoid with associated orbifold groupoid $\B:=\G/\sT$ and let $\underline{\Delta}\subset \underline{M}$ be a Delzant subspace. There is a Grothendieck spectral sequence \cite{Gro} associated to the commutative triangle of functors:
\begin{center}
\begin{tikzcd} \textsf{Sh}(\B\vert_\Delta) \arrow[r,"\Gamma( \cdot )^\B"] \arrow[d,"(q_\Delta)_*"'] & \textsf{Ab} \\
 \textsf{Sh}(\underline{\Delta}) \arrow[ur,"\Gamma( \cdot )"'] & 
\end{tikzcd}
\end{center} which is a version of the usual Leray spectral sequence for the canonical map of topological groupoids $q_\Delta$ from $\B\vert_\Delta$ to the unit groupoid of $\underline{\Delta}$. The associated exact sequence in low degrees (\ref{eqn:fivetermexseqlowdeg}) has as first few terms:
\begin{equation}\label{eqn:threetermexseqlowdeg} 0\to H^1(\underline{\Delta},\underline{\L})\to H^1(\B\vert_\Delta,\L)\to \H^1(\underline{\Delta}).
\end{equation} Concretely, $\mathcal{H}^1$ is the sheafification of the pre-sheaf on $\underline{\Delta}$ that assigns to an open $\underline{U}$ the abelian group $H^1(\B\vert_{U},\L)$ (with the restriction maps induced by pull-back along the inclusion $\B\vert_V\hookrightarrow\B\vert_U$ for opens $\underline{V}\subset \underline{U}$) and the map:
\begin{equation}\label{eqn:secondmapsesorblerspecseq}
 H^1(\B\vert_\Delta,\L)\to \H^1(\underline{\Delta})
\end{equation} in the above sequence takes a global section of this pre-sheaf to the induced global section of its sheafification. To define the $\H^1$-action on the ext-sheaf $\mathcal{I}^1:=\mathcal{I}^1_{((\G,\Omega),\underline{\Delta})}$, let $x\in \Delta$. Given an open $\underline{U}$ in $\underline{\Delta}$ such that $x\in U$, we consider the map:
\begin{equation}\label{stalkmapfirsthigherimage}
\Phi_{x}:=(\textrm{ev}_x)_*\circ (i_x)^*:H^\bullet(\B\vert_U,\L)\xrightarrow{(i_x)^*}H^\bullet(\B_x,\L_x)\xrightarrow{(\textrm{ev}_x)_*} H^\bullet(\B_x,\sT_x),
\end{equation} where $i_x:\B_x\hookrightarrow \B$ is the inclusion of the isotropy group $\B_x$ at $x$ (viewed as groupoid over $\{x\}$) and $\textrm{ev}_x:\L_x\to \sT_x$ is the evaluation map from the stalk of $\L$ at $x$ into the fiber of $\sT$ over $x$. The last two groups denote group cohomology. Combining the map $\Phi_x$ (in degree $1$) and the action in Remark \ref{degonegpcohomtorsorrem}, we obtain an $H^1(\B\vert_U,\L)$-action on the set $I^1(\G_x,\sT_x)$ \textemdash the stalk of $\mathcal{I}^1$ at the {orbit} $\O_x$ of $\G$ through $x$ (see Remark \ref{rem:stalkofextsheaf}). For varying $\underline{U}$ the actions on this set are compatible with restriction of opens and so descend to an action of the stalk of $\H^1$ at $\O_x$. 

\begin{thm}\label{thm:extsheaftorsorfirstdirectimage} The stalk-wise actions given above define an action of the sheaf of abelian groups $\H^1$ on the ext-sheaf $\mathcal{I}^1$ that is free and transitive (meaning that for each open the associated group action is free and transitive). Moreover, via the map (\ref{eqn:secondmapsesorblerspecseq}), the induced $\H^1(\underline{\Delta})$-action on $\mathcal{I}^1(\underline{\Delta})$ restricts to the action in Theorem \ref{thirdstrthm} on the image of (\ref{extinvariantmapthrdstrthm}).
\end{thm}
To prove this we will use the lemma below. 
\begin{lemma}\label{linthmintaffmoreqprop} Let $G$ be an infinitesimally abelian compact Lie group and consider the regular and proper symplectic groupoid $(\G,\Omega):=(G\ltimes \g^*,-\d \lambda_{\textrm{can}})$ over $M:=\g^*$, with associated orbifold groupoid $\B:=\G/\sT=\Gamma\ltimes \g^*$, where $\Gamma:=G/T$ with $T$ the identity component of $G$. Suppose that $\underline{\Delta}\subset \underline{M}=\g^*/G$ is a Delzant subspace with the property that the corresponding invariant subspace $\Delta$ in $\g^*$ is convex. Let $x\in \Delta^\Gamma$ be a $\Gamma$-fixed point in $\Delta$. 
\begin{itemize}\item[a)] The map as in (\ref{stalkmapfirsthigherimage}):
\begin{equation*} \Phi_x:H^n(\Gamma\ltimes\Delta,\L)\to H^n(\Gamma,T)
\end{equation*} is an isomorphism in all degrees $n>0$.
\item[b)] For any $y\in \Delta$, the maps $\Phi_x$ and $\Phi_y$ are related via restriction, meaning that these fit into a commutative diagram:
\begin{center} 
\begin{tikzcd}
H^n(\Gamma\ltimes\Delta,\L)\arrow[r,"\Phi_x"]\arrow[dr,"\Phi_y"'] & H^n(\Gamma,T)\arrow[d,"(i_{\Gamma_y})^*"]\\
& H^n(\Gamma_y,T)
\end{tikzcd}
\end{center} {for each $n>0$}, where $i_{\Gamma_y}:\Gamma_y\hookrightarrow{\Gamma}$ denotes the inclusion.
\end{itemize}
\end{lemma}
\begin{proof} Let $\U$ be the basis of $\Delta$ consisting of all convex opens. By assumption, $\Delta\in \U$. Moreover, by Lemma \ref{vanlem}, $\L$ is $\U$-acyclic. To prove part $a$, we will show that the composition:
\begin{equation}\label{eqn:linthmintaffmoreqprop}
\check{H}^n_\U(\Gamma\ltimes\Delta,\L)\xrightarrow{(\ref{eqn:cechtoderivedmorph:equivshf})} H^n(\Gamma\ltimes\Delta,\L) \xrightarrow{\Phi_x} H^n(\Gamma,T)
\end{equation} is an isomorphism. Consider the $\Gamma$-module $\L(\Delta)$ of global section of $\L$, with action given by: 
\begin{equation*} (\tau\cdot \gamma)(y)=\tau(\gamma\cdot y)\cdot (\gamma,y)\in \sT,\quad \gamma\in \Gamma, \quad y\in \Delta.
\end{equation*} The map (\ref{eqn:linthmintaffmoreqprop}) is equal to the composition: 
\begin{equation*} \check{H}^n_\U(\Gamma\ltimes\Delta,\L)\xrightarrow{(i_\Delta)^*}H^n(\Gamma,\L(\Delta))\xrightarrow{(\textrm{ev}_x)_*} H^n(\Gamma,T),
\end{equation*} where $\textrm{ev}_x:\L(\Delta)\to T$ is the map of $\Gamma$-modules given by evaluation at $x$ and $i_\Delta$ is the functor:
\begin{equation*} \Gamma\to \textsf{Emb}_\U(\Gamma\ltimes\Delta),\quad \gamma \mapsto (\Delta\xleftarrow{\sigma_\gamma} \Delta),
\end{equation*} where we view $\Gamma$ as category with a single object and $\sigma_\gamma$ denotes the bisection with $\Gamma$-component constantly equal to $\gamma$. Consider the short exact sequence of $(\Gamma\ltimes\Delta)$-sheaves:
\begin{equation}\label{eqn:sesBsheaveslagrsec} 0\to \mathcal{C}^\infty(\Lambda)\to \Omega^1_{\textrm{cl}}\xrightarrow{(\ref{exptorbun})} \L\to 0,
\end{equation} where $\mathcal{C}^\infty(\Lambda)$ is the sheaf of smooth local sections of $\Lambda\vert_\Delta$, with $\Lambda$ the lattice bundle associated to $(\G,\Omega)$, and $\Omega^1_{\textrm{cl}}$ is the sheaf of closed $1$-forms on the manifold with corners $\Delta$, each equipped with the $(\Gamma\ltimes\Delta)$-action defined like that on $\L$ (see Example \ref{ex:orbifoldsheafoflagsec}). 
Since $\U$ consists of convex opens, any section of $\L$ over an open $U\in \U$ lifts to a closed $1$-form on $U$ along the universal covering map $T^*\g^*\to \sT$. It follows from this that global sections of (\ref{eqn:sesBsheaveslagrsec}) form an exact sequence of $\Gamma$-modules and the sequence of pre-sheaves on $\textsf{Emb}_\U(\Gamma\ltimes\Delta)$ associated to (\ref{eqn:sesBsheaveslagrsec}) is exact too. Moreover, we have a short exact sequence of $\Gamma$-modules:
\begin{equation*} 0\to \Lambda_T\to \g\to T\to 0. 
\end{equation*}
The associated long exact sequences in cohomology give rise to a commutative diagram:
\begin{center}
\begin{tikzcd}
\check{H}^n_\U(\Gamma\ltimes\Delta,\Omega^1_{\textrm{cl}})\arrow[r,""]\arrow[d,"(i_\Delta)^*"] &  \check{H}^n_\U(\Gamma\ltimes\Delta,\L)\arrow[r]\arrow[d,"(i_\Delta)^*"]  &  \check{H}^{n+1}_\U(\Gamma\ltimes\Delta,\mathcal{C}^\infty(\Lambda))\arrow[r]\arrow[d,"(i_\Delta)^*"]  & \check{H}^{n+1}_\U(\Gamma\ltimes\Delta,\Omega^1_{\textrm{cl}})\arrow[d,"(i_\Delta)^*"] \\ 
H^n(\Gamma,\Omega^1_{\textrm{cl}}(\Delta))\arrow[r]\arrow[d,"(\textrm{ev}_x)_*"] & H^n(\Gamma,\L(\Delta))\arrow[r]\arrow[d,"(\textrm{ev}_x)_*"] & H^{n+1}(\Gamma,\mathcal{C}^\infty(\Lambda)(\Delta))\arrow[r]\arrow[d,"(\textrm{ev}_x)_*"] & H^{n+1}(\Gamma,\Omega^1_{\textrm{cl}}(\Delta))\arrow[d,"(\textrm{ev}_x)_*"]\\
H^n(\Gamma,\g)\arrow[r] & H^n(\Gamma,T)\arrow[r] & H^{n+1}(\Gamma,\Lambda_T)\arrow[r] & H^{n+1}(\Gamma,\g)
\end{tikzcd}
\end{center} with exact rows. The outer left and right groups in the middle and bottom row vanish, because the group cohomology of an $\R$-linear representation of a finite group vanishes in all degrees $n>0$. The outer left and right groups in the top row vanish as well. One way to see this is to note that an explicit null-homotopy in strictly positive degree is given by:
\begin{equation*} h:\check{C}^{\bullet+1}_\U(\Gamma\ltimes\Delta,\Omega^1_{\textrm{cl}})\to \check{C}^{\bullet}_\U(\Gamma\ltimes\Delta,\Omega^1_{\textrm{cl}}), \quad h(c)(U_0\xleftarrow{\sigma_1}\dots \xleftarrow{\sigma_{n}} U_{n})=\frac{1}{|\Gamma|}\sum_{\gamma\in \Gamma}c(\Delta\xleftarrow{\sigma_\gamma}U_0\xleftarrow{\sigma_1} \dots \xleftarrow{\sigma_{n}} U_{n}).
\end{equation*} In view of this, it is enough to show that each of the maps in the composition:
\begin{equation*} \check{H}^{n+1}_\U(\Gamma\ltimes\Delta,\mathcal{C}^\infty(\Lambda)) \xrightarrow{(i_\Delta)^*} H^{n+1}(\Gamma,\mathcal{C}^\infty(\Lambda)(\Delta))\xrightarrow{(\textrm{ev}_x)_*} H^{n+1}(\Gamma,\Lambda_T)
\end{equation*} is an isomorphism. For this, note that $\Lambda=\Lambda_T\times \g^*$ (viewed as subset of $\g\times \g^*=T^*\g^*$). In view of this, the map of $\Gamma$-modules $\textrm{ev}_x:\mathcal{C}^\infty(\Lambda)(\Delta)\to \Lambda_T$ is an isomorphism (since $\Delta$ is connected), hence so is the induced map of cohomology groups $(\textrm{ev}_x)_*$. Moreover, the pre-sheaf on $\textsf{Emb}_\U(\Gamma\ltimes\Delta)$ induced by $\mathcal{C}^\infty(\Lambda)$ is the pull-back of the $\Gamma$-module $\mathcal{C}^\infty(\Lambda)(\Delta)$ along the functor $r_\Delta:\textsf{Emb}_\U(\Gamma\ltimes\Delta)\to \Gamma$ that sends an arrow $V\xleftarrow{\sigma} U$ to the value of the $\Gamma$-component of $\sigma$ (which is constant as $U$ is connected). So, there is the morphism: 
\begin{equation*} (r_\Delta)^*:H^{\bullet}(\Gamma,\mathcal{C}^\infty(\Lambda)(\Delta))\to\check{H}_\U^{\bullet}(\Gamma\ltimes\Delta,\mathcal{C}^\infty(\Lambda)).
\end{equation*} This is inverse to $(i_\Delta)^*$. Indeed, on the level of complexes $(i_\Delta)^*\circ (r_\Delta)^*$ is the identity map (since $r_\Delta\circ i_\Delta=\textrm{id}_{\Gamma}$) and $(r_\Delta)^*\circ (i_\Delta)^*$ is homotopic to the identity map via the homotopy $H$ given by:
\begin{equation*} H(c)(U_0\xleftarrow{\sigma_1}\dots \xleftarrow{\sigma_{n}} U_{n})=\sum_{i=0}^{n}(-1)^ic(\Delta\xleftarrow{\sigma_{\gamma_1}}\dots \xleftarrow{\sigma_{\gamma_i}} \Delta\hookleftarrow U_i \xleftarrow{\sigma_{i+1}}\dots \xleftarrow{\sigma_{n}} U_{n}),
\end{equation*} where $\Delta\hookleftarrow U_i$ is the arrow with underlying bisection the unit map and $\gamma_i$ is the constant value of the $\Gamma$-component of $\sigma_i$ (so that $\sigma_i=\sigma_{\gamma_i}\vert_{U_i}$). So, both $(i_\Delta)^*$ and $(\textrm{ev}_x)_*$ are isomorphisms. This proves part $a$. For part $b$, note that the composition:
\begin{equation*} \check{H}^n_\U(\Gamma\ltimes\Delta,\L)\xrightarrow{(\ref{eqn:cechtoderivedmorph:equivshf})} H^n(\Gamma\ltimes\Delta,\L) \xrightarrow{\Phi_y} H^n(\Gamma_y,T)
\end{equation*} is equal to the composition:
\begin{equation*}
\check{H}^n_\U(\Gamma\ltimes\Delta,\L)\xrightarrow{(i_\Delta)^*}H^n(\Gamma,\L(\Delta))\xrightarrow{(i_{\Gamma_y})^*}H^n(\Gamma_y,\L(\Delta))\xrightarrow{(\textrm{ev}_y)_*} H^n(\Gamma_y,T),
\end{equation*} where $\textrm{ev}_y:\L(\Delta)\to T$ is the map of $\Gamma_y$-modules given by evaluation at $y$ and $i_\Delta$ is as before. So, it suffices to show that:
\begin{center}
\begin{tikzcd}
H^n(\Gamma,\L(\Delta))\arrow[r,"(\textrm{ev}_x)_*"]\arrow[d,"(i_{\Gamma_y})^*"] & H^n(\Gamma,T)\arrow[d,"(i_{\Gamma_y})^*"] \\
H^n(\Gamma_y,\L(\Delta))\arrow[r,"(\textrm{ev}_y)_*"'] & H^n(\Gamma_y,T)
\end{tikzcd} 
\end{center} commutes. Consider the map of $\Gamma$-modules $\tau_\Delta:T\to \L(\Delta)$ that sends $t\in T$ to the section of $ \sT=T\times \g^*$ with $T$-component constantly equal to $t$. Since $\textrm{ev}_x\circ \tau_\Delta=\textrm{id}_T$, the induced map: 
\begin{equation*} (\tau_\Delta)_*:H^n(\Gamma,T)\to H^n(\Gamma,\L(\Delta))
\end{equation*} is the inverse of the isomorphism $(\textrm{ev}_x)_*$. Therefore, it is enough to check the commutativity of:
\begin{center}
\begin{tikzcd}
H^n(\Gamma,\L(\Delta))\arrow[d,"(i_{\Gamma_y})^*"] & H^n(\Gamma,T)\arrow[d,"(i_{\Gamma_y})^*"]\arrow[l,"(\tau_\Delta)_*"'] \\
H^n(\Gamma_y,\L(\Delta))\arrow[r,"(\textrm{ev}_y)_*"'] & H^n(\Gamma_y,T)
\end{tikzcd} 
\end{center} and this is readily verified on the level of cocycles, using the fact that $\textrm{ev}_y\circ \tau_\Delta=\textrm{id}_T$. 
\end{proof} 
We will also use the lemma below, which we leave for the reader to verify.
\begin{lemma}\label{lem:moreqcompwithstalkhomfirstdirim} Suppose we are given a symplectic Morita equivalence between regular and proper symplectic groupoids $(\G_1,\Omega_1)\rightrightarrows M_1$ and $(\G_2,\Omega_2)\rightrightarrows M_2$ that relates a Delzant subspace $\underline{\Delta}_1\subset \underline{M}_1$ to a Delzant subspace $\underline{\Delta}_2\subset \underline{M}_2$. For any $p\in P$, the isomorphism in cohomology induced (as in Corollary \ref{cor:morinvshfcohom:iaequiv}) by the associated integral affine Morita equivalence (Example \ref{sympmoreqiamoreq}) fits into a commutative square: 
\begin{center}
\begin{tikzcd}
H^1(\B\vert_{\Delta_1},\L_1)\arrow[d,"\sim" {anchor=south, rotate=90}]\arrow[r,"\Phi_{x_1}"] & H^1(\B_{x_1},\sT_{x_1})\arrow[d,"(\phi_p)_*", "\sim"' {anchor=south, rotate=90}]\\
H^1(\B_2\vert_{\Delta_2},\L_2)\arrow[r,"\Phi_{x_2}"] & H^1(\B_{x_2},\sT_{x_2}) 
\end{tikzcd}
\end{center} where $x_1=\alpha_1(p)$, $x_2=\alpha_2(p)$ and $\phi_p:\G_{x_1}\xrightarrow{\sim} \G_{x_2}$ is the isomorphism of isotropy groups (\ref{moreqindisoisotgps}). 
\end{lemma}
\begin{proof}[Proof of Theorem \ref{thm:extsheaftorsorfirstdirectimage}] Given an open $\underline{U}$ in $\underline{\Delta}$, the $H^1(\B\vert_U,\L)$-actions on the sets $I^1(\G_x,\sT_x)$ for $x\in U$ (described before the statement of Theorem \ref{thm:extsheaftorsorfirstdirectimage})
induce an action on the set $\mathcal{I}^1_\textrm{Set}(\underline{U})$ (denoting $\mathcal{I}^1_\textrm{Set}:=\mathcal{I}^1_{\textrm{Set},(\G,\Omega,\underline{\Delta})}$ as in Definition \ref{setheoreticinvsecshfextinvdefi}) given by:
\begin{equation*} 
(\textrm{c}\cdot \sigma)(x):=\Phi_{x}(\textrm{c})\cdot \sigma(x),\quad \textrm{c}\in H^1(\B\vert_U,\L),\quad \sigma\in \mathcal{I}^1_\textrm{Set}(\underline{U}),\quad x\in U,
\end{equation*} where on the right we use the action in Remark \ref{degonegpcohomtorsorrem}. Note here that $\textrm{c}\cdot \sigma$ is indeed $\B$-invariant and hence defines an element of $\mathcal{I}^1_\textrm{Set}(\underline{U})$. This can be seen using Lemma \ref{lem:moreqcompwithstalkhomfirstdirim}, applied to the identity equivalence. For varying $\underline{U}$, these actions are compatible with restriction of opens and so descend to an action of the sheaf $\H^1$ on the sheaf $\mathcal{I}^1_\textrm{Set}$. As we will now show, this action is free in the sense that the $\H^1(\underline{U})$-action on $\mathcal{I}^1_\textrm{Set}(\underline{U})$ is free for each open $\underline{U}$ in $\underline{\Delta}$. For this it is enough to show that each $\O\in \underline{U}$ admits an open neighbourhood $\underline{V}$ in $\underline{U}$ such that the $\H^1(\underline{V})$-action on $\mathcal{I}^1_\textrm{Set}(\underline{V})$ is free. To this end, note that for each such $\O$ there is (as in Subsection \ref{flatsecdefsec}) an invariant open $W$ around $\O$ in $M$, together with an infinitesimally abelian compact Lie group $G$, a $G$-invariant open $W_{\g^*}$ around the origin in $\g^*$, and a symplectic Morita equivalence:
\begin{center}
\begin{tikzpicture} \node (G1) at (-1.3,0) {$(\G,\Omega)\vert_{W}$};
\node (M1) at (-1.3,-1.3) {$W$};
\node (S) at (1.4,0) {$(P,\omega_P)$};
\node (M2) at (4.2,-1.3) {$W_{\g^*}$};
\node (G2) at (4.2,0) {$(G\ltimes \g^*,-\d \lambda_{\textrm{can}})\vert_{W_{\g^*}}$};
 
\draw[->,transform canvas={xshift=-\shift}](G1) to node[midway,left] {}(M1);
\draw[->,transform canvas={xshift=\shift}](G1) to node[midway,right] {}(M1);
\draw[->,transform canvas={xshift=-\shift}](G2) to node[midway,left] {}(M2);
\draw[->,transform canvas={xshift=\shift}](G2) to node[midway,right] {}(M2);
\draw[->](S) to node[pos=0.25, below] {$\text{ }\text{ }\alpha_1$} (M1);
\draw[->] (0.65,-0.15) arc (315:30:0.25cm);
\draw[<-] (2.05,0.15) arc (145:-145:0.25cm);
\draw[->](S) to node[pos=0.25, below] {$\alpha_2$\text{ }} (M2);
\end{tikzpicture}
\end{center} that relates the orbit $\O$ of $\G$ to the origin in $\g^*$. After possibly shrinking $W$, we can further arrange that $V:=W\cap \Delta\subset U$ and that the invariant subset $\Delta_{\g^*}$ of $\g^*$ related to $V$ is convex. By Lemma \ref{lem:moreqcompwithstalkhomfirstdirim}, for any choice of $x\in \O$ and $p\in \alpha_1^{-1}(x)$ we have a commutative square:
\begin{center}
\begin{tikzcd}
H^1(\B\vert_V,\L)\arrow[d, "\sim"'{anchor=south, rotate=90}]\arrow[r,"\Phi_x"] & H^1(\B_x,\sT_x)\arrow[d,"(\phi_p)_*", "\sim"' {anchor=south, rotate=90}]\\
H^1(\Gamma\ltimes \Delta_{\g^*},\L_{\Delta_{\g^*}})\arrow[r,"\Phi_x"] & H^1(\Gamma,T)
\end{tikzcd}
\end{center} in which the vertical maps are isomorphisms. This and part $a$ of Lemma \ref{linthmintaffmoreqprop} show that: 
\begin{equation}\label{eqn:stalkmaphigherdirimiso} \Phi_x:H^1(\B\vert_V,\L)\to H^1(\B_x,\sT_x)
\end{equation} is an isomorphism as well. It follows that if $\textrm{c}\in H^1(\B\vert_V,\L)$ and $\sigma\in \mathcal{I}^1_\textrm{Set}(\underline{V})$ are such that $\textrm{c}\cdot \sigma=\sigma$, then $\Phi_x(\textrm{c})\cdot \sigma(x)=\sigma(x)$, hence $\Phi_x(\textrm{c})=0$ (since the action in Remark \ref{degonegpcohomtorsorrem} is free), and so $\textrm{c}=0$. So, the $\H^1(\underline{V})$-action on $\mathcal{I}^1_\textrm{Set}(\underline{V})$ is free. This shows freeness of the $\H^1$-action on $\mathcal{I}^1_\textrm{Set}$. \\

Next, we will show that it restricts to a transitive action on $\mathcal{I}^1$. This would prove the first part of the theorem, since stalk-wise this $\H^1$-action on $\mathcal{I}^1$ would be given by the actions described before the statement of Theorem \ref{thm:extsheaftorsorfirstdirectimage}. To show that the action restricts, it is enough to show that for each open $\underline{U}$ the $H^1(\B\vert_U,\L)$-action on $\mathcal{I}^1_\textrm{Set}(\underline{U})$ preserves $\mathcal{I}^1(\underline{U})$. Let $\sigma\in \mathcal{I}(\underline{U})$ and $\O\in \underline{U}$. Since $\sigma$ is flat, there is a symplectic Morita equivalence as above that relates $\sigma\vert_{\underline{V}}$ (via Lemma \ref{shfisodiscinvseclem}) to a centered section (Definition \ref{centeredsecdefi}). In view of Lemma \ref{lem:moreqcompwithstalkhomfirstdirim} and part $b$ of Lemma \ref{linthmintaffmoreqprop}, for any choice of $x\in \O$, $p\in \alpha_1^{-1}(x)$ and $q\in Q$ the diagram:
\begin{center}
\begin{tikzcd} 
H^1(\B\vert_V,\L)\arrow[r,"\Phi_x"]\arrow[rd,"\Phi_y"']  & H^1(\B_x,\sT_x)\arrow[r,"(\phi_p)_*"] & H^1(\Gamma,T)\arrow[d,"(i_{\Gamma_\alpha})^*"]\\
& H^1(\B_y,\sT_y) \arrow[r,"(\phi_q)_*"] & H^1(\Gamma_{\alpha},T)
\end{tikzcd}
\end{center} commutes, where $y:=\alpha_1(q)$ and $\alpha:=\alpha_2(q)$. Because of this, the section related to $(\textrm{c}\cdot \sigma)\vert_{\underline{V}}$ by the Morita equivalence is centered as well, and so $\textrm{c}\cdot \sigma$ is flat at $\O$, for any $\textrm{c}\in H^1(\B\vert_U,\L)$. So, the $H^1(\B\vert_U,\L)$-action indeed preserves $\mathcal{I}^1(\underline{U})$. For transitivitity, let $\sigma_1,\sigma_2\in \mathcal{I}^1(\underline{U})$ and $\O\in \underline{U}$. In view of Remark \ref{extinvflatindchartrem}, there is a symplectic Morita equivalence as above that relates both $\sigma_1\vert_{\underline{V}}$ and $\sigma_2\vert_{\underline{V}}$ to centered sections. Consider an $x\in \O$. By surjectivity of (\ref{eqn:stalkmaphigherdirimiso}) and transitivity of the $H^1(\B_x,\sT_x)$-action on $I^1(\G,\sT_x)$, there is a $\textrm{c}_V\in H^1(\B\vert_V,\L)$ such that $\Phi_x(\textrm{c}_V)\cdot \sigma_1(x)=\sigma_2(x)$. From commutativity of the last diagram above and the fact that both $\sigma_1\vert_{\underline{V}}$ and $\sigma_2\vert_{\underline{V}}$ are related to centered sections by the Morita equivalence, it follows that in fact $\textrm{c}_V\cdot {\sigma_1}\vert_{\underline{V}}=\sigma_2\vert_{\underline{V}}$. This shows that there is an open cover $\underline{\mathcal{V}}$ of $\underline{U}$ with a collection $\{\textrm{c}_{V}\in H^1(\B\vert_V,\L)\mid \underline{V}\in \underline{\mathcal{V}}\}$ such that $\textrm{c}_V\cdot {\sigma_1}\vert_{\underline{V}}=\sigma_2\vert_{\underline{V}}$ for each $\underline{V}\in \underline{\mathcal{V}}$. Since the $\H^1$-action is free, for any $\underline{V}_1,\underline{V}_2$ the $\H^1$-germs of $\textrm{c}_{V_1}$ and $\textrm{c}_{V_2}$ must coincide at all points in $\underline{V}_1\cap\underline{V}_2$. So, these glue to a section $\textrm{c}\in \H^1(\underline{U})$ such that $\textrm{c}\cdot \sigma_1=\sigma_2$. This shows that the $\mathcal{H}^1$-action on $\mathcal{I}^1$ is indeed transitive. \\

It remains to prove the second part of the theorem. We ought to show that the map (\ref{extinvariantmapthrdstrthm}) is $H^1(\B\vert_\Delta,\L)$-equivariant with respect to the action in the first structure theorem and that on $\mathcal{I}^1(\underline{\Delta})$ obtained from the $\H^1(\underline{\Delta})$-action via the group homomorphism (\ref{eqn:secondmapsesorblerspecseq}). After restricting to a complete transversal the proof reduces to the case in which $\B$ is etale. To prove the equivariance in that case, suppose that $\B$ is etale, that $\textrm{c}\in H^1(\B\vert_\Delta,\L)$ and that $J:(S,\omega)\to M$ is a faithful multiplicity-free $(\G,\Omega)$-space with momentum map image $\Delta$. Let $\U$ be a basis of $\Delta$ for which $\L$ is $\U$-acyclic, consisting of connected opens. Consider a $1$-cocycle $\tau$ representing the class $[\tau]\in\check{H}^1_\U(\B\vert_\Delta,\L)$ corresponding to $\textrm{c}$ via the isomorphism (\ref{eqn:cechtoderivedmorph:equivshf}). Let $x\in \Delta$. In view of Lemma \ref{proetgpoidlinlem}, there is a $U\in \U$ with the property that for each $\gamma\in \B_x$ there is a (necessarily unique) smooth bisection $\sigma_\gamma:U\to \B\vert_U$ such that $\sigma(x)=\gamma$, and the map $\B_x\times U\to \B\vert_U$, $(\gamma,y)\mapsto \sigma_\gamma(y)$ is surjective. Then, in fact, for each $\gamma\in \B\vert_U$ there is a unique smooth bisection $\sigma_\gamma:U\to \B\vert_U$ that maps $s(\gamma)$ to $\gamma$. The composition:
\begin{equation*} \check{H}^1_\U(\B\vert_\Delta,\L)\xrightarrow{(\ref{eqn:cechtoderivedmorph:equivshf})} H^1(\B\vert_\Delta,\L)\xrightarrow{\Phi_x} H^1(\B_x,\sT_x)
\end{equation*} is equal to the composition:
\begin{equation*} \check{\Phi}_x:=(\textrm{ev}_x)_*\circ(i_U)^*:\check{H}^1_\U(\B\vert_\Delta,\L)\xrightarrow{(i_U)^*} H^1(\B_x,\L(U))\xrightarrow{(\textrm{ev}_x)_*} H^1(\B_x,\sT_x),
\end{equation*} with $i_U:\B_x\to \textsf{Emb}_\U(\B\vert_\Delta)$ the functor (where we view $\B_x$ as category over $\{x\}$) that sends $\{x\}$ to $U$ and that sends an arrow $\gamma\in \B_x$ to the arrow $U\xleftarrow{\sigma_\gamma} U$. So, we have to show that:
\begin{equation}\label{extinvequivariantpropeqtoprove} \check{\Phi}_x([\tau])\cdot e(J)_x=e(J_\tau)_x.
\end{equation}  
Consider the map:
\begin{equation*} \phi_\tau:(\G,\Omega)\vert_U\to (\G,\Omega)\vert_U,\quad g\mapsto g\tau(U\xleftarrow{\sigma_{{[g]}}} U)(s(g))^{-1}.
\end{equation*} This is an automorphism of symplectic groupoids with corners. To see this, let $\gamma\in \B\vert_U$. The image $\sigma_\gamma(U)$ in $\B\vert_\Delta$ is open, hence the pre-image of this under the projection $\G\vert_\Delta\to \B\vert_\Delta$ is an open in $\G\vert_U$ on which the map $\phi_\tau$ is given by the smooth map: 
\begin{equation*} g \mapsto g\tau(U\xleftarrow{\sigma_\gamma} U)(s(g))^{-1}.
\end{equation*} Since $\G\vert_U$ can be covered by such opens, $\phi_\tau$ is smooth. Furthermore, the fact that each $\tau(U\xleftarrow{\sigma_\gamma} U)$ is Lagrangian implies that $\phi_\tau$ is symplectic, and $\tau$ being a $1$-cocycle implies that $\phi_\tau$ a morphism of groupoids. Moreover, $\phi_{-\tau}$ is inverse to $\phi_{\tau}$. So, $\phi_\tau$ is indeed an automorphism. The pair $(\phi_\tau,j_U)$ consisting of $\phi_\tau$ and the symplectomorphism (\ref{injopenintoquot}) is compatible with the $\G\vert_U$-actions, in the sense that for all $g\in \G\vert_U$ and $p\in J^{-1}(U)$ such that $s(g)=J(p)$ it holds that:
\begin{equation*} j_U(g\cdot p)=\phi_\tau(g)\cdot j_U(p).
\end{equation*} Consequently, this pair induces an equivalence of symplectic representations:
\begin{equation*} (\phi_\tau,(\d j_U)_p):(\G_p,(\S\No_p,\omega_p))\xrightarrow{\sim} (\G_{[p,U]},(\S\No_{[p,U]},(\omega_\tau)_{[p,U]})) 
\end{equation*} for each $p\in S$ such that $J(p)=x$. Using Lemma \ref{smpeqtorrepindisoextinvprop} and the observation that $\phi_\tau$ restricts to the identity on $\sT$, (\ref{extinvequivariantpropeqtoprove}) readily follows from this.  
\end{proof} 

\subsubsection{{The structure groups for the case of group actions}}\label{sec:caseofgpactions} In this subsection we give the remaining details for Examples \ref{almabcompLiegpexintro} and \ref{gencompLiegpexintro}. The facts used in Example \ref{almabcompLiegpexintro} that remain to be proved are collected in the following proposition. 
\begin{prop}\label{prop:strgps:caseofgpactions}
Let $G$ be an infinitesimally abelian compact Lie group with identity component $T$ and group of connected components $\Gamma:=G/T$. Consider the cotangent groupoid $(\G,\Omega):=(G\ltimes \g^*, -\d\lambda_\textrm{can})$ and let $\underline{\Delta}$ be a non-empty Delzant subspace of $\underline{M}=\g^*/G$ such that the corresponding invariant subspace $\Delta$ of $\g^*$ is convex. Then the following hold. 
\begin{itemize}
\item[a)] {For each $n>0$,} there is a canonical group isomorphism:
\begin{equation*} H^n(\B\vert_\Delta,\L)\xrightarrow{\sim} H^n(\Gamma,T)
\end{equation*}
\item[b)] $H^1(\underline{\Delta},\underline{\L})=0$.  
\item[c)] There is a canonical bijection between the set $\mathcal{I}^1(\underline{\Delta})$ of global sections of the ext-sheaf (\ref{flatsecshintro}) and $I^1(G,T)$, which together with the above group isomorphism is compatible with the action of $H^1(\B\vert_\Delta,\L)$ on $\mathcal{I}^1(\underline{\Delta})$ and the canonical action of $H^1(\varGamma,T)$ on $I^1(G,T)$. 
\end{itemize}
\end{prop}
\begin{proof} Since the finite group $\Gamma$ acts linearly on $\g^*$ and $\Delta$ is convex and non-empty, there is a $\Gamma$-fixed point $x\in \Delta^\Gamma$ (as follows by averaging over $\Gamma$). The map $\Phi_x:H^n(\B\vert_\Delta,\L)\to H^n(\Gamma,T)$ is an isomorphism by part $a$ of Lemma \ref{linthmintaffmoreqprop}, and by part $b$ of the same lemma this does not depend on the choice of $x\in \Delta^\Gamma$. This proves part $a$. For part $b$, note that $\Phi_x$ factors as:
\begin{center}
\begin{tikzcd} 
H^1(\B\vert_\Delta,\L)\arrow[dr,"\Phi_x"']\arrow[r] & \mathcal{H}^1(\underline{\Delta})\arrow[d,"\widehat{\Phi}_x",dashed] \\
& H^1(\Gamma,T) 
\end{tikzcd}
\end{center} with the horizontal map as in the sequence (\ref{eqn:threetermexseqlowdeg}) and with $\widehat{\Phi}_x$ given by composition of the evaluation at the stalk of $\H^1$ at $\O_x$ with the map from this stalk into $H^1(\Gamma,T)$ induced by the maps (\ref{stalkmapfirsthigherimage}). As (\ref{eqn:threetermexseqlowdeg}) is exact, part $b$ follows. For part $c$, consider the evaluation map:
\begin{equation}\label{eqn:evatfixpt:extsheaf}\textrm{ev}_x:\mathcal{I}^1(\underline{\Delta})\to I^1(G,T). 
\end{equation} Clearly, this and the map $\Phi_x$ are compatible with the action of $H^1(\B\vert_\Delta,\L)$ and that of $H^1(\Gamma,T)$. To prove part $c$, we will show that this is a bijection that does not depend on the choice of $x\in \Delta^\Gamma$. For injectivity, suppose that $\sigma_1,\sigma_2\in \mathcal{I}^1(\underline{\Delta})$ are such that $\sigma_1(x)=\sigma_2(x)$. Let $y\in \Delta$. 
Then $(1-t)x+ty\in \Delta$ for all $t\in [0,1]$, by convexity. Since $x$ is fixed by $\Gamma$ and the action is linear, it holds that $\Gamma_{(1-t)x+ty}=\Gamma_y$ for all $t\in ]0,1]$. Hence, the interval $]0,1]$ is partitioned by the sets:
\begin{equation*} S([c])=\{t\in]0,1]\mid \sigma_1((1-t)x+ty)=[c]\},\quad [c]\in I^1(\Gamma_y,T). 
\end{equation*} From flatness of $\sigma_1$ and $\sigma_2$ and the fact that $\sigma_1(x)=\sigma_2(x)$, it follows that $S(\sigma_2(y))$ is non-empty. Moreover, $S([c])$ is open in $]0,1]$ for each $[c]\in I^1(G,T)$ due to flatness of $\sigma_1$. So, by connectedness of $]0,1]$ it must hold that $\sigma_1(y)=\sigma_2(y)$, which shows that (\ref{eqn:evatfixpt:extsheaf}) is indeed injective. For the remainder, consider the map:
\begin{equation*} I^1(G,T)\to \mathcal{I}^1(\underline{\Delta})
\end{equation*}
that sends $[c]\in I^1(G,T)$ to the global section of $\mathcal{I}^1$ defined by the section of (\ref{set-theoric-bundle-ext-inv-cent}) with value at $y\in \Delta$ given by $[c\vert_{G_y}]\in I^1(G_y,T)$ (which is indeed invariant, as follows from Lemma \ref{conjtoreltrivincohomlem}). Since this is a section of (\ref{eqn:evatfixpt:extsheaf}), we conclude that (\ref{eqn:evatfixpt:extsheaf}) is a bijection with this map as its inverse. As its inverse is independent of the choice of $x\in \Delta^\Gamma$, so is the map (\ref{eqn:evatfixpt:extsheaf}). 
\end{proof}
{
Next, we return to Example \ref{gencompLiegpexintro}. 
\begin{ex}\label{ex:gencompLiegpexintro:details}
Let $G$ be any compact Lie group. Fix a maximal torus $T$ and an open Weyl chamber $\c$ in $\t^*$. Further, let $N(\c)$ be the normalizer of $\c$ in $G$. Consider the regular and proper symplectic groupoid $(\G,\Omega):=(G\ltimes \g^*_\textrm{reg},-\d\lambda_\textrm{can})$ over $M:=\g^*_\textrm{reg}$. Since $\c$ is a complete transversal, $(\G,\Omega)$ is canonically Morita equivalent to its restriction $(\G,\Omega)\vert_\c$, which coincides with the restriction to $\c$ of cotangent symplectic groupoid of $N(\c)$. This Morita equivalence induces a bijection, as in \eqref{eqn:moreq:indbijisoclasseshamsp}, between isomorphism classes of compact, connected, regular, multiplicity-free Hamiltonian $G$-spaces (in the sense of Example \ref{gencompLiegpexintro}) and isomorphism classes of compact, connected, multiplicity-free Hamiltonian $N(\c)$-spaces with momentum map image contained in $\c$. Explicitly, this is given by associating to a regular Hamiltonian $G$-space its restriction to $\c$. To a Hamiltonian $N(\c)$-space $J:(S,\omega)\to \t^*$ as above we can associate the triple $(K,\Delta,e)$ consisting of:
\begin{itemize}\item the kernel $K$ of the morphism $T\to \textrm{Diff}(S)$ obtained by restricting the $N(\c)$-action to $T$,  
\item the image $\Delta:=J(S)$,
\item the ext-invariant $e$ of the induced maximally toric Hamiltonian $N(\c)/K$-action (as in Example \ref{almabcompLiegpexintro}), with momentum map $\mu-x:(S,\omega)\to (\t/\mathfrak{k})^*$, for any $N(\c)$-fixed point $x\in \Delta$.
\end{itemize} Note here that $e$ does not depend on the choice of $N(\c)$-fixed point. All together, this gives a bijection between isomorphism classes of compact, connected, regular, multiplicity-free Hamiltonian $G$-spaces and triples as in Example \ref{gencompLiegpexintro}. 
\end{ex}
}
Finally, using the above proposition we can also complete the argument in Example \ref{example:concretecompsecondstrthm}.
\begin{ex}
\label{example:concretecompsecondstrthm2}
To show that $H^1(\underline{\Delta},\underline{\L})$ vanishes in Example \ref{example:concretecompsecondstrthm}, consider the open cover $\underline{\U}=\{\underline{U}^+,\underline{U}^-\}$ with $U^{\pm}:=W^{\pm}\cap \Delta$ for $W^{\pm}:=\{(x,\mu)\in \R\times \mathbb{S}^1\mid \mu\neq \pm 1\}$. Since $(\G,\Omega)\vert_{W^+}$ and $(\G,\Omega)\vert_{W^-}$ are both isomorphic to restrictions of symplectic groupoids of the form in Proposition \ref{prop:strgps:caseofgpactions} via isomorphisms that identify $U^+$ and $U^-$ with convex sets, the cohomology groups $H^1(\underline{U}^+,\underline{\L})$ and $H^1(\underline{U}^-,\underline{\L})$ vanish. So, in view of the Mayer-Vietoris sequence, to show that $H^1(\underline{\Delta},\underline{\L})=0$ it suffices to show that every $\tau \in \underline{\L}(\underline{U}^+\cap \underline{U}^-)$ can be written as the difference of a section of $\underline{\L}$ on $\underline{U}^+$ and one on $\underline{U}^-$. For this, note that by $\Gamma$-invariance of $\tau$, after possibly subtracting the section $[\frac{1}{2}\d x]\in \underline{\L}(\underline{\Delta})$ (obtained from the closed $1$-form $\frac{1}{2}\d x$ by projection along $T^*M\to \sT_\Lambda$), we can assume that $\tau(0,i)=[c\,\d\theta_{(0,i)}]$ for some $c\in \R$. Let $V^{\pm}:=\{(x,\mu)\in \Delta\mid \pm\im(\mu)>0\}$ denote the connected components of $U^+\cap U^-$. Since $V^+$ is simply-connected, there is a unique $f\in C^\infty(V^+)$ such that $\d f$ projects to $\tau\vert_{V^+}$ along the covering map $T^*M\to \sT_\Lambda$ and such that $\d f_{(0,i)}=c\,\d\theta_{(0,i)}$ and $f(0,i)=0$. Since $(0,i)$ and $\d f_{(0,i)}$ are fixed by $\Z_2\times \{1\}$, the function $f$ is $\Z_2\times \{1\}$-invariant. Hence, the function $\widehat{f}\in C^\infty(U^+\cap U^-)$ given by:
\begin{equation*}
\widehat{f}(x,\mu)=\begin{cases} f(x,\mu) \quad&\text{ if } (x,\mu)\in V^+ \\
f(x,\mu^{-1})\quad &\text{ if }(x,\mu)\in V^-
\end{cases} 
\end{equation*} is $\Gamma$-invariant. Choose a smooth function $\rho\in C^\infty(\mathbb{S}^1)$ that has constant value $1$ on a neighbourhood of $1$, that is supported in $\mathbb{S}^1-\{-1\}$ and that is $\Z_2$-invariant with respect to the action given by $\varepsilon\cdot \mu=\mu^\varepsilon$, for $\varepsilon\in \Z_2$ and $\mu\in \mathbb{S}^1$. The respective functions: 
\begin{equation*} (x,\mu)\mapsto \rho(\mu)\widehat{f}(x,\mu)\quad\&\quad (x,\mu)\mapsto (\rho(\mu)-1)\widehat{f}(x,\mu)
\end{equation*} on $U^+\cap U^-$ extend by zero to $\Gamma$-invariant smooth functions $f_+$ and $f_-$ on $U^+$ and $U^-$, and the closed $1$-forms $\d f_+$ and $\d f_-$ project to respective sections $\tau_+\in \underline{\L}(\underline{U}^+)$ and $\tau_-\in \underline{\L}(\underline{U}^-)$ with the property that $\tau_+\vert_{U^+\cap U^-}-\tau_-\vert_{U^+\cap U^-}=\tau$. This proves the claim. 
\end{ex}

\newpage
\appendix
\section{{On the multiplicity-free condition}}
In this appendix we use results on the Poisson geometry of the orbit space of a Hamiltonian action to derive the following (in the same spirit as \cite[Proposition A.1]{Wo}).
\begin{prop}\label{toracttorrepprop-cd} Let $(\G,\Omega)\rightrightarrows M$ be a regular and proper symplectic groupoid. A Hamiltonian $(\G,\Omega)$-action along $J:(S,\omega)\to M$ is {faithful multiplicity-free} if and only if the following four conditions hold.
\begin{itemize} \item[i)] The induced action of $\mathcal{T}$ is free on a dense subset of $S$.
\item[ii)] The equality:
\begin{equation}\label{compzero} \dim(S)=2\dim(M)-\rk(\pi)
\end{equation} holds, where $\pi$ is the Poisson structure on $M$ induced by $(\G,\Omega)$.
\end{itemize}
\begin{itemize}
\item[iii)] The momentum map $J$ has connected fibers.
\item[iv)] The transverse momentum map $\underline{J}:\underline{S}\to \underline{M}$ is closed as a map into its image. 
\end{itemize} 
\end{prop}

To prove Proposition \ref{toracttorrepprop-cd}, let us recall some facts on the Poisson geometry of Hamiltonian actions. Let $(\G,\Omega)$ be a proper symplectic groupoid and let $J:(S,\omega)\to M$ be a Hamiltonian $(\G,\Omega)$-space. 
\begin{itemize} \item The sheaf of smooth functions $\mathcal{C}^\infty_{\underline{S}}$ on the orbit space $\underline{S}:=S/\G$ is the sheaf of $\R$-algebras consisting of $\G$-invariant smooth functions on invariant opens in $S$. This can naturally be viewed as a subsheaf of the sheaf of continuous functions on $\underline{S}$. For each invariant open $U$ in $S$, the subalgebra $\mathcal{C}^\infty_{\underline{S}}(\underline{U})$ of $\mathcal{C}^\infty_S(U)$ is a Poisson subalgebra with respect to the Poisson bracket associated to the symplectic form $\omega$ on $S$. These Poisson brackets make $\mathcal{C}^\infty_{\underline{S}}$ into a sheaf of Poisson algebras, and as for smooth manifolds these brackets are uniquely determined by the single Poisson bracket on the algebra of global smooth functions on $\underline{S}$. The stratification $\S_\textrm{Gp}(\underline{S})$ of $\underline{S}$ induced by the $\G$-action has the property that each stratum $\underline{\Sigma}$ admits a natural structure of Poisson manifold, uniquely determined by the fact that restriction along the inclusion $i:\underline{\Sigma}\hookrightarrow \underline{S}$ induces a surjective map of sheaves $i^*:\mathcal{C}^\infty_{\underline{S}}\vert_{\underline{\Sigma}}\to \mathcal{C}^\infty_{\underline{\Sigma}}$ that respects the Poisson brackets.
\item The \textbf{complexity} \textrm{C}($J$) of the Hamiltonian $(\G,\Omega)$-action is, by definition, half of the maximum of the dimensions of the symplectic leaves on all of these strata. The dimension of the symplectic leaves in $\underline{S}$ is locally non-decreasing (\cite[Proposition 2.92]{Mol1}). Therefore, the union of the symplectic leaves in $\underline{S}$ of maximal dimension is open in $\underline{S}$. This can be used to deduce (using for instance \cite[Proposition 2.69 and Remark 2.93]{Mol1}) that, if the set of points in $S$ at which the momentum map $J$ is a submersion is dense in $S$, then the complexity of the Hamiltonian action is:
\begin{equation}\label{complexityeq} \textrm{C}(J)=\frac{1}{2}\left(\dim(S)-2\dim(M)+\max_{x\in J(S)}\rk(\pi_x)\right),
\end{equation} where $\pi$ is the Poisson structure on $M$ induced by $(\G,\Omega)$. 
\item The symplectic leaves of the orbit space $\underline{S}$ can be described in terms of the symplectic reduced spaces of the Hamiltonian action. As topological spaces, the reduced spaces are the subspaces of $\underline{S}$ of the form:
\begin{equation*} J^{-1}(\O)/\G=\underline{J}^{-1}(\O),
\end{equation*} where $\O$ is an {orbit} of $\G$ in $M$. The reduced spaces can naturally be stratified into symplectic manifolds (generalizing the main result of \cite{LeSj}) and the symplectic leaves of the orbit space $\underline{S}$ coincide with the symplectic strata of the reduced spaces. 
\end{itemize} The facts mentioned above are probably well-known for Hamiltonian actions of compact Lie groups. {Such Hamiltonian actions that satisfy the equivalent conditions in the proposition below are called \textbf{multiplicity-free}} \cite{GS3}. 
\begin{prop}\label{complexityzerocharprop} Let $(\G,\Omega)$ be a proper symplectic groupoid and let $J:(S,\omega)\to M$ be a Hamiltonian $(\G,\Omega)$-space. Then the following are equivalent.
\begin{itemize}\item[a)] The induced Poisson bracket on $C^\infty(\underline{S})$ is the zero bracket.
\item[b)] The Hamiltonian action has complexity zero.
\item[c)] All of the reduced spaces are discrete topological subspaces of $\underline{S}$. 
\end{itemize}
\end{prop}
\begin{proof} As mentioned above, the Poisson brackets on the algebras associated to the sheaf of Poisson algebras $\mathcal{C}^\infty_{\underline{S}}$ are uniquely determined by the single Poisson bracket on the algebra $C^\infty(\underline{S})$. So, if $a$ holds, then the Poisson bracket on $\mathcal{C}^\infty_{\underline{S}}(\underline{U})$ is zero for each invariant open $U$ in $S$. The Poisson structure on each stratum of $\S_\textrm{Gp}(\underline{S})$ must then also be zero, so that $b$ holds. Conversely, from $b$ it is immediate that the Poisson bracket on each stratum is zero. Since the inclusion of each stratum is a Poisson map, the bracket on $C^\infty(\underline{S})$ is zero as well, as follows from pointwise evaluation. So, $b$ implies $a$. This proves the equivalence between $a$ and $b$. The equivalence between $b$ and $c$ is clear from the description of the symplectic leaves as strata of the symplectic reduced spaces. 
\end{proof}
We now turn to the proof of the main result of this appendix. 
\begin{proof}[Proof of Proposition \ref{toracttorrepprop-cd}] First notice that both sets of conditions contain the assumption that the induced $\sT$-action is free on a dense subset of $S$. The momentum map $J$ is a submersion at all points in $S$ at which the $\sT$-action is free. Therefore, under both sets of conditions, the set of points in $S$ at which $J$ is a submersion is dense in $S$, which implies that the complexity of the Hamiltonian action is given by (\ref{complexityeq}). Since $(\G,\Omega)$ is regular, this means that:
\begin{equation*} \textrm{C}(J)=\frac{1}{2}\left(\dim(S)-2\dim(M)+\rk(\pi)\right).
\end{equation*}
Now, suppose that all properties in Proposition \ref{toracttorrepprop-cd} are satisfied. By the above equation for the complexity of the Hamiltonian action, the second condition in Proposition \ref{toracttorrepprop-cd} means that the action has complexity zero, which by Proposition \ref{complexityzerocharprop} means that for every $x\in M$, the subspace $\underline{J}^{-1}({\O}_x)$ of $\underline{S}$ is discrete. Since $\underline{J}^{-1}(\O_x)$ is (canonically) homeomorphic to the quotient $J^{-1}(x)/\G_x$ and the $J$-fibers are assumed to be connected, it follows that both $J^{-1}(x)/\G_x$ and $\underline{J}^{-1}(\O_x)$ consist of a single point. Firstly, it follows from this that the $\G_x$-orbit (which is embedded, seeing as $\G_x$ is compact) is the subspace $J^{-1}(x)$ of $S$. Since $J^{-1}(x)$ is connected, the $\mathcal{T}_x$-orbit must then also be the entire space $J^{-1}(x)$. Secondly, it follows that $\underline{J}$ is injective (its fibers being points). So, since it is assumed to be closed as map into its image, it must be a topological embedding. This proves that the Hamiltonian action is faithful multiplicity-free. \\

Next, suppose that the action is faithful multiplicity-free. Then clearly the transverse momentum map is closed as map into its image. Furthermore, since the $J$-fibers coincide with the $\sT$-orbits, they are connected and for each $x\in M$ the quotient $J^{-1}(x)/\G_x$ consists of a single point. So, for every $x\in M$ the set $\underline{J}^{-1}(\O_x)$ consists of a single point as well, which by Proposition \ref{complexityzerocharprop} implies that the Hamiltonian action has complexity zero. By the above equation for the complexity, this proves that the Hamiltonian action satisfies all conditions in Proposition \ref{toracttorrepprop-cd}. 
\end{proof}
\newpage
\section{Background on manifolds with corners}\label{appmanwithcornsec}
Let $X$ be a topological space. Given integers $n\geq 0$ and $k\in \{0,...,n\}$, we use the notation:
\begin{equation*} \R^n_k:=[0,\infty[^k\times \R^{n-k}.
\end{equation*}
By an \textbf{$n$-dimensional chart with corners} for $X$ we mean a pair $(U,\chi)$ consisting of an open $U$ in $X$ and a homeomorphism $\chi$ from $U$ onto an open in $\R^n_k$, for some $k \in \{0,...,n\}$. Given two subsets $A\subset \R^n$ and $B\subset \R^m$, we say that a map $f:A\to B$ is smooth if for every $x\in A$ there is an open $U_x$ in $\R^n$ around $x$ and a smooth map $U_x\to \R^m$ that coincides with $f$ on $U_x\cap A$. Two charts with corners $(U,\chi)$ and $(V,\phi)$ on $X$ are called smoothly compatible if both transition maps between them are smooth maps in this sense. Just as for manifolds without corners, this leads to a notion of smooth atlas consisting of charts with corners for $X$ (that we require to consist of charts of a fixed dimension) and every such atlas is contained in a unique maximal one. We refer to a maximal such atlas as a \textbf{smooth structure with corners} on $X$.
\begin{defi} A \textbf{smooth manifold with corners} is a second countable and Hausdorff space $X$ together with a smooth structure with corners on $X$. Henceforth, by a manifold with corners we always mean a \textit{smooth} manifold with corners and we omit the smooth structure from the notation. Furthermore, we use the following terminology and notation.
\begin{itemize} \item The common dimension of the charts for $X$ is called the \textbf{dimension} of $X$, denoted $\dim(X)$. 
\item By a \textbf{smooth map} $f:X\to Y$ between two manifolds with corners we will mean a continuous map with the property that for any chart with corners $(U,\chi)$ for $X$ and any chart with corners $(V,\phi)$ for $Y$, the coordinate representation: 
\begin{equation*} \phi\circ f\circ \chi^{-1}:\chi(U\cap f^{-1}(V))\to \phi(V)
\end{equation*} is smooth in the sense above. 
\item A homeomorphism $f:X\to Y$ between manifolds with corners is called a \textbf{diffeomorphism} if $f$ and $f^{-1}$ are smooth. 
\end{itemize}
\end{defi}
\begin{rem} The above definition of smooth map coincides with that used in \cite{NiWeXu,KaLe,Ni}. For a further comparison to the literature on manifolds with corners and maps between them, see \cite{Jo} (where a smooth map in the above sense is called weakly smooth).  
\end{rem}
\begin{rem}\label{maxsmatlasdetbysmfunrem} Any open subspace of a manifold with corners $X$ inherits a smooth structure with corners and the $\R$-valued smooth functions on these opens form a sheaf of algebras $\mathcal{C}^\infty_X$ on $X$. Two smooth structures with corners coincide if their associated sheaves of smooth functions coincide.  
\end{rem}
\begin{rem}\label{manwithcornreddiffbsprem} A manifold with corners $X$ is a reduced differentiable space with structure sheaf $\mathcal{C}^\infty_X$, in the sense of \cite{GoSa} (also see \cite{Mol1} for a more direct introduction to these). Furthermore, a continuous map $f:X\to Y$ between manifolds with corners is smooth if and only if it is a morphism of the underlying reduced differentiable spaces (meaning that for every smooth function on an open in $Y$, the pull-back along $f$ is again smooth).  As for any second countable and Hausdorff reduced differentiable space, there exist $\mathcal{C}^\infty_X$-partitions of unity subordinate to any open cover. 
\end{rem}

Let $X$ be an $n$-dimensional manifold with corners. Given $x\in X$, let $\textrm{Charts}_x(X)$ denote the set of charts for $X$ around $x$. As for manifolds without corners, one can define the tangent space $T_xX$ of $X$ at $x$ as the $n$-dimensional real vector space of maps:
\begin{equation*} \textrm{Charts}_x(X)\to \R^n
\end{equation*} that are compatible with coordinate changes (so that any such map is determined by its value on a single chart), and for every smooth map $f:X\to Y$ between manifolds with corners one can define its differential $\d f_x:T_xX\to T_{f(x)}Y$ at $x$. The composition of two smooth maps is again smooth, and the chain rule still holds. 
\begin{rem}\label{tanspderdescriprem} The tangent space $T_xX$ is naturally isomorphic to the vector space of derivations at $x$ of the stalk of $\mathcal{C}^\infty_X$ at $x$, and to that of the algebra of global smooth functions on $X$. 
\end{rem}
\begin{rem} On manifolds with corners one can define smooth vector fields and differential forms, the pull-back of differential forms along smooth maps, their wedge-product and their exterior derivative, as for manifolds without corners. The Poincar\'{e} Lemma still holds: every closed differential form is locally exact, by arguments as in the case without corners (e.g. as in \cite{BoTu}). 
\end{rem}
Unlike for manifolds without corners, there may be tangent vectors that cannot be realized as the derivative of a smooth curve in $X$. A tangent vector $v\in T_xX$ is called \textbf{inward pointing} if there is an $\varepsilon>0$ and a smooth curve $\gamma:[0,\varepsilon[\to X$ such that $v=\dot{\gamma}(0)$. The inward pointing tangent vectors form a polyhedral cone $C_xX$ in $T_xX$, that we call the \textbf{tangent cone} of $X$ at $x$. We let $F_xX$ denote the largest linear subspace of $T_xX$ that is contained in $C_xX$. This consists of those $v\in T_xX$ for which there is an $\varepsilon>0$ and a smooth curve $\gamma:]-\varepsilon,\varepsilon[\to X$ such that $v=\dot{\gamma}(0)$. For any smooth map $f:X\to Y$ and $x\in X$, it holds that:
\begin{equation}\label{smmapconrel} \d f_x(C_xX)\subset C_{f(x)}Y \quad \& \quad \d f_x(F_xX)\subset F_{f(x)}Y. 
\end{equation} For any chart $(U,\chi)$ for $X$ onto an open in $\R^n_k$ that sends $x\in U$ to the origin, the differential $\d\chi_x:T_xX\xrightarrow{\sim} \R^n$ identifies $C_xX$ with $\R^n_k$ and $F_xX$ with $\{0\}\times \R^{n-k}$. The \textbf{depth} of $x\in X$ is: \begin{equation*}\textrm{depth}_X(x):=\dim(X)-\dim(F_xX).
\end{equation*} For any chart $(U,\chi)$ for $X$ around $x$ mapping onto an open in $\R^n_k$, the depth of $x$ equals the number of $j\in\{1,...,k\}$ such that $\chi^j(x)=0$. \\

Next, we turn to embeddings. Following \cite{KaLe}, we use the definition below.
\begin{defi}\label{emsubmanwithcorndef} We call a topological embedding $i:X\to Y$ between manifolds with corners a \textbf{smooth embedding} if it is smooth and at each point in $X$ its differential is injective. 
\end{defi} 
\begin{ex}\label{graphembmanwithcorn} Given manifolds with corners $X$ and $Y$, the product $X\times Y$ inherits a natural structure of smooth manifold with corners. For any smooth map $f:X\to Y$, the graph map: \begin{equation*} 
X\to X\times Y, \quad x\mapsto (x,f(x)),
\end{equation*} is a smooth embedding. 
\end{ex}
\begin{rem}\label{charembofmanwithcornrem} A topological embedding $i:X\to Y$ between manifolds with corners is a smooth embedding if and only if $i:(X,\mathcal{C}^\infty_X)\to (Y,\mathcal{C}^\infty_Y)$ is an embedding of reduced differentiable spaces, meaning that $i$ is smooth and for every smooth function $g$ on an open $U$ in $X$ and every $x\in U$ there is a smooth function $\widehat{g}$ on an open $U_x$ in $Y$ around $i(x)$ such that $\widehat{g}\circ i$ coincides with $g$ on $U\cap i^{-1}(U_x)$. Here, the forward implication follows using the immersion theorem for smooth maps between opens in Euclidean spaces, while the backwards implication is clear from Remark \ref{tanspderdescriprem}. In view of Remark \ref{manwithcornreddiffbsprem} we conclude from this characterization that: if $i:X\to Y$ is a smooth embedding of manifolds with corners, then a map $f:Z\to X$ from another manifold with corners into $X$ is smooth if and only if $i\circ f$ is smooth.
\end{rem}
\begin{rem}\label{smstrembuniquerem} In view of Remark \ref{maxsmatlasdetbysmfunrem} and Remark \ref{charembofmanwithcornrem} it holds that: given a topological embedding $i:X\to Y$ from a topological space $X$ into a manifold with corners $Y$, there is at most one smooth structure with corners for $X$ that makes $i:X\to Y$ a smooth embedding. 
\end{rem}
\begin{defi}\label{embsubmanwithcorndefi} We call a subspace of a manifold with corners an \textbf{embedded submanifold} if it admits a (necessarily unique) smooth structure with corners that makes the inclusion map a smooth embedding. 
\end{defi}
\begin{ex}\label{stratmanwithcornex} Each of the subspaces:
\begin{equation*} X_k:=\{x\in X\mid \textrm{depth}_X(x)=k\}
\end{equation*}
is an embedded submanifold of $X$ without corners, with tangent space at $x\in X_k$ equal to $F_xX$. Their connected components -- called the \textbf{open faces} -- form a stratification of $X$. The open and dense subset $X_0$ of $X$ is the regular part of this stratification (that is, the union of all open strata). We usually denote $X_0$ as $\mathring{X}$.
\end{ex}
We now turn to submersions. Following \cite{NiWeXu,Ni}, we use the notion below. 
\begin{defi}\label{tamesubmdefi} By a \textbf{submersion} $f:X\to Y$ between manifolds with corners we mean a smooth map with the property that, for each $x\in X$, the differential $\d f_x:T_xX\to T_{f(x)}Y$ is surjective. Such a submersion is called \textbf{tame} if in addition, for each $x\in X$ it holds that:
\begin{equation*} \d f_x^{-1}(C_{f(x)}Y)=C_xX.
\end{equation*} In other words: if $v\in T_xX$ and $\d f_x(v)$ is an inward pointing, then $v$ is inward pointing.
\end{defi} Note here that tame submersions are simply called submersions in \cite{NiWeXu}.
\begin{ex}\label{prototamesubmex} The prototypical example of a tame submersion is the following. Let $f:M\to N$ be a submersion between smooth manifolds without corners and let $Z$ be an embedded submanifold with corners in $N$. Then it follows from the submersion theorem that the pre-image $f^{-1}(Z)$ is an embedded submanifold with corners in $M$, with tangent space and tangent cone: 
\begin{equation*} T_xf^{-1}(Z)=\d f_x^{-1}(T_{f(x)}Z)\quad \& \quad C_xf^{-1}(Z)=\d f_x^{-1}(C_{f(x)}Z), \quad\quad x\in f^{-1}(Z),
\end{equation*} and the restriction $f:f^{-1}(Z)\to Z$ is a tame submersion.
\end{ex}
The following shows that tame submersions behave much like submersions between manifolds without corners.
\begin{prop}\label{fundproptamesubm} Let $f:X\to Y$ be a submersion between manifolds with corners. The following are equivalent.
\begin{itemize}\item[a)] The submersion $f$ is tame.
\item[b)] The submersion $f$ preserves depth. That is, for each $x\in X$ it holds that: 
\begin{equation*} \textrm{depth}_Y(f(x))=\textrm{depth}_X(x).
\end{equation*} 
\item[c)] For each $x\in X$, there is a chart $(U,\chi)$ around $x$ onto an open in $\R^n_k$ that sends $x$ to the origin and there is a chart $(V,\phi)$ around $f(x)$ onto an open in $\R^m_k$ that sends $f(x)$ to the origin, such that $f(U)=V$ and the coordinate representation of $f$ is:
\begin{equation*} \phi\circ f\circ \chi^{-1}: \chi(U)\to \phi(V),\quad (x_1,...,x_n)\mapsto (x_1,...,x_m).
\end{equation*} Here $n=\dim(X)$, $m=\dim(Y)$ and $k$ is the depth of $x$ and $f(x)$. 
\end{itemize}
\end{prop}
\begin{proof} For the implication from $a$ to $b$ notice that, since $f$ is tame, the linear subspace $\d f_x^{-1}(F_{f(x)}Y)$ of $T_xX$ is contained in $C_xX$ and hence it must be contained in the largest such subspace $F_xX$. Combining this with (\ref{smmapconrel}), we conclude that $\d f_x^{-1}(F_{f(x)}Y)=F_xX$. Using this and surjectivity of $\d f_x$, it follows from the rank-nullity theorem that $\textrm{depth}_Y(f(x))=\textrm{depth}_X(x)$. \\

For the implication from $b$ to $c$, we essentially follow the proof of \cite[Lemma 1.3]{Ni}. To this end, suppose that $f$ preserves depth and let $x\in X$. Since $\textrm{depth}_X(x)=\textrm{depth}_Y(f(x))$, we can assume (after fixing appropriate charts with corners for $X$ and $Y$ around $x$ and $f(x)$) that $f$ is a depth-preserving submersion between opens $U$ in $\R^n_k$ and $V$ in $\R^m_k$ around the respective origins, and that $x$ and $f(x)$ are the respective origins. After possibly further shrinking $U$, we can assume that there is an open $W$ in $\R^n$ and a smooth map $\widehat{f}:W\to \R^m$ such that $U=W\cap \R^n_k$ and $\widehat{f}$ coincides with $f$ on $U$. Then the differential of $\widehat{f}$ at the origin coincides with that of $f$, so that it is surjective. Hence, by the submersion theorem for maps between opens in Euclidean spaces, we can (after possibly shrinking $U$ and $W$) further arrange for there to be a diffeomorphism $\chi:W\to ]-\varepsilon,\varepsilon[^n$, for some $\varepsilon>0$, that maps the origin to the origin and is such that $\widehat{f}\circ \chi^{-1}:]-\varepsilon,\varepsilon[^n\to \R^m$ is the projection onto the first $m$ coordinates. To prove part $b$, it is now enough to show that $\chi(U)=]-\varepsilon,\varepsilon[^n\cap \R^n_k$. The inclusion from left to right is clear. For the other inclusion, suppose that $y\in ]-\varepsilon,\varepsilon[^n\cap \R^n_k$. To show that $\chi^{-1}(y) \in \R^n_k$, we fix an $\widetilde{x}\in \mathring{U}=U\cap\mathring{\R}^n_k$ (meaning that its first $k$ components are strictly positive; cf. Rem \ref{stratmanwithcornex}) and we show that the curve: 
\begin{equation*} \gamma:[0,1]\to W,\quad \gamma(t)=\chi^{-1}((1-t)\chi(\widetilde{x})+ty).
\end{equation*} takes values in $\R^n_k$. For this, by continuity of $\gamma$ it is enough to show that $\gamma(t)\in \R^n_k$ for all $t\in [0,1[$. The set of $t\in [0,1[$ such that $\gamma(t)\in \R^n_k$ is non-empty (for $\gamma(0)\in \R^n_k$) and is clearly closed in $[0,1[$. Next, we will show that for $t\in [0,1[$: $\gamma(t)\in \R^n_k$ if and only if $\gamma(t)\in \mathring{\R}^n_k$, so that the set of $t\in [0,1[$ such that $\gamma(t)\in \R^n_k$ is open in $[0,1[$ as well ($\mathring{\R}^n_k$ being open in $\R^n$) and hence it must be all of $[0,1[$, by connectedness. To this end, let $t\in [0,1[$ such that $\gamma(t)\in \R^n_k$. Then $\gamma(t)\in U$, and so: 
\begin{equation*} f(\gamma(t))=\widehat{f}(\gamma(t))=(1-t)(\chi^1(\widetilde{x}),...,\chi^m(\widetilde{x}))+t(y_1,...,y_m). 
\end{equation*} Since $f$ preserves depth, the first $k$ components of $(\chi^1(\widetilde{x}),...,\chi^m(\widetilde{x}))=f(\widetilde{x})$ are strictly positive. So, because $y\in \R^n_k$, the first $k$ components of $f(\gamma(t))$ are strictly positive as well. Therefore, the depth of $f(\gamma(t))$ in $V$ is zero, hence so is the depth of $\gamma(t)$ in $U$, meaning that $\gamma(t)$ indeed belongs to $\mathring{\R}^n_k$. This proves that $b$ implies $c$. Furthermore, the implication from $c$ to $a$ is clear and so we conclude that $a$, $b$ and $c$ are indeed equivalent. 
\end{proof}
\begin{cor}\label{fundproptamesubmcor} Every tame submersion $f:X\to Y$ has the following properties.
\begin{itemize}
\item[a)] The map $f:X\to Y$ is open and for every $x\in X$ there is a smooth local section of $f$, defined on an open around $f(x)$ in $Y$, that maps $f(x)$ to $x$.
\item[b)] For each $y\in Y$ the fiber $f^{-1}(y)$ is an embedded submanifold of $X$ without corners, with tangent space:
\begin{equation*} T_xf^{-1}(y)=\ker(\d f_x).
\end{equation*} 
\item[c)] For every smooth map $g:Z\to Y$ from another manifold with corners into $Y$, the set-theoretic fiber product $X\times_YZ$ is an embedded submanifold with corners of $X\times Z$, with tangent space:
\begin{equation*} T_{(x,z)}(X\times_YZ)=\{(v,w)\in T_xX\times T_zZ\mid \d f_x(v)=\d g_z(w)\},
\end{equation*} and tangent cone:
\begin{align*} C_{(x,z)}(X\times_YZ)&=\{(v,w)\in C_xX\times C_zZ\mid \d f_x(v)=\d g_z(w)\}\\
&=\{(v,w)\in T_xX\times C_zZ\mid \d f_x(v)=\d g_z(w)\}.
\end{align*} The analogous statement holds when interchanging the roles of $X$ and $Z$. 
\end{itemize}
\end{cor} 
\begin{proof} 
We leave it to the reader to derive property $a$ from Proposition \ref{fundproptamesubm} and $b$ from $c$. To verify property $c$, let $g:Z\to Y$ be a smooth map from another manifold with corners into $Y$. By Remark \ref{smstrembuniquerem} it is enough to show that every point $(x,z)\in X\times_YZ$ admits an open neighbourhood in $X\times_YZ$ that is an embedded submanifold of $X\times Z$, with the prescribed tangent space and tangent cone at $(x,z)$. To this end, let $(x,z)\in X\times_YZ$ and for this $x\in X$ consider charts $(U,\chi)$ and $(V,\phi)$ as in Proposition \ref{fundproptamesubm}, such that: 
\begin{align*} \chi(U)&=[0,\varepsilon[^k\times]-\varepsilon,\varepsilon[^{m-k}\times ]-\varepsilon,\varepsilon[^{n-m},\\
 \phi(V)&=[0,\varepsilon[^k\times]-\varepsilon,\varepsilon[^{m-k},
\end{align*} for some $\varepsilon >0$. Then $U\times_Yg^{-1}(V)$ is an open in $X\times_YZ$ around $(x,z)$ and the map:
\begin{equation*} i:]-\varepsilon,\varepsilon[^{n-m}\times g^{-1}(V)\to X\times Z,\quad (p,q)\mapsto (\chi^{-1}(\phi(g(q)),p),q)
\end{equation*} is a smooth embedding (it is essentially a graph map, cf. Example \ref{graphembmanwithcorn}) with image $U\times_Yg^{-1}(V)$. Therefore $U\times_Yg^{-1}(V)$ is an embedded submanifold of $X\times Z$, with tangent space:
\begin{align*} T_{(x,z)}(U\times_Yg^{-1}(V))&=\d i_{(0,z)}(T_0\R^{n-m} \oplus T_{z}Z),\\
&=\{(v,w)\in T_{x}X\times T_{z}Z\mid \d f_{x}(v)=\d g_{z}(w)\}
\end{align*} and tangent cone:
\begin{align*}
C_{(x,z)}(U\times_Yg^{-1}(V))&=\d i_{(0,z)}(T_0\R^{n-m} \oplus C_{z}Z),\\
&=\{(v,w)\in C_{x}X\times C_{z}Z\mid \d f_{x}(v)=\d g_{z}(w)\},\\
&=\{(v,w)\in T_{x}X\times C_{z}Z\mid \d f_{x}(v)=\d g_{z}(w)\},
\end{align*} where the last equality follows from (\ref{smmapconrel}) (applied to $g$) and the equality $\d f_{x}^{-1}(C_{f(x)}Y)=C_{x}X$. This concludes the proof of property $c$.
\end{proof}
\begin{ex} Unlike for manifolds without corners, for the conclusion of Proposition \ref{fundproptamesubm}$c$ to hold at a given point $x\in X$, it is not enough to require the conditions in Definition \ref{tamesubmdefi} just at that point. To see this, consider for instance the map: 
\begin{equation*} f:[0,\infty[^2\to [0,\infty[^2,\quad (x,y)\mapsto (x,y+x^2).
\end{equation*} This is a submersion that satisfies the tameness condition in Definition \ref{tamesubmdefi} at the origin, but does not satisfy the conclusion of Proposition \ref{fundproptamesubm}$c$ there.
\end{ex}
Now, we turn to Lie groupoids with corners.
\begin{defi}[\cite{NiWeXu}]\label{liegpwithcorndef} A (Hausdorff) \textbf{Lie groupoid with corners} $\G\rightrightarrows X$ is a groupoid for which $\G$ and $X$ are manifolds with corners, the source and target map are tame submersions and all structure maps are smooth as maps between manifolds with corners.
\end{defi} 
Note here that, because the source and target of $\G$ are tame submersions, by Corollary \ref{fundproptamesubmcor}$c$ the space of composable arrows $\G^{(2)}$ is an embedded submanifold with corners of $\G\times \G$, so that (as usual) the requirement for the multiplication map of $\G$ to be smooth makes sense. 
\begin{rem} Let $\G\rightrightarrows X$ be a Lie groupoid with corners. In view of Proposition \ref{fundproptamesubm}, for each integer $0\leq k\leq \dim (X)$, the embedded submanifold $\G_k$ of $\G$ (as in Example \ref{stratmanwithcornex}) coincides with both $s^{-1}(X_k)$ and $t^{-1}(X_k)$, so that $X_k$ is $\G$-invariant and the structure maps of $\G$ restrict to give $\G_k$ the structure of Lie groupoid without corners over $X_k$. From this and the standard theory of Lie groupoids without corners we conclude that for each $x\in X$ the following hold. 
\begin{itemize}\item[a)] The isotropy group $\G_x$ of $\G$ is an embedded submanifold of $\G$ without corners and, as such, it is a Lie group.
\item[b)] The source-fiber $s^{-1}(x)$ is an embedded submanifold of $\G$ without corners, and the {orbit} $\O_x$ is an initial submanifold of $X_k$ without corners, for $k=\textrm{depth}_X(x)$, with smooth manifold structure uniquely determined by the fact that:
\begin{equation*} t:s^{-1}(x)\to \O_x,
\end{equation*} is a (right) principal $\G_x$-bundle. 
\end{itemize}
\end{rem}
As in the case without corners, we can define Morita equivalences. 
\begin{defi}\label{def:moreqLiegpoidswithcorners}
Let $\G_1\rightrightarrows X_1$ and $\G_2\rightrightarrows X_2$ be Lie groupoid with corners. A \textbf{Morita equivalence} from $\G_1$ to $\G_2$ is a principal $(\G_1,\G_2)$-bibundle $(P,\alpha_1,\alpha_2)$. This consists of:
\begin{itemize} \item A manifold with corners $P$ with two surjective tame submersions $\alpha_i:P\to X_i$.
\item A smooth left action of $\G_1$ along $\alpha_1$ that is free and the orbits of which coincide with the $\alpha_2$-fibers.
\item A smooth right action of $\G_2$ along $\alpha_2$ that is free and the orbits of which coincide with the $\alpha_1$-fibers.
\end{itemize} Furthermore, the two actions are required to commute.
\end{defi} Here, smoothness of the actions means that the action maps $\G_1\times_{X_1}P\to P$ and $P\times_{X_2}\G_2\to P$ are smooth as maps between manifolds with corners (which makes sense in view of Corollary \ref{fundproptamesubmcor}$c$). As in the case without corners, the following holds.
\begin{prop}\label{altdefprincbunliegpoidwithcorn} The respective conditions on the left and right action above are equivalent to the requirement that the respective maps:
\begin{align} \G_1\times_{X_1}P&\to P\times_{X_2}P,\quad (g,p)\mapsto (g\cdot p,p), \label{moreqleftprinbunmap}\\
P\times_{X_2}\G_2&\to P\times_{X_1}P,\quad (p,g)\mapsto (p,p\cdot g), \label{moreqrightprinbunmap}
\end{align} are well-defined diffeomorphisms of manifolds with corners. 
\end{prop} 
\begin{proof} The respective conditions on the left and right action above mean that the respective maps (\ref{moreqleftprinbunmap}) and (\ref{moreqrightprinbunmap}) are well-defined, smooth and bijective. So, we ought to show that if (\ref{moreqleftprinbunmap}) respectively (\ref{moreqrightprinbunmap}) is a smooth bijection, then it is in fact a diffeomorphism. In view of Corollary \ref{fundproptamesubmcor}$a$ it is enough to show for each of the respective maps that if it is a smooth bijection, then at every point its differential is bijective and it is tame. Bijectivity of the differential follows as in the case without corners and tameness is immediate from the second description of the tangent cone in Corollary \ref{fundproptamesubmcor}$c$. 
\end{proof}
As in the case without corners, Morita equivalences can be adapted to the (pre-)symplectic setting \cite{Xu,Xu1,BuCrWeZh}. Explicitly:
\begin{defi}\label{sympgpwithcorndef} A pre-symplectic form $\omega$ on a manifold with corners $P$ is a closed differential $2$-form. Such a pre-symplectic form is called \textbf{symplectic} if it is non-degenerate. We call $(P,\omega)$ a \textbf{(pre-)symplectic manifold with corners}. A \textbf{pre-symplectic groupoid with corners} $(\G,\Omega)\rightrightarrows X$ is a Lie groupoid with corners equipped with a pre-symplectic form $\Omega$ on $\G$, satisfying:
\begin{itemize}\item[i)] $\dim(\G)=2\dim(X)$,
\item[ii)] the form $\Omega$ is \textbf{multiplicative}, in the sense that:
\begin{equation*} m^*\Omega=(\textrm{pr}_1)^*\Omega+(\textrm{pr}_2)^*\Omega, 
\end{equation*} where we denote by: 
\begin{equation*} m,\textrm{pr}_1,\textrm{pr}_2:\G^{(2)}\to \G
\end{equation*} the multiplication and projection maps from the space of composable arrows $\G^{(2)}$ to $\G$, 
\item[iii)] the form $\Omega$ satisfies the \textbf{non-degeneracy} condition:
\begin{equation*} \ker(\Omega)_{1_x}\cap \ker(\d s)_{1_x}\cap \ker(\d t)_{1_x}=0,\quad \forall x\in X. 
\end{equation*}
\end{itemize} This is called a \textbf{symplectic groupoid with corners} if $\Omega$ is symplectic. A \textbf{Morita equivalence $(P,\omega_P,\alpha_1,\alpha_2)$ between (pre-)symplectic groupoids with corners} $(\G_1,\Omega_1)\rightrightarrows X_1$ and $(\G_2,\Omega_2)\rightrightarrows X_2$ consists of:
\begin{itemize}\item a (pre-)symplectic manifold with corners $(P,\omega_P)$, 
\item a Morita equivalence $(P,\alpha_1,\alpha_2)$ between the underlying Lie groupoids with corners, with the additional property that both actions are Hamiltonian, in the sense that:
\begin{equation*} (m^L_P)^*\omega_P=(\textrm{pr}_{\G_1})^*\Omega_1+(\textrm{pr}_{P}^L)^*\omega_P \quad \& \quad (m_{P}^R)^*\omega_P=(\textrm{pr}_{\G_2})^*\Omega_2+(\textrm{pr}_{P}^R)^*\omega_P
\end{equation*} where we denote by:
\begin{align*} &m_{P}^L,\textrm{pr}_{P}^L:\G_1\times_{X_1}P\to P, \quad \textrm{pr}_{\G_1}:\G_1\times_{X_1}P\to \G_1,\\ 
&m_{P}^R,\textrm{pr}_{P}^R:P\times_{X_2}\G_2\to P,\quad \textrm{pr}_{\G_2}:P\times_{X_2}\G_2\to \G_2,
\end{align*} the maps defining the actions and the projections onto $P$, $\G_1$ and $\G_2$. 
\end{itemize}
\end{defi}
\begin{rem}\label{presympmoreqissymprem} A pre-symplectic Morita equivalence $((P,\omega_P),\alpha_1,\alpha_2)$ between two symplectic groupoids with corners is automatically symplectic. 
\end{rem}
\begin{prop}\label{moreqwithcorniseqrelprop} Morita equivalence between Lie or (pre-)symplectic groupoids with corners is an equivalence relation. 
\end{prop}
\begin{proof} This proposition follows from the observation that, as in the case without corners, for Lie and pre-symplectic groupoids with corners there is the identity Morita equivalence and Morita equivalences can be inverted and composed. The only extra technicality arising here is in the construction of composition of two Morita equivalences, for it involves quotients by actions of Lie groupoids with corners. To be more precise, suppose that we are given two Morita equivalences between Lie groupoids with corners:
\begin{center}
\begin{tikzpicture} 
\node (H1) at (-3.1,0) {$\G_1$};
\node (S1) at (-3.1,-1.3) {$X_1$};
\node (Q) at (-1.35,0) {$P$};
\node (S2) at (0.4,-1.3) {$X_2$};
\node (H2) at (0.4,0) {$\G_2$};

\node (G1) at (2,0) {$\G_2$};
\node (M1) at (2,-1.3) {$X_2$};
\node (P) at (3.75,0) {$Q$};
\node (M2) at (5.5,-1.3) {$X_3$};
\node (G2) at (5.5,0) {$\G_3$};
 
\draw[->,transform canvas={xshift=-\shift}](H1) to node[midway,left] {}(S1);
\draw[->,transform canvas={xshift=\shift}](H1) to node[midway,right] {}(S1);
\draw[->,transform canvas={xshift=-\shift}](H2) to node[midway,left] {}(S2);
\draw[->,transform canvas={xshift=\shift}](H2) to node[midway,right] {}(S2);
\draw[->](Q) to node[pos=0.25, below] {$\quad\text{ }\alpha_1$} (S1);
\draw[->] (-2,-0.15) arc (315:30:0.25cm);
\draw[<-] (-0.7,0.15) arc (145:-145:0.25cm);
\draw[->](Q) to node[pos=0.25, below] {$\alpha_2$\text{ }} (S2); 
 
\draw[->,transform canvas={xshift=-\shift}](G1) to node[midway,left] {}(M1);
\draw[->,transform canvas={xshift=\shift}](G1) to node[midway,right] {}(M1);
\draw[->,transform canvas={xshift=-\shift}](G2) to node[midway,left] {}(M2);
\draw[->,transform canvas={xshift=\shift}](G2) to node[midway,right] {}(M2);
\draw[->](P) to node[pos=0.25, below] {$\quad\text{ }\beta_2$} (M1);
\draw[->] (3.1,-0.15) arc (315:30:0.25cm);
\draw[<-] (4.4, 0.15) arc (145:-145:0.25cm);
\draw[->](P) to node[pos=0.25, below] {$\beta_3$\text{ }} (M2);
\end{tikzpicture} 
\end{center} Consider the induced left anti-diagonal $\G_2$-action along $\alpha_2\circ \textrm{pr}_P:P\times_{X_2}Q\to X_2$. We will show that the topological quotient space (which is second countable and Hausdorff):
\begin{equation}\label{compmoreqbun} P\ast_{\G_2} Q:=\frac{P\times_{X_2} Q}{\G_2}
\end{equation} admits a unique smooth structure with corners with respect to which the quotient map:
\begin{equation}\label{quotmapmoreqcomp} P\times_{X_2}Q\to P\ast_{\G_2} Q
\end{equation} is a tame submersion. As for Lie groupoids without corners, one can then define the composite Morita equivalence:
\begin{center}
\begin{tikzpicture} \node (G1) at (-1.5,0) {$\G_1$};
\node (M1) at (-1.5,-1.3) {$X_1$};
\node (S) at (1.4,0) {$P\ast_{\G_2} Q$};
\node (M2) at (4.3,-1.3) {$X_3$};
\node (G2) at (4.3,0) {$\G_3$};
 
\draw[->,transform canvas={xshift=-\shift}](G1) to node[midway,left] {}(M1);
\draw[->,transform canvas={xshift=\shift}](G1) to node[midway,right] {}(M1);
\draw[->,transform canvas={xshift=-\shift}](G2) to node[midway,left] {}(M2);
\draw[->,transform canvas={xshift=\shift}](G2) to node[midway,right] {}(M2);
\draw[->](S) to node[pos=0.5, below] {$\quad\quad\quad{ }\text{ }\underline{\alpha}_1\circ \underline{\textrm{pr}}_P$} (M1);
\draw[->] (0.05,-0.15) arc (315:30:0.25cm);
\draw[<-] (2.65,0.15) arc (145:-145:0.25cm);
\draw[->](S) to node[pos=0.35, below] {$\underline{\beta}_3\circ \underline{\textrm{pr}}_Q$\quad\text{ }} (M2);
\end{tikzpicture}
\end{center} 
Moreover, if the given groupoids and Morita equivalences are (pre-)symplectic with corners, with (pre-)symplectic forms $\omega_P$ and $\omega_Q$, then as for (pre-)symplectic groupoids without corners (see \cite{Xu,Xu1}) the form $\omega_{P}\oplus\omega_{Q}$ on $P\times_{X_2}Q$ descends to a (pre-)symplectic form on $P\ast_{\G_2}Q$ that makes the composite Morita equivalence (pre-)symplectic. \\

To prove that (\ref{compmoreqbun}) indeed admits a unique smooth structure with corners with respect to which (\ref{quotmapmoreqcomp}) is a tame submersion, we give an adaptation of the proof of \cite[Lemma B.1.4]{Vi} (in fact, we believe that Lie groupoids with corners fall into the much more general framework developed in that thesis). Uniqueness follows from Corollary \ref{fundproptamesubmcor}$a$. By this uniqueness, to prove existence it is enough to show that every $[p,q]\in P\ast_{\G_2} Q$ admits an open neighbourhood with a smooth structure with corners, with respect to which the restriction of (\ref{quotmapmoreqcomp}) is a tame submersion. To this end, let $[p,q]\in P\ast_{\G_2} Q$. By Corollary \ref{fundproptamesubmcor}$a$ there is a smooth local section $\sigma:U\to Q$ of $\beta_3$, defined on an open neighbourhood $U$ of $\beta_3(q)$ in $X_3$, that maps $\beta_3(q)$ to $q$. This induces a diffeomorphism:
\begin{equation*} \Phi_\sigma:\G_2\times_{X_2}U \to \beta_3^{-1}(U), \quad (g,x)\mapsto g\cdot \sigma(x),
\end{equation*} where the fiber-product $\G_2\times_{X_2}U$ is taken with respect to the source-map of $\G_2$ and the map $\beta_2\circ \sigma:U\to X_2$. To see that the inverse of this map is indeed smooth, consider the division map:
\begin{equation*}  Q\times_{M_3} Q\to \G_2, \quad (q_1,q_2)\mapsto [q_1:q_2],
\end{equation*} that assigns the unique element of $\G_2$ satisfying $[q_1:q_2]\cdot q_2=q_1$. This is smooth, as a consequence of Proposition \ref{altdefprincbunliegpoidwithcorn}. Therefore, so is $\Phi_{\sigma}^{-1}$, for it is given by:
\begin{equation*} \Phi_\sigma^{-1}:\beta_3^{-1}(U)\to \G_2\times_{X_2}U,\quad q\mapsto ([q:(\sigma\circ\beta_3)(q)],\beta_3(q)).
\end{equation*}  
Now consider the composition:
\begin{equation}\label{locsmstrquotmapmoreqcomp0} P\times_{X_2}\beta_3^{-1}(U)\xrightarrow{\textrm{Id}_P\times\Phi_\sigma^{-1}} P\times_{X_2}\G_2\times_{X_2}U \xrightarrow{(\ref{moreqrightprinbunmap})\times\textrm{Id}_U} P\times_{X_1}P\times_{X_2} U\xrightarrow{\textrm{pr}_{P,2}\times\textrm{Id}_U} P\times_{X_2}U,
\end{equation} where the last fiber-product is taken with respect to $\alpha_2:P\to X_2$ and $\beta_2\circ \sigma:U\to X_2$. This being a composition of two diffeomorphisms and a tame submersion, the composite (\ref{locsmstrquotmapmoreqcomp0}) is a tame submersion. It factors through a homeomorphism from $P\ast_{\G_2}\beta_3^{-1}(U)$ to $P\times_{X_2}U$, with inverse:
\begin{equation}\label{locsmstrquotmapmoreqcomp} P\times_{X_2}U\to P\ast_{\G_2}\beta_3^{-1}(U), \quad (p,x)\mapsto [p,\sigma(x)].
\end{equation} Since the composite (\ref{locsmstrquotmapmoreqcomp0}) is a tame submersion, this homeomorphism induces a smooth structure with corners on the open $P\ast_{\G_2}\beta_3^{-1}(U)$ around $[p,q]$, with respect to which the restriction of (\ref{quotmapmoreqcomp}) to $P\times_{X_2}\beta_3^{-1}(U)$ is a tame submersion, as was to be constructed. \end{proof}
\newpage
\bibliographystyle{plain}

\bibliography{ref}
\Addresses
\end{document}